\documentclass[letterpaper,leqno,10pt,twoside]{amsart}
\usepackage{amsmath,amsthm,amsfonts,amssymb,euscript,mathrsfs,graphics,color,latexsym,marginnote}
\usepackage[dvips]{graphicx}
\usepackage[left=0.75 in, right=0.75 in,top=0.8 in, bottom=0.8 in]{geometry}
\usepackage{hyperref}
\usepackage[toc,page]{appendix}
\usepackage{relsize}
\usepackage[shortlabels]{enumitem}
\usepackage[toc,page]{appendix}

\usepackage[dvipsnames]{xcolor}

\usepackage{hyperref}
\hypersetup{colorlinks=true, pdfstartview=FitV,linkcolor=blue!70!black,citecolor=red!70!black, urlcolor=green!60!black}
\definecolor{labelkey}{rgb}{0.6,0,0} 

\newtheorem{theorem}{Theorem}[section]
\newtheorem{lemma}[theorem]{Lemma}
\newtheorem{proposition}[theorem]{Proposition}
\newtheorem{corollary}[theorem]{Corollary}

\newtheorem{remark}[theorem]{Remark}
\newtheorem{definition}[theorem]{Definition}

\makeatletter
\renewcommand \theequation {%
\ifnum \c@section>\z@ \@arabic\c@section.
\fi\@arabic\c@equation} \@addtoreset{equation}{section}
\makeatother

\DeclareMathOperator*{\esssup}{ess\,sup}

\providecommand{\abs}[1]{\left\vert#1\right\vert}
\providecommand{\babs}[1]{\big\vert#1\big\vert}
\providecommand{\Babs}[1]{\Big\vert#1\Big\vert}
\providecommand{\nm}[1]{\left\Vert#1\right\Vert}

\providecommand{\Bnm}[1]{\Big\Vert#1\Big\Vert}
\providecommand{\BBnm}[1]{\bigg\Vert#1\bigg\Vert}
\providecommand{\nnm}[1]{\left\vert\kern-0.25ex\left\vert\kern-0.25ex\left\vert#1\right\vert\kern-0.25ex\right\vert\kern-0.25ex\right\vert}

\providecommand{\br}[1]{\left\langle #1 \right\rangle}
\providecommand{\bbr}[1]{\Big\langle #1 \Big\rangle}
\providecommand{\brx}[1]{\left\langle #1 \right\rangle_x}
\providecommand{\bbrx}[1]{\Big\langle #1 \Big\rangle_x}
\providecommand{\brv}[1]{\left\langle #1 \right\rangle_v}
\providecommand{\bbrv}[1]{\Big\langle #1 \Big\rangle_v}
\providecommand{\brr}[1]{\left( #1 \right)_{\mathcal{L}}}
\providecommand{\bbrr}[1]{\Big( #1 \Big)_{\mathcal{L}}}
\providecommand{\brb}[2]{\left\langle #1 \right\rangle_{#2}}
\providecommand{\bbrb}[2]{\Big\langle #1 \Big\rangle_{#2}}

\providecommand{\xnm}[1]{\left\Vert#1\right\Vert_{X}}

\providecommand{\xxnm}[1]{\left\Vert#1\right\Vert_{\widetilde{X}}}
\providecommand{\bnm}[1]{\left\Vert#1\right\Vert_{\text{BV}}}

\providecommand{\pnm}[2]{\left\Vert#1\right\Vert_{L^{#2}}}
\providecommand{\tnm}[1]{\left\Vert#1\right\Vert_{L^{2}}}
\providecommand{\lnm}[1]{\left\Vert#1\right\Vert_{L^{\infty}}}
\providecommand{\lnmm}[1]{\left\Vert#1\right\Vert_{L^{\infty}_{\varrho,\vartheta}}}
\providecommand{\um}[1]{\left\Vert#1\right\Vert_{L^2_{\nu}}}

\providecommand{\pnms}[3]{\left\vert#1\right\vert_{L^{#2}_{#3}}}
\providecommand{\tnms}[2]{\left\vert#1\right\vert_{L^{2}_{#2}}}
\providecommand{\lnms}[2]{\left\vert#1\right\vert_{L^{\infty}_{#2}}}
\providecommand{\lnmms}[2]{\left\vert#1\right\vert_{L^{\infty}_{#2,\varrho,\vartheta}}}

\providecommand{\lnmmss}[1]{\left\vert#1\right\vert_{L^{\infty}_{\varrho,\vartheta}}}

\providecommand{\jnm}[1]{\left\Vert#1\right\Vert_{L^{\N}}}

\providecommand{\knm}[1]{\left\Vert#1\right\Vert_{L^{\frac{\N}{\N-1}}}}

\providecommand{\jnms}[2]{\left\vert#1\right\vert_{L^{\frac{2\N}{3}}_{#2}}}
\providecommand{\knms}[2]{\left\vert#1\right\vert_{L^{\frac{2\N}{2\N-3}}_{#2}}}

\def\ud{\mathrm{d}}

\def\p{\partial}
\def\ls{\lesssim}
\def\gs{\gtrsim}
\def\half{\frac{1}{2}}
\def\rt{\rightarrow}
\def\r{\mathbb{R}}
\def\no{\nonumber}
\def\ue{\mathrm{e}}

\def\ds{\displaystyle}

\def\oo{o(1)}
\def\oot{o_{\tq}}
\def\cc{o}

\def\id{\textbf{1}}
\def\od{\textbf{0}}

\def\N{r}
\def\NN{s}
\def\P{p}

\def\fs{\mathfrak{F}}
\def\f{f}
\def\fb{f^{B}}

\def\pp{\mathcal{P}_{\gamma}}

\def\e{\varepsilon}
\def\al{\alpha}
\def\s{\mathbb{S}}

\def\vx{x}
\def\vv{v}
\def\vuu{{\mathfrak{u}}}
\def\vvv{{\mathfrak{v}}}
\def\nx{\nabla_{x}}
\def\dx{\Delta_x}
\def\d{\delta}
\def\vn{n}

\def\k{\kappa}
\def\re{R}
\def\ss{S}

\def\hk{\mathscr{H}}
\def\hh{\mathcal{H}}
\def\si{\sigma}
\def\nk{\mathcal{N}}
\def\nnk{\mathcal{N}^{\perp}}

\def\sb{\overline{\ss}}
\def\sp{\ss_6}
\def\ch{\overline{\chi}}
\def\sss{\mathcal{S}}
\def\sx{\ss_{3a}}
\def\sy{\ss_{3b}}
\def\sz{\ss_{3c}}

\def\ff{\mathfrak{F}_{\text{a}}}
\def\ffe{\widetilde{\mathfrak{F}}_{\text{a}}}

\def\m{\mu}
\def\mh{\m^{\frac{1}{2}}}
\def\mhh{\m^{-\frac{1}{2}}}
\def\mb{\mu_{w}}
\def\mbh{\mb^{\frac{1}{2}}}

\def\ms{M_{w}}
\def\mss{m_w}

\def\rq{\rho}
\def\uq{u}
\def\tq{T}

\def\tb{T_{w}}

\def\lc{\mathcal{L}}
\def\llc{\lc^1}
\def\li{\lc^{-1}}
\def\pk{\textbf{P}}
\def\bpk{\overline{\textbf{P}}}
\def\ik{\textbf{I}}

\def\a{\mathscr{A}}
\def\ab{\overline{\mathscr{A}}}
\def\bb{\textbf{b}}
\def\bd{\textbf{d}}
\def\be{\textbf{e}}
\def\z{Z}

\def\b{\mathscr{B}}
\def\bbb{\overline\b}

\def\nn{\mathcal{V}}
\def\vh{w}
\def\tvh{\widetilde{w}}
\def\vth{\vartheta}
\def\vrh{\varrho}

\def\bv{\br{\vv}^{\vth}\ue^{\varrho\frac{\abs{\vv}^2}{2\tm}}}
\def\vo{\omega}

\def\gb{\mathcal{G}}

\def\gg{g}
\def\gbg{g^B}
\def\fbf{\mathfrak{g}}

\def\va{v_{\eta}}
\def\vb{v_{\phi}}
\def\vc{v_{\psi}}
\def\vr{\mathbf{r}}
\def\vt{\varsigma}
\def\kk{\kappa}
\def\mn{\mathfrak{n}}

\def\tm{T_M}
\def\mm{\mu_M}
\def\mmh{\mm^{\frac{1}{2}}}
\def\mmhh{\mm^{-\frac{1}{2}}}
\def\rem{\re_M}
\def\remm{\overline{\rem}}
\def\mmss{m_{M,w}}
\def\tvhh{\widetilde{w}_M}

\def\bbq{\mathfrak{B}}
\def\ppq{\mathfrak{p}}

\def\blf{\Phi}
\def\blff{\widetilde{\Phi}}
\def\blg{\Psi}
\def\blgg{\widetilde{\Psi}}

\def\to{\mathfrak{T}}
\def\hb{h^B}
\def\zf{\mathfrak{y}}

\def\g{G} 
\def\ggg{\mathfrak{G}}

\def\rr{\mathscr{R}}

\def\pl{L}
\def\hh{\mathfrak{h}}
\def\te{\theta}
\def\ggt{\ggg^{L,\te}}
\def\ggl{\ggg^L}
\def\wwt{\mathfrak{w}^{L,\te}}
\def\qqt{\mathfrak{q}^{L,\te}}

\def\za{z}
\def\hhh{\overline{h}}

\def\as{\abs{\nabla\tb}_{W^{3,\infty}}}
\def\test{\mathfrak{f}}

\def\vxx{\mathfrak{x}}

\begin{document}

\title{Ghost Effect from Boltzmann Theory}

\author[R. Esposito]{Raffaele Esposito}
\address[R. Esposito]{
   \newline\indent M\&MOCS - Universita' dell'Aquila}
\email{raff.esposito@gmail.com}

\author[Y. Guo]{Yan Guo}
\address[Y. Guo]{
   \newline\indent Division of Applied Mathematics, Brown University}
\email{yan\_guo@brown.edu}
\thanks{Y. Guo was supported by NSF Grant DMS-2106650.}

\author[R. Marra]{Rossana Marra}
\address[R. Marra]{
   \newline\indent Dipartimento di Fisica and Unit\`a INFN, Universit\`a di Roma Tor Vergata}
\email{marra@roma2.infn.it}
\thanks{R. Marra is supported by INFN}

\author[L. Wu]{Lei Wu}
\address[L. Wu]{
   \newline\indent Department of Mathematics, Lehigh University}
\email{lew218@lehigh.edu}
\thanks{L. Wu was supported by NSF Grant DMS-2104775.}

\date{}



\maketitle

\makeatletter
\renewcommand \theequation {%
0.%
\ifnum \c@section>\z@ \@arabic\c@section.%
\fi
% \ifnum\c@subsection>\z@\@arabic\c@subsection.%
%\fi\ifnum \c@subsubsection>\z@\@arabic\c@subsubsection.
% \fi
\@arabic\c@equation} \@addtoreset{equation}{section}
% \@addtoreset{equation}{subsection} 
\makeatother

\begin{abstract}
    Taking place naturally in a gas subject to a given wall temperature distribution (Maxwell \cite{Maxwell1879}), the ``ghost effect'' exhibits a rare kinetic effect beyond the prediction of classical fluid theory and Fourier law in such a classical problem in physics. As the Knudsen number $\e$ goes to zero, the finite variation of temperature in the bulk is determined by an $\e$ infinitesimal, ghost-like velocity field, created by a given \textit{finite} variation of the tangential wall temperature as predicted by Maxwell's slip boundary condition. Mathematically, such a finite variation leads to the presence of a severe $\e^{-1}$ singularity and a Knudsen layer approximation in the fundamental energy estimate. Neither difficulty is within the reach of any existing PDE theory on the steady Boltzmann equation in a general 3D bounded domain. Consequently, in spite of the discovery of such a ghost effect from temperature variation in as early as 1960’s, its mathematical validity has been a challenging and intriguing open question, causing confusion and suspicion.
    We settle this open question in affirmative if the temperature variation is small but finite, by developing a new $L^2-L^6-L^\infty$ framework with four major innovations: 1) a key $\a$-Hodge decomposition and its corresponding local $\a$-conservation law eliminate the severe $\e^{-1}$ bulk singularity, leading to a reduced energy estimate; 2) A surprising $\e^{\frac{1}{2}}$ gain in $L^2$ via momentum conservation and a dual Stokes solution; 3) the $\a$-conservation, energy conservation and a coupled dual Stokes-Poisson solution reduces to a $\e^{-\frac{1}{2}}$ boundary singularity; 4) a crucial construction of $\e$-cutoff boundary layer eliminates such boundary singularity via new Hardy and BV estimates.
\end{abstract}

\pagestyle{myheadings} \thispagestyle{plain} \markboth{ESPOSITO AND GUO AND MARRA AND WU}{GHOST EFFECT FROM BOLTZMANN THEORY}

\setcounter{tocdepth}{1}
\tableofcontents

%%%%%%%%%%%%%%%%%%%%%%%%%%%%%%%%%%%%%%%%%%%%%%%%%%%%%%%%%%%%%%%%%%%%%%%%
%%%%%%%%%%%%%%%%%%%%%%%%%%%%%%%%%%%%%%%%%%%%%%%%%%%%%%%%%%%%%%%%%%%%%%%%

\makeatletter
\renewcommand \theequation {%
\ifnum \c@section>\z@ \@arabic\c@section.%
%\fi \ifnum\c@subsection>\z@\@arabic\c@subsection.%
\fi\@arabic\c@equation} \@addtoreset{equation}{section}
%\@addtoreset{equation}{subsection}
\makeatother

%%%%%%%%%%%%%%%%%%%%%%%%%%%%%%%%%%%%%%%%%%%%%%%%%%%%%%%%%%%%%%%%%%%%%%%%
\section{Formulation and Introduction}
%%%%%%%%%%%%%%%%%%%%%%%%%%%%%%%%%%%%%%%%%%%%%%%%%%%%%%%%%%%%%%%%%%%%%%%%

%%%%%%%%%%%%%%%%%%%%%%%%%%%%%%%%%%%%%%%%%%%%%%%%%%%%%%%%%%%%%%%%%%%%%%%%
\subsection{Introduction to Ghost Effect}
%%%%%%%%%%%%%%%%%%%%%%%%%%%%%%%%%%%%%%%%%%%%%%%%%%%%%%%%%%%%%%%%%%%%%%%%

One of the important applications of the Boltzmann theory is to derive hydrodynamic (fluid) equations as  $\e\rt 0$, where $\e$ is   a dimensionless number called the Knudsen number (the ratio between the mean free path $\ell$ and the characteristic macroscopic length). There has been explosive literature on such hydrodynamic limits both from physical and mathematical standpoints. Almost all of the fluid equations obtained from such derivations are compatible with standard fluid theory, and the role of such limits can be viewed as further justification
rather than new discovery of fluid equations from the Boltzmann theory.

A rare exception occurs in the study of the fundamental hydrodynamic limit for a steady gas in a motionless bounded domain $\Omega$.
In his seminal paper \cite{Maxwell1879}, J. C. Maxwell introduces his kinetic formulation of boundary conditions to investigate such a fundamental problem in Physics, and proposes a slip boundary condition for a fluid flow, by assuming thermal equilibria (local Maxwellian) for the gas. 
In the case of diffuse-reflection Maxwell BC (with  the accommodation coefficient equal to $1$), his now famous slip boundary condition takes the form of 
\begin{align}\label{maxwell-boundary}
    u_{\iota}-u_w\simeq G\left(\frac{\p u_{\iota}}{\p n}-\frac{3\lambda}{2\rho_w T_w}\frac{\p^2\tq}{\p\iota\p n}\right)+\frac{3\lambda}{4\rho_w T_w}\frac{\p\tq_w}{\p\iota}.
\end{align}
Here $(\rho,u,T)$ is the fluid density, velocity and temperature with $(\rho_w,u_w,T_w)$ being its counterparts at the wall, $n$ is the normal direction and $\iota$ the tangential direction, 
and $\lambda$ the viscosity and $G$ the slip coefficient proportional to the mean free path $\e$. The second term is of order  $\e$ \cite{Sone2002}. 
This boundary condition also predicts and explains that a dilute gas slips from colder to hotter regions (the thermal transpiration or thermal creep phenomenon) at the boundary.
Maxwell's work has inspired further research in physics and engineering ever since, see recent works on micro fluids including constructions of thermal pumps \cite{
%Chernyak.Margilevskii.Porodnov.Suetin1975, Han.Muntz.Alexeenko.Young2007, Lockerby.Reese.Emerson.Barber2004, %Wang.Su.Zhang.Zhang.Zhang2020, 
Akhlaghi.Roohi.Stefanov2022} and references quoted therein. 

In the case of a motionless boundary $u_w=0$ with Maxwell diffuse BC, Maxwell's kinetic formulation leads to the following stationary Boltzmann equation in a bounded three-dimensional $C^3$ domain $\Omega\ni\vx=(x_1,x_2,x_3)$
with velocity $\vv=(v_1,v_2,v_3)\in\r^3$:
\begin{align}\label{large system-}
\vv\cdot\nx \fs=\e^{-1}Q[\fs,\fs],\quad \fs\big|_{\vv\cdot\vn<0}=\ms\displaystyle\int_{\vuu\cdot\vn>0}
    \fs(\vuu)\abs{\vuu\cdot\vn}\ud{\vuu},
\end{align}
 and its $\e\rt0$ behavior is one of the most classical and basic hydrodynamic problems in the kinetic theory.
Here $\fs$ is the phase space density,  
$Q$ is the hard-sphere collision operator
\begin{align}
Q[F,G]:=&\int_{\r^3}\int_{\s^2}q(\vo,\vv-\vuu)\Big(F(\vuu_{\ast})G(\vv_{\ast})-F(\vuu)G(\vv)\Big)\ud{\vo}\ud{\vuu},\label{QQ}
\end{align}
with $\vo\in\s^2$ a unit vector,  $\vuu_{\ast}:=\vuu+\vo\big((\vv-\vuu)\cdot\vo\big)$, $\vv_{\ast}:=\vv-\vo\big((\vv-\vuu)\cdot\vo\big)$, and the hard-sphere collision kernel $q(\vo,\vv-\vuu):=q_0\abs{\vo\cdot(\vv-\vuu)}$ for a positive constant $q_0$.

Moreover, the wall Maxwellian in the diffuse-reflection boundary condition is 
\begin{align}
    \ds\ms(\vx_0,\vv):=\frac{1}{2\pi\big(\tb(\vx_0)\big)^2}
\exp\left(-\frac{\abs{\vv}^2}{2\tb(\vx_0)}\right)
\end{align}
for $x_0\in\p\Omega$
where $\vn$ is the unit outward normal to the bounded domain $\Omega$
with a wall temperature 
\begin{align}
    \tb=1+O\big(\abs{\nabla\tb}_{L^{\infty}}\big)
\end{align}
satisfying
\begin{align}
\int_{\vv\cdot\vn(\vx_0)>0}\ms(\vx_0,\vv)\abs{\vv\cdot\vn(\vx_0)}\ud{\vv}=1.
\end{align}
In the absence of a velocity field of with non-zero Mach number inside a bounded gas, Sone, among others \cite{Sone2002, Sone2007} establishes a systematic Hilbert expansion 
in the bulk with careful and precise boundary layer corrections.

In the case of $\abs{\nabla\tb}=O(\e)$, 
an $L^6-L^{\infty}$ framework has
been developed in \cite{Esposito.Guo.Kim.Marra2013, Esposito.Guo.Kim.Marra2015} which leads to the validity of
the Hilbert expansion
\begin{align}\label{INSF limit}
    \fs\approx (2\pi)^{-\frac{3}{2}}\ue^{-\frac{\abs{v}^2}{2}}\bigg(1+\e\Big(\rq_1+\tq_1\frac{\abs{v}^2-3}{2}\Big)\bigg)
\end{align}
where $(2\pi)^{-\frac{3}{2}}\ue^{-\frac{\abs{v}^2}{2}}$ is a global Maxwellian and $\tq_1$ satisfies the celebrated Fourier law for the infinitesimal temperature variation $T_1$ of order $\e$: 
\begin{align}
    \Delta_x\tq_1=0
\end{align}
with $u_1=\od$. We note that in this case, Maxwell's slip 
condition \eqref{maxwell-boundary} is trivial at the $\e$ order. 

In contrast, for the more natural and general case of $\abs{\nabla\tb}_{L^{\infty}}=O(1)$, $\e^{0}$ order only results in $\nx\big(\rho T\big)\equiv 0$. Even though $\abs{\nabla\tb}_{L^{\infty}}=O(1)$ is within the regime of the compressible Euler limit, the vanishing of the velocity at $\e^{0}$ order requires further expansion to determine uniquely $\rho $ and $T$. 
This process (see details of our companion paper \cite{AA024}) leads to that
for any constant $P>0$
\begin{align}\label{expansion}
    \fs\approx \m+\e\bigg\{\m\bigg(\rq_1+\uq_1\cdot\vv+\tq_1\frac{\abs{v}^2-3T}{2}\bigg)-\mh\left(\a\cdot\frac{\nx\tq}{2\tq^2}\right)\bigg\}
\end{align}
where $\ds\m(\vx,\vv):=\frac{\rq(\vx)}{\big(2\pi\tq(\vx)\big)^{\frac{3}{2}}}
\exp\bigg(-\frac{\abs{\vv}^2}{2\tq(\vx)}\bigg)$, 
\begin{align}\label{final 22}
    &\ab:=v\cdot\left(\abs{v}^2-5T\right)\mh\in\r^3,\quad \a:=\li\left[\ab\right]\in\r^3,
\end{align}
where $\lc$ is defined in \eqref{att 11},
and $(\rq,\uq_1,\tq,\mathfrak{p})$ is determined by a Navier-Stokes-Fourier system with ``ghost" effect 
\begin{align}\label{fluid system-}
\left\{
\begin{array}{rcl}
    P&=&\rq\tq,\\\rule{0ex}{1.2em}
    \rq(\uq_1\cdot\nx\uq_1)+\nx \mathfrak{p}&=&\nx\cdot\left(\tau^{(1)}-\tau^{(2)}\right),\\\rule{0ex}{1.2em}
    \nx\cdot(\rq\uq_1)&=&0,\\\rule{0ex}{1.8em}
    \nx\cdot\left(\k\dfrac{\nx\tq}{2\tq^2}\right)&=&5P\big(\nx\cdot\uq_1\big),
    \end{array}
    \right.
\end{align}
with the boundary condition 
\begin{align}\label{boundary condition}
    \tq\Big|_{\p\Omega}=\tb,\quad \uq_1\Big|_{\p\Omega}:=\big(u_{1,\iota_1},u_{1,\iota_2},u_{1,n}\big)\Big|_{\p\Omega}=\Big(\beta_0\p_{\iota_1}\tb,\beta_0\p_{\iota_2}\tb,0\Big).
\end{align}
Here $\tau^{(1)}:=\lambda\left(\nx\uq_1+(\nx\uq_1)^t-\frac{2}{3}(\nx\cdot\uq_1)\id\right)$, $\tau^{(2)}:=\frac{\lambda^2}{P}\Big(K_1\big(\nx^2\tq-\frac{1}{3}\dx\tq\id\big)+\frac{K_2}{\tq}\big(\nx\tq\otimes\nx\tq-\frac{1}{3}\abs{\nx\tq}^2\id\big)\Big)$, $K_1$ and $K_2$ are positive constants,  $\lambda[T]>0$ is the viscosity coefficient, $\k[\tq]>0$ is the heat conductivity, $(u_{1,\iota_1},u_{1,\iota_2})$ are two tangential components and $u_{1,n}$ is the normal component of $u_1$, $\beta_0=\beta_0[\tb]$ is a function of $\tb$. 
The system \eqref{fluid system-} has been discovered by various researches as early as in 1960's. 
We refer to \cite{Kogan1958, Kogan.Galkin.Fridlender1976, Sone2002, Sone2007} for the stationary problem and to  \cite{Masi.Esposito.Lebowitz1989, Bobylev1995, Bardos.Levermore.Ukai.Yang2008} for the time dependent one.

Sone’s precise formula \eqref{boundary condition} confirms Maxwell's boundary condition \eqref{maxwell-boundary} at $O(\e)$ since the first term in \eqref{maxwell-boundary} are of the order $O(\e^2)$, and his $\beta_0$ formula matches the second term in \eqref{maxwell-boundary}:
\begin{align}
    \beta_0\simeq \text{const}\times \frac{\sqrt{\tq_w}}{\rq_w \tq_w}
\end{align}
which can be solved from the kinetic boundary layer (the Milne problem) \cite{Sone2007, Sone2002, AA024}. Sone's formula \eqref{boundary condition} serves as a bench mark for many numerical simulations, see \cite{Akhlaghi.Roohi.Stefanov2022}.
%\cite{Chernyak.Margilevskii.Porodnov.Suetin1975, Han.Muntz.Alexeenko.Young2007, Wang.Su.Zhang.Zhang.Zhang2020}.

Furthermore, Sone \cite{Sone2002, Sone2007}, among others (see also \cite{Kogan.Galkin.Fridlender1976}), makes an important and surprising observation that new limiting system \eqref{fluid system-} cannot be predicted by any classical fluid theory, which can be viewed as an exciting example of a new kinetic effect. Thanks to the mismatch of $\e$
orders, the infinitesimal first-order velocity $\e u_1$ acting like a ``ghost'' (a term introduced by Y. Sone \cite{Sone1972}), has an $O(1)$ impact in determining zeroth order $\rq$ and $\tq$, as long as the tangential
temperature variation is of $O(1)$ in \eqref{boundary condition}. Therefore,
\eqref{fluid system-} is not compatible with any classical continuum fluid theory, in
which all fluid quantities are determined
at the same level of order of $\e$.  In particular, $\tau^{(2)}$ is a new contribution different from the standard fluid theories.
Subsequently, different types of ghost effects have been discovered,
such as the ghost effect from curvature or %ghost effect
from mixture of gases \cite{Sone.Aoki.Doi1992, Sone.Aoki.Takata.Sugimoto.Bobylev1996}. In \cite{Esposito.Marra2023}, a formal derivation of ghost-like equations is
obtained starting from the Newtonian many particle system.

Due to fundamental mathematical difficulty created by $\abs{\nabla\tb}_{L^{\infty}}=O(1)$, it has remained an intriguing outstanding question to rigorously derive \eqref{fluid system-}. 
Particularly, the presence of new term $\mh\left(\a\cdot\frac{\nx\tq}{2\tq^2}\right)$ in \eqref{expansion} leads to a boundary layer correction of order $\e$, which can not be avoided in mathematical analysis.
The main goal of this paper is to settle this open question in affirmative in a bounded domain under the assumption
\begin{align}\label{assumption:boundary}
    \as=\oo.
\end{align}
Throughout this paper, in the low Mach number regime, we always assume \eqref{assumption:boundary} and the physically important hard-sphere collision kernel, even though our analysis can be easily extended to
kernels of hard potential with an angular cutoff.

%%%%%%%%%%%%%%%%%%%%%%%%%%%%%%%%%%%%%%%%%%%%%%%%%%%%%%
\subsection{Asymptotic Expansion and Remainder Equation}\label{sec:geometric-setup}
%%%%%%%%%%%%%%%%%%%%%%%%%%%%%%%%%%%%%%%%%%%%%%%%%%%%%%%

We follow the approach in \cite{BB002} to define the geometric quantities and more details can be found in our companion paper \cite{AA024}. 
For smooth manifold $\p\Omega$, there exists an orthogonal curvilinear coordinates system $(\iota_1,\iota_2)$ such that the coordinate lines coincide with the principal directions at any $\vx_0\in\p\Omega$ (at least locally).
Assume $\p\Omega$ is parameterized by $\vr=\vr(\iota_1,\iota_2)$. Let the vector length be $\pl_i=\abs{\p_{\iota_i}\vr}$ and unit vector $\vt_i=\pl_i^{-1}\p_{\iota_i}\vr$. Based on sign of the flow direction $\vv\cdot\vn(\vx_0)$, we can divide the boundary $\gamma:=\big\{(\vx_0,\vv):\ \vx_0\in\p\Omega,\vv\in\r^3\big\}$ into the incoming boundary $\gamma_-$, the outgoing boundary $\gamma_+$, and the grazing set $\gamma_0$. In particular, the boundary condition of \eqref{large system-} is only given on $\gamma_{-}$.

Consider the corresponding new coordinate system $\vxx:=(\iota_1,\iota_2,\mn)$, where $\mn$ denotes the normal distance to boundary surface $\p\Omega$, i.e.
\begin{align}
\vx=\vr-\mn\vn.
\end{align}
Define the orthogonal velocity substitution for $\vvv:=(\va,\vb,\vc)$ as
\begin{align}\label{aa 44}
-\vv\cdot\vn:=\va,\quad
-\vv\cdot\vt_1:=\vb,\quad
-\vv\cdot\vt_2:=\vc.
\end{align}
Finally, we define the scaled variable $\eta=\dfrac{\mn}{\e}$, which implies $\dfrac{\p}{\p\mn}=\dfrac{1}{\e}\dfrac{\p}{\p\eta}$.

We seek a solution in the form 
\begin{align}\label{aa 08}
    \fs(x,v)=&\f+\fb+\e\mh\re
    =\m+\mh\Big(\e\f_1+\e^2\f_2\Big)+\mbh\Big(\e\fb_1\Big)+\e\mh\re,
\end{align}
where the interior solution is
\begin{align}\label{expand 1}
\f(x,v):= \m(x,v)+\mh(x,v)\Big(\e\f_1(x,v)+\e^2\f_2(x,v)\Big),
\end{align}
and the new $\e$-cutoff (singularity in $\va$) boundary layer is
\begin{align}\label{expand 2}
\fb(\vxx,\vvv):= \mbh(\iota_1,\iota_2,\vvv)\Big(\e\fb_1(\vxx,\vvv)\Big).
\end{align}
Here $\f_1$, $\f_2$ and $\fb_1$ are constructed precisely in our companion paper \cite{AA024} (see Appendix \ref{sec:asymptotic} for definitions), $\re(x,v)$ is the remainder with the minimum and natural $\e$ pre-factor, $\m(x,v)$ denotes a local Maxwellian and $\mb(\iota_1,\iota_2,\vvv)=\m(x_0,v)$ the boundary Maxwellian. The norms are defined in Appendix \ref{sec:norms}.

In \cite{AA024}, we proved the following
\begin{theorem}\label{thm:ghost}
    Under the assumption \eqref{assumption:boundary}, for any given $P>0$, there exists a unique solution $(\rq,\uq_1,\tq; \mathfrak{p})$ (where $\mathfrak{p}$ has zero average) to the ghost-effect equation \eqref{fluid system-} and \eqref{boundary condition} satisfying for any $\NN\in[2,\infty)$
    \begin{align}
    \nm{u_1}_{W^{3,\NN}}+\nm{\mathfrak{p}}_{W^{2,\NN}}+\nm{T-1}_{W^{4,\NN}}\ls\oot.
    \end{align}
    Also, we can construct $\f_1$, $\f_2$ and $\fb_1$ as in \eqref{final 34}\eqref{final 35}\eqref{final 14} such that 
    \begin{align}
    \nm{f_1}_{W^{3,\NN}L^{\infty}_{\varrho,\vartheta}}+\abs{f_1}_{W^{3-\frac{1}{\NN},\NN}L^{\infty}_{\varrho,\vartheta}}&\ls\oot,\quad
    \nm{f_2}_{W^{2,\NN}L^{\infty}_{\varrho,\vartheta}}+\abs{f_2}_{W^{2-\frac{1}{\NN},\NN}L^{\infty}_{\varrho,\vartheta}}\ls\oot,
    \end{align}
and for some $K_0>0$ and any $0<\N\leq 3$, (uniform in $\e$-cutoff)
\begin{align}
    \lnmm{\ue^{K_0\eta}\fb_1}+\lnmm{\ue^{K_0\eta}\p_{\iota_1}^\N\fb_1}+\lnmm{\ue^{K_0\eta}\p_{\iota_2}^\N\fb_1}&\ls\oot.
\end{align}
\end{theorem}

We denote $\oo$ to be a sufficiently small constant independent of the data. Also, let $\oot$ be a small constant depending on $\tb$ satisfying
\begin{align}\label{def:oot}
    \oot=\oo\rt0\ \ \text{as}\ \ \as\rt0.
\end{align}
In principle, while $\oot$ is determined by $\nabla\tb$ a priori, we are free to choose $\oo$ depending on the estimate.
When necessary, we will use $\cc_1,\cc_2,\cc_3\cdots$ to mark different small constants.

Based on the analysis in Section \ref{sec:remainder}, in order to show the validity of \eqref{aa 08}, it suffices to consider the remainder equation for $\re$: 
\begin{align}\label{remainder}
\left\{
\begin{array}{l}
\vv\cdot\nx\left(\mh\re\right)+\e^{-1}\mh\lc[\re]=\mh\ss\ \ \text{in}\ \
\Omega\times\r^3,\\\rule{0ex}{1.2em} \re(\vx_0,\vv)=\pp[\re](\vx_0,\vv)+h(\vx_0,\vv) \ \ \text{for}\ \ \vx_0\in\p\Omega\ \ \text{and}\ \ \vv\cdot\vn(\vx_0)<0.
\end{array}
\right.
\end{align}
Here the linearized Boltzmann operator  $\lc$ is defined in \eqref{att 11}, the boundary and bulk sources $h$ and $\ss$ are defined in \eqref{aa 32} and \eqref{aa 31} (In order to show positivity of $\fs$, we choose to add an artificial term $\sp$ to $\ss$),
\begin{align}
\pp[\re](\vx_0,\vv):=\mss(\vx_0,\vv)\displaystyle\int_{\vuu\cdot\vn(\vx_0)>0}
\mh(\vx_0,\vuu)\re(\vx_0,\vuu)\abs{\vuu\cdot\vn(\vx_0)}\ud{\vuu},
\end{align}
with $\mss(\vx_0,\vv):=\ms\mb^{-\frac{1}{2}}$
satisfying the normalization condition
\begin{align}
    \mb^{\frac{1}{2}}(\vx_0,\vv)=\mss(\vx_0,\vv)\displaystyle\int_{\vuu\cdot\vn(\vx_0)>0}
    \mb(\vx_0,\vuu)\abs{\vuu\cdot\vn(\vx_0)}\ud{\vuu}=\frac{P}{\big(2\pi T_w(x_0)\big)^{\frac{1}{2}}}\mss(\vx_0,\vv).
\end{align}
Note that 
\begin{align}\label{oo 40}
\vv\cdot\nx\left(\mh\re\right)&=\mh(\vv\cdot\nx\re)+\frac{1}{2}\mhh(\vv\cdot\nx\m)\re=\mh(\vv\cdot\nx\re)+\left(\ab\cdot\frac{\nx\tq}{4\tq^2}\right)\re,
\end{align}
where $\ab$ is defined in \eqref{final 22}.
Hence, we can rewrite \eqref{remainder} in the equivalent form
\begin{align}\label{remainder.}
\left\{
\begin{array}{l}
\vv\cdot\nx\re+\left(\mhh\ab\cdot\dfrac{\nx\tq}{4\tq^2}\right)\re+\e^{-1}\lc[\re]=\ss\ \ \text{in}\ \
\Omega\times\r^3,\\\rule{0ex}{1.0em} \re(\vx_0,\vv)=\pp[\re](\vx_0,\vv)+h(\vx_0,\vv) \ \ \text{for}\ \ \vx_0\in\p\Omega\ \ \text{and}\ \ \vv\cdot\vn(\vx_0)<0.
\end{array}
\right.
\end{align}

%%%%%%%%%%%%%%%%%%%%%%%%%%%%%%%%%%%%%%%%%%%%%%%%%%%%%%%%%%%%%%%%%%%%%%%%
\subsection{Decomposition and Reformulation}
%%%%%%%%%%%%%%%%%%%%%%%%%%%%%%%%%%%%%%%%%%%%%%%%%%%%%%%%%%%%%%%%%%%%%%%%

%%%%%%%%%%%%%%%%%%%%%%%%%%%%%%%%%%%%%%%%%%%%%%%%%%%%%%%%%%%%%%%%%%%%%%%%
\subsubsection{$\a$-Hodge Decomposition and Local $\a$-Conservation Law}
%%%%%%%%%%%%%%%%%%%%%%%%%%%%%%%%%%%%%%%%%%%%%%%%%%%%%%%%%%%%%%%%%%%%%%%%

The presence of new  contributions involving $\a\cdot\frac{\nx\tq}{2\tq^2}$
in \eqref{expansion} and $\left(\mhh\ab\cdot\frac{\nx\tq}{4\tq^2}\right)\re$ in \eqref{remainder.} creates fundamental mathematical
difficulties, and we define the following projection and $\a$-Hodge
decomposition, as well as corresponding local $\a$-conservation law to capture.

Note that the null space $\nk$ of $\lc$ is a five-dimensional space spanned by the orthogonal basis
\begin{align}\label{att 32}
\mh\Big\{1,\vv,\left(\abs{\vv}^2-3\tq\right)\Big\}.
\end{align}
We denote $\nnk$ the orthogonal complement of $\nk$ in $L^2(\r^3)$.
Define the kernel operator $\pk$
as the orthogonal projection onto $\nk$, and the non-kernel operator $\ik-\pk$.
We decompose
\begin{align}
    \re=&\pk[\re]+(\ik-\pk)[\re]
    :=\mh(\vv)\Big\{\P_{\re}(\vx)+\vv\cdot
    \bb_{\re}(\vx)+\Big(\abs{\vv}^2-5\tq\Big)c_{\re}(\vx)\Big\}+(\ik-\pk)[\re],
\end{align}
Note that we have chosen a different decomposition of $\pk[\re]$ in contrast with the classical $(a_{\re},\bb_{\re},c_{\re})$ decomposition in previous work \cite{Esposito.Guo.Kim.Marra2015} ($\P_{\re}=a_{\re}+2\tq c_{\re}$).
Moreover, we introduce a crucial further orthogonal split
\begin{align}
(\ik-\pk)[{\re}]
&=\a\cdot\bd_{\re}(\vx)+(\ik-\bpk)[\re],
\end{align}
where $\a$ is defined in \eqref{final 22}, $(\ik-\bpk)[{\re}]$ is the orthogonal complement to $\a\cdot\bd_{\re}(\vx)$ in $\nnk$ with respect to $\brr{\ \cdot\ ,\ \cdot\ }:=\bbrv{\ \cdot\ ,\lc[\ \cdot\ ]}$
\begin{align}
\bbrr{\a,(\ik-\bpk)[{\re}]}=\bbrv{\ab,(\ik-\bpk)[{\re}]}=0.
\end{align}
From now on, when there is no confusion, we will omit the subscript ${\re}$ and simply write $\P,\bb,c,\bd$.
In summary, we decompose the remainder as
\begin{align}
\re&=\Big(\P+\bb\cdot v+c\big(\abs{v}^2-5T\big)\Big)\mh+\bd\cdot\a+(\ik-\bpk)[\re],\label{pp 01}
\end{align}
and from \eqref{remainder.}, $\re$ satisfies 
\begin{align}
v\cdot\nx\re+\left(\mhh\ab\cdot\frac{\nx T}{4T^2}\right)\re+\e^{-1}\bd\cdot\ab+\e^{-1}\lc\Big[(\ik-\bpk)[\re]\Big]&=\ss.\label{pp 02}
\end{align}
We can further define the $\a$-Hodge decomposition $\bd=\nx\xi+\be$
with the potential $\xi$ solving the Poisson equation
\begin{align}\label{tt 01}
\nx\cdot\left(\k\nx\xi\right)=\nx\cdot(\k\bd)\ \ \text{in}\ \ \Omega,\quad
\xi=0\ \ \text{on}\ \ \p\Omega.
\end{align}
We can directly compute
\begin{align}\label{tt 13}
\nx\cdot(\k\be)=\nx\cdot(\k\bd)-\nx\cdot(\k\nx\xi)=0.
\end{align}
This implies the crucial orthogonality: for any $\eta(\vx)$ such that $\eta=0$ on $\p\Omega$, we have
\begin{align}\label{final 63}
\int_{\Omega}(\k\be)\cdot\nx\eta\ud{\vx}=-\int_{\Omega}\nx\cdot(\k\be)\eta\ud{\vx}=0.
\end{align}
Thus, taking $\eta=\xi$, we know
\begin{align}\label{energy 11}
\bbrr{\bd\cdot\a,\bd\cdot\a}&=\bbr{\bd\cdot\a,\lc[\bd\cdot\a]}=\bbr{\bd\cdot\a,\bd\cdot\ab}=\int_{\Omega}\k\abs{\nx\xi+\be}^2\ud\vx\\
&=\int_{\Omega}\k\abs{\nx\xi}^2\ud\vx+\int_{\Omega}\k\abs{\be}^2\ud\vx+2\int_{\Omega}\k\left(\nx\xi\cdot\be\right)\ud\vx\no\\
&=\int_{\Omega}\k\abs{\nx\xi}^2\ud\vx+\int_{\Omega}\k\abs{\be}^2\ud\vx-2\int_{\Omega}\nx\cdot(\k\be)\xi\ud\vx=\int_{\Omega}\k\abs{\nx\xi}^2\ud\vx+\int_{\Omega}\k\abs{\be}^2\ud\vx.\no
\end{align} 

We note the local conservation laws of mass, momentum and energy (with test functions $\mh,v\mh,\abs{v}^2\mh$):
\begin{align}
    P\big(\nx\cdot\bb\big)=&\brv{\mh,\ss},\label{conservation law 1}\\
    P\nx\P+\nx\cdot\int_{\r^3}(v\otimes v)\mh(\ik-\bpk)[\re]=&\brv{v\mh,\ss},\label{conservation law 2}\\
    5P\Big(\nx\cdot(\bb T)\Big)+\nx\cdot(\kappa\bd)=&\brv{\abs{v}^2\mh,\ss}.\label{conservation law 3}
\end{align}
Notice that the construction of $\xi$ in \eqref{revise 01} leads to a natural kinetic equation for the combination $\nx\big(c+\e^{-1}\xi\big)$ as
\begin{align}\label{revise 02}
    v\cdot\nx\re+\e^{-1}\big(\nx\xi+\be\big)\cdot\ab
    \approx& \nx\big(c+\e^{-1}\xi\big)\cdot\ab+v\cdot\nx\Big(p\mh+\bb\cdot v\mh+(\ik-\pk)[\re]\Big)+\e^{-1}\be\cdot\ab,
\end{align}
along $\ab$ direction. This allows us to introduce the local $\a$-conservation law (with test function $\a$)
\begin{align}\label{new 3=}
\k\nx\big(c+\e^{-1}\xi\big)+\frac{\nx T}{2T^2}\big(\kappa\P+\sigma c\big)+\e^{-1}\k\be\\
+\brv{v\cdot\nx\big((\ik-\bpk)[\re]\big),\a}+\brv{\frac{\nx T}{4T^2}(\ik-\bpk)[\re],\a}&=\brv{\ss,\a}.\no
\end{align}
This new $\a$-Hodge decomposition and its
$\a$-conservation plays the key role for us to circumvent the analytical difficulty.

%%%%%%%%%%%%%%%%%%%%%%%%%%%%%%%%%%%%%%%%%%%%%%%%%%%%%%%%%%%%%%%%%%%%%%%%%%%%
\subsubsection{Reformulation with Global Maxwellian}
%%%%%%%%%%%%%%%%%%%%%%%%%%%%%%%%%%%%%%%%%%%%%%%%%%%%%%%%%%%%%%%%%%%%%%%%%%%%

Due to the cubic velocity growth term in \eqref{remainder.}, we have to reformulate the remainder equation with a global Maxwellian in order to obtain $L^{\infty}$ estimates. Considering $\lnm{\nx\tq}\ls\oot$ for $\oot$ defined in \eqref{def:oot}, choose constant $\ds\tm: \tm<\min_{x\in\Omega}T<\max_{x\in\Omega}T<2\tm$ and $\ds\max_{x\in\Omega}T-\tm=\oot$.
Define a global Maxwellian
\begin{align}\label{final 13}
    \mm:=\frac{P}{(2\pi)^{\frac{3}{2}}\tm^{\frac{5}{2}}}\exp\bigg(-\frac{\abs{v}^2}{2\tm}\bigg).
\end{align}
We can rewrite \eqref{remainder} as
\begin{align}\label{ll 01}
\left\{
\begin{array}{l}
\vv\cdot\nx\rem+\e^{-1}\lc_M[\re]=\ss_M\ \ \text{in}\ \
\Omega\times\r^3,\\\rule{0ex}{1.5em} \rem(\vx_0,\vv)=\mathcal{P}_M[\rem](\vx_0,\vv)+h_M(\vx_0,\vv) \ \ \text{for}\ \ \vx_0\in\p\Omega\ \ \text{and}\ \ \vv\cdot\vn(\vx_0)<0,
\end{array}
\right.
\end{align}
where $\rem=\mmhh\mh\re$, $\ss_M=\mmhh\mh\ss$, $h_M=\mmhh\mh h$ and for $\mmss:=\mmhh\mh(\vx_0,\vv)\mss(\vx_0,\vv)=M_w\mmhh$
\begin{align}
    \lc_M[\rem]:=&-2\mmhh Q^{\ast}\left[\m,\mmh \rem\right]:=\nu_M\rem-K_M[\rem],\\
    \mathcal{P}_M[\re_M](\vx_0,\vv):=&\mmss(\vx_0,\vv)\displaystyle\int_{\vuu\cdot\vn(\vx_0)>0}   \mmh\rem(\vx_0,\vuu)\abs{\vuu\cdot\vn(\vx_0)}\ud{\vuu}.
\end{align}

%%%%%%%%%%%%%%%%%%%%%%%%%%%%%%%%%%%%%%%%%%%%%%%%%%%%%%%%%%%%%%%%%%%%%%%%%%%%
\subsubsection{Working Space}
%%%%%%%%%%%%%%%%%%%%%%%%%%%%%%%%%%%%%%%%%%%%%%%%%%%%%%%%%%%%%%%%%%%%%%%%%%%%

Denote the working space $X$ via the norm
\begin{align}\label{ss 00}
\xnm{\re}:=&\e^{-1}\Big(\tnm{\P}+\tnm{\xi}+\tnm{\be}+\um{(\ik-\bpk)[\re]}+\pnm{\xi}{6}\Big)+\e^{-\frac{1}{2}}\Big(\tnm{\bb}+\nm{\xi}_{H^2}+\tnms{(1-\pp)[\re]}{\gamma_+}+\tnms{\nx\xi}{\p\Omega}\Big)\\
&+\Big(\pnm{c}{2}+\pnm{\P}{6}+\pnm{\bb}{6}+\pnm{c}{6}+\nm{\xi}_{W^{2,6}}+\pnm{\be}{6}+\pnm{(\ik-\bpk)[\re]}{6}+\tnms{\pp[\re]}{\gamma}+\pnms{\m^{\frac{1}{4}}(1-\pp)[\re]}{4}{\gamma_+}\Big)\no\\
&+\e^{\frac{1}{2}}\Big(\lnmm{\rem}+\lnmms{\rem}{\gamma}\Big).\no
\end{align}

%%%%%%%%%%%%%%%%%%%%%%%%%%%%%%%%%%%%%%%%%%%%%%%%%%%%%%%%%%%%%%%%%%%%%%%%
\subsection{Main Theorem}
%%%%%%%%%%%%%%%%%%%%%%%%%%%%%%%%%%%%%%%%%%%%%%%%%%%%%%%%%%%%%%%%%%%%%%%%

\begin{theorem}\label{main}
Assume that $\Omega$ is a bounded $C^3$ domain and \eqref{assumption:boundary} holds. Then for any given $P>0$, there exists $\e_0>0$ such that for any $\e\in(0,\e_0)$, there exists a non negative solution $\fs$ to the equation \eqref{large system-} represented by \eqref{aa 08} satisfying
\begin{align}\label{final 31}
    \int_{\Omega}\P(x)\ud x=0,
\end{align}
and
\begin{align}\label{final 21}
    \xnm{\re}\ls\oot,
\end{align}
where the $X$ norm is defined in \eqref{ss 00}. Such a solution is unique among all solutions satisfying \eqref{final 31} and \eqref{final 21}. This further yields that in the expansion \eqref{aa 08}, $\m+\e\m\big(\uq_1\cdot\vv\big)$ is the leading-order terms in the sense of
\begin{align}
    \Bnm{\mhh\big[\fs-\m\big]}_{L^2_{x,v}}\ls\e
\end{align}
and 
\begin{align}\label{final 211}
    \nm{\int_{\r^3}\Big[\fs-\m-\e \m\big(\uq_1\cdot v\big)\Big]v}_{L^2_{x}}\ls\e^{\frac{3}{2}},
\end{align}
where $(\rq,\uq_1,\tq)$
is determined by the ghost-effect equations \eqref{fluid system-} and \eqref{boundary condition}.
\end{theorem}

\begin{remark}
    There is no restriction on $\nabla T_{w}$ except smallness condition \eqref{assumption:boundary}. Moreover, our result is valid for all general smooth bounded domains (including non-convex domains), despite the presence of boundary-layer approximation.
\end{remark}

\begin{remark}
    We stress that the boundedness of $\xnm{\re}$ only implies that $\e^{-\frac{1}{2}}\nx\xi$ is bounded in $L^2$, not $\e^{-1}\nx\xi$ is bounded in $L^2$. This is in contrast to the full $\e^{-1}\um{(\ik-\pk)[\re]}$ control in $L^2$ in \cite{Esposito.Guo.Kim.Marra2015}, 
\end{remark}

%%%%%%%%%%%%%%%%%%%%%%%%%%%%%%%%%%%%%%%%%%%%%%%%%%%%%%%%%%%%%%%%%%%%%%%%
\subsection{Literature Review}
%%%%%%%%%%%%%%%%%%%%%%%%%%%%%%%%%%%%%%%%%%%%%%%%%%%%%%%%%%%%%%%%%%%%%%%%

The hydrodynamic limit of the Boltzmann equation has been the subject of many studies since the pioneering work by Hilbert, who introduced his famous expansion in the Knudsen number $\e$ in \cite{Hilbert1916, Hilbert1953} to answer the sixth of his famous questions \cite{Hilbert1900}. 
For dynamic and unsteady problems, mathematical results on the control of the remainder of the Hilbert expansion of the Boltzmann equation to show the convergence to the solutions of the compressible Euler equations for small Knudsen number $\e$, were obtained by Caflisch \cite{Caflisch1980}, and Lachowicz \cite{Lachowicz1987}, while Nishida \cite{Nishida1978}, Asano-Ukai \cite{Asano.Ukai1983} proved the convergence of the Boltzmann equation to the compressible Euler equations in the hydrodynamic limit by different methods. 
We also refer to \cite{Huang.Wang.Yang2013, Huang.Wang.Wang.Yang2016, Guo.Huang.Wang2021}. 
More recently, the convergence in the presence of singularities for the Euler equations have been obtained in \cite{Yu2005} and \cite{Huang.Wang.Yang2010, Huang.Wang.Yang2010(=)}. The relativistic Euler limit has been studied in \cite{Speck.Strain2011}.
We also refer to the recent progress on incompressible Euler limit \cite{Jang.Kim2021, Kim.La2022, Bardos.Golse.Paillard2012, Golse2013, Golse.Saint-Raymond2009, Saint-Raymond2003, Saint-Raymond2008}.
On a longer time scale $\e^{-1}$, where diffusion effects become signiﬁcant,  low Mach numbers regime (Mach number of order $\e$ or smaller) is assumed due to the lack of scaling invariance of the compressible Navier–Stokes equations. 
Hence the Boltzmann solution has been solved close to the incompressible Navier–Stokes–Fourier system. 
Mathematical results were obtained, among
others, in \cite{Bardos.Ukai1991, Masi.Esposito.Lebowitz1989, Guo2006, Guo.Jang2010, Guo.Jang.Jiang2010, Cao.Jang.Kim2021, Gallagher.Tristani2020} for smooth solutions. 
In particular, the convergence of the renormalized solutions to the Navier-Stokes-Fourier system has been obtained by Golse and Saint-Raymond \cite{Golse.Saint-Raymond2004}, see also contributions in \cite{Bardos.Golse.Levermore1991, Bardos.Golse.Levermore1993, Bardos.Golse.Levermore1998, Bardos.Golse.Levermore2000, Lions.Masmoudi2001, Masmoudi.Saint-Raymond2003, Saint-Raymond2009, Arsenio2012, Jiang.Masmoudi2016, Jiang.Levermore.Masmoudi2010}. We also refer to the survey \cite{Mouhot.Villani2015}.
In a pioneering study of the unsteady ghost effect for 1D geometry $-\infty<x<\infty$ \cite{Huang.Wang.Wang.Yang2016}, convergence is established for a temperature variation along the 1D normal direction via delicate analysis in a Sobolev space with a crucial sign condition for $\tq'$.

Much less is known about the hydrodynamic limits for steady Boltzmann solution, due to lack of basic $L^1$ and entropy estimates. 
Only the control of entropy production from $\int Q[\fs,\fs]\ln(\fs)$ is available analytically. 
Despite progress \cite{Villani2002, Desvillettes.Villani2005, Desvillettes.Villani2001, Carlen.Carvalho1994, Aoki.Golse.Kosuge2015} on the control of entropy production in terms of $\fs-\m_{\fs}$ (where $\m_{\fs}$ is the local
Maxwellian with the same mass, momentum and energy), its nonlinear nature prevents useful applications for Boltzmann solutions with large amplitudes so far (see an interesting progress in \cite{Arkeryd.Nouri2000} and also \cite{Liu.Yu2013,Chen.Chen.Liu.Sone2007}).
In \cite{Arkeryd.Esposito.Marra.Nouri2010}, in the Rayleigh-Benard context existence results for small Knudsen and Mach numbers in two dimensional slab are obtained and the first bifurcation is studied. In \cite{Esposito.Lebowitz.Marra1994, Esposito.Lebowitz.Marra1995} the geometry is strictly one-dimensional and the Knudsen number is small, but not the Mach number. 
We refer to the recent review \cite{Esposito.Marra2020} for more details. 
In the above cases no slip boundary conditions are considered. 
Recently an interesting extension to the case of slip boundary condition has been obtained in \cite{Duan.Liu.Yang.Zhang2022}. 
Very few rigorous results, beyond the one-dimensional case, are available for the ghost effect. In \cite{Brull2008,Brull2008(=)}, the ghost effect for a one-dimensional mixture is established, and in \cite{Arkeryd.Esposito.Marra.Nouri2011} the ghost effect is studied for a one-dimensional problem with cylindrical symmetry.  
We remark that classical ghost effect from temperature variation \eqref{large system-} requires non-trivial tangential temperature variation which has a multi-dimensional nature.

%%%%%%%%%%%%%%%%%%%%%%%%%%%%%%%%%%%%%%%%%%%%%%%%%%%%%%%%%%%%%%%%%%%%%%%%

\section{Methodology}\label{sec:framework}

\subsection{$L^6-L^{\infty}$ Framework for Fourier Law}
%%%%%%%%%%%%%%%%%%%%%%%%%%%%%%%%%%%%%%%%%%%%%%%%%%%%%%%%%%%%%%%%%%%%%%%%

For a general 3D domain, an improved $L^6-L^\infty$ framework is developed in \cite{Esposito.Guo.Kim.Marra2015}, in which steady hydrodynamic limits to the celebrated Fourier law are established, along with their dynamical stability.
For Boltzmann solutions close to a Maxwellian $\m$, the 
entropy production $\int Q[\fs,\fs]\ln(\fs)$ is approximated by the fundamental a priori bound for the microscopic part $\um{(\ik-\pk)[\re]}^2$ associated with the linearized Boltzmann operator $\lc$ around $\m$, where $\mh\re\sim\fs-\m$ instead of the natural
 difference with the local Maxwellian $\fs-\m_{\fs}$.
In such a setup, the fundamental analytic difficulty is to control the missing macroscopic part $\pk[\re]\sim\m-\m_{\fs}$ for the nonlinear closure.
 
In a series of papers \cite{Guo2003, Guo2012}, an elliptic structure is discovered for $\pk[\re]$, and $\pk[\re]$ can be bounded via $(\ik-\pk)[\re]$ through the Boltzmann equation in high order Sobolev norms. 
More
applications can be found in \cite{Gressman.Strain2011, Chaturvedi.Luk.Nguyen2022}. 
Unfortunately, it is well-known that
Boltzmann solutions exhibit only limited regularity (or even singularity) in the presence of physical boundary conditions \cite{Kim2011, Guo.Kim.Tonon.Trescases2013} due to the characteristic nature of the grazing set in kinetic theory. 
To overcome this difficulty, a $L^{2}-L^{\infty }$ framework is established in \cite{Guo2010}, in which $\tnm{\pk[\re]}$ is bounded via $\tnm{(\ik-\pk)[\re]}$ to bootstrap weighted $\lnm{\re}$ bound via double-Duhamel principle along the
characteristics, thanks to the velocity mixing feature of $K$. An important new methodology is developed in \cite{Esposito.Guo.Kim.Marra2013}, where $\tnm{\pk[\re]}$ is estimated quantitatively in terms of $\tnm{(\ik-\pk)[\re]}$ via proper choices of test functions in the weak formulation with boundary effect, by solving dual Poisson equations of the form $\Delta_x\phi\sim a,\bb,c$.
Such a methodology entails a robust and flexible
approach to grasp the ellipticity (positivity) estimates for $\pk[\re]$ in the presence of boundary conditions, reminiscing in spirit
elliptic estimates in weak forms.

In recent papers \cite{Esposito.Guo.Kim.Marra2013, Esposito.Guo.Kim.Marra2015, BB002}, the steady solution to the Boltzmann equation close to Maxwellians was constructed, in 3D smooth domains, for a gas in contact with a boundary with a prescribed temperature profile modeled by the diffuse-reflection boundary condition. 
Also, the Navier-Stokes-Fourier limit was established in \cite{Esposito.Guo.Kim.Marra2015, BB002} for the diffusive scaling of $\abs{\nx\tq}=\oo\e$ based on an improved $L^6-L^{\infty}$ framework. 
In particular, the proof in \cite{Esposito.Guo.Kim.Marra2015} relies on the $L^2$ energy/coercivity estimates combined with $L^6$ kernel estimates (by solving dual Poisson equations of $\dx\phi\sim a\abs{a}^4, \bb\abs{\bb}^4, c\abs{c}^4$)
\begin{align}
    \e^{-1}\um{(\ik-\pk)[\re]}+\pnm{\pk[\re]}{6}+\e^{-\frac{1}{2}}\tnms{(1-\pp)[\re]}{\gamma_+}\ls \oo\e^{\frac{1}{2}}\lnmm{\re}+\text{$\oo$ boundary/source terms},
\end{align}
and the $L^{\infty}$ estimates
\begin{align}
    \lnmm{\re}\ls \e^{-\frac{3}{2}}\um{(\ik-\pk)[\re]}+\e^{-\frac{1}{2}}\pnm{\pk[\re]}{6}+\text{$\oo$ boundary/source terms}.
\end{align}
The $L^6$ bound is crucial to control nonlinearity in 3D and to close the $L^{\infty}$ estimates:  weaker $\tnm{\pk[\re]}$ bound leads to extra $\e^{-1}$ loss (compared with $L^6$ bound) as $\lnmm{\re}\ls \e^{-\frac{3}{2}}\tnm{\pk[\re]}$.

We also refer to the recent papers on the diffusive limit of the Boltzmann equation and related models \cite{Briant.Guo2016, Briant.Merino-Aceituno.Mouhot2019, Jiang.Masmoudi2022, Jang2009}. We list some recent developments along $L^p-L^{\infty}$ framework \cite{Esposito.Guo.Kim.Marra2013, Guo2010, Guo.Jang2010, Guo.Jang.Jiang2010, Guo.Kim.Tonon.Trescases2013, Guo.Kim.Tonon.Trescases2016, Kim2011, Kim2014, Speck.Strain2011, AA003}.

%%%%%%%%%%%%%%%%%%%%%%%%%%%%%%%%%%%%%%%%%%%%%%%%%%%%%%%%%%%%%%%%%%%%%%%%
\subsection{New $L^2-L^6-L^{\infty}$ Framework}
%%%%%%%%%%%%%%%%%%%%%%%%%%%%%%%%%%%%%%%%%%%%%%%%%%%%%%%%%%%%%%%%%%%%%%%%

For hydrodynamic limits of \eqref{large system-}, the basic energy estimate yields
\begin{align}
    \e^{-2}\um{(\ik-\pk)[\re]}^2=&-\e^{-1}\br{\mhh\ab\cdot\frac{\nx\tq}{4\tq^2},\re^2}+\text{good terms}
    =-\e^{-1}5P\bbr{\nx\tq,\bb c}+\text{good terms}.
\end{align}

In the case of $\lnm{\nx\tq}=o(1)\e$, we have
\begin{align}\label{basicenergy}
    \e^{-1}\um{(\ik-\pk)[\re]}\ls \oo\pnm{\pk[\re]}{6}+1.
\end{align}
With a new $L^{6}$ gain for $\pk[\re]$ in \cite{Esposito.Guo.Kim.Marra2015}, quantitative estimates in \cite{Esposito.Guo.Kim.Marra2013} lead to
\begin{align}\label{6bound}
    \pnm{\pk[\re]}{6}\ls \e^{-1}\um{(\ik-\pk)[\re]}
\end{align}
where $c$ and $\bb$ are estimated via test functions of the form $\nx\phi\cdot\ab$ and $\nx\psi:\bbb$, and the convergence to the incompressible Navier-Stokes-Fourier system as $\e\rt 0$ in \cite{Esposito.Guo.Kim.Marra2015} is established. We remark that thanks to the new $L^6$ estimate and $\lnm{\nx\tq}=o(1)\e$, it is possible to avoid boundary layer approximation
completely.

In the case of $\lnm{\nx\tq}=o(1)$, however,
\begin{align}\label{final 46}
    -\e^{-1}5P\bbr{\nx\tq,\bb c}\approx \oo\e^{-1}\pnm{\pk[\re]}{6}^2
\end{align}
which is impossible to close via \eqref{6bound} with a severe loss of $\e^{-1}$. We note that the term $-\e^{-1}5P\bbr{\nx\tq,\bb c}$ presents a fundamental major difficulty, and the main mathematical achievement of our contribution is to develop a systematic methodology to overcome this loss of $\e$ in the presence of boundary effects. Furthermore, \eqref{assumption:boundary}  forces us to introduce boundary layer approximation at the leading order, in a stark contrast to the case $\lnm{\nx\tq}=o(1)\e$. This presents another major technical challenge due to the well-known singularity at the grazing set: no mathematical theory is available for non-convex sets.\\

\paragraph{\underline{Reduced Energy Estimate of $(\ik-\bpk)[\re]+\be\cdot\a$}}
The first key idea is to use $\a$-Hodge decomposition to split
\begin{align}\label{revise 01}
    (\ik-\pk)[\re]=\bd\cdot\a+(\ik-\bpk)[\re]=\nx\xi\cdot\a+\be\cdot\a+(\ik-\bpk)[\re].
\end{align}
A reduced energy estimate is then established as the building block of our analysis:

\begin{proposition}[Proof in Section \ref{sec:energy-est}]\label{thm:energy}
Let $\re$ be a solution to \eqref{remainder}. Under the assumption \eqref{assumption:boundary}, we have
\begin{align}\label{ss 01}
    \e^{-\frac{1}{2}}\tnms{(1-\pp)[\re]}{\gamma_+}+\e^{-1}\tnm{\be}+\e^{-1}\um{(\ik-\bpk)[\re]}
    \ls&\oot\xnm{\re}+\xnm{\re}^2+\oot.
\end{align}
\end{proposition}

Remarkably, this new reduced energy
estimate \eqref{ss 01}, in comparison to \eqref{basicenergy}, sacrifices the control of $\nx\xi\cdot \a$ in $(\ik-\pk)[\re]$ to cancel $-\e^{-1}5P\br{\nx\tq,\bb c}$ via a string of delicate manipulations (ignoring boundary contributions):
\begin{align}\label{string}
    \e^{-1}\bbrx{\nx\tq,\bb c} \approx&-\e^{-1}\bbrx{\nx\cdot\big(\k\bd\big),c}\qquad \big[\text{Mass and Energy Local Conservation Laws \eqref{conservation law 1}\eqref{conservation law 3}}\big]\\
    =&-\e^{-1}\bbrx{\nx\cdot\big(\k\nx\xi\big),c}\qquad \big[\text{$\a$-Hodge \eqref{tt 13}}\big]\no\\
    \approx&\e^{-1}\bbrx{\k\nx\xi,\nx c}\qquad \big[\text{Integration by parts}\big]\no\\
    \approx&-\e^{-1}\bbrx{\nx\xi,\e^{-1}\k\nx\xi}\qquad \big[\text{Local $\a$-Conservation Law \eqref{new 3=} and $\a$-Hodge \eqref{final 63}}\big]\no\\
    =&-\e^{-2}\iint_{\Omega\times\r^3}\k\abs{\nx\xi}^2,\no
\end{align}
which is cancelled with its counterpart in $\ds\e^{-2}\iint_{\Omega\times\r^3}\re(\ik-\pk)[\re]$.
We note that the ignored boundary contribution $\brb{\nx\xi,(1-\pp)[\re]}{\gamma_+}$ is bounded by $o(1)\e\xnm{\re}^2$ thanks to \eqref{ss 00}.

We remark that the a priori control of the full $(\ik-\pk)[\re]$, as entropy production, has been the starting point for PDE study for Boltzmann solutions near Maxwellians in the past.
To our knowledge, this is the first time one needs to study fine structure within $(\ik-\pk)[\re]$ itself.

\paragraph{\underline{Estimate of $\P$ and $\z:=c+\e^{-1}\xi$}}
We can control $\P\approx O(\e )$ directly from local conservation laws \eqref{conservation law 2} as in \eqref{6bound}. 

\begin{proposition}[Proof in Section \ref{sec:p-est}]\label{prop:p-bound}
Let $\re$ be a solution to \eqref{remainder}. Under the assumption \eqref{assumption:boundary} and \eqref{final 31}, we have 
\begin{align}\label{oo 28=}
    \e^{-1}\tnm{\P}+\pnm{\P}{6}\ls \oot\xnm{\re}+\xnm{\re}^2+\oot.
\end{align}
\end{proposition}

Thanks to the local $\a $-conservation law \eqref{new 3=},  the $L^6$ norm of the natural combination
\begin{align}
    \z=c+\e^{-1}\xi
\end{align}
is bounded by the $L^6$ framework in \cite{Esposito.Guo.Kim.Marra2015}.
\begin{proposition}[Proof in Section \ref{sec:z-est}]\label{prop:z-bound}
    Let $\re$ be a solution to \eqref{remainder}. Under the assumption \eqref{assumption:boundary}, for $\al\geq1$, we have 
    \begin{align}\label{oo 33=}
        \pnm{\z}{6}\ls&\oot\xnm{\re}+\xnm{\re}^2+\oot.
    \end{align}
\end{proposition}

In order to cope with the new $\e$-cutoff boundary layer interaction, we
must further split $\z$ as regular part $\z^R$ and
small singular part $\z^S$: $\z=\z^{S}+\z^{R}$ so that

\begin{proposition}[Proof in Section \ref{sec:z-split}]\label{prop:z-splitting}
Let $\re$ be a solution to \eqref{remainder}. Under the assumption \eqref{assumption:boundary}, we have $\z=\z^R+\z^S$ where 
\begin{align}
    \nm{\z^R}_{H^1_0}\ls&\oot\xnm{\re}+\xnm{\re}^2+\oot,\label{qq 06}\\
    \tnm{\z^S}\ls& \oot\e^{\frac{1}{2}}\xnm{\re}+\e^{\frac{1}{2}}\xnm{\re}^2+\oot\e.\label{qq 07}
\end{align}
\end{proposition}

\paragraph{\underline{Dual Stokes Problem and Improved Estimate $\tnm{\bb}\ls \e^{\frac{1}{2}}$}}
We now obtain a surprising gain of $\e^{\frac{1}{2}}$ for $\bb$, which is necessary to justify the ghost effect contribution $\e u_{1}\cdot v\mh$ beyond the first order of $\e$. 

\begin{proposition}[Proof in Section \ref{sec:b-est}]\label{prop:b-bound}
    Let $\re$ be a solution to \eqref{remainder}. Under the assumption \eqref{assumption:boundary}, we have
    \begin{align}\label{oo 34=}
        \e^{-\frac{1}{2}}\tnm{\bb}+\pnm{\bb}{6}\ls&\oot\xnm{\re}+\xnm{\re}^2+\oot.
    \end{align}
\end{proposition}

Such a gain $\e^{\frac{1}{2}}$ is crucial to justify the $\e$ order ghost effect $\e u_1$ with a key removal of $\e^{-1}(\ik-\bpk)[\re]$ term via compensating with an extra local conservation law. Such a removal also plays the key role in the subsequent estimate for $c$.

Notice that the test function $\nx\psi:\b$, where
\begin{align}\label{final 22=}
    &\bbb=\bigg(v\otimes v-\frac{\abs{v}^2}{3}\mathbf{1}\bigg)\mh\in\r^{3\times3},\quad \b=\lc^{-1}\left[\bbb\right]\in\r^{3\times3},
\end{align} 
yields 
\begin{align}\label{pp 04=}
    &\br{v\cdot\nx\re+\e^{-1}\bd\cdot\ab+\e^{-1}\lc\Big[(\ik-\bpk)[\re]\Big],\nx\psi:\b}\\
    =&\br{v\cdot\nx\re+\e^{-1}\lc\Big[(\ik-\bpk)[\re]\Big],\nx\psi:\b}\no\\
    \approx& \br{\abs{v}^2\mh\bb,\nx\Big(\nx\psi:\b\Big)}+\br{\e^{-1}(\ik-\bpk)[\re],\nx\psi:\bbb},\no
\end{align}
and thus the combination of \eqref{pp 04=} and the local momentum conservation law \eqref{conservation law 2}
\begin{align}
    \e^{-1}\br{(\ik-\bpk)[\re],\nx\psi:\bbb}\approx -\e^{-1}\br{\nx\P,\psi}\approx \e^{-1}\br{\P,\nx\cdot\psi}.
\end{align}
Hence, the choice of the \textit{new} test function $\nx\psi:\b+\e^{-1}\psi\cdot v\mh$ with a smooth function $\psi(x)$ satisfying $\nx\cdot\psi=0$, $\psi\big|_{\p\Omega}=0$ exactly eliminates this most singular term for \eqref{remainder} and leads to 
\begin{align}\label{pp 16=}
&-\bbrx{\lambda\Delta_x\psi,\bb}=\text{good terms}.
\end{align}
The key feature in \eqref{pp 16=} is the absence of the $O(1)$ term $\e^{-1}(\ik-\bpk)[\re]$, so that a gain of $\e^{\frac{1}{2}}$ is possible by constructing a
solution to the dual Stokes problem for $(\psi, q)$ (with an artificial pressure $q$)
\begin{align}\label{pp 17=}
    -\lambda\Delta_x\psi+\nx q\approx\bb\abs{\bb}^{\N-2}\ \ \text{in}\ \ \Omega,\qquad
    \nx\cdot\psi=0\ \ \text{in}\ \ \Omega,\qquad\psi=0\ \ \text{on}\ \ \p\Omega.
\end{align}

\paragraph{\underline{Dual Stokes-Poisson Problem and Estimate of $c$}}
The estimate of $c$ is
the most delicate part of the paper.

\begin{proposition}[Proof in Section \ref{sec:c-est}]\label{prop:c-bound}
    Let $\re$ be a solution to \eqref{remainder}. Under the assumption \eqref{assumption:boundary}, we have
    \begin{align}\label{oo 35=}
        \pnm{c}{6}\ls&\oot\xnm{\re}+\xnm{\re}^{2}+\xnm{\re}^{3}+\oot.
    \end{align}
\end{proposition}

Because of \eqref{revise 02}, it seems impossible to split $\z=c+\e^{-1}\xi$ and to obtain $c$ estimate independent of $\e^{-1}\xi$
via any test functions. The key new idea is to recombine $\e^{-1}\nx\xi+\e^{-1}\be=\e^{-1}\bd$ in the local $\a$-conservation law \eqref{new 3=} with a \textit{new} test function $\nx\phi\cdot\a$:
\begin{align}
\label{pp 03=}
&\br{v\cdot\nx\re+\e^{-1}\bd\cdot\a,\nx\phi\cdot\ab}=\br{v\cdot\nx\re,\nx\phi\cdot\a}+\br{\e^{-1}\bd\cdot\ab,\nx\phi\cdot\a}\\
\approx&\br{\ab\cdot\nx c,\nx\phi\cdot\a}+\e^{-1}\brx{\nx\cdot(\k\bd),\phi},\no
\end{align}
and thus the combination of \eqref{pp 03=} and the mass/energy conservation \eqref{conservation law 1}\eqref{conservation law 3} yields
\begin{align}
    \e^{-1}\brx{\nx\cdot(\k\bd),\phi}\approx -\e^{-1}\bbrx{\nx\tq,\bb\phi}.
\end{align}
Hence, the choice of a \textit{new pair} of test functions $\nx\phi\cdot\a+\e^{-1}\phi\big(\abs{v}^2-5T\big)\mh$ where $\phi(x)$ is a smooth function satisfying $\phi\big|_{\p\Omega}=0$ for \eqref{remainder} leads to 
\begin{align}\label{pp 10=}
&-\bbrx{\nx\cdot(\kappa\nx\phi),c}+\e^{-1}5P\bbrx{\phi,\nx T\cdot\bb}=\text{good terms},
\end{align}
where the most singular term is in terms of $\bb$.
Then we can choose choose $\psi$ 
in \eqref{pp 16=} coupled with \eqref{pp 10=} to kill $\e^{-1}5P\bbrx{\phi,\nx T\cdot\bb}$. 
To this end, we solve
a \textit{coupled} dual Stokes-Poisson system for the triple $(\psi, q,\phi)$
\begin{align}\label{pp 20=}
    -\lambda\Delta_x\psi+\nx q\approx-5P\phi\nx T\ \ \text{in}\ \ \Omega,\qquad
    \nx\cdot\psi=0\ \ \text{in}\ \ \Omega,\qquad
    \psi=0\ \ \text{on}\ \ \p\Omega,
\end{align}
\begin{align}\label{pp 23=}
    -\nx\cdot\big(\kappa\nx\phi\big)=c\abs{c}^{\N-2}\ \ \text{in}\ \ \Omega,\qquad
    \phi=0\ \ \text{on}\ \ \p\Omega.
\end{align}
Thanks to the precise cancellation of \eqref{final 05},  all but the boundary contribution
\begin{align}\label{final 75}
-\e^{-1}\bbrb{\nx\psi:\b,(1-\pp)[\re]}{\gamma_+}
\end{align}
in $\e^{-1}\times\text{\eqref{pp 16=}}+\text{\eqref{pp 10=}}$ are under control. We note that $(1-\pp)[\re]$ is of $
O(\e^{\frac{1}{2}})$ and $\nx\psi:\b$ is of $O(1)$, so there is still a loss of $\e^{\frac{1}{2}}$ here for closure.\\

\paragraph{\underline{$\e$-cutoff boundary layer $\gbg$ and Interior $\gg$ Compensation}}
Through an extensive effort, the original loss of $\e^{-1}$ in $-\e^{-1}5P\br{\nx\tq,\bb c}$ is now transferred to a
boundary loss of $\e^{-\frac{1}{2}}$ in $-\e^{-1}\bbrb{\nx\psi:\b,(1-\pp)[\re]}{\gamma_+}$. Motivated by the gain of $%
\e^{\frac{1}{2}}$ in $L^{2}$ for any boundary layers in the bulk, we carefully design an $\e$-cutoff boundary layer $\gbg$ and its interior counterpart $\gg$ to compensate such a loss with 
\begin{align}
\big(\gbg+\gg\big)\big|_{\gamma _{+}}=-\big(\nx\psi :\b\big)\big|_{\gamma _{+}}.
\end{align}
Thanks to the fact that $\p_n\psi_n\big|_{\p\Omega}=0$ as well as the parity of $\b$, we can
construct 
\begin{align}
    \gg=\mh(\vv)\big(\vv\cdot\bbq \big),\quad \nx\cdot\bbq=0.
\end{align}
This crucial and precise
structure and parity of $\gbg$ lead to two crucial cancellations in \eqref{final 04} and \eqref{final 06},
which ensure the final closure of estimates with no singular power of
$\e$: 
\begin{align}
    \e^{-1}\br{\mh v\cdot\nx\Big(\mhh\gg\Big),\re}=\e^{-1}\br{\mh v\cdot\nx\big(\bbq\cdot v\big),\re}=&0,\\
    \e^{-1}\br{\gbg,\Gamma\Big[\P\mh+c\left(\abs{v}^2-5T\right)\mh,\P\mh+c\left(\abs{v}^2-5T\right)\mh\Big]}=&0.
\end{align}

Even though boundary layer
approximations have been established in kinetic theory for matching \textit{given} boundary data, to our knowledge, this is the first time
boundary layer construction is based on (unknown) $\psi$ from \eqref{pp 20=} and \eqref{pp 23=} to estimate the remainder $R$ itself.\\

\paragraph{\underline{New $\e$-Cutoff Boundary Layer Estimates for Non-Convex Domain}}

\subparagraph{$\bullet$ Hardy Inequality with $\e$ Gain}
One of the important challenges in the hydrodynamic limit of \eqref{large system-} is the necessary inclusion of
$\e$-cutoff boundary layers $\fb_1$ and $\gbg$. In fact, the determination of the ghost-effect equations depends on 
$\fb_1$ implicitly. Unfortunately, in \cite{AA003,AA007,AA009}, it is discovered that for non-flat domains the classical boundary 
layer theory in kinetic theory breaks down, due again to the characteristic and singular nature of the grazing set. Even though an alternative satisfactory theory has been established in convex domains \cite{AA004,AA013, BB002}, the non-convex case is completely open. In the proof of Lemma \ref{lemma:k1}, via the Hardy's inequality, the most difficult contribution in $\e^{-1}\br{\hb,\re}$ (where $\hb$ denotes a generic quantity related to $\e$-cutoff boundary layer) is treated as
\begin{align}
\e^{-1}\br{\hb,c\left(\abs{v}^{2}-5T\right)\mh}
=&\e^{-1}\br{\hb,\z\left(\abs{v}^{2}-5T\right)\mh}-\e^{-1}\br{\hb,\e^{-1}\xi\left(\abs{v}^{2}-5T\right)\mh}  \\
=&\e^{-1}\br{\hb,\z^R\left(\abs{v}^{2}-5T\right)\mh}
+\e^{-1}\br{\hb,\z^S\left(\abs{v}^{2}-5T\right)\mh}-\e^{-2}\br{\hb,\xi\left(\abs{v}^{2}-5T\right)\mh}.\no
\end{align}
Thanks to the fact that $\z^R\in H_0^1$ and $\xi\in H_{0}^{2}$, we express
\begin{align}
    \frac{\z^R}{\mn}=\frac{1}{\mn}\int_0^{\mn}\p_{\mn}\z^R,\quad \frac{\xi}{\mn}=\frac{1}{\mn}\int_0^{\mn}\p_{\mn}\xi,
\end{align}
and apply Hardy's inequality \cite{Hardy1920, Masmoudi2011} along the normal $\mn$ direction to obtain
\begin{align}
    \nm{\frac{\z^R}{\mn}}_{L^2_{\mn}}\ls\nm{\p_{\mn}\z^R}_{L^2_{\mn}},\quad \nm{\frac{\xi}{\mn}}_{L^2_{\mn}}\ls\nm{\p_{\mn}\xi}_{L^2_{\mn}}.
\end{align}
Note that $\mn=\e\eta$ in the $\e$-cutoff boundary layer scaling can be
absorbed by $\hb$ to produce extra $\e$ in $\mn\hb=\e\eta\hb$:
\begin{align}
    \abs{\e^{-2}\br{\hb,\xi\left(\abs{v}^{2}-5T\right)\mh}}\ls \Big(\e^{-\frac{1}{2}}\nm{\eta\hb}_{L^2_{\mn}}\Big)\Big(\e^{-\frac{1}{2}}\nm{\p_{\mn}\xi}_{L^2_{\mn}}\Big)
\end{align}
where $\e^{-\frac{1}{2}}\nm{\eta\hb}_{L^2_{\eta}}$ is bounded thanks to the change-of-variable $\frac{\mn}{\e}=\eta$ which yields $\nm{\cdot}_{L^2_{\mn}}\ls\e^{\frac{1}{2}}\nm{\cdot}_{L^2_{\eta}}$, and $\e^{-\frac{1}{2}}\nm{\p_{\mn}\xi}_{L^2_{\mn}}$ is bounded via an interpolation of
$\a$-Hodge decomposition \eqref{tt 01} with $\e^{-1}\tnm{\xi}$.

\subparagraph{$\bullet$ BV Estimate to Bound $\varepsilon ^{-1}\br{\p_{\va}\fb_1,\gbg}$}
It is well known from the boundary layer theory that the sharp pointwise bound near the grazing set holds $\abs{\p_{\va}\fb_1}\approx \abs{\va}^{-1}$.
With a cutoff $\abs{\va}\geq\e$, $\nm{\p_{\va}\fb_1}_{L_{\va}^{1}}\approx \abs{\ln \e}$, which creates a fatal logarithmic loss for closure, even with $\varepsilon $ in $\eta $ integration. To overcome such a $\abs{\ln \e}$ loss, we
establish a new BV estimate for $\fb_1$ in Theorem \ref{boundary regularity}, which amounts to a
subtle but crucial gain in joint $(\eta, \va)$ integration with no
loss of $\abs{\ln \e}$ for the cutoff $\nm{\p_{\va}\fb_1}_{L_{\va}^{1}}\ls 1$. For example, we can bound \eqref{final 36}-type term
\begin{align}
    \e^{-1}\br{\p_{\va}\fb_1,\gbg}\ls \e^{-1}\bnm{\fb_1}\lnm{\gbg}\ls\lnm{\gbg}^2+\e^{-2}\bnm{\fb_1}^2\ls \oot\tnm{c}^2+\oot.
\end{align}
We design an $\e$-cutoff of the standard boundary-layer Milne solution to achieve this goal. Such a cutoff leads to a non-local commutator of $K$ as in \eqref{mm 00}, which needs to be carefully controlled (see Lemma \ref{ss3-estimate}). Moreover, since it is extremely difficult to go beyond the first order boundary layer approximation in general domains, the presence of boundary layer dictates $\e\mh\re$ in \eqref{aa 08}. 

We can bootstrap $L^2-L^6$ estimates to $L^{\infty}$ bound by the method developed in \cite{Esposito.Guo.Kim.Marra2013, Esposito.Guo.Kim.Marra2015, Guo.Jang.Jiang2010, Guo.Jang.Jiang2009}. We note that the singular boundary layer term $\p_{\va}\fb_1$ in $\ss$ (see \eqref{mm 00}) is bounded by $\oot\e^{-1}$ in $L^{\infty}$ norm, which is under control based on the mild formulation.  
\\

\paragraph{\underline{Well-Posedness and Positivity}} 
The main effort of our paper is to obtain an a priori estimate for $\xnm{\re}$. For the actual construction, we need to ensure positivity. Our method is different from those in \cite{Esposito.Guo.Marra2018} and \cite{Arkeryd.Nouri2007}. We prove it via a new artificial compensation \eqref{auxiliary system'}. The existence relies on the Fredholm alternative, which is subtle in $L^{\infty}$ space as in Proposition \ref{prop:linear}. We switch to the remainder around a suitable global Maxwellian to avoid the cubic velocity
growth from $\mhh\ab\cdot\frac{\nx\tq}{4\tq^2}$ in \eqref{pp 02}, and design a
huge artificial $\e^{-1}\pk_M[\rem]$ in \eqref{linear remainder'} to overcome a severe boundary error terms, whose
$\e^{-1}$ dependence does not affect the compactness property. Finally, the nonlinearity $\Gamma[\re,\re]$ has a strong effect, so we must employ the crucial gain of $L^6$ estimates in \cite{Esposito.Guo.Kim.Marra2015} to close the estimate.

In summary, we develop a systematic approach to study the steady Boltzmann equation near a {\it local Maxwellian} in a 3D bounded domain. Our analysis relies on elaborate and integrated schemes with several exact cancellations and sharp estimates with no room to spare. These new techniques have led to the final resolution to the diffusive limit of the
neutron transport equation in \cite{AA020}, an important open question for one of the most classical problems in the kinetic theory. Moreover, dynamical stability of new steady Boltzmann solution is being studied now with similar methods \cite{AA029}.

%%%%%%%%%%%%%%%%%%%%%%%%%%%%%%%%%%%%%%%%%%%%%%%%%%%%%%%%%%%%%%%%%%%%%%%%
\subsection{Appendix and Notation}
%%%%%%%%%%%%%%%%%%%%%%%%%%%%%%%%%%%%%%%%%%%%%%%%%%%%%%%%%%%%%%%%%%%%%%%%

Throughout this paper, $C>0$ denotes a constant that only depends on
the domain $\Omega$, but does not depend on the data or $\e$. It is
referred as universal and can change from one inequality to another.
When we write $C(z)$, it means a certain positive constant depending
on the quantity $z$. We write $a\ls b$ to denote $a\leq Cb$ and $a\gs b$ to denote $a\geq Cb$. Also, we write $a\simeq b$ if both $a\ls b$ and $a\gs b$ are valid.
Due to the complexity of the problem, we put all the notation, including definitions of $\k$, $\si$, $\cdots$ etc, in Appendix for the convenience of
the reader.

%%%%%%%%%%%%%%%%%%%%%%%%%%%%%%%%%%%%%%%%%%%%%%%%%%%%%%%%%%%%%%%%%%%%%%%%
\section{Boundary Layer Analysis}\label{sec: boundary-analysis}
%%%%%%%%%%%%%%%%%%%%%%%%%%%%%%%%%%%%%%%%%%%%%%%%%%%%%%%%%%%%%%%%%%%%%%%%

In this section, we consider the Milne problem for $\g(\eta,\iota_1,\iota_2,\vvv)$ (see Section \ref{sec:geometric-setup}):
\begin{align}\label{Milne}
    \left\{
    \begin{array}{l}
    \va\dfrac{\p\g }{\p\eta}+\nu_w\g- K_w\left[\g\right]=0,\\\rule{0ex}{1.2em}
    \g(0,\iota_1,\iota_2,\vvv)=\hhh(\iota_1,\iota_2,\vvv)\ \ \text{for}\ \ \va>0,\quad
    \ds\int_{\r^3}\va\mbh(\iota_1,\iota_2,\vvv)\g(0,\iota_1,\iota_2,\vvv)\ud \vvv=0.
    \end{array}
    \right.
\end{align}
Here $\nu_w$ and $K_w$ are actually determined by the boundary Maxwellian $\mb$. In this section, we temporarily ignore the subscript $w$ for convenience. Also, when there is no confusion, we will not write $(\iota_1,\iota_2)$ explicitly.

\begin{theorem}\label{boundary well-posedness}
Assume $\hhh\in W^{k,\infty}_{\iota_1,\iota_2}$ for some $k\in\mathbb{N}$. Then there exists $\g_{\infty}(\vvv)\in\nk$ and a unique solution $\g(\eta,\vvv)$ to \eqref{Milne} such that $\ggg:=\g-\g_{\infty}$ satisfies 
\begin{align}\label{Milne.}
        \va\dfrac{\p\ggg }{\p\eta}+\nu \ggg-K\left[\ggg\right]=0,\quad
        \ggg(0,\vvv)=\hhh(\vvv)-   \g_{\infty}(\vvv):=\hh(\vvv)\ \text{for}\ \va>0,\quad
        \ds\int_{\r^3}\va\mh(\vvv)\ggg(0,\vvv)\ud \vvv=0.
\end{align}
and for some $K_0>0$ and any $0<\N\leq k$
\begin{align}
    \abs{\g_{\infty}}+\lnmm{\ue^{K_0\eta}\ggg}\ls& \lnmms{\hhh}{\gamma_-},\label{final 41}\\ \lnmm{\ue^{K_0\eta}\va\p_{\eta}\ggg}+\lnmm{\ue^{K_0\eta}\va\p_{\va}\ggg}\ls& \lnmms{\hhh}{\gamma_-}+\lnmms{\nabla_{\vvv}\hhh}{\gamma_-},\label{final 42}\\
    \lnmm{\ue^{K_0\eta}\p_{\vb}\ggg}+\lnmm{\ue^{K_0\eta}\p_{\vc}\ggg}\ls& \lnmms{\hhh}{\gamma_-}+\lnmms{\nabla_{\vvv}\hhh}{\gamma_-},\label{final 43}\\
    \lnmm{\ue^{K_0\eta}\p_{\iota_1}^{\N}\ggg}+\lnmm{\ue^{K_0\eta}\p_{\iota_2}^{\N}\ggg}\ls& \lnmms{\hhh}{\gamma_-}+\sum_{j=1}^{\N}\lnmms{\p_{\iota_1}^j\hhh}{\gamma_-}+\sum_{j=1}^{\N}\lnmms{\p_{\iota_2}^j\hhh}{\gamma_-}.\label{final 44}
\end{align}
Also, if $\hhh$ is odd in $\vb$ and $\vc$, then $\g$, $\ggg$ and $\g_{\infty}$ are also odd in $\vb$ and $\vc$.
\end{theorem}

\begin{proof}
    This is essentially Theorem 2.6 in the companion paper \cite{AA024}. Based on \cite{Bardos.Caflisch.Nicolaenko1986} and \cite{BB002}, we have the well-posedness of \eqref{Milne} as well as estimates \eqref{final 41}\eqref{final 42} and \eqref{final 43}. Then \eqref{final 44} follows from taking $\iota_i$ derivatives for $i=1,2$ on both sides of \eqref{Milne.} and study the resulting Milne problem.
    The oddness in Lemma \ref{lemma:m6} comes from the invariance of the collision kernel and the uniqueness. When we make the transformation $\vb\rt-\vb$ or $\vc\rt-\vc$, the equation in \eqref{Milne=} remains unchanged but the sign of the boundary data is flipped \cite{Glassey1996, Cercignani.Illner.Pulvirenti1994}.
\end{proof}

Note that $\p_{\eta}\ggg$ and $\p_{\va}\ggg$ estimates require the extra weight $\va$. In this section, we will focus on the BV regularity of $\ggg$ without the weight.

For $f(\eta,\iota_1,\iota_2,\vvv)$, denote the semi-norm
\begin{align}
    \nm{f}_{\widetilde{\text{BV}}}:=\sup\bigg\{\iint_{\eta,\vvv}f(\nabla_{\eta,\vvv}\cdot\psi)\ud\eta\ud\vvv:\ \psi\in C^1_c\ \text{and}\ \nm{\psi}_{L^{\infty}}\leq 1\bigg\},
\end{align}
and thus the BV norm can be defined as
\begin{align}
    \bnm{f}:=\pnm{f}{1}+\nm{f}_{\widetilde{\text{BV}}}.
\end{align}
It is classical that $W^{1,1}\hookrightarrow \text{BV}$, and Theorem \ref{boundary well-posedness} helps bound $L^1$ norm of $\ggg$, $\ds\frac{\p\ggg}{\p\iota_1}$, $\ds\frac{\p\ggg}{\p\iota_2}$, $\ds\frac{\p\ggg}{\p\vb}$, $\ds\frac{\p\ggg}{\p\vc}$. Hence, in the following, we will focus on $L^1$ estimate of $\ds\frac{\p\ggg}{\p\eta}$ and $\ds\frac{\p\ggg}{\p\va}$.

%%%%%%%%%%%%%%%%%%%%%%%%%%%%%%%%%%%%%%%%%%%%%%%%%%%%%%%%%%%%%%%%%%%%%%%%
\subsection{$L$-Approximate Milne Problem}
%%%%%%%%%%%%%%%%%%%%%%%%%%%%%%%%%%%%%%%%%%%%%%%%%%%%%%%%%%%%%%%%%%%%%%%%

In order to study the regularity of \eqref{Milne.}, we first consider an approximate Milne problem for $\ggl(\eta,\vvv)$ with $\eta\in[0,L]$ and $\vvv\in\r^3$:
\begin{align}\label{Milne=}
    \va\dfrac{\p\ggl }{\p\eta}+\nu \ggl-K\left[\ggl\right]=0,\quad
    \ggl(0,\vvv)=\hh(\vvv)\ \text{for}\ \va>0,\quad
    \ggl(L,\vvv)=\ggl(L,\rr[\vvv]),
\end{align}
Here $\rr[\vvv]=(-\va,\vb,\vc)$ and $L>0$. 
This is justified since $\g_{\infty}\in\nk$ with no $\va$ component, then $\g_{\infty}$ satisfies both \eqref{Milne=} and specular-reflection boundary condition at $L$. Therefore, we deduce from the constructions of the $L$- approximate solutions in \cite{Bardos.Caflisch.Nicolaenko1986} and \cite{BB002} that: 
\begin{proposition}\label{lemma:m6}
There exists $\ggl_{L}(\vvv)\in\nk$ and a unique solution $\ggl(\eta,\vvv)$ to \eqref{Milne=} such that
\begin{align}
    \abs{\ggl_{L}}+\lnmm{\ue^{K_0\eta}\big(\ggl-\ggl_L\big)}\ls \lnmms{\hh}{\gamma_-}.
\end{align}
Also, if $\hh$ is odd in $\vb$ and $\vc$, then $\ggl$ and $\ggl_L$ are also odd in $\vb$ and $\vc$.
\end{proposition}

We can define an extension of $\ggl$ by letting $\ggl=0$ for $\eta>L$ (for the convenience of assigning norms). Based on \cite{Bardos.Caflisch.Nicolaenko1986}, \cite{BB002} and \cite[Lemma 4.4]{AA004}, we have 
\begin{lemma}\label{lemma:m8}
We have $\lnmmss{\ue^{K_0L}\ggl_L}\rt0$ as $L\rt\infty$, and $\ggl\rt\ggg$ weakly in $L^2([0,\infty)\times\r^3)$ as $L\rt\infty$.
\end{lemma}

%%%%%%%%%%%%%%%%%%%%%%%%%%%%%%%%%%%%%%%%%%%%%%%%%%%%%%%%%%%%%%%%%%%%%%%%
\subsection{$(L,\te)$-Approximate Milne Problem}
%%%%%%%%%%%%%%%%%%%%%%%%%%%%%%%%%%%%%%%%%%%%%%%%%%%%%%%%%%%%%%%%%%%%%%%%

In order to study \eqref{Milne=}, we further consider a truncated approximate Milne problem for $\ggt(\eta,\vvv)$ with $\eta\in[0,L]$ and $\vvv\in\r^3$:
\begin{align}\label{Milne approximate}
    \left\{
    \begin{array}{l}
    \va\dfrac{\p\ggt }{\p\eta}+\nu \ggt-K\left[\ggt\right]=-\chi\left(\te^{-1}\va\right)K\left[\ggt\right],\\\rule{0ex}{1.0em}
    \ggt(0,\vvv)=\ch\left(\te^{-1}\va\right)\hh(\vvv)\ \ \text{for}\ \ \va>0,\quad
    \ggt(L,\vvv)=\ggt(L,\rr[\vvv]).
    \end{array}
    \right.
\end{align}
Here $0<\te\ll1$. Clearly, the first equation of \eqref{Milne approximate} is equivalent to
\begin{align}\label{mm 04}
    \va\dfrac{\p\ggt }{\p\eta}+\nu \ggt-\ch\left(\te^{-1}\va\right)K\left[\ggt\right]=0.
\end{align}
Compared with \eqref{Milne.}, the key difference lies in the cutoff $\ch\left(\te^{-1}\va\right)$. This term breaks the orthogonality, and thus we do not need to prescribe the mass-flux at $\eta=0$. Instead since $\ds\int_{\r^3}\va\mh(\vvv)\ggt(L,\vvv)\ud\vvv=0$,
by integrating $\vvv\in\r^3$ on both sides of \eqref{mm 04}, we have
\begin{align}
    \int_{\r^3}\va\mh(\vvv)\ggt(0,\vvv)\ud\vvv=-\int_0^L\int_{\r^3}\mh(\vvv)\chi\left(\te^{-1}\va\right)K\left[\ggt\right](\eta,\vvv)\ud\eta\ud\vvv.
\end{align}

\begin{lemma}\label{remark boundary}
    We have
    $\ggt(\eta,\vvv)=0$ for any $\eta\in[0,L]$ and $-\te\leq\va<\te$.
\end{lemma}
\begin{proof}
Note that the characteristics of the approximate Milne problem \eqref{mm 04} are straight lines $\va=\text{const}$.
\begin{itemize}
    \item 
    For any $\eta\in[0,L]$ and $0<\va\leq \te$, we have
    \begin{align}
        \ggt(\eta,\vvv)=\ue^{-\frac{\nu}{\va}\eta}\ggt(0,\vvv)+\int_0^{\eta}\ue^{-\frac{\nu}{\va}(\eta-y)}\ch\left(\te^{-1}\va\right)K\left[\ggt\right](y,\vvv)\ud y.
    \end{align}
    Considering $\ggt(0,\vvv)=\ch\left(\te^{-1}\va\right)\hh(\vvv)=0$ and $\ch\left(\te^{-1}\va\right)K\left[\ggt\right]=0$ for $0\leq\va\leq \te$, we have $\ggt(\eta,\vvv)=0$ for any $\eta\in[0,L]$ and $0\leq\va\leq \te$. 
    \item
    For any $\eta\in[0,L]$ and $-\te\leq\va< 0$, we have
    \begin{align}
        \ggt(\eta,\vvv)=\ue^{\frac{\nu}{\va}(L-\eta)}\ggt(L,\vvv)+\int_{\eta}^L\ue^{\frac{\nu}{\va}(y-\eta)}\ch\left(\te^{-1}\va\right)K\left[\ggt\right](y,\vvv)\ud y.
    \end{align}
    Due to the specular reflection boundary condition at $\eta=L$, we have $\ggt(L,\vvv)=0$ for $-\te\leq\va< 0$. Also, considering $\ch\left(\te^{-1}\va\right)K\left[\ggt\right]=0$ for $0\leq\va\leq \te$, we have
    $\ggt(\eta,\vvv)=0$ for any $\eta\in[0,L]$ and $-\te\leq\va< 0$.
\end{itemize}
\end{proof}

\begin{lemma}[$L^2$ Bound]\label{lemma:m1'}
Assume $\te\ll L^{-8}$. Then there exists $\ggt_{L}(\vvv)\in\nk$ and a unique solution $\ggt(\eta,\vvv)$ to \eqref{Milne approximate} satisfying (uniformly in $\te$ and $L$)
\begin{align}
    \lnmmss{\ggt_L}+\tnm{\ggt-\ggt_L}\ls \tnms{\hh}{\gamma_-}.
\end{align}
\end{lemma}
\begin{proof}
This proof follows from a similar argument as in \cite[Section 4.1.1]{BB002} with the source term $\sss^{L,\te}:=-\chi\left(\te^{-1}\va\right)K\left[\ggt\right]$.
Based on \cite[(3.101)]{Glassey1996}, we have
\begin{align}
    \tnm{\sss^{L,\te}}\ls\te^{\frac{1}{2}}\tnm{K\left[\ggt\right]}\ls\te^{\frac{1}{2}}\tnm{\ggt}.
\end{align}
We split $\ggt=:\wwt+\qqt$, where $\qqt\in\nk$ and $\wwt\in\nnk$.
Following the similar argument as the proof of \cite[Lemma 4.1.5]{BB002} (and noticing that $\sss^{L,\te}$ is not necessarily in $\nnk$), we have
\begin{align}\label{mm 05'}
    \um{\wwt}^2\ls& \tnms{\hh}{\gamma_-}^2+\int_0^L\ggt\sss^{L,\te}
    \ls\tnms{\hh}{\gamma_-}^2+\te^{-\frac{1}{2}}\tnm{\sss^{L,\te}}^2+\te^{\frac{1}{2}}\tnm{\wwt}^2+\te^{\frac{1}{2}}\tnm{\qqt}^2,
\end{align}
which yields for $\te\ll1$ 
\begin{align}
    \um{\wwt}\ls \tnms{\hh}{\gamma_-}+\te^{\frac{1}{4}}\tnm{\ggt}+\te^{\frac{1}{4}}\tnm{\qqt}.
\end{align}
Following the similar argument as the proof of \cite[Lemma 4.1.7, Lemma 4.1.10]{BB002} and considering that $L$ is fixed, we obtain that there exists $q_L^{\te}(\vvv)\in\nk$ such that
\begin{align}\label{mm 06}
    \abs{q_L^{\te}}\ls&\tnms{\hh}{\gamma_-}+L \tnm{\sss^{L,\te}}+L\tnm{\wwt}
    \ls\tnms{\hh}{\gamma_-}+\te^{\frac{1}{4}} L\tnm{\ggt}+\te^{\frac{1}{4}}L\tnm{\qqt},
\end{align}
and
\begin{align}\label{mm 07'}
    \tnm{\qqt-q_L^{\te}}\ls& \tnms{\hh}{\gamma_-}+\bigg[\int_0^L\bigg(\int_{\eta}^L\sss^{L,\te}(y)\bigg)^2\ud\eta\bigg]^{\frac{1}{2}}+\bigg[\int_0^L\bigg(\int_{\eta}^L\wwt(y)\bigg)^2\ud\eta\bigg]^{\frac{1}{2}}\\
    \ls& \tnms{\hh}{\gamma_-}+ L\tnm{\sss^{L,\te}}+ L\tnm{\wwt}
    \ls\tnms{\hh}{\gamma_-}+\te^{\frac{1}{4}} L\tnm{\ggt}+\te^{\frac{1}{4}}L\tnm{\qqt}.\no
\end{align}
Hence, summarizing \eqref{mm 05'}, \eqref{mm 06} and \eqref{mm 07'}, letting $\ggt_L=q_L^{\te}$, we have 
\begin{align}\label{mm 08'}
    \abs{\ggt_L}+\tnm{\ggt-\ggt_L}\ls
    \ls&\tnms{\hh}{\gamma_-}+\te^{\frac{1}{4}}L^2\abs{\ggt_L}+\te^{\frac{1}{4}}L\tnm{\ggt-\ggt_L},
\end{align}
which, when $\te\ll L^{-8}$, yields the desired result.
\end{proof}

\begin{lemma}[$L^2$ Decay]\label{lemma:m1}
Assume $\te\ll L^{-8}\ue^{-4K_0L}$ with $0\leq K_0<1$. Then the unique solution $\ggt(\eta,\vvv)$ to \eqref{Milne approximate} satisfies (uniformly in $\te$ and $L$)
\begin{align}
    \tnm{\ue^{K_0\eta}\big(\ggt-\ggt_L\big)}\ls \tnms{\hh}{\gamma_-}.
\end{align}
\end{lemma}
\begin{proof}
Based on \cite[(3.101)]{Glassey1996}, we have
\begin{align}
    \tnm{\ue^{K_0\eta}\sss^{L,\te}}\ls\te^{\frac{1}{2}}\tnm{K\left[\ue^{K_0\eta}\ggt\right]}\ls\te^{\frac{1}{2}}\tnm{\ue^{K_0\eta}\ggt}.
\end{align}
Following the similar argument as the proof of \cite[Lemma 4.1.9]{BB002}, we have
\begin{align}\label{mm 05}
    \um{\ue^{K_0\eta}\wwt}\ls& \tnms{\hh}{\gamma_-}+\te^{-\frac{1}{4}}\tnm{\ue^{K_0\eta}\sss^{L,\te}}+\te^{\frac{1}{4}}\tnm{\ue^{K_0\eta}\qqt}
    \ls \tnms{\hh}{\gamma_-}+\te^{\frac{1}{4}}\tnm{\ue^{K_0\eta}\ggt}+\te^{\frac{1}{4}}\tnm{\ue^{K_0\eta}\qqt}.
\end{align}
Following the similar argument as the proof of \cite[Lemma 4.1.10]{BB002}, we have
\begin{align}\label{mm 07}
    \tnm{\ue^{K_0\eta}\left(\qqt-q_L^{\te}\right)}\ls& \tnms{\hh}{\gamma_-}+\bigg[\int_0^L\bigg(\int_{\eta}^L\ue^{K_0y}\sss^{L,\te}(y)\bigg)^2\ud\eta\bigg]^{\frac{1}{2}}+\bigg[\int_0^L\bigg(\int_{\eta}^L\ue^{K_0y}\wwt(y)\bigg)^2\ud\eta\bigg]^{\frac{1}{2}}\\
    \ls& \tnms{\hh}{\gamma_-}+L\tnm{\ue^{K_0\eta}\sss^{L,\te}}+ L\tnm{\ue^{K_0\eta}\wwt}
    \ls\tnms{\hh}{\gamma_-}+\te^{\frac{1}{4}} L\tnm{\ue^{K_0\eta}\ggt}+\te^{\frac{1}{4}}L\tnm{\ue^{K_0\eta}\qqt}.\no
\end{align}
Hence, summarizing \eqref{mm 05} and \eqref{mm 07}, we have 
\begin{align}\label{mm 08}
    \tnm{\ue^{K_0\eta}\big(\ggt-\ggt_L\big)}
    \ls&\tnms{\hh}{\gamma_-}+\te^{\frac{1}{4}}L^2\ue^{K_0L}\abs{\ggt_L}+\te^{\frac{1}{4}}L\tnm{\ue^{K_0\eta}\big(\ggt-\ggt_L\big)},
\end{align}
which, when $\te^{\frac{1}{4}}\ll L^{-2}\ue^{-K_0L}$, yields the desired result.
\end{proof}

\begin{lemma}[$L^{\infty}$ Bound and Decay]\label{lemma:m2}
There exists $\ggt_{L}(\vvv)\in\nk$ and a unique solution $\ggt(\eta,\vvv)$ to \eqref{Milne approximate} satisfying (uniformly in $\te$ and $L$)
\begin{align}
    \lnmm{\ue^{K_0\eta}\big(\ggt-\ggt_L\big)}\ls \lnmms{\hh}{\gamma_-},
\end{align}
\end{lemma}
\begin{proof}
This follows from \cite[Theorem 4.1.24]{BB002}.
\end{proof}

Summarizing Lemma \ref{lemma:m1'}, Lemma \ref{lemma:m1} and Lemma \ref{lemma:m2}, we arrive at the following:

\begin{proposition}\label{prop:Milne-approximate}
Assume $\te\ll L^{-8}\ue^{-4K_0L}$ with $0\leq K_0<1$. Then the unique solution $\ggt(\eta,\vvv)$ to \eqref{Milne approximate} satisfying
\begin{align}
    \lnmm{\ggt}\ls \lnmms{\hh}{\gamma_-}.
\end{align}
\end{proposition}

\begin{remark}\label{rem:boundary}
Due to the presence of $\ggt_L$, we cannot simply deduce that 
\begin{align}
    \tnm{\ue^{K_0\eta}\ggt}+\lnmm{\ue^{K_0\eta}\ggt}\ls \lnmms{\hh}{\gamma_-}.
\end{align}
\end{remark}

%%%%%%%%%%%%%%%%%%%%%%%%%%%%%%%%%%%%%%%%%%%%%%%%%%%%%%%%%%%%%%%%%%%%%%%%
\subsection{$L^1$ Estimate for $\p_{\eta}\ggt$}
%%%%%%%%%%%%%%%%%%%%%%%%%%%%%%%%%%%%%%%%%%%%%%%%%%%%%%%%%%%%%%%%%%%%%%%%

\begin{lemma}\label{lemma:m3}
$\ggt$ satisfies (uniformly in $\te$ and $L$)
\begin{align}
    \pnm{\nu\p_{\eta}\ggt }{1}\ls\pnm{\nu^2\ggt}{1}+\nm{\nu\ggt}_{L^1_{\eta}L^{2}_{\vvv}}+\lnmm{\ggt}.
\end{align}
\end{lemma}

\begin{proof}
For some $0<\d\ll1$, considering \eqref{mm 04}, we have
\begin{align}\label{mm 11}
    \pnm{\nu\p_{\eta}\ggt \id_{\{\abs{\va}\geq\d\}}}{1}\ls& \d^{-1}\pnm{\nu\va\p_{\eta}\ggt }{1}
    \ls\d^{-1}\pnm{\nu^2\ggt}{1}+\d^{-1}\pnm{\nu K\left[\ggt\right]}{1}\ls\d^{-1}\pnm{\nu^2\ggt}{1},
\end{align}
and thus it suffices to bound $\pnm{\nu\p_{\eta}\ggt \id_{\{\abs{\va}\leq\d\}}}{1}$.
Taking $\eta$ derivative in \eqref{mm 04}, we obtain
\begin{align}\label{mm 09}
    \va\p_{\eta}\left(\p_{\eta}\ggt \right)+\nu \left(\p_{\eta}\ggt \right)=\ch\left(\te^{-1}\va\right)K\left[\p_{\eta}\ggt \right].
\end{align}
This equation holds for any $(\eta,\vvv)$, so we focus on the $\abs{\va}\leq\d$ part:
\begin{align}\label{mm 02}
    &\va\p_{\eta}\left(\p_{\eta}\ggt \id_{\{\abs{\va}\leq\d\}}\right)+\nu \left(\p_{\eta}\ggt \id_{\{\abs{\va}\leq\d\}}\right)=\ch\left(\te^{-1}\va\right)\id_{\{\abs{\va}\leq\d\}}K\left[\p_{\eta}\ggt \right]\\
    =&\ch\left(\te^{-1}\va\right)\id_{\{\abs{\va}\leq\d\}}\int_{\abs{\vuu_{\eta}}\geq\d}k(\vvv,\vuu)\p_{\eta}\ggt (\vuu)\ud \vuu+\ch\left(\te^{-1}\va\right)\id_{\{\abs{\va}\leq\d\}}\int_{\abs{\vuu_{\eta}}\leq\d}k(\vvv,\vuu)\p_{\eta}\ggt (\vuu)\ud \vuu.\no
\end{align}
Here, $\vuu=(\vuu_{\eta},\vuu_{\phi},\vuu_{\psi})$. By multiplying $\text{sgn}\left[\p_{\eta}\ggt \id_{\{\abs{\va}\leq\d\}}\right]$ and integrating over $(\eta,\vvv)$ in \eqref{mm 02}, we obtain
\begin{align}
    \pnm{\nu\p_{\eta}\ggt \id_{\{\abs{\va}\leq\d\}}}{1}\ls&\int_{\r^3}\abs{\p_{\eta}\ggt(0)\id_{\{\abs{\va}\leq\d\}}}\va\ud\vvv-\int_{\r^3}\abs{\p_{\eta}\ggt(L)\id_{\{\abs{\va}\leq\d\}}}\va\ud\vvv\\
    &+\BBnm{\ch\left(\te^{-1}\va\right)\id_{\{\abs{\va}\leq\d\}}\int_{\abs{\vuu_{\eta}}\geq\d}k(\vvv,\vuu)\p_{\eta}\ggt (\vuu)\ud \vuu}_{L^1}\no\\
    &+\BBnm{\ch\left(\te^{-1}\va\right)\id_{\{\abs{\va}\leq\d\}}\int_{\abs{\vuu_{\eta}}\leq\d}k(\vvv,\vuu)\p_{\eta}\ggt (\vuu)\ud\vuu}_{L^1}.\no
\end{align}
Plugging in \eqref{Milne approximate} yields
\begin{align}\label{mm 10}
    \int_{\r^3}\abs{\p_{\eta}\ggt(0)\id_{\{\abs{\va}\leq\d\}}}\va\ud\vvv\ls& \lnmms{\nu\ggt(0)}{\gamma}+\lnmms{\chi\left(\te^{-1}\va\right)K\big[\ggt\big](0)}{\gamma}
    \ls\lnmms{\ggt(0)}{\gamma}\ls\lnmm{\ggt},
\end{align}
and similarly
\begin{align}
    \int_{\r^3}\abs{\p_{\eta}\ggt(L)\id_{\{\abs{\va}\leq\d\}}}\va\ud\vvv\ls\lnmm{\ggt}.
\end{align}
Also, we know
\begin{align}
    &\BBnm{\ch\left(\te^{-1}\va\right)\id_{\{\abs{\va}\leq\d\}}\int_{\abs{\vuu_{\eta}}\geq\d}k(\vvv,\vuu)\p_{\eta}\ggt (\vuu)\ud \vuu}_{L^1}\\
    \ls&\BBnm{\int_{\abs{\vuu_{\eta}}\geq\d}k(\vvv,\vuu)\ud \vuu}_{L^2_{\vvv}}\nm{\p_{\eta}\ggt \id_{\{\abs{\va}\geq\d\}}}_{L^1_{\eta}L^{2}_{\vvv}}
    \ls \d^{-1}\nm{\va\p_{\eta}\ggt }_{L^1_{\eta}L^{2}_{\vvv}}\ls\d^{-1}\nm{\nu\ggt}_{L^1_{\eta}L^{2}_{\vvv}}.\no
\end{align}
The remaining principal term is
\begin{align}\label{final 24}
    &\BBnm{\ch\left(\te^{-1}\va\right)\id_{\{\abs{\va}\leq\d\}}\int_{\abs{\vuu_{\eta}}\leq\d}k(\vvv,\vuu)\p_{\eta}\ggt (\vuu)\ud \vuu}_{L^1}\\
    \ls&\left(\sup_{\vuu}\int_{\r^3}\id_{\{\abs{\va}\leq\d\}}k(\vvv,\vuu)\ud \vvv\right)\pnm{\id_{\{\abs{\vuu_{\eta}}\leq\d\}}\p_{\eta}\ggt(\vuu)}{1}
    \ls\d\pnm{\p_{\eta}\ggt \id_{\{\abs{\va}\leq\d\}}}{1},\no
\end{align}
where for $\vvv'=(\vb,\vc)$ and $\vuu'=(u_{\phi},u_{\psi})$, due to \cite[Lemma 3]{Guo2010}
\begin{align}\label{final 24'}
    \sup_{\vuu}\int_{\r^3}\id_{\{\abs{\va}\leq\d\}}k(\vvv,\vuu)\ud \vvv\ls&\sup_{\vuu}\int_{\r^3}\id_{\{\abs{\va}\leq\d\}}\abs{\vvv-\vuu}^{-1}\ue^{-\frac{1}{8}\abs{\vvv-\vuu}^2}\ud \vvv\\
    \ls&\sup_{\vuu}\int_{\r}\id_{\{\abs{\va}\leq\d\}}\bigg(\int_{\r^2}\abs{\vvv'-\vuu'}^{-1}\ue^{-\frac{1}{8}\abs{\vvv-\vuu}^2}\ud \vvv'\bigg)\ud\va
    \ls\sup_{\vuu}\int_{\r}\id_{\{\abs{\va}\leq\d\}}\ud\va \ls\d.\no
\end{align}
In summary, we have shown that
\begin{align}
    \pnm{\nu\p_{\eta}\ggt \id_{\{\abs{\va}\leq\d\}}}{1}\ls \d^{-1}\pnm{\nu\ggt}{1}+\d^{-1}\nm{\nu\ggt}_{L^1_{\eta}L^{2}_{\vvv}}+\d\pnm{\p_{\eta}\ggt \id_{\{\abs{\va}\leq\d\}}}{1}+\lnmm{\ggt},
\end{align}
which yields
\begin{align}\label{mm 12}
    \pnm{\nu\p_{\eta}\ggt \id_{\{\abs{\va}\leq\d\}}}{1}\ls \d^{-1}\pnm{\nu\ggt}{1}+\d^{-1}\nm{\nu\ggt}_{L^1_{\eta}L^{2}_{\vvv}}.
\end{align}
Combining \eqref{mm 11} and \eqref{mm 12}, we have
\begin{align}
    \pnm{\nu\p_{\eta}\ggt }{1}\ls \d^{-1}\pnm{\nu^2\ggt}{1}+\d^{-1}\nm{\nu\ggt}_{L^1_{\eta}L^{2}_{\vvv}}+\lnmm{\ggt}.
\end{align}
\end{proof}

%%%%%%%%%%%%%%%%%%%%%%%%%%%%%%%%%%%%%%%%%%%%%%%%%%%%%%%%%%%%%%%%%%%%%%%%
\subsection{$L^1$ Estimate for $\p_{\va}\ggt$}
%%%%%%%%%%%%%%%%%%%%%%%%%%%%%%%%%%%%%%%%%%%%%%%%%%%%%%%%%%%%%%%%%%%%%%%%

\begin{lemma}\label{lemma:m4}
$\ggt$ satisfies (uniformly in $\te$ and $L$)
\begin{align}
    \pnm{\nu\p_{\va}\ggt }{1}\ls&\pnm{\nu^2\ggt}{1}+\nm{\nu\ggt}_{L^1_{\eta}L^{2}_{\vvv}}+\lnmm{\ggt}
    +\int_{\va>0}\abs{\p_{\va}\hh }\va\ud\vvv+\int_{\va>0}\abs{\hh}\ud\vvv.
\end{align}
\end{lemma}

\begin{proof}
Taking $\va$ derivative in \eqref{Milne approximate}, we obtain
\begin{align}\label{mm 03}
    \va\p_{\eta}\left(\p_{\va}\ggt \right)+\p_{\eta}\ggt +\nu\p_{\va}\ggt +\p_{\va}\nu\ggt=\ch\left(\te^{-1}\va\right)\p_{\va}K\left[\ggt\right]+\te^{-1}\ch'\left(\te^{-1}\va\right)K\left[\ggt\right].
\end{align}
By multiplying $\text{sgn}\left[\p_{\va}\ggt \right]$ and integrating over $(\eta,\vv)$ in \eqref{mm 03}, 
we obtain
\begin{align}
    \pnm{\nu\p_{\va}\ggt }{1}\ls&\abs{\int_{\r^3}\abs{\p_{\va}\ggt(0)}\va\ud\vvv-\int_{\r^3}\abs{\p_{\va}\ggt(L)}\va\ud\vvv}+\pnm{\p_{\eta}\ggt}{1}+\pnm{\p_{\va}\nu\ggt}{1}\\
    &+\pnm{\ch\left(\te^{-1}\va\right)\p_{\va}K\left[\ggt\right]}{1}+\pnm{\te^{-1}\ch'\left(\te^{-1}\va\right)K\left[\ggt\right]}{1}.\no
\end{align}
Using Lemma \ref{remark boundary}, 
we have
\begin{align}
    \int_{\r^3}\abs{\p_{\va}\ggt(0)}\va\ud\vvv=&\int_{\va>0}\abs{\p_{\va}\ggt(0)}\va\ud\vvv+\int_{\va<0}\abs{\p_{\va}\ggt(0)}\va\ud\vvv\\
    \leq&\int_{\va>0}\abs{\p_{\va}\ggt(0)}\va\ud\vvv=\int_{\va>0}\abs{\p_{\va}\Big(\ch\left(\te^{-1}\va\right)\hh(\vvv)\Big)}\va\ud\vvv\no\\
    \leq&\int_{\va>0}\abs{\ch\left(\te^{-1}\va\right)\p_{\va}\hh }\va\ud\vvv+\int_{\va>0}\abs{\te^{-1}\ch'\left(\te^{-1}\va\right)\hh}\va\ud\vvv
    \ls\int_{\va>0}\abs{\p_{\va}\hh }\va\ud\vvv+\int_{\va>0}\abs{\hh}\ud\vvv.\no
\end{align}
Due to specular-reflection boundary at $\eta=L$, we have $\ds\int_{\r^3}\abs{\p_{\va}\ggt(L)}\va\ud\vvv=0$.
Lemma \ref{lemma:m3} yields
\begin{align}
    \pnm{\p_{\eta}\ggt}{1}\ls \pnm{\nu^2\ggt}{1}+\nm{\nu\ggt}_{L^1_{\eta}L^{2}_{\vvv}}+\lnmm{\ggt}.
\end{align}
Direct estimate implies $\pnm{\p_{\va}\nu\ggt}{1}\ls\pnm{\ggt}{1}$ and $\pnm{\ch\left(\te^{-1}\va\right)\p_{\va}K\left[\ggt\right]}{1}\ls\pnm{\nu\ggt}{1}$.
The remaining principal term is
\begin{align}
    \pnm{\te^{-1}\ch'\left(\te^{-1}\va\right)K\left[\ggt\right]}{1}\ls&\te^{-1}\pnm{\ch'\left(\te^{-1}\va\right)\int_{\r^3}k(\vvv,\vuu)\ggt(\vuu)\ud \vuu}{1}\\
    \ls&\te^{-1}\left(\sup_{\vuu}\int_{\r^3}\id_{\{\abs{\va}\leq\te\}}k(\vvv,\vuu)\ud \vvv\right)\pnm{\ggt}{1}
    \ls\pnm{\ggt}{1},\no
\end{align}
based on \eqref{final 24'} $\ds\sup_{\vuu}\int_{\r^3}\id_{\{\abs{\va}\leq\te\}}k(\vvv,\vuu)\ud \vvv\ls\te$.
In summary, we have shown that
\begin{align}
    \pnm{\nu\p_{\va}\ggt }{1}\ls\pnm{\nu^2\ggt}{1}+\nm{\nu\ggt}_{L^1_{\eta}L^{2}_{\vvv}}+\lnmm{\ggt}+\int_{\va>0}\abs{\p_{\va}\hh }\va\ud\vvv+\int_{\va>0}\abs{\hh}\ud\vvv.
\end{align}
\end{proof}

%%%%%%%%%%%%%%%%%%%%%%%%%%%%%%%%%%%%%%%%%%%%%%%%%%%%%%%%%%%%%%%%%%%%%%%%
\subsection{Limit $\te\rt0$}
%%%%%%%%%%%%%%%%%%%%%%%%%%%%%%%%%%%%%%%%%%%%%%%%%%%%%%%%%%%%%%%%%%%%%%%%

We formally take $\te\rt0$ in \eqref{Milne approximate} and recover \eqref{Milne=}.\\

\begin{lemma}\label{lem:m6}
We have
\begin{align}
    \lnmmss{\ggt_L-\ggl_L}+\tnm{\ue^{K_0\eta}\big(\ggt-\ggl\big)}\rt0\ \ \text{as}\ \ \te\rt0.
\end{align}
\end{lemma}

\begin{proof}
This can be easily shown from the difference equation of $\ggt-\ggl$:
\begin{align}\label{Milne approximate difference}
    \left\{
    \begin{array}{l}
    \va\dfrac{\p\big(\ggt-\ggl\big) }{\p\eta}+\nu\big(\ggt-\ggl\big)-\ch\left(\te^{-1}\va\right)K\left[\ggt-\ggl\right]=\chi\left(\te^{-1}\va\right)K\left[\ggl\right],\\\rule{0ex}{1.2em}
    \big(\ggt-\ggl\big)(0,\vvv)=-\chi\left(\te^{-1}\va\right)\hh(\vvv)\ \ \text{for}\ \ \va>0,\\\rule{0ex}{1.5em}
    \big(\ggt-\ggl\big)(L,\vvv)=\big(\ggt-\ggl\big)(L,\rr[\vvv]).
    \end{array}
    \right.
\end{align}
Using the proof of Lemma \ref{lemma:m1'}, Lemma \ref{lemma:m1} and Lemma \ref{lemma:m2}, we have
\begin{align}
    \lnmmss{\ggt_L-\ggl_L}+\tnm{\ue^{K_0\eta}\big(\ggt-\ggl\big)}\ls\te\tnms{\hh}{\gamma_-}+\te^{\frac{1}{4}}L\tnm{\ue^{K_0\eta}\big(\ggt-\ggt_L\big)}+\te^{\frac{1}{4}}\ue^{K_0L}\tnm{\ggt_L}.
\end{align}
Then for fixed $L$, the result naturally follows as $\te\rt0$.
\end{proof}

\begin{corollary}\label{prop:Milne-approximate-improve}
Assume $\te\ll L^{-8}\ue^{-4K_0L}$ with $0\leq K_0<1$. Then the unique solution $\ggt(\eta,\vvv)$ to \eqref{Milne approximate} satisfies
\begin{align}
    \tnm{\ue^{K_0\eta}\ggt}+\lnmm{\ggt}\ls \lnmms{\hh}{\gamma_-}.
\end{align}
\end{corollary}
\begin{proof}
    Combining Lemma \ref{lem:m6} and Proposition \ref{lemma:m6}, we obtain the $L^2$ bound. Then $L^{\infty}$ bound is from Proposition \ref{prop:Milne-approximate}.
\end{proof}

\begin{remark}\label{rem:boundary-uniform}
    Using H\"older's inequality and the utilizing the $L^2$ decay, the similar results as in Corollary \ref{prop:Milne-approximate-improve} also holds for any $\N\in[1,\infty)$:
    \begin{align}
    \pnm{\ue^{K_0\eta}\ggt}{\N}\ls \lnmms{\hh}{\gamma_-}.
\end{align}
\end{remark}

Next we turn to the BV estimates.
\begin{proposition}\label{lemma:m7}
We have
\begin{align}
    \bnm{\nu\ggl}\ls\lnmms{\hh}{\gamma_-}+\int_{\va>0}\abs{\p_{\va}\hh }\va\ud\vvv+\int_{\va>0}\abs{\hh}\ud\vvv.
\end{align}
\end{proposition}

\begin{proof}
Using Corollary \ref{prop:Milne-approximate-improve} and Remark \ref{rem:boundary-uniform}, we conclude that the estimates in Lemma \ref{lemma:m3} and Lemma \ref{lemma:m4} with $K_0=0$ are uniform in $\te$. In combination with Theorem \ref{boundary well-posedness}, we have shown that 
\begin{align}\label{final 45}
    \nm{\ggt}_{W^{1,1}}\ls\lnmms{\hh}{\gamma_-}+\int_{\va>0}\abs{\p_{\va}\hh }\va\ud\vvv+\int_{\va>0}\abs{\hh}\ud\vvv,
\end{align}
uniformly in $\te$, which yields a weakly $W^{1,1}$ convergent subsequence as $\te\rt0$. 
Also, Proposition \ref{lem:m6} yields strong $L^{2}$ convergence as $\te\rt0$.

Combining both of the convergence, we obtain the BV convergence of the subsequence $\ggt\rt\ggl$ as $\te\rt0$, and thus, with the help of weak lower semi-continuity, we can pass to limit in \eqref{final 45} to obtain the desired result.
\end{proof}

%%%%%%%%%%%%%%%%%%%%%%%%%%%%%%%%%%%%%%%%%%%%%%%%%%%%%%%%%%%%%%%%%%%%%%%%
\subsection{Limit $L\rt\infty$}
%%%%%%%%%%%%%%%%%%%%%%%%%%%%%%%%%%%%%%%%%%%%%%%%%%%%%%%%%%%%%%%%%%%%%%%%

Based on Lemma \ref{lemma:m8}, we take $L\rt\infty$ in \eqref{Milne=} and recover the equation \eqref{Milne.}.
\begin{theorem}\label{boundary regularity}
We have
\begin{align}
    \bnm{\nu\ggg}\ls\lnmms{\hh}{\gamma_-}+\int_{\va>0}\abs{\p_{\va}\hh }\va\ud\vvv+\int_{\va>0}\abs{\hh}\ud\vvv.
\end{align}
\end{theorem}

\begin{proof}
    For any $C_c^{\infty}$ test function $\phi(\eta,\vvv)$ with compact support contained in $(\eta,\vvv)\in [0,L']\times B_r$, choosing $L>L'$ and using Lemma \ref{lemma:m8}, we have as $L\rt\infty$
    \begin{align}
        \iint_{[0,L]\times\r^3}\frac{\p\ggl}{\p\eta}\phi=-\iint_{[0,L]\times\r^3}\ggl\frac{\p\phi}{\p\eta}\rt \iint_{[0,L]\times\r^3}\ggg\frac{\p\phi}{\p\eta}=\iint_{[0,L]\times\r^3}\frac{\p\ggg}{\p\eta}\phi.
    \end{align}
    Hence, we know that the weak derivative $\dfrac{\p\ggl}{\p\eta}\rt \dfrac{\p\ggg}{\p\eta}$ in the sense of measure. Similarly, we can show that the weak derivative $\dfrac{\p\ggl}{\p\va}\rt \dfrac{\p\ggg}{\p\va}$ in the sense of measure. In addition, we know that
    \begin{align}
        \nm{\ggg}_{\widetilde{\text{BV}}}\ls\nm{\ggl}_{\widetilde{\text{BV}}}.
    \end{align}
    Hence, using Proposition  \ref{lemma:m7}, our result follows.
\end{proof}

%%%%%%%%%%%%%%%%%%%%%%%%%%%%%%%%%%%%%%%%%%%%%%%%%%%%%%%%%%%%%%%%%%%%%%%%
\section{Setup of Remainder Estimates}\label{sec:remainder}
%%%%%%%%%%%%%%%%%%%%%%%%%%%%%%%%%%%%%%%%%%%%%%%%%%%%%%%%%%%%%%%%%%%%%%%%

Now we begin to estimate the remainder equation for $\re$ in \eqref{aa 08}, or equivalently the nonlinear Boltzmann equation \eqref{large system-}. 
Denote 
\begin{align}
    Q[F,F]=&Q_{\text{gain}}[F,F]-Q_{\text{loss}}[F,F]\\
    :=&\int_{\r^3}\int_{\s^2}q(\vo,\abs{\vuu-\vv})F(\vuu_{\ast})F(\vv_{\ast})\ud{\vo}\ud{\vuu}-F(\vv)\int_{\r^3}\int_{\s^2}q(\vo,\abs{\vuu-\vv})F(\vuu)\ud{\vo}\ud{\vuu}.\no
\end{align}
Denote $\fs=\ff+\e\mh\re$,
where $\ff:=\m+\mh\big(\e\f_1+\e^2\f_2\big)+\mbh\big(\e\fb_1\big)$.
We can split $\fs=\fs_+-\fs_-$ where $\fs_+=\max\{\fs,0\}$ and $\fs_-=\max\{-\fs,0\}$ denote the positive and negative parts, and the similar notation also applies to $\ff$ and $\re$. 

In order to study \eqref{large system-}, we first consider an auxiliary equation
\begin{align}\label{auxiliary system}
    \vv\cdot\nx \fs+\e^{-1}\Big(Q_{\text{loss}}[\fs,\fs]-Q_{\text{gain}}[\fs_+,\fs_+]\Big)=\mathfrak{z}\ds\iint_{\Omega\times\r^3}\e^{-1}\Big(Q_{\text{loss}}[\fs,\fs]-Q_{\text{gain}}[\fs_+,\fs_+]\Big),
\end{align}
with diffuse-reflection boundary condition
\begin{align}
    \fs(\vx_0,\vv)=\ms(\vx_0,\vv)\displaystyle\int_{\vuu\cdot\vn(\vx_0)>0}
    \fs(\vx_0,\vuu)\abs{\vuu\cdot\vn(\vx_0)}\ud{\vuu} \ \ \text{for}\ \ \vx_0\in\p\Omega\ \ \text{and}\ \ \vv\cdot\vn(\vx_0)<0.
\end{align}
Here $\mathfrak{z}=\mathfrak{z}(v)>0$ is a smooth function with support contained in $\{\abs{v}\leq 1\}$ such that $\ds\iint_{\Omega\times\r^3}\mathfrak{z}=1$.
Due to orthogonality of $Q$, the auxiliary system \eqref{auxiliary system} is equivalent to
\begin{align}\label{auxiliary system'}
\vv\cdot\nx \fs-\e^{-1}Q[\fs,\fs]
=&-\e^{-1}\Big(Q_{\text{gain}}[\fs,\fs]-Q_{\text{gain}}[\fs_+,\fs_+]\Big)+\mathfrak{z}\ds\iint_{\Omega\times\r^3}\e^{-1}\Big(Q_{\text{gain}}[\fs,\fs]-Q_{\text{gain}}[\fs_+,\fs_+]\Big).
\end{align}

\begin{remark}
The extra terms
\begin{align}
    -\e^{-1}\Big(Q_{\text{gain}}[\fs,\fs]-Q_{\text{gain}}[\fs_+,\fs_+]\Big)+\mathfrak{z}\ds\iint_{\Omega\times\r^3}\e^{-1}\Big(Q_{\text{gain}}[\fs,\fs]-Q_{\text{gain}}[\fs_+,\fs_+]\Big).
\end{align}
on the RHS of \eqref{auxiliary system'} plays a significant role in justifying the positivity of $\fs$ (see Section \ref{sec:well-posedness}).
Clearly, when $\fs\geq0$, i.e. $\fs=\fs_+$, the above extra terms vanish and the auxiliary equation \eqref{auxiliary system'} reduces to \eqref{large system-}.
\end{remark}

Denote $\ff=\m+\ffe$ where $\ffe:=\mh\big(\e\f_1+\e^2\f_2\big)+\mbh\big(\e\fb_1\big)$.
Inserting $\fs=\ff+\e\mh\re=\m+\ffe+\e\mh\re$ into \eqref{auxiliary system'}, we have
\begin{align}\label{wt 02}
    &\vv\cdot\nx \left(\mh\re\right)+\e^{-1}\mh\lc[\re]\\
    =&\ \mathscr{\ss}+\e^{-1}\bigg(2Q^{\ast}\left[\ffe,\mh\re\right]+\e Q^{\ast}\left[\mh\re,\mh\re\right]\bigg)\no\\
    &-\e^{-1}\bigg(\e^{-1}Q_{\text{gain}}\left[\ff+\e\mh\re,\ff+\e\mh\re\right]-\e^{-1}Q_{\text{gain}}\left[\left(\ff+\e\mh\re\right)_+,\left(\ff+\e\mh\re\right)_+\right]\bigg)\no\\
    &+\e^{-1}\mathfrak{z}\ds\iint_{\Omega\times\r^3}\bigg(\e^{-1}Q_{\text{gain}}\left[\ff+\e\mh\re,\ff+\e\mh\re\right]-\e^{-1}Q_{\text{gain}}\left[\left(\ff+\e\mh\re\right)_+,\left(\ff+\e\mh\re\right)_+\right]\bigg),\no
\end{align}
where $\mathscr{\ss}:=-\e^{-1}v\cdot\nx\ff+\e^{-2}Q^{\ast}\left[\ff,\ff\right]$.
Then we arrive at \eqref{remainder} with $h$ and $\ss$ defined in \eqref{aa 32} and \eqref{aa 31}.

\begin{lemma}\label{lemma:final 1}
    For $h$ and $\ss$ defined in \eqref{aa 32} and \eqref{aa 31}, we have the compatibility condition
    \begin{align}\label{compatibility}
    \iint_{\Omega\times\r^3}\mh S+\int_{\gamma_-}\mh h\ud{\gamma}=0.
    \end{align}
\end{lemma}
\begin{proof}
We can directly verify that
\begin{align}\label{aa 36'}
    &\int_{\gamma_-}h(\vx,\vv)\mh(x,v)\ud{\gamma}=\e^{-1}\left(\int_{\gamma_-}\pp\left[\mhh\ff\right]\mh\ud{\gamma}-\int_{\gamma_-}\ff\ud{\gamma}\right)\\
    =&\e^{-1}\left(\int_{\gamma_+}\pp\left[\mhh\ff\right]\mh\ud{\gamma}-\int_{\gamma_-}\ff\ud{\gamma}\right)\no
    =\e^{-1}\left(\int_{\gamma_+}\ff\ud{\gamma}-\int_{\gamma_-}\ff\ud{\gamma}\right)=\e^{-1}\int_{\gamma}\ff(v\cdot n)\ud\vv\ud S_x.\no
\end{align}
Using the orthogonality of $Q^{\ast}$ (see \eqref{qstar}) and the designed zero-mass condition of $\sp$, we verify that
\begin{align}
    \iint_{\Omega\times\r^3}S(\vx,\vv)\mh(x,v)\ud{\vv}\ud{\vx}
    =&-\e^{-1}\iint_{\Omega\times\r^3}\Big(v\cdot\nx\ff\Big)\ud{\vv}\ud{\vx}
    =-\e^{-1}\int_{\gamma}\ff(v\cdot n)\ud\vv\ud S_x.
\end{align}
Hence, \eqref{compatibility} is always guaranteed.
\end{proof}

\begin{proposition}\label{lemma:final 2}
    For $h$ and $\ss$ defined in \eqref{aa 32} and \eqref{aa 31}, we have for each $\vx_0\in\p\Omega$
    \begin{align}\label{aa 36}
    \int_{\vv\cdot\vn<0}\mh h(\vx_0)(\vv\cdot\vn)\ud{\vv}=0,\quad \iint_{\Omega\times\r^3}\mh\ss=0.
    \end{align}
\end{proposition}
\begin{proof}
    This is guaranteed by \eqref{aa 36'} for each $\vx_0\in\p\Omega$, $u_1\cdot n=0$ and the choice of $u_2$ (see \eqref{aa 38}). Then the result follows from Lemma \ref{lemma:final 1}.
\end{proof}

\begin{remark}\label{remark 02}
Using Proposition \ref{lemma:final 2}, we can directly compute that for $x_0\in\p\Omega$
\begin{align}\label{final 53}
    \bb(x_0)\cdot n=&\int_{\r^3}\re(\vx_0)\mh(v\cdot n)\ud v=\int_{v\cdot n<0}h(\vx_0)\mh(v\cdot n)\ud v=0.
\end{align}
\end{remark}

\begin{lemma}[Green's Identity, Lemma 2.2 of \cite{Esposito.Guo.Kim.Marra2013}]\label{remainder lemma 2}
Assume $f(\vx,\vv),\ g(\vx,\vv)\in L^2_{\nu}(\Omega\times\r^3)$ and
$\vv\cdot\nx f,\ \vv\cdot\nx g\in L^2(\Omega\times\r^3)$ with $f,\
g\in L^2_{\gamma}$. Then
\begin{align}
\iint_{\Omega\times\r^3}\Big(\left(\vv\cdot\nx f\right)g+(\vv\cdot\nx
g)f\Big)\ud{\vx}\ud{\vv}=\int_{\gamma}fg(v\cdot n)=\int_{\gamma_+}fg\ud{\gamma}-\int_{\gamma_-}fg\ud{\gamma}.
\end{align}
\end{lemma}

Using Lemma \ref{remainder lemma 2}, we can derive the weak formulation of \eqref{remainder.}. For any test function $\test(x,v)\in L^2_{\nu}(\Omega\times\r^3)$ with $\vv\cdot\nx\test\in L^2(\Omega\times\r^3)$ with $\test\in L^2_{\gamma}$, we have
\begin{align}\label{weak formulation}
    \int_{\gamma}\re\test(v\cdot n)-\iint_{\Omega\times\r^3}\re\big(v\cdot\nx \test\big)+\iint_{\Omega\times\r^3}\left(\mhh\ab\cdot\dfrac{\nx\tq}{4\tq^2}\right)\re \test+\e^{-1}\iint_{\Omega\times\r^3}\lc[\re]\test=\iint_{\Omega\times\r^3}\ss\test.
\end{align}

%%%%%%%%%%%%%%%%%%%%%%%%%%%%%%%%%%%%%%%%%%%%%%%%%%%%%%%%%%%%%%%%%%%%%%%%
\subsection{Preliminary Estimates}
%%%%%%%%%%%%%%%%%%%%%%%%%%%%%%%%%%%%%%%%%%%%%%%%%%%%%%%%%%%%%%%%%%%%%%%%

\begin{lemma}[Lemma 2.3 of \cite{Guo2003}]\label{lemma:a1}
Let $\Gamma(f,g)$ be given by \eqref{gamma}. We have
\begin{align}
    \abs{\brv{\Gamma[f_1,f_2],g}}\ls&\Bigg\{\left(\int_{\r^3}\nu\abs{f_1}^2\right)^{\frac{1}{2}}\left(\int_{\r^3}\abs{f_2}^2\right)^{\frac{1}{2}}+\left(\int_{\r^3}\nu\abs{f_2}^2\right)^{\frac{1}{2}}\left(\int_{\r^3}\abs{f_1}^2\right)^{\frac{1}{2}}\Bigg\}\left(\int_{\r^3}\nu\abs{g}^2\right)^{\frac{1}{2}},\\
    \tnm{\brv{\Gamma[f_1,f_2],g}}\ls& \Big(\sup_{x,v}\abs{\nu^3g}\Big)\min\Bigg\{\sup_x\left(\int_{\r^3}\abs{f_1}^2\right)^{\frac{1}{2}}\tnm{f_2}, \sup_x\left(\int_{\r^3}\abs{f_2}^2\right)^{\frac{1}{2}}\tnm{f_1}\Bigg\},\\
    \tnm{\Gamma[f_1,f_2]g}\ls& \Big(\sup_{x,v}\abs{\nu g}\Big)\min\Bigg\{\sup_x\left(\int_{\r^3}\abs{f_1}^2\right)^{\frac{1}{2}}\tnm{f_2}, \sup_x\left(\int_{\r^3}\abs{f_2}^2\right)^{\frac{1}{2}}\tnm{f_1}\Bigg\}.
\end{align}
\end{lemma}

\begin{lemma}[Lemma 5 of \cite{Guo2010}]\label{lemma:a2}
We have
\begin{align}
    \lnmm{\nu^{-1}\Gamma[f_1,f_2]}\ls\lnmm{f_1}\lnmm{f_2}.
\end{align}
\end{lemma}

\begin{lemma}\label{m-estimate}
    If $\abs{\tq-\tq_w}\ls\za$ for some $0<\za\ll1$, then we have
    \begin{align}
        \abs{\m-\m_w}\ls \za\left(\abs{\vv}^2\ue^{\za\abs{\vv}^2}\right)\m.
    \end{align}
\end{lemma}
\begin{proof}
    Notice that, since $P=\rq\tq$ is a constant,
    \begin{align}
        \m(\vx,\vv)=\frac{P}{\big(2\pi\big)^{\frac{3}{2}}\big(\tq(\vx)\big)^{\frac{5}{2}}}
        \exp\bigg(-\frac{\abs{\vv}^2}{2\tq(\vx)}\bigg),\quad \m_w(\vx,\vv)=\frac{P}{\big(2\pi\big)^{\frac{3}{2}}\big(\tq_w(\vx)\big)^{\frac{5}{2}}}
        \exp\bigg(-\frac{\abs{\vv}^2}{2\tq_w(\vx)}\bigg).
    \end{align}
    Then we have $\m-\m_w=I+II$,
    where
    \begin{align}
        I:=\frac{P}{\big(2\pi\big)^{\frac{3}{2}}}\bigg\{\frac{1}{\tq^{\frac{5}{2}}}-\frac{1}{\tq_w^{\frac{5}{2}}}\bigg\}\exp\bigg(-\frac{\abs{\vv}^2}{2\tq}\bigg),\quad
        II:=\frac{P}{\big(2\pi\big)^{\frac{3}{2}}}\frac{1}{\tq_w^{\frac{5}{2}}}\bigg\{\exp\bigg(-\frac{\abs{\vv}^2}{2\tq}\bigg)-\exp\bigg(-\frac{\abs{\vv}^2}{2\tq_w}\bigg)\bigg\}.
    \end{align}
    Since $\abs{\tq-\tq_w}\ls\za$, we have
    \begin{align}
        \abs{I}=\frac{P}{\big(2\pi\big)^{\frac{3}{2}}}\frac{\abs{\tq_w^{\frac{5}{2}}-\tq^{\frac{5}{2}}}}{\tq^{\frac{5}{2}}\tq_w^{\frac{5}{2}}}\exp\left(-\frac{\abs{\vv}^2}{2\tq}\right)\ls \abs{\tq_w^{\frac{5}{2}}-\tq^{\frac{5}{2}}}\m\ls\za\m.
    \end{align}
    Also, using mean-value theorem, for some $\d\in(0,\za)$, we know
    \begin{align}
        \abs{II}\ls& \m\bigg\{1-\exp\bigg(\frac{\abs{\vv}^2}{2\tq}-\frac{\abs{\vv}^2}{2\tq_w}\bigg)\bigg\}\ls\m\left\{1-\ue^{\za\abs{\vv}^2}\right\}
        \leq\m\left(\za\abs{\vv}^2\right)\ue^{\d\abs{\vv}^2}\ls \za\left(\abs{\vv}^2\ue^{\za\abs{\vv}^2}\right)\m.
    \end{align}
    Then our result naturally follows.
\end{proof}

Using a similar argument, we can derive the following estimates.
\begin{corollary}\label{m-estimate.}
    If $\abs{\tq-\tq_w}\ls\za$ for some $0<\za\ll1$, then we have
    \begin{align}
        \abs{\mh-\mh_w}\ls \za\left(\abs{\vv}^2\ue^{\za\abs{\vv}^2}\right)\mh,\quad
        \abs{\mhh-\mhh_w}\ls \za\left(\abs{\vv}^2\ue^{\za\abs{\vv}^2}\right)\mhh.
    \end{align}
\end{corollary}

\begin{corollary}\label{m-estimate..}
    Let $\m_{\#}$ be $\m$ or $\m_w$. If $\abs{\tq-\tq_w}\ls\za$ for some $0<\za\ll1$, then we have
    \begin{align}
        \tnm{\m_{\#}^{-\frac{1}{2}} Q^{\ast}\Big[\m-\m_w,\m_{\#}^{\frac{1}{2}}g\Big]}&=\tnm{\mh\m_{\#}^{-\frac{1}{2}} \Gamma\Big[\mh-\mhh\m_w,\m_{\#}^{\frac{1}{2}}\mhh g\Big]}\ls \za\tnm{\nu g},\\
        \lnmm{\m_{\#}^{-\frac{1}{2}} Q^{\ast}\Big[\m-\m_w,\m_{\#}^{\frac{1}{2}}g\Big]}&=\lnmm{\mh\m_{\#}^{-\frac{1}{2}} \Gamma\Big[\mh-\mhh\m_w,\m_{\#}^{\frac{1}{2}}\mhh g\Big]}\ls \za\lnmm{g}.
    \end{align}
\end{corollary}

%%%%%%%%%%%%%%%%%%%%%%%%%%%%%%%%%%%%%%%%%%%%%%%%%%%%%%%%%%%%%%%%%%%%%%%%
\subsection{Estimates of Boundary Terms and Source Terms}\label{sec:bs}
%%%%%%%%%%%%%%%%%%%%%%%%%%%%%%%%%%%%%%%%%%%%%%%%%%%%%%%%%%%%%%%%%%%%%%%%

In this section, we will give several types of estimates regarding $h$ and $\ss$. In the following, we will assume that the constant $\N\in[2,6]$ and $g(x,v)$ is a given function.

%%%%%%%%%%%%%%%%%%%%%%%%%%%%%%%%%%%%%%%%%%%%%%%%%%%%%%%%%%%%%%%%%%%%%%%%
\paragraph{\underline{Estimates of $h$}}
%%%%%%%%%%%%%%%%%%%%%%%%%%%%%%%%%%%%%%%%%%%%%%%%%%%%%%%%%%%%%%%%%%%%%%%%

\begin{lemma}\label{h-estimate}
Under the assumption \eqref{assumption:boundary}, for $h$ defined in \eqref{aa 32}, we have
\begin{align}
    \tnms{h}{\gamma_-}&\ls\oot\e,\quad
    \jnms{h}{\gamma_-}\ls\oot\e^{\frac{3}{\N}},\quad
    \lnmms{h}{\gamma_-}\ls\oot,\quad \sup_{\iota_1,\iota_2}\int_{v\cdot n<0}\abs{h}\abs{v\cdot n}\ud v\ls\oot\e.
\end{align}
\end{lemma}

\begin{proof}
Note that from Theorem \ref{thm:ghost}, it holds that $\lnmms{\f_2}{\gamma}\ls\nm{\f_2}_{W^{1,\NN}L^{\infty}_{\varrho,\vartheta}}\ls\oot$.
Then we know
\begin{align}
    \lnmms{\e\left(\mss\displaystyle\int_{\vuu\cdot\vn>0}
    \mbh(v')\f_2(v')\abs{\vuu\cdot\vn}\ud{\vuu}-\f_2\Big|_{v\cdot n<0}\right)}{\gamma_-}&\ls\e\lnmms{\f_2}{\gamma}\ls\oot\e.
\end{align}
Then we obtain the similar estimates for $L^2_{\gamma_-}$ and $L^{\frac{2\N}{3}}_{\gamma_-}$ norms. 

On the other hand, noticing that $\lnmms{\blff}{\gamma}\ls\oot$,
we have
\begin{align}
    &\lnmms{\left(\ms\displaystyle\int_{\vuu\cdot\vn>0}
    \mbh\chi\left(\e^{-1}\va\right)\blff\abs{\vuu\cdot\vn}\ud{\vuu}-\mbh\chi\left(\e^{-1}\va\right)\blff\Big|_{v\cdot n<0}\right)}{\gamma_-}
    \ls\lnmms{\blff}{\gamma_-}\ls\oot.
\end{align}
Also, noticing that the cutoff $\chi\left(\e^{-1}\va\right)$ implies a restriction to the domain $\abs{\va}\leq\e$ and the $\gamma$ norm has an extra $\va$, we have
\begin{align}
    &\jnms{\left(\ms\displaystyle\int_{\vuu\cdot\vn>0}
    \mbh\chi\left(\e^{-1}\va\right)\blff\abs{\vuu\cdot\vn}\ud{\vuu}-\mbh\chi\left(\e^{-1}\va\right)\blff\Big|_{v\cdot n<0}\right)}{\gamma_-}
    \ls\jnms{\chi\left(\e^{-1}\va\right)\blff}{\gamma_-}\ls\oot\e^{\frac{3}{\N}},
\end{align}
and the $L^2_{\gamma_-}$ estimate follows when $\N=3$, and
\begin{align}
    &\int_{v\cdot n<0}\abs{\left(\ms\displaystyle\int_{\vuu\cdot\vn>0}
    \mbh\chi\left(\e^{-1}\va\right)\blff\abs{\vuu\cdot\vn}\ud{\vuu}-\mbh\chi\left(\e^{-1}\va\right)\blff\Big|_{v\cdot n<0}\right)}\abs{v\cdot n}\ud v
    \ls\int_{v\cdot n<0}\chi\left(\e^{-1}\va\right)\blff\abs{\va}\ud v\ls\e.
\end{align}
Then our estimates follow.
\end{proof}

%%%%%%%%%%%%%%%%%%%%%%%%%%%%%%%%%%%%%%%%%%%%%%%%%%%%%%%%%%%%%%%%%%%%%%%%
\paragraph{\underline{Estimates of $\llc[\re]$}}
%%%%%%%%%%%%%%%%%%%%%%%%%%%%%%%%%%%%%%%%%%%%%%%%%%%%%%%%%%%%%%%%%%%%%%%%

\begin{lemma}\label{ssl-estimate}
Under the assumption \eqref{assumption:boundary}, for $\llc[\re]$ defined in \eqref{final 51}, we have
\begin{align}
    \abs{\brv{\llc[\re],g}}\ls \oot\left(\int_{\r^3}\nu\abs{g}^2\right)^{\frac{1}{2}}\left(\int_{\r^3}\nu\abs{\re}^2\right)^{\frac{1}{2}},
\end{align}
and thus
\begin{align}
    \abs{\br{\llc[\re],g}}\ls \oot\um{g}\um{\re}\ls\oot\um{g}\left(\tnm{\pk[\re]}+\um{(\ik-\pk)[\re]}\right).
\end{align}
Also, we have
\begin{align}
    \tnm{\llc[\re]}&\ls\oot\um{\re},\quad
    \lnmm{\nu^{-1}\llc[\re]}\ls\oot\lnmm{\re}.
\end{align}
\end{lemma}

\begin{proof}
    The desired estimates follow from Lemma \ref{lemma:a1} and Lemma \ref{lemma:a2} since, by Theorem \ref{thm:ghost}, we have
    \begin{align}
        \lnmm{f_1}\ls\nm{f_1}_{W^{1,\NN}}\ls\oot.
    \end{align}
\end{proof}

%%%%%%%%%%%%%%%%%%%%%%%%%%%%%%%%%%%%%%%%%%%%%%%%%%%%%%%%%%%%%%%%%%%%%%%%
\paragraph{\underline{Estimates of $\ss_0$}}
%%%%%%%%%%%%%%%%%%%%%%%%%%%%%%%%%%%%%%%%%%%%%%%%%%%%%%%%%%%%%%%%%%%%%%%%

\begin{lemma}\label{ss0-estimate}
Under the assumption \eqref{assumption:boundary}, for $\ss_0$ defined in \eqref{def-s0}, we have
\begin{align}
    \abs{\brv{\ss_0,g}}\ls \oot\e\left(\int_{\r^3}\nu\abs{g}^2\right)^{\frac{1}{2}}\left(\int_{\r^3}\nu\abs{\re}^2\right)^{\frac{1}{2}},
\end{align}
and thus
\begin{align}
    \abs{\br{\ss_0,g}}\ls \oot\e\um{g}\um{\re}\ls\oot\e\um{g}\left(\tnm{\pk[\re]}+\um{(\ik-\bpk)[\re]}\right).
\end{align}
Also, we have
\begin{align}
    \tnm{\ss_0}&\ls\oot\e\um{\re},\quad
    \lnmm{\nu^{-1}\ss_0}\ls\oot\e\lnmm{\re}.
\end{align}
\end{lemma}

\begin{proof}
    The desired estimates follow from Lemma \ref{lemma:a1} and Lemma \ref{lemma:a2} since, by  Theorem \ref{thm:ghost}, we have 
    \begin{align}
        \lnmms{\f_2}{\gamma}\ls\nm{\f_2}_{W^{1,\NN}L^{\infty}_{\varrho,\vartheta}}\ls&\oot.
    \end{align}
\end{proof}

%%%%%%%%%%%%%%%%%%%%%%%%%%%%%%%%%%%%%%%%%%%%%%%%%%%%%%%%%%%%%%%%%%%%%%%%
\paragraph{\underline{Estimates of $\ss_1$}}
%%%%%%%%%%%%%%%%%%%%%%%%%%%%%%%%%%%%%%%%%%%%%%%%%%%%%%%%%%%%%%%%%%%%%%%%

\begin{lemma}\label{ss1-estimate}
Under the assumption \eqref{assumption:boundary}, for $\ss_1$ defined in \eqref{def-s1}, we have
\begin{align}
    \abs{\brv{\ss_1,g}}\ls \left(\int_{\r^3}\nu\abs{g}^2\right)^{\frac{1}{2}}\left(\int_{\r^3}\nu\abs{\fb_1}^2\right)^{\frac{1}{2}}\left(\int_{\r^3}\nu\abs{\re}^2\right)^{\frac{1}{2}},
\end{align}
and thus
\begin{align}
    \abs{\br{\ss_1,g}}\ls& \oot\um{g}\um{\re}\ls\oot\um{g}\left(\tnm{\pk[\re]}+\um{(\ik-\pk)[\re]}\right),\\
    \abs{\br{\ss_1,g}}\ls& \oot\um{\fb_1}\lnmm{g}\um{\re}\ls\oot\e^{\frac{1}{2}}\lnmm{g}\left(\tnm{\pk[\re]}+\um{(\ik-\pk)[\re]}\right).
\end{align}
Also, we have
\begin{align}
    \tnm{\ss_1}&\ls\oot\um{\re},\quad
    \lnmm{\nu^{-1}\ss_1}\ls\oot\lnmm{\re}.
\end{align}
\end{lemma}

\begin{proof}
    The desired estimates from Lemma \ref{lemma:a1} and Lemma \ref{lemma:a2} using Theorem \ref{thm:ghost}
    \begin{align}
        \lnmm{\fb_1}\ls&\oot,\quad
        \um{\fb_1}\ls\oot\e^{\frac{1}{2}}.
    \end{align}
\end{proof}

%%%%%%%%%%%%%%%%%%%%%%%%%%%%%%%%%%%%%%%%%%%%%%%%%%%%%%%%%%%%%%%%%%%%%%%%
\paragraph{\underline{Estimates of $\ss_2$}}
%%%%%%%%%%%%%%%%%%%%%%%%%%%%%%%%%%%%%%%%%%%%%%%%%%%%%%%%%%%%%%%%%%%%%%%%

\begin{lemma}\label{ss2-estimate}
Under the assumption \eqref{assumption:boundary}, for $\ss_2$ defined in \eqref{def-s2}, we have
\begin{align}\label{aa 39}
    \abs{\brv{\ss_2,g}}\ls \left(\int_{\r^3}\nu\abs{g}^2\right)^{\frac{1}{2}}\left(\int_{\r^3}\nu\abs{\re}^2\right),
\end{align}
and thus
\begin{align}\label{aa 40}
    \abs{\br{\ss_2,g}}\ls& \um{g}\Big(\um{(\ik-\bpk)\re}\lnmm{\re}+\pnm{\pk[\re]+\bd\cdot\a}{3}\pnm{\pk[\re]+\bd\cdot\a}{6}\Big)
    \ls\um{g}\xnm{\re}^2.
\end{align}
Also, we have
\begin{align}\label{aa 41}
    &\tnm{\ss_2}\ls\Big(\um{(\ik-\bpk)\re}\lnmm{\re}+\pnm{\pk[\re]+\bd\cdot\a}{3}\pnm{\pk[\re]+\bd\cdot\a}{6}\Big)
    \ls\xnm{\re}^2,\\
    &\lnmm{\nu^{-1}\ss_2}\ls\lnmm{\re}^2.\label{aa 42}
\end{align}
\end{lemma}

\begin{proof}
    Estimate \eqref{aa 39} follows from Lemma \ref{lemma:a1}. Then utilizing H\"older's inequality
    \begin{align}
        \abs{\int_{\r^3}\nu\re(\ik-\bpk)[\re]}\ls&\um{(\ik-\bpk)\re}\lnmm{\re},\\
        \int_{\r^3}\nu\abs{\pk[\re]+\bd\cdot\a}^2\ls&\pnm{\pk[\re]+\bd\cdot\a}{3}\pnm{\pk[\re]+\bd\cdot\a}{6},
    \end{align}
    we obtain \eqref{aa 40}. Then \eqref{aa 41} follows from duality. Finally, \eqref{aa 42} holds due to Lemma \ref{lemma:a2}.
\end{proof}

%%%%%%%%%%%%%%%%%%%%%%%%%%%%%%%%%%%%%%%%%%%%%%%%%%%%%%%%%%%%%%%%%%%%%%%%
\paragraph{\underline{Estimates of $\ss_3$}}
%%%%%%%%%%%%%%%%%%%%%%%%%%%%%%%%%%%%%%%%%%%%%%%%%%%%%%%%%%%%%%%%%%%%%%%%

\begin{lemma}\label{ss3-estimate}
Under the assumption \eqref{assumption:boundary}, for $\ss_3$ defined in \eqref{mm 00}, we have
\begin{align}
    \pnm{\ss_3}{1}+\pnm{\eta\left(\sy+\sz\right)}{1}+\pnm{\eta^2\left(\sy+\sz\right)}{1}&\ls \oot \e,\label{ss3-estimate1}\\
    \tnm{\br{v}^2\ss_3}+\tnm{\eta\left(\sy+\sz\right)}+\tnm{\eta^2\left(\sy+\sz\right)}&\ls \oot,\label{ss3-estimate2}\\
    \pnm{\ss_3}{\N}+\pnm{\eta\left(\sy+\sz\right)}{\N}+\pnm{\eta^2\left(\sy+\sz\right)}{\N}&\ls\oot\e^{\frac{2}{\N}-1},\label{ss3-estimate3}\\
    \nm{\ss_3}_{L^{\N}_{\iota_1\iota_2}L^1_{\mn}L^1_v}+\nm{\eta\left(\sy+\sz\right)}_{L^{\N}_{\iota_1\iota_2}L^1_{\mn}L^1_v}&\ls\oot\e,\label{ss3-estimate4},
\end{align}
and
\begin{align}
    \nm{\sy+\sz}_{L^{\N}_{x}L^1_v}+\nm{\eta\left(\sy+\sz\right)}_{L^{\N}_{x}L^1_v}&\ls\oot\e^{\frac{1}{\N}},\label{ss3-estimate5}\\
    \abs{\br{\sx,g}}+\abs{\br{\eta\sx,g}}+\abs{\br{\eta^2\sx,g}}&\ls\knm{\br{v}^2\fb_1}\jnm{\nabla_v g}\ls\oot\e^{1-\frac{1}{\N}}\jnm{\nabla_v g}.\label{ss3-estimate6}
\end{align}
Also, we have
\begin{align}\label{final 15}
    \lnmm{\ss_3}&\ls\oot\e^{-1}.
\end{align}
\end{lemma}

\begin{proof} 
We start from the $L^{\N}$ estimate \eqref{ss3-estimate3}, and then \eqref{ss3-estimate1} and \eqref{ss3-estimate2} will naturally follow. We first focus on $\sx$. Notice that
    \begin{align}
    \frac{\p\fb_1}{\p\va}(\eta,\vvv)=\e^{-1}\ch'\left(\e^{-1}\va\right)\chi(\e\eta)\blff(\eta,\vvv)+\ch\left(\e^{-1}\va\right)\chi(\e\eta)\frac{\p\blff(\eta,\vvv)}{\p\va}.
\end{align}
From Theorem \ref{boundary well-posedness} and Theorem \ref{boundary regularity}, we have
\begin{align}
    \lnmm{\ch\left(\e^{-1}\va\right)\chi(\e\eta)\frac{\p\blff(\eta,\vvv)}{\p\va}}\ls\e^{-1}\lnmm{\va\frac{\p\blff(\eta,\vvv)}{\p\va}}\ls\e^{-1},\quad \pnm{\ch\left(\e^{-1}\va\right)\chi(\e\eta)\frac{\p\blff(\eta,\vvv)}{\p\va}}{1}\ls 1.
\end{align}
Then by change of variable $\eta=\e^{-1}\mn$, we have
\begin{align}
    \pnm{\e^{-1}\ch'\left(\e^{-1}\va\right)\chi(\e\eta)\blff(\eta,\vvv)}{\N}
    =&\left(\iint_{\Omega\times\r^3}\abs{\e^{-1}\ch'\left(\e^{-1}\va\right)\chi(\e\eta)\blff(\eta,\vvv)}^{\N}\right)^{\frac{1}{\N}}\\
    \ls&\e^{-1}\left(\iint_{\Omega\times\r^3}\abs{\ch'\left(\e^{-1}\va\right)\chi(\e\eta)\blff(\eta,\vvv)}^{\N}\right)^{\frac{1}{\N}}\ls\e^{\frac{2}{\N}-1},\no
\end{align}
and
\begin{align}\label{aa 43}
    &\pnm{\ch\left(\e^{-1}\va\right)\chi(\e\eta)\frac{\p\blff(\eta,\vvv)}{\p\va}}{\N}
    =\left(\iint_{\Omega\times\r^3}\abs{\ch\left(\e^{-1}\va\right)\chi(\e\eta)\frac{\p\blff(\eta,\vvv)}{\p\va}}^{\N}\right)^{\frac{1}{\N}}\\
    \ls&\left(\lnmm{\ch\left(\e^{-1}\va\right)\frac{\p\blff(\eta,\vvv)}{\p\va}}^{\N-1}\iint_{\Omega\times\r^3}\abs{\ch\left(\e^{-1}\va\right)\chi(\e\eta)\frac{\p\blff(\eta,\vvv)}{\p\va}}\right)^{\frac{1}{\N}}
    \ls\big(\e^{-(\N-1)}\e\big)^{\frac{1}{\N}}\ls\e^{\frac{2}{\N}-1}.\no
\end{align}
Hence, we know $\pnm{\sx}{\N}\ls \e^{\frac{2}{\N}-1}$.
Similarly $\sy$ estimates follow from Theorem \ref{boundary well-posedness} and Theorem \ref{boundary regularity}. 

Noticing that
\begin{align}
    \sz=&\e^{-1}\mhh\mbh\chi(\e\eta)\bigg(\chi(\e^{-1}\va)K_w\Big[\blff\Big]-K_w\Big[\chi(\e^{-1}\va)\blff\Big]\bigg),
\end{align}
which has one less $\e$-power but produces an $\e$-size cutoff $\chi(\e^{-1}\va)$. Clearly, the $\chi(\e^{-1}\va)K_w\Big[\blff\Big]$ can be estimated by a similar argument as \eqref{aa 43}. Then noticing that by the change of variable $w_{\eta}=\e^{-1}\vuu_{\eta}$
\begin{align}
    \int_{\r^3}\abs{K_w\Big[\chi(\e^{-1}\va)\blff\Big]}\ud v=&\int_{\r^3}\abs{\int_{\r^3}k_w(\vuu,v)\chi(\e^{-1}\vuu_{\eta})\blff(\vuu)\ud\vuu}\ud v\\
    \ls&\sup_{\vuu}\int_{\r^3}\abs{k_w(\vuu,v)}\ud v\abs{\int_{\r^3}\chi(\e^{-1}\vuu_{\eta})\blff(\vuu) \ud\vuu}
    \ls\abs{\int_{\r^3}\chi(\e^{-1}\vuu_{\eta})\blff(\vuu)\ud\vuu}\ls\e,\no
\end{align}
we can bound $K_w\Big[\chi(\e^{-1}\va)\blff\Big]$ by a similar argument as \eqref{aa 43}. Hence, we complete the proof of \eqref{ss3-estimate3}.

Noting the rescaling $\eta=\e^{-1}\mn$ and the cutoff $\chi(\e^{-1}\va)$, using Theorem \ref{boundary well-posedness} and Theorem \ref{boundary regularity}, \eqref{ss3-estimate4} and \eqref{ss3-estimate5} follow from substitution in the integral.

Then we turn to \eqref{ss3-estimate6}. The most difficult term in $\abs{\br{\sx,g}}$ is essentially $\abs{\br{\frac{\p\fb_1}{\p\va},g}}$. Note that $\frac{\p\fb_1}{\p\va}=0$ for $\abs{\va}\leq \e$ due to the cutoff in $\fb_1$.
Integration by parts with respect to $\va$ implies
\begin{align}
   \abs{\br{\frac{\p\fb_1}{\p\va},g}}\ls& \abs{\br{\fb_1, \frac{\p g}{\p\va}}}
   \ls\jnm{\fb_1}\knm{\frac{\p g}{\p\va}}.
\end{align}
From \eqref{aa 44} and $\dfrac{\p\vx}{\p\va}\equiv\od$, we know the substitution $(\mn,\iota_1,\iota_2,\vv)\rt (\mn,\iota_1,\iota_2,\vvv)$ implies $-\frac{\p\vv}{\p\va}\cdot\vn=1$,
$-\frac{\p\vv}{\p\va}\cdot\vt_1=0$, $
-\frac{\p\vv}{\p\va}\cdot\vt_2=0$.
Hence, we know $\abs{\frac{\p\vv}{\p\va}}\ls 1$,
and thus $\abs{\frac{\p g}{\p\va}}\ls \abs{\nabla_v g}\abs{\frac{\p \vv}{\p\va}}\ls\abs{\nabla_v g}$.
Hence, we know that 
\begin{align}
   \abs{\br{\frac{\p\fb_1}{\p\va},g}}\ls&\jnm{\fb_1}\knm{\nabla_v g}\ls\oot\e^{\frac{1}{\N}}\knm{\nabla_v g}.
\end{align}
Finally, \eqref{final 15} holds due to the cutoff $\ch$.
\end{proof}

\begin{remark}\label{ss3-remark}
    Notice that the BV estimate in Theorem \ref{boundary regularity} does not contain exponential decay in $\eta$, and thus we cannot directly bound $\eta\sx$ and $\eta^2\sx$. Instead, we should first integrate by parts with respect to $\va$ as in \eqref{ss3-estimate6} to study $\fb_1$ instead:
    \begin{align}
        \pnm{\fb_1}{\N}+\pnm{\eta\fb_1}{\N}+\pnm{\eta^2\fb_1}{\N}&\ls\oot\e^{\frac{2}{\N}-1},\label{ss3-estimate7}\\
        \nm{\fb_1}_{L^{\N}_{\iota_1\iota_2}L^1_{\mn}L^1_v}+\nm{\eta\fb_1}_{L^{\N}_{\iota_1\iota_2}L^1_{\mn}L^1_v}&\ls\oot\e,\label{ss3-estimate8}\\
        \nm{\fb_1}_{L^{\N}_{x}L^1_v}+\nm{\eta\fb_1}_{L^{\N}_{x}L^1_v}&\ls\oot\e^{\frac{1}{\N}}.
    \end{align}
\end{remark}

%%%%%%%%%%%%%%%%%%%%%%%%%%%%%%%%%%%%%%%%%%%%%%%%%%%%%%%%%%%%%%%%%%%%%%%%
\paragraph{\underline{Estimates of $\ss_4$}}
%%%%%%%%%%%%%%%%%%%%%%%%%%%%%%%%%%%%%%%%%%%%%%%%%%%%%%%%%%%%%%%%%%%%%%%%

\begin{lemma}\label{ss4-estimate}
Under the assumption \eqref{assumption:boundary}, for $\ss_4$ defined in \eqref{def-s4}, we have
\begin{align}
    \tnm{\br{v}^2\ss_4}&\ls \oot\e,\quad
    \pnm{\ss_4}{\N}\ls\oot\e,\quad
    \lnmm{\ss_4}\ls\oot\e.
\end{align}
Also, we have
    \begin{align}\label{final 25}
    \br{v\mh,\ss_4}=\e\br{v, v\cdot\nx\left(\mh\f_2\right)}=0.
    \end{align}
\end{lemma}

\begin{proof}
    The desired estimates follow from Lemma \ref{lemma:a1} and Lemma \ref{lemma:a2} using Theorem \ref{thm:ghost}
    \begin{align}
        \lnmm{\nx f_2}\ls\nm{f_2}_{W^{2,\NN}}\ls\oot.
    \end{align}
    Also, note that
    \begin{align}
    \br{v\mh,\ss_4}=\e\br{v, v\cdot\nx\left(\mh\f_2\right)}=0,
    \end{align}
    due to the crucial orthogonality condition $\mhh\left(v\cdot\nx\left(\mh\f_2\right)\right)\perp v\mh$ in the companion paper \cite[(2.8)]{AA024}, which is equivalent to the second equation for the ghost-effect equations \eqref{fluid system-}.
\end{proof}

%%%%%%%%%%%%%%%%%%%%%%%%%%%%%%%%%%%%%%%%%%%%%%%%%%%%%%%%%%%%%%%%%%%%%%%%
\paragraph{\underline{Estimates of $\ss_5$}}
%%%%%%%%%%%%%%%%%%%%%%%%%%%%%%%%%%%%%%%%%%%%%%%%%%%%%%%%%%%%%%%%%%%%%%%%

\begin{lemma}\label{ss5-estimate}
Under the assumption \eqref{assumption:boundary}, for $\ss_5$ defined in \eqref{def-s5}, we have
\begin{align}
    \abs{\brv{\ss_5,g}}\ls \oot\left(\int_{\r^3}\nu\abs{g}^2\right)^{\frac{1}{2}},
\end{align}
and thus
\begin{align}
    \abs{\br{\ss_5,g}}\ls&\oot\e^{\frac{1}{2}}\um{g},\quad
    \abs{\br{\ss_5,g}}\ls\oot\e\lnmm{g}.
\end{align}
Also, we have
\begin{align}
    \tnm{\ss_5}&\ls\oot\e^{\frac{1}{2}},\quad
    \lnmm{\nu^{-1}\ss_5}\ls\oot.
\end{align}
\end{lemma}

\begin{proof}
    We obtain the desired estimates from Lemma \ref{lemma:a1} and Lemma \ref{lemma:a2} using Theorem \ref{thm:ghost}
    \begin{align}
        \lnmm{f_1}+\lnmm{f_2}\ls\nm{f_1}_{W^{2,\NN}}+\nm{f_2}_{W^{2,\NN}}\ls&\oot,\quad
        \lnmm{\fb_1}\ls\oot.
    \end{align} 
    In particular, the estimate of $\m-\mb$ follows from Lemma \ref{m-estimate} and Corollary \ref{m-estimate..} with $z=\e$, which provides an extra power of $\e$.
\end{proof}

%%%%%%%%%%%%%%%%%%%%%%%%%%%%%%%%%%%%%%%%%%%%%%%%%%%%%%%%%%%%%%%%%%%%%%%%
\paragraph{\underline{Estimates of $\sp$}}
%%%%%%%%%%%%%%%%%%%%%%%%%%%%%%%%%%%%%%%%%%%%%%%%%%%%%%%%%%%%%%%%%%%%%%%%

\begin{lemma}\label{ssp-estimate}
Under the assumption \eqref{assumption:boundary},  for $\sp$ defined in \eqref{def-sp} and any $m\in\mathbb{N}$, we have
\begin{align}
    \tnm{\br{v}^2\sp}&\ls \oot\e^{m-2},\quad
    \pnm{\sp}{\N}\ls\oot\e^{m-2},\quad
    \lnmm{\sp}\ls\oot\e^{m-2}.
\end{align}
\end{lemma}

\begin{proof}
Note that $\ff=\m+O(\e)$, and thus $\ff\approx\m$ when $\e$ is small. Also, for $w(v)=\ue^{\vrh\frac{\abs{v}^2}{2\tm}}$ with $0\leq\vrh<\frac{1}{2}$, we have
\begin{align}
    \abs{\e\mh\re}= \abs{\e\mh w^{-1}\big(w\re\big)}\ls\e\mh \ue^{-\vrh\frac{\abs{v}^2}{2\tm}}\lnmm{\re} \ls \e^{\frac{1}{2}}\m^{\frac{1}{2}+\vrh}\xnm{\re}.
\end{align}
Hence, for $\e$ sufficiently small, as long as $\e^{\frac{1}{2}}\m^{\frac{1}{2}+\vrh}\xnm{\re}\ls\m$
we will have $\ff+\e\mh\re\approx \m+\e\mh\re\geq0$,
and thus 
\begin{align}
    \overline{Q}:=Q_{\text{gain}}\left[\ff+\e\mh\re,\ff+\e\mh\re\right]- Q_{\text{gain}}\left[\left(\ff+\e\mh\re\right)_+,\left(\ff+\e\mh\re\right)_+\right]=0.
\end{align}
In other words, $\overline{Q}$ vanishes except in the region 
\begin{align}\label{oo 01}
    D:=\bigg\{(x,v): \e^{\frac{1}{2}}\m^{\frac{1}{2}+\vrh}\xnm{\re}\gs\m\bigg\}.
\end{align}
The requirement in \eqref{oo 01} is equivalent to $\m^{\frac{1}{2}-\vrh}\ls\e^{\frac{1}{2}}\xnm{\re}$, 
which indicates 
\begin{align}
    \abs{v}\gs \frac{\tq}{\frac{1}{2}-\varrho}\left(\abs{\ln(\e)}^{\frac{1}{2}}+\abs{\ln\xnm{\re}}^{\frac{1}{2}}\right).
\end{align}
In addition, in the region $D$, we have $\m\ls\e^{\frac{\frac{1}{2}}{\frac{1}{2}-\vrh}}\xnm{\re}^{\frac{1}{\frac{1}{2}-\vrh}}$.
Hence, for any $m\in\mathbb{N}$, we can choose $\vrh$ with $\frac{1}{2}-\vrh\ll 1$ such that in the region $D$, we have $\m\ls\e^{2m}\xnm{\re}^{4m}$,
and thus
\begin{align}\label{oo 02}
    \abs{\ff+\e\mh\re}\ls \e^{m}\xnm{\re}^{2m}.
\end{align}

For the two groups of terms in $\sp$:
the effective region of $Q_{\text{gain}}$ integral is actually only on 
\begin{align}
    \abs{\vuu_{\ast}}\gs\abs{\ln(\e)}^{\frac{1}{2}}\abs{\ln\xnm{\re}}^{\frac{1}{2}}\ \ \text{and}\ \  \abs{\vv_{\ast}}\gs\abs{\ln(\e)}^{\frac{1}{2}}\abs{\ln\xnm{\re}}^{\frac{1}{2}}.
\end{align}
Therefore \eqref{oo 02} is valid in both sets above so that
\begin{align}
    \lnmm{\sp}
    \ls&\e^{m-2}\xnm{\re}^{2m}.
\end{align}
\end{proof}

%%%%%%%%%%%%%%%%%%%%%%%%%%%%%%%%%%%%%%%%%%%%%%%%%%%%%%%%%%%%%%%%%%%%%%%%%%%%
\subsection{Local Conservation Laws and Weak Formulation}
%%%%%%%%%%%%%%%%%%%%%%%%%%%%%%%%%%%%%%%%%%%%%%%%%%%%%%%%%%%%%%%%%%%%%%%%%%%%

%%%%%%%%%%%%%%%%%%%%%%%%%%%%%%%%%%%%%%%%%%%%%%%%%%%%%%%%%%%%%%%%%%%%%%%%%%%%
\subsubsection{Consequences of Local Conservation Laws \eqref{conservation law 1}\eqref{conservation law 2}\eqref{conservation law 3}}
%%%%%%%%%%%%%%%%%%%%%%%%%%%%%%%%%%%%%%%%%%%%%%%%%%%%%%%%%%%%%%%%%%%%%%%%%%%%

Now we investigate the consequences of local conservation laws \eqref{conservation law 1}\eqref{conservation law 2}\eqref{conservation law 3}.
\begin{lemma}\label{remark 01}
    Under the assumption \eqref{assumption:boundary}, we have for $\N\in[2,6]$
    \begin{align}\label{tt 03}
    \nm{\xi}_{W^{2,\N}}\ls& \oot\jnm{\bb}+\oot\e^{\frac{1}{\N}},\\ \jnm{\bd}\ls&\oot\jnm{\bb}+\jnm{\be}+\oot\e^{\frac{1}{\N}},\label{tt 03'}\\
    \label{qq 33}
    \tnms{\nx\xi}{\p\Omega}\ls&\nm{\xi}_{H^2}\ls\oot\tnm{\bb}+\oot\e^{\frac{1}{2}}.
    \end{align}
    In particular, we have
    \begin{align}
        \e^{-\frac{1}{2}}\nm{\xi}_{H^2}+\nm{\xi}_{W^{2,6}}+\e^{-\frac{1}{2}}\tnms{\nx\xi}{\p\Omega}\ls\oot\xnm{\re}+\oot.
    \end{align}
\end{lemma}
\begin{proof}
\eqref{conservation 
 law 1} and \eqref{conservation law 3} as well as \eqref{tt 01}, imply
\begin{align}\label{aa 35}
    \nx\cdot(\k\nx\xi)=&\nx\cdot(\k\bd)=\brv{\abs{v}^2\mh,\ss_3+\ss_4+\sp}-5P\Big(\nx\cdot(\bb\tq)\Big)\\
    =&\brv{\abs{v}^2\mh,\ss_3+\ss_4+\sp}-5P\Big((\nx\cdot\bb)\tq+\bb\cdot\nx\tq\Big)\no\\
    =&\brv{\left(\abs{v}^2-5\tq\right)\mh,\ss_3+\ss_4}+\brv{\abs{v}^2\mh,\sp}-5P(\bb\cdot\nx\tq).\no
\end{align}
By the standard elliptic estimate and Lemma \ref{ss4-estimate} and Lemma \ref{ssp-estimate}, we have for $1<r<\infty$ 
\begin{align}
    \nm{\xi}_{W^{2,\N}}\ls&\oot\jnm{\bb}+\jnm{\brv{\left(\abs{v}^2-5\tq\right)\mh,\ss_3}}+\jnm{\br{v}^2\ss_4}+\jnm{\sp}\\
    \ls& \oot\jnm{\bb}+\oot\e^{\frac{1}{\N}}+\jnm{\brv{\left(\abs{v}^2-5\tq\right)\mh,\ss_3}}.\no
\end{align}
Using \eqref{ss3-estimate5} and \eqref{ss3-estimate6}, we know
\begin{align}
    \jnm{\brv{\left(\abs{v}^2-5\tq\right)\mh,\ss_3}}\ls\oot\e^{\frac{1}{\N}}. 
\end{align}
Then \eqref{tt 03} is verified. Hence, we have
\begin{align}
    \jnm{\bd}\ls\jnm{\nx\xi}+\jnm{\be}\ls\oot\jnm{\bb}+\jnm{\be}+\oot\e^{\frac{1}{\N}}.
\end{align}
Then by Sobolev embedding and trace theorem
\begin{align}\label{qq 33,}
    \tnms{\nx\xi}{\p\Omega}\ls\nm{\xi}_{H^2}\ls\oot\tnm{\bb}+\oot\e^{\frac{1}{2}}.
\end{align}

\end{proof}

%%%%%%%%%%%%%%%%%%%%%%%%%%%%%%%%%%%%%%%%%%%%%%%%%%%%%%%%%%%%%%%%%%%%%%%%%%%%
\subsubsection{Weak Formulation with Test Function $\nx\phi\cdot\a$}
%%%%%%%%%%%%%%%%%%%%%%%%%%%%%%%%%%%%%%%%%%%%%%%%%%%%%%%%%%%%%%%%%%%%%%%%%%%%

\begin{lemma}
    Under the assumption \eqref{assumption:boundary}, for any smooth function $\phi(x)$, we have %\eqref{pp 03}.
\begin{align}
\label{pp 03}
&-\bbrx{\nx\cdot\big(\kappa\nx\phi\big),c}+\e^{-1}\bbrx{\nx\phi,\k\bd}+\brx{\nx\phi,\frac{\nx T}{2T^2}\big(\kappa\P+\sigma c\big)}+\bbr{\nx\phi\cdot\a,\llc[\re]}\\
=&\bbrb{\nx\phi\cdot\a,h}{\gamma_-}-\bbrb{\nx\phi\cdot\a,(1-\pp)[\re]}{\gamma_+}+\br{v\cdot\nx\Big(\nx\phi\cdot\a\Big),(\ik-\bpk)[\re]}\no\\
&-\br{\nx\phi\cdot\a,\left(\mhh\ab\cdot\frac{\nx T}{2T^2}\right)(\ik-\bpk)[\re]}+\bbr{\nx\phi\cdot\a,\bar\ss}.\no
\end{align}
\end{lemma}
\begin{proof}
Letting test function $\test=\nx\phi\cdot\a$ in \eqref{weak formulation}, we obtain
\begin{align}
    &-\br{v\cdot\nx \Big(\nx\phi\cdot\a\Big),\re}+\e^{-1}\bbr{ \nx\phi\cdot\a, \lc[\re]}+\iint_{\Omega\times\r^3}\left(\mhh\ab\cdot\dfrac{\nx\tq}{4\tq^2}\right)\re\Big(\nx\phi\cdot\a\Big)\\
    =&-\int_{\gamma}\re \Big(\nx\phi\cdot\a\Big)\ud\gamma+\bbr{\nx\phi\cdot\a,\ss}.\no
\end{align}
Using \eqref{pp 01}, oddness and the orthogonality of $\a$ with $\pk[\re]$ and $(\ik-\bpk)[\re]$, we simplify
\begin{align}
&-\bbrx{\nx\cdot\big(\kappa\nx\phi\big),c}+\e^{-1}\bbrx{\nx\phi,\k\bd}+\brx{\nx\phi,\frac{\nx T}{2T^2}\big(\kappa\P+\sigma c\big)}+\bbr{\nx\phi\cdot\a,\llc[\re]}\\
=&-\int_{\p\Omega\times\r^3}\Big(\nx\phi\cdot\a\Big)\re (v\cdot n)+\br{v\cdot\nx\Big(\nx\phi\cdot\a\Big),(\ik-\bpk)[\re]}\no\\
&-\br{\nx\phi\cdot\a,\left(\mhh\ab\cdot\frac{\nx T}{2T^2}\right)(\ik-\bpk)[\re]}+\bbr{\nx\phi\cdot\a,\bar\ss}.\no
\end{align}
We rewrite the boundary condition in \eqref{remainder} as
\begin{align}\label{qq 109}
    \re\id_{\gamma}=\pp[\re]\id_{\gamma}+(1-\pp)[\re]\id_{\gamma_+}+h\id_{\gamma_-}.
\end{align}
Using orthogonality of $\a$ and $\vv\mh$, i.e. $\ds\int_{\r^3}\Big(\nx\phi\cdot\a\Big)\mh(v\cdot n)=0$,
we deduce \eqref{pp 03}.
\end{proof}

\begin{remark}
    Using \eqref{tt 01} and by integration by parts for $\e^{-1}\br{\nx\phi,\k\bd}$, we can further rewrite for $\z=c+\e^{-1}\xi$
    \begin{align}\label{pp 03.}
    &-\bbrx{\nx\cdot\big(\kappa\nx\phi\big),\z}+\e^{-1}\bbrx{\nx\phi,\k\be}+\brx{\nx\phi,\frac{\nx T}{2T^2}\big(\kappa\P+\sigma c\big)}+\bbr{\nx\phi\cdot\a,\llc[\re]}\\
    =&\bbrb{\nx\phi\cdot\a,h}{\gamma_-}-\bbrb{\nx\phi\cdot\a,(1-\pp)[\re]}{\gamma_+}+\br{v\cdot\nx\Big(\nx\phi\cdot\a\Big),(\ik-\bpk)[\re]}\no\\
    &-\br{\nx\phi\cdot\a,\left(\mhh\ab\cdot\frac{\nx T}{2T^2}\right)(\ik-\bpk)[\re]}+\bbr{\nx\phi\cdot\a,\bar\ss}.\no
    \end{align}
\end{remark}

%%%%%%%%%%%%%%%%%%%%%%%%%%%%%%%%%%%%%%%%%%%%%%%%%%%%%%%%%%%%%%%%%%%%%%%%%%%%
\subsubsection{Weak Formulation with Test Function $\nx\psi:\b$}
%%%%%%%%%%%%%%%%%%%%%%%%%%%%%%%%%%%%%%%%%%%%%%%%%%%%%%%%%%%%%%%%%%%%%%%%%%%%

\begin{lemma}
    Under the assumption \eqref{assumption:boundary}, for any smooth function $\psi(x)$, we have 
    \begin{align}\label{pp 04}
    &-\br{\nx\cdot\Big(\bbb(\nx\psi:\b)\Big),\bb}+\e^{-1}\bbrx{\nx\psi,\varpi}+\br{\nx\psi:\b,(\bb\cdot v)\left(\ab\cdot\frac{\nx T}{2T^2}\right)}+\bbr{\nx\psi:\b,\llc[\re]}\\
    =&\bbrb{\nx\psi:\b,h}{\gamma_-}-\bbrb{\nx\psi:\b,(1-\pp)[\re]}{\gamma_+}+\br{v\cdot\nx\Big(\nx\psi:\b\Big),\bd\cdot\a+(\ik-\bpk)[\re]}\no\\
    &-\br{\nx\psi:\b,\left(\mhh\ab\cdot\frac{\nx T}{2T^2}\right)\Big(\bd\cdot\a+(\ik-\bpk)[\re]\Big)}+\bbrv{\nx\psi:\b,\bar\ss}.\no
    \end{align}
    where $\ds\varpi:=\int_{\r^3}(v\otimes v)\mh(\ik-\bpk)[\re]$.
\end{lemma}
\begin{proof}
Let the test function in \eqref{weak formulation} be chosen as $\test=\nx\psi:\b$ . We obtain
\begin{align}
    &-\br{v\cdot\nx \Big(\nx\psi:\b\Big),\re}+\e^{-1}\bbr{ \nx\psi:\b, \lc[\re]}+\iint_{\Omega\times\r^3}\left(\mhh\ab\cdot\dfrac{\nx\tq}{4\tq^2}\right)\re\Big(\nx\psi:\b\Big)\\
    =&-\int_{\gamma}\re \Big(\nx\psi:\b\Big)\ud\gamma+\bbr{\nx\psi:\b,\ss}.\no
\end{align}
Using \eqref{pp 01}, oddness and orthogonality of $\b$ with $\pk[\re]+\bd\cdot\a$,
we obtain
\begin{align}
&-\br{\nx\cdot\Big(\bbb(\nx\psi:\b)\Big),\bb}+\e^{-1}\bbrx{\nx\psi,\varpi}+\br{\nx\psi:\b,(\bb\cdot v)\left(\ab\cdot\frac{\nx T}{2T^2}\right)}+\bbr{\nx\psi:\b,\llc[\re]}\\
=&-\int_{\p\Omega\times\r^3}\Big(\nx\psi:\b\Big)\re (v\cdot n)+\br{v\cdot\nx\Big(\nx\psi:\b\Big),\bd\cdot\a+(\ik-\bpk)[\re]}\no\\
&-\br{\nx\psi:\b,\left(\mhh\ab\cdot\frac{\nx T}{2T^2}\right)\Big(\bd\cdot\a+(\ik-\bpk)[\re]\Big)}+\bbrv{\nx\psi:\b,\bar\ss}.\no
\end{align}
Using \eqref{qq 109} and orthogonality of $\b$ with $\pp[\re]$, i.e. $\ds\int_{\r^3}\Big(\nx\psi:\b\Big)\mh(v\cdot n)=0$,
we obtain \eqref{pp 04}.
\end{proof}

%%%%%%%%%%%%%%%%%%%%%%%%%%%%%%%%%%%%%%%%%%%%%%%%%%%%%%%%%%%%%%%%%%%%%%%%%%%%
\subsubsection{Weak Formulation with Test Function $\nx\phi\cdot\a+\e^{-1}\phi\big(\abs{v}^2-5T\big)\mh$}
%%%%%%%%%%%%%%%%%%%%%%%%%%%%%%%%%%%%%%%%%%%%%%%%%%%%%%%%%%%%%%%%%%%%%%%%%%%%

\begin{lemma}
    Under the assumption \eqref{assumption:boundary}, for any smooth test function $\phi(x)$ satisfying $\phi\big|_{\p\Omega}=0$, we have 
\begin{align}\label{pp 10}
&-\bbrx{\nx\cdot(\kappa\nx\phi),c}+\e^{-1}5P\bbrx{\phi,\nx T\cdot\bb}+\brx{\nx\phi,\frac{\nx T}{2T^2}\big(\kappa\P+\sigma c\big)}+\br{\nx\phi\cdot\a,\llc[\re]}\\
=&\bbrb{\nx\phi\cdot\a,h}{\gamma_-}-\bbrb{\nx\phi\cdot\a,(1-\pp)[\re]}{\gamma_+}\no\\
&+\br{v\cdot\nx\Big(\nx\phi\cdot\a\Big),(\ik-\bpk)[\re]}-\br{\nx\phi\cdot\a,\left(\mhh\ab\cdot\frac{\nx T}{2T^2}\right)(\ik-\bpk)[\re]}\no\\
&+\e^{-1}\br{\phi\big(\abs{v}^2-5T\big)\mh,\ss_3+\ss_4+\sp}+\bbr{\nx\phi\cdot\a,\bar\ss}.\no
\end{align}
\end{lemma}
\begin{proof}
Multiplying $\phi(x)\in\r$ on both sides of \eqref{aa 35} and integrating over $x\in\Omega$, we obtain
\begin{align}\label{pp 06}
5P\bbrx{\phi,\nx T\cdot\bb}-\bbrx{\nx\phi,\kappa\bd}+\int_{\p\Omega}\phi(\kappa\bd)\cdot n=&\br{\phi\left(\abs{v}^2-5T\right)\mh,\ss_3+\ss_4+\sp}.
\end{align}
Hence, adding $\e^{-1}\times$\eqref{pp 06} and \eqref{pp 03} to eliminate $\e^{-1}\brx{\nx\phi,\kappa\bd}$ yields
\begin{align}\label{pp 08}
&-\bbrx{\nx\cdot\big(\kappa\nx\phi\big),c}+\e^{-1}5P\bbrx{\phi,\nx T\cdot\bb}+\e^{-1}\int_{\p\Omega}\phi(\kappa\bd)\cdot n\\
&+\brx{\nx\phi,\frac{\nx T}{2T^2}\big(\kappa\P+\sigma c\big)}+\bbr{\nx\phi\cdot\a,\llc[\re]}\no\\
=&\bbrb{\nx\phi\cdot\a,h}{\gamma_-}-\bbrb{\nx\phi\cdot\a,(1-\pp)[\re]}{\gamma_+}\no\\
&+\br{v\cdot\nx\Big(\nx\phi\cdot\a\Big),(\ik-\bpk)[\re]}-\br{\nx\phi\cdot\a,\left(\mhh\ab\cdot\frac{\nx T}{2T^2}\right)(\ik-\bpk)[\re]}\no\\
&+\e^{-1}\br{\phi\left(\abs{v}^2-5T\right)\mh,\ss_3+\ss_4+\sp}+\bbr{\nx\phi\cdot\a,\bar\ss}.\no
\end{align}
The assumption $\phi\big|_{\p\Omega}=0$ completely eliminates the boundary term $\ds\e^{-1}\int_{\p\Omega}\phi(\kappa\bd)\cdot n$ in \eqref{pp 08}. Hence, we have \eqref{pp 10}.
\end{proof}

%%%%%%%%%%%%%%%%%%%%%%%%%%%%%%%%%%%%%%%%%%%%%%%%%%%%%%%%%%%%%%%%%%%%%%%%%%%%
\subsubsection{Weak Formulation with Test Function $\nx\psi:\b+\e^{-1}\psi\cdot v\mh$}
%%%%%%%%%%%%%%%%%%%%%%%%%%%%%%%%%%%%%%%%%%%%%%%%%%%%%%%%%%%%%%%%%%%%%%%%%%%%

\begin{lemma}
    Under the assumption \eqref{assumption:boundary}, for any smooth function $\psi(x)$ satisfying $\nx\cdot\psi=0$, $\psi\big|_{\p\Omega}=0$, we have %\eqref{pp 16}.
\begin{align}\label{pp 16}
&-\bbrx{\lambda\Delta_x\psi,\bb}-\bbr{\bbb\cdot\big(\nx\psi:\nx\b\big)+\big(\nx\cdot\bbb\big)\big(\nx\psi:\b\big),\bb}\\
&+\br{\nx\psi:\b,(\bb\cdot v)\left(\ab\cdot\frac{\nx T}{2T^2}\right)}+\bbr{\nx\psi:\b,\llc[\re]}\no\\
=&\bbrb{\nx\psi:\b,h}{\gamma_-}-\bbrb{\nx\psi:\b,(1-\pp)[\re]}{\gamma_+}\no\\
&+\br{v\cdot\nx\Big(\nx\psi:\b\Big),\bd\cdot\a+(\ik-\bpk)[\re]}-\br{\nx\psi:\b,\left(\mhh\ab\cdot\frac{\nx T}{4T^2}\right)\Big(\bd\cdot\a+(\ik-\bpk)[\re]\Big)}\no\\
&+\e^{-1}\br{\psi\cdot v\mh,\ss_3+\sp}+\bbr{\nx\psi:\b,\bar\ss}.\no
\end{align}
\end{lemma}
\begin{proof}
Multiplying $\psi(x)\in\r^3$ on both sides of \eqref{conservation law 2} and integrating over $x\in\Omega$, we obtain
\begin{align}\label{pp 12}
-P\bbrx{\nx\cdot\psi, \P}-\bbrx{\nx\psi,\varpi}+\int_{\p\Omega}\Big(P\P\psi+\psi\cdot\varpi\Big)\cdot n= &\br{\psi\cdot v\mh,\ss_3+\sp}.
\end{align}
Hence, adding $\e^{-1}\times$\eqref{pp 12} and \eqref{pp 04} to eliminate $\e^{-1}\brx{\nx\psi,\varpi}$ yields
\begin{align}\label{pp 14}
&-\br{\nx\cdot\Big(\bbb(\nx\psi:\b)\Big),\bb}-\e^{-1}P\bbrx{\nx\cdot\psi, \P}+\e^{-1}\int_{\p\Omega}\Big(P\P\psi+\psi\cdot\varpi\Big)\cdot n\\
&+\br{\nx\psi:\b,(\bb\cdot v)\left(\ab\cdot\frac{\nx T}{2T^2}\right)}+\bbr{\nx\psi:\b,\llc[\re]}\no\\
=&\bbrb{\nx\psi:\b,h}{\gamma_-}-\bbrb{\nx\psi:\b,(1-\pp)[\re]}{\gamma_+}\no\\
&+\br{v\cdot\nx\Big(\nx\psi:\b\Big),\bd\cdot\a+(\ik-\bpk)[\re]}-\br{\nx\psi:\b,\left(\mhh\ab\cdot\frac{\nx T}{4T^2}\right)\Big(\bd\cdot\a+(\ik-\bpk)[\re]\Big)}\no\\
&+\e^{-1}\br{\psi\cdot v\mh,\ss_3+\sp}+\bbr{\nx\psi:\b,\bar\ss}.\no
\end{align}
The assumptions $\nx\cdot\psi=0$ and $\psi\big|_{\p\Omega}=0$ eliminates $\e^{-1}P\bbrx{\nx\cdot\psi, \P}$ and $\e^{-1}\ds\int_{\p\Omega}\Big(P\P\psi+\psi\cdot\varpi\Big)\cdot n$ in \eqref{pp 14}. Hence, we have
\begin{align}\label{pp 15}
&-\br{\nx\cdot\Big(\bbb(\nx\psi:\b)\Big),\bb}+\br{\nx\psi:\b,(\bb\cdot v)\left(\ab\cdot\frac{\nx T}{2T^2}\right)}+\br{\nx\psi:\b,\llc[\re]}\\
=&\bbrb{\nx\psi:\b,h}{\gamma_-}-\bbrb{\nx\psi:\b,(1-\pp)[\re]}{\gamma_+}\no\\
&+\br{v\cdot\nx\Big(\nx\psi:\b\Big),\bd\cdot\a+(\ik-\bpk)[\re]}-\br{\nx\psi:\b,\left(\mhh\ab\cdot\frac{\nx T}{4T^2}\right)\Big(\bd\cdot\a+(\ik-\bpk)[\re]\Big)}\no\\
&+\e^{-1}\br{\psi\cdot v\mh,\ss_3+\sp}+\br{\nx\psi:\b,\bar\ss}.\no
\end{align}
Using $\nx\cdot\psi=0$, a standard simplification
\begin{align}
\\
    \nx\cdot\Big(\bbb(\nx\psi:\b)\Big)=&\bbb\cdot\nx\big(\nx\psi:\b\big)+\big(\nx\cdot\bbb\big)\big(\nx\psi:\b\big)
    =\lambda\dx\psi+\bbb\cdot\big(\nx\psi:\nx\b\big)+\big(\nx\cdot\bbb\big)\big(\nx\psi:\b\big),\no
\end{align}
further yields \eqref{pp 16}.
\end{proof}

%%%%%%%%%%%%%%%%%%%%%%%%%%%%%%%%%%%%%%%%%%%%%%%%%%%%%%%%%%%%%%%%%%%%%%%%
\section{Basic Energy Estimates}\label{sec:energy-est}
%%%%%%%%%%%%%%%%%%%%%%%%%%%%%%%%%%%%%%%%%%%%%%%%%%%%%%%%%%%%%%%%%%%%%%%%

In this section, we plan to prove Proposition \ref{thm:energy}.

%%%%%%%%%%%%%%%%%%%%%%%%%%%%%%%%%%%%%%%%%%%%%%%%%%%%%%%%%%%%%%%%%%%%%%%%
\subsection{Energy Structure}
%%%%%%%%%%%%%%%%%%%%%%%%%%%%%%%%%%%%%%%%%%%%%%%%%%%%%%%%%%%%%%%%%%%%%%%%

\begin{lemma}\label{lem:final 9}
Under the assumption \eqref{assumption:boundary}, we have
\begin{align}\label{qq 34}
    \e^{-1}\tnms{(1-\pp)[\re]}{\gamma_+}^2+\e^{-2}\tnm{\be}^2+\e^{-2}\um{(\ik-\bpk)[\re]}^2
    \ls&\e^{-1}\br{\re-c\mh\left(\abs{\vv}^2-5\tq\right)-\Big(\nx\xi\cdot\a\Big),\ss}+\oot\xnm{\re}^2+\oot.
\end{align}
\end{lemma}

\begin{proof}
Taking $\test=\e^{-1}\re$ in the weak formulation \eqref{weak formulation}, we have
\begin{align}
    \e^{-1}\bbr{\vv\cdot\nx\re,\re}+\e^{-1}\br{\mhh\ab\cdot\frac{\nx\tq}{4\tq^2},\re^2}+\e^{-2}\bbr{\re,\lc[\re]}
    =\e^{-1}\bbr{\re,\ss}.
\end{align}

\paragraph{\underline{Step 1: Weak Formulation}}
Using Green's identity in Lemma \ref{remainder lemma 2}, we obtain
\begin{align}
&\frac{\e^{-1}}{2}\iint_{\Omega\times\r^3}\vv\cdot\nx(\re^2)=\frac{\e^{-1}}{2}\int_{\p\Omega\times\r^3}\re^2(\vv\cdot\vn)=\frac{\e^{-1}}{2}\left(\int_{\gamma_+}\babs{\re}^2\ud\gamma-\int_{\gamma_-}\babs{\pp[\re]+h}^2\ud\gamma\right)\\
=&\frac{\e^{-1}}{2}\left(\int_{\gamma_+}\abs{\re}^2\ud\gamma-\int_{\gamma_-}\babs{\pp[\re]}^2\ud\gamma-\int_{\gamma_-}\babs{h}^2\ud\gamma-2\int_{\gamma_-}\pp[\re]h\ud\gamma\right)
=\frac{\e^{-1}}{2}\left(\tnms{(1-\pp)[\re]}{\gamma_+}^2-\tnms{h}{\gamma_-}^2\right).\no
\end{align}
Here, we use Proposition \ref{lemma:final 2} to conclude that
\begin{align}
    \int_{\gamma_-}\pp[\re]h\ud\gamma=0.
\end{align}
Also, we use the decomposition of $(\ik-\pk)[\re]=\bd\cdot\a+(\ik-\bpk)[\re]$ and orthogonality to obtain
\begin{align}
\br{\re,\lc[\re]}
=&\br{(\ik-\bpk)[\re],\lc\Big[(\ik-\bpk)[\re]\Big]}+\bbr{\bd\cdot\a,\lc[\bd\cdot\a]}
\end{align}
where the middle term vanishes due to the orthogonality of $\a$ and $(\ik-\bpk)[\re]$. On the other hand, using the designed orthogonality \eqref{energy 11}, we have
\begin{align}
\bbr{\bd\cdot\a,\lc[\bd\cdot\a]}=\bbr{\bd\cdot\a,\bd\cdot\ab}=\int_{\Omega}\k\abs{\nx\xi}^2+\int_{\Omega}\k\abs{\be}^2.
\end{align}
Noticing the coercivity property of $\lc$ as \eqref{coercivity}, there exists a uniform $\d_0>0$ such that
\begin{align}
    \br{(\ik-\bpk)[\re],\lc\Big[(\ik-\bpk)[\re]\Big]}\geq \d_0\um{(\ik-\bpk)[\re]}^2.
\end{align}
Hence, we have
\begin{align}
    \iint_{\Omega\times\r^3}{\re\lc[\re]}\geq\d_0\um{(\ik-\bpk)[\re]}^2+\int_{\Omega}\k\abs{\nx\xi}^2+\int_{\Omega}\k\abs{\be}^2.
\end{align}
In summary, using Lemma \ref{h-estimate}, we have shown that
\begin{align}\label{energy 01}
    &\frac{\e^{-1}}{2}\tnms{(1-\pp)[\re]}{\gamma_+}^2+\e^{-2}\d_0\um{(\ik-\bpk)[\re]}^2+\e^{-2}\int_{\Omega}\k\abs{\nx\xi}^2+\e^{-2}\int_{\Omega}\k\abs{\be}^2\\
    \leq&-\e^{-1}\br{\mhh\ab\cdot\frac{\nx\tq}{4\tq^2},\re^2}+ \frac{\e^{-1}}{2}\bbr{\re,\ss}
    +\e^{-1}\tnms{h}{\gamma_-}^2
    \leq-\e^{-1}\br{\mhh\ab\cdot\frac{\nx\tq}{4\tq^2},\re^2}+ \e^{-1}\bbr{\re,\ss}
    +\oot\e.\no
\end{align}

\paragraph{\underline{Step 2: Estimates of $\ds\br{\mhh\ab\cdot\frac{\nx\tq}{4\tq^2},\re^2}$}}
We decompose
\begin{align}
\br{\mhh\ab\cdot\frac{\nx\tq}{4\tq^2},\re^2}&=\br{\mhh\ab\cdot\frac{\nx\tq}{4\tq^2},\Big(\pk[\re]+\bd\cdot\a+(\ik-\bpk)[\re]\Big)^2}
=\hk_1+\hk_2+\hk_3+\hk_4,
\end{align}
where
\begin{align}
\hk_1&:=\br{\mhh\ab\cdot\frac{\nx\tq}{4\tq^2},\Big(\pk[\re]\Big)^2}=\br{\mh\ab\cdot\frac{\nx\tq}{4\tq^2},\bigg(\P+\vv\cdot
\bb+\Big(\abs{\vv}^2-5\tq\Big)c\bigg)^2}\\
\hk_2&:=\br{\mhh\ab\cdot\frac{\nx\tq}{4\tq^2},\Big(\bd\cdot\a\Big)^2},\\
\hk_3&:=2\br{\mhh\ab\cdot\frac{\nx\tq}{4\tq^2},\pk[\re](\bd\cdot\a)}=\br{\ab\cdot\frac{\nx\tq}{2\tq^2},\bigg(\P +\vv\cdot
\bb+\Big(\abs{\vv}^2-5\tq\Big)c\bigg)(\bd\cdot\a)},\\
\hk_4&:=\br{\mhh\ab\cdot\frac{\nx\tq}{4\tq^2},\Big(2\pk[\re]+2(\bd\cdot\a)+(\ik-\bpk)[\re]\Big)(\ik-\bpk)[\re]}.
\end{align}
Due to symmetry/oddness, $\P^2$, $\bb\otimes\bb$, $c^2$ and $\P c$ terms vanish. Due to orthogonality, $\P \bb$ term also vanishes. Hence, recalling the precise computation $\ds\int_{\r^3}\m\abs{\vv}^2\Big(\abs{\vv}^2-5\tq\Big)^2=30PT^2$, we are left with
\begin{align}
\hk_1&=2\br{\mh\ab\cdot\frac{\nx\tq}{4\tq^2},(\vv\cdot\bb)\Big(\abs{\vv}^2-5\tq\Big)c}
=2\int_{\Omega}\frac{\nx\tq}{4\tq^2}\cdot\int_{\r^3}\mh\ab(\vv\cdot\bb)\Big(\abs{\vv}^2-5\tq\Big)c\\
&=\frac{2}{3}\int_{\Omega}\frac{\nx\tq}{4\tq^2}\cdot(\bb c)\int_{\r^3}\m\abs{\vv}^2\Big(\abs{\vv}^2-5\tq\Big)^2
=5P\bbrx{\nx\tq,\bb c}.\no
\end{align}
Due to symmetry/oddness, we know $\hk_2=0$. On the other hand, for $\hk_3$,
due to symmetry, the $\bb$ term vanishes. Hence, we have
\begin{align}
\hk_3&=\br{\ab\cdot\frac{\nx\tq}{2\tq^2},\bigg(\P +\Big(\abs{\vv}^2-5\tq\Big)c\bigg)(\bd\cdot\a)}\\
&=\br{\ab\cdot\frac{\nx\tq}{2\tq^2},\P (\bd\cdot\a)}
+\br{\ab\cdot\frac{\nx\tq}{2\tq^2},\Big(\abs{\vv}^2-5\tq\Big)c(\bd\cdot\a)}=\brx{\bd\cdot\frac{\nx\tq}{2\tq^2},\P \k+c\si}.\no
\end{align}
In summary, we have
\begin{align}\label{energy 02}
\br{\mhh\ab\cdot\frac{\nx\tq}{4\tq^2},\re^2}=&5P\bbrx{\nx\tq,\bb c}+\brx{\bd\cdot\frac{\nx\tq}{2\tq^2},\k\P+\si c}\\
&+\br{\mhh\ab\cdot\frac{\nx\tq}{4\tq^2},\Big(2\pk[\re]+2(\bd\cdot\a)+(\ik-\bpk)[\re]\Big)(\ik-\bpk)[\re]}.\no
\end{align}
Inserting \eqref{energy 02} into \eqref{energy 01}, we obtain
\begin{align}\label{final 37}
    &\frac{\e^{-1}}{2}\tnms{(1-\pp)[\re]}{\gamma_+}^2+\e^{-2}\d_0\um{(\ik-\bpk)[\re]}^2+\e^{-2}\int_{\Omega}\k\abs{\nx\xi}^2+\e^{-2}\int_{\Omega}\k\abs{\be}^2\\
    \leq&-\e^{-1}5P\bbrx{\nx\tq,\bb c}-\e^{-1}\brx{\bd\cdot\frac{\nx\tq}{2\tq^2},\k\P+\si c}\no\\
    &-\e^{-1}\br{\mhh\ab\cdot\frac{\nx\tq}{4\tq^2},\Big(2\pk[\re]+2(\bd\cdot\a)+(\ik-\bpk)[\re]\Big)(\ik-\bpk)[\re]}+ \e^{-1}\bbr{\re,\ss}+\oot\e.\no
\end{align}
Taking inner product of $c$ and \eqref{aa 35}, we have
\begin{align}
5P\bbrx{\nx\tq,\bb c}&=-\bbrx{c,\nx\cdot(\k\bd)}+\br{c\mh\left(\abs{\vv}^2-5\tq\right),\ss}.
\end{align}
Hence, we have
\begin{align}\label{energy 03}
    &\frac{\e^{-1}}{2}\tnms{(1-\pp)[\re]}{\gamma_+}^2+\e^{-2}\d_0\um{(\ik-\bpk)[\re]}^2+\e^{-2}\int_{\Omega}\k\abs{\nx\xi}^2+\e^{-2}\int_{\Omega}\k\abs{\be}^2\\
    \leq&\e^{-1}\bbrx{c,\nx\cdot(\k\bd)}-\e^{-1}\brx{\bd\cdot\frac{\nx\tq}{2\tq^2},\k\P+\si c}\no\\
    &-\e^{-1}\br{\mhh\ab\cdot\frac{\nx\tq}{4\tq^2},\Big(2\pk[\re]+2(\bd\cdot\a)+(\ik-\bpk)[\re]\Big)(\ik-\bpk)[\re]}+\e^{-1}\br{\re-c\mh\left(\abs{\vv}^2-5\tq\right),\ss}
    +\oot\e.\no
\end{align}

\paragraph{\underline{Step 3: Estimates of $\ds\bbrx{c,\nx\cdot(\k\bd)}$ and $\ds\brx{\nx\xi\cdot\frac{\nx\tq}{2\tq^2},\k\P+\si c}$}}
Letting $\phi=\xi$ in \eqref{pp 03} and using \eqref{tt 01}, we obtain
\begin{align}\label{new 3}
&-\bbrx{c,\nx\cdot(\k\bd)}+\brx{\frac{\nx T}{2T^2}\cdot\nx\xi,\kappa\P+\sigma c}+\e^{-1}\int_{\Omega}\kappa\abs{\nx\xi}^2\\
=&\bbrb{\nx\xi\cdot\a,h}{\gamma_-}-\bbrb{\nx\xi\cdot\a,(1-\pp)[\re]}{\gamma_+}\no\\
&+\br{v\cdot\nx\Big(\nx\xi\cdot\a\Big),(\ik-\bpk)[\re]}-\br{\Big(\nx\xi\cdot\a\Big)\left(\mhh\ab\cdot\frac{\nx T}{2T^2}\right),(\ik-\bpk)[\re]}+\bbr{\nx\xi\cdot\a,\ss}.\no
\end{align}
Subtracting $\e^{-1}\times$ \eqref{new 3} from \eqref{energy 03}, we arrive at
\begin{align}\label{energy 04}
    &\frac{\e^{-1}}{2}\tnms{(1-\pp)[\re]}{\gamma_+}^2+\e^{-2}\tnm{\be}^2+\e^{-2}\d_0\um{(\ik-\bpk)[\re]}^2\\
    \leq&-\e^{-1}\brx{\be\cdot\frac{\nx\tq}{2\tq^2},\k\P+\si c}-\e^{-1}\br{\mhh\ab\cdot\frac{\nx\tq}{4\tq^2},\Big(2\pk[\re]+2(\bd\cdot\a)+(\ik-\bpk)[\re]\Big)(\ik-\bpk)[\re]}\no\\
    &-\e^{-1}\br{v\cdot\nx\Big(\nx\xi\cdot\a\Big),(\ik-\bpk)[\re]}+\e^{-1}\br{(\nx\xi\cdot\a)\left(\mhh\ab\cdot\frac{\nx T}{2T^2}\right),(\ik-\bpk)[\re]}\no\\
    &+\e^{-1}\bbrb{\nx\xi\cdot\a,(1-\pp)[\re]}{\gamma_+}-\e^{-1}\bbrb{\nx\xi\cdot\a,h}{\gamma_-}+\e^{-1}\br{\re-c\mh\left(\abs{\vv}^2-5\tq\right)-\Big(\nx\xi\cdot\a\Big),\ss}+\oot\e.\no
\end{align}
Note that the crucial term $\e^{-1}\ds\int_{\Omega}\kappa\abs{\nx\xi}^2$ on the LHS of \eqref{energy 03} is cancelled at this step. This is the cost paid to remove the difficult term $-\e^{-1}5P\bbrx{\nx\tq,\bb c}$ from \eqref{final 37}.

We can further estimate each term on the RHS using H\"older's inequality, Lemma \ref{h-estimate} and Lemma \ref{remark 01}: 
\begin{align}
    &\abs{\e^{-1}\brx{\be\cdot\frac{\nx\tq}{2\tq^2},\k\P+\si c}}\ls \oot\tnm{\P}^2+\oot\tnm{c}^2+\oot\e^{-2}\tnm{\be}^2\ls\oot\xnm{\re}^2,
\end{align}
\begin{align}
    &\abs{\e^{-1}\br{\mhh\ab\cdot\frac{\nx\tq}{4\tq^2},\Big(2\pk[\re]+2(\bd\cdot\a)\Big)(\ik-\bpk)[\re]}}\\
    \ls&\ \oot\tnm{\pk[\re]}^2+\oot\tnm{\nx\xi}^2+\oot\tnm{\be}^2+\oot\e^{-2}\um{(\ik-\bpk)[\re]}^2\ls\oot\xnm{\re}^2,\no
\end{align}
\begin{align}
    \abs{\e^{-1}\br{v\cdot\nx\Big(\nx\xi\cdot\a\Big),(\ik-\bpk)[\re]}}&\ls\cc_1^{-1}\nm{\xi}^2_{H^2}+\cc_1\e^{-2}\um{(\ik-\bpk)[\re]}^2\\
    &\ls\cc_1\e^{-2}\um{(\ik-\bpk)[\re]}^2+\oot\xnm{\re}^2+\oot\e,\no
\end{align}
\begin{align}
    \abs{\e^{-1}\br{(\nx\xi\cdot\a)\left(\mhh\ab\cdot\frac{\nx T}{2T^2}\right),(\ik-\bpk)[\re]}}&\ls \nm{\xi}^2_{H^1}+\oot\e^{-2}\um{(\ik-\bpk)[\re]}^2\ls\oot\xnm{\re}^2+\oot\e,
\end{align}
\begin{align}
    \abs{\e^{-1}\bbrb{\nx\xi\cdot\a,(1-\pp)[\re]}{\gamma_+}}&\ls \cc_2^{-1}\e^{-1}\tnms{\nx\xi}{\p\Omega}^2+\cc_2\e^{-1}\tnms{(1-\pp)[\re]}{\gamma_+}^2\\
    &\ls\cc_2\e^{-1}\tnms{(1-\pp)[\re]}{\gamma_+}^2+\oot\xnm{\re}^2+\oot,\no
\end{align}
\begin{align}
    \abs{\e^{-1}\bbrb{\nx\xi\cdot\a,h}{\gamma_-}}\ls \e^{-1}\tnms{\nx\xi}{\p\Omega}^2+\e^{-1}\tnms{h}{\gamma_-}^2\ls \oot\xnm{\re}^2+\oot.
\end{align}
Also, using the weighted $L^{\infty}$ estimate to control the large velocity, we have
\begin{align}
    &\abs{\e^{-1}\br{\mhh\ab\cdot\frac{\nx\tq}{4\tq^2},\Big((\ik-\bpk)[\re]\Big)^2}}\\
    \ls&\abs{\e^{-1}\int_{\Omega}\int_{\abs{v}\leq\e^{-\frac{1}{4}}}\left(\mhh\ab\cdot\frac{\nx\tq}{4\tq^2}\right)\Big((\ik-\bpk)[\re]\Big)^2}+\abs{\e^{-1}\int_{\Omega}\int_{\abs{v}\geq\e^{-\frac{1}{4}}}\left(\mhh\ab\cdot\frac{\nx\tq}{4\tq^2}\right)\Big((\ik-\bpk)[\re]\Big)^2}\no\\
    \ls&\oot\e^{-2}\um{(\ik-\bpk)[\re]}^2+\oot\e^3\lnmm{\re}^2\ls\oot\xnm{\re}^2,\no
\end{align}
Collecting all above, for $\cc_1,\cc_2\ll1$, we have \eqref{qq 34}.
\end{proof}

%%%%%%%%%%%%%%%%%%%%%%%%%%%%%%%%%%%%%%%%%%%%%%%%%%%%%%%%%%%%%%%%%%%%%%%%
\subsection{Estimates of Source Terms}
%%%%%%%%%%%%%%%%%%%%%%%%%%%%%%%%%%%%%%%%%%%%%%%%%%%%%%%%%%%%%%%%%%%%%%%%

In this section, we consider the source term in \eqref{qq 34}:
\begin{align}
    \e^{-1}\br{\re-c\mh\left(\abs{\vv}^2-5\tq\right)-\Big(\nx\xi\cdot\a\Big),\ss}
    =&\e^{-1}\br{\P\mh+\bb\cdot v\mh+\be\cdot\a+(\ik-\bpk)[\re],\ss}.
\end{align}

\begin{lemma}\label{lem:final 10}
Under the assumption \eqref{assumption:boundary}, we have
\begin{align}\label{qq 36}
    \abs{\e^{-1}\br{\re-c\mh\Big(\abs{\vv}^2-5\tq\Big)-\Big(\nx\xi\cdot\a\Big),\ss}}
    \ls&\oo\Big(\e^{-2}\tnm{\be}^2+\e^{-2}\tnm{(\ik-\bpk)[\re]}^2\Big)\\
    &+\oot\xnm{\re}^2+\xnm{\re}^4+\oot.\no
\end{align}
\end{lemma}

\begin{proof}
Due to the orthogonality of $\Gamma$, using Lemma \ref{ssl-estimate}, Lemma \ref{ss0-estimate}, Lemma \ref{ss1-estimate}, Lemma \ref{ss5-estimate}, we know
\begin{align}
    &\abs{\e^{-1}\br{\P\mh+\bb\cdot v\mh+\be\cdot\a+(\ik-\bpk)[\re],\llc[\re]+\ss_0+\ss_1+\ss_5}}\\
    =&\e^{-1}\abs{\bbr{\be\cdot\a+(\ik-\bpk)[\re],\llc[\re]+\ss_0+\ss_1+\ss_5}}
    \ls\oot\e^{-2}\tnm{\be}^2+\oot\e^{-2}\tnm{(\ik-\bpk)[\re]}^2+\oot\um{\re}^2+\oot\e.\no
\end{align}
Also, using Lemma \ref{ss4-estimate} and Lemma \ref{ssp-estimate}, we have
\begin{align}
    &\abs{\e^{-1}\br{\P\mh+\bb\cdot v\mh+\be\cdot\a+(\ik-\bpk)[\re],\ss_4+\ss_6}}\\
    \ls&\oot\e^{-2}\tnm{\be}^2+\oot\e^{-2}\um{(\ik-\bpk)[\re]}^2+\oot\e^{-1}\tnm{\P}^2+\oot\e^{-1}\tnm{\bb}^2+\e^{-1}\tnm{\ss_4}^2+\e^{-1}\tnm{\ss_6}^2\no\\
    \ls&\oot\e^{-2}\tnm{\be}^2+\oot\e^{-2}\um{(\ik-\bpk)[\re]}^2+\oot\e^{-1}\tnm{\P}^2+\oot\e^{-1}\tnm{\bb}^2+\oot\e.\no
\end{align}
Also, we know
\begin{align}\label{final 09}
    \abs{\e^{-1}\br{\P\mh+\bb\cdot v\mh+\be\cdot\a+(\ik-\bpk)[\re],\ss_3}}
    \ls&\abs{\e^{-1}\br{\P\mh+\bb\cdot v\mh,\ss_3}}+\abs{\e^{-1}\br{\be\cdot\a+(\ik-\bpk)[\re],\ss_3}}.
\end{align}
In the first term of \eqref{final 09}, using \eqref{ss3-estimate5} and \eqref{ss3-estimate6} with $\N=2$, we obtain
\begin{align}
    &\abs{\e^{-1}\br{\P\mh+\bb\cdot v\mh,\ss_3}}\ls \abs{\e^{-1}\br{\P\mh+\bb\cdot v\mh,\sx}}+\abs{\e^{-1}\br{\P\mh+\bb\cdot v\mh,\sy+\sz}}\\
    \ls&\oot\e^{-\frac{1}{2}}\tnm{\nabla_v\Big(\P\mh+\bb\cdot v\mh\Big)}+\e^{-1}\nm{\P\mh+\bb\cdot v\mh}_{L^2_xL^{\infty}_v}\nm{\sy+\sz}_{L^2_xL^1_v}\no\\
    \ls&\oot\e^{-1}\tnm{\P}^2+\oot\e^{-1}\tnm{\bb}^2+\oot.\no
\end{align}
For the second term of \eqref{final 09}, we can directly bound
\begin{align}
    \abs{\e^{-1}\br{\be\cdot\a+(\ik-\bpk)[\re],\ss_3}}
    \ls&\e^{-1}\Big(\tnm{\be}^2+\tnm{(\ik-\bpk)[\re]}\Big)\tnm{\ss_3}\\
    \ls&\oot\e^{-2}\tnm{\be}^2+\oot\e^{-2}\tnm{(\ik-\bpk)[\re]}^2+\oot.\no
\end{align}
In addition, due to the orthogonality of $\Gamma$, using Lemma \ref{ss2-estimate}, we know
\begin{align}
    &\abs{\e^{-1}\br{\P\mh+\bb\cdot v\mh+\be\cdot\a+(\ik-\bpk)[\re],\ss_2}}=\e^{-1}\abs{\bbr{\be\cdot\a+(\ik-\bpk)[\re],\ss_2}}\\
    \ls& \oo\e^{-2}\tnm{\be}^2+\oo\e^{-2}\tnm{(\ik-\bpk)[\re]}^2+\xnm{\re}^4.\no
\end{align}
In summary, we have
\begin{align}\label{qq 36,}
    \abs{\e^{-1}\br{\re-c\Big(\abs{\vv}^2-5\tq\Big)-\Big(\nx\xi\cdot\a\Big),\ss}}
    \ls&\oo\Big(\e^{-2}\tnm{\be}^2+\e^{-2}\tnm{(\ik-\bpk)[\re]}^2\Big)\\
    &+\oot\xnm{\re}^2+\xnm{\re}^4+\oot.\no
\end{align}
\end{proof}

\begin{proof}[Proof of Proposition \ref{thm:energy}]
Inserting \eqref{qq 36} into \eqref{qq 34}, our desired result follows.
\end{proof}

%%%%%%%%%%%%%%%%%%%%%%%%%%%%%%%%%%%%%%%%%%%%%%%%%%%%%%%%%%%%%%%%%%%%%%%%
\subsection{Extended Energy Estimates}
%%%%%%%%%%%%%%%%%%%%%%%%%%%%%%%%%%%%%%%%%%%%%%%%%%%%%%%%%%%%%%%%%%%%%%%%

\begin{lemma}\label{lem:final 11}
    Let $\re$ be a solution to \eqref{remainder}. Under the assumption \eqref{assumption:boundary}, we have
    \begin{align}\label{final 57}
        \pnm{\be}{6}+\pnm{(\ik-\bpk)[\re]}{6}+\pnms{\m^{\frac{1}{4}}(1-\pp)[\re]}{4}{\gamma_+}
        \ls&\oot\xnm{\re}+\xnm{\re}^2+\oot.
    \end{align}
\end{lemma}
\begin{proof}
By interpolation and Proposition \ref{thm:energy}, we obtain
\begin{align}
    \pnm{\be}{6}\ls&\tnm{\be}^{\frac{1}{3}}\lnmm{\be}^{\frac{2}{3}}\ls \Big(\oot\xnm{\re}^{\frac{1}{3}}+\xnm{\re}^{\frac{1}{3}}+\oot\Big)\left(\e^{\frac{1}{2}}\lnmm{\re}\right)^{\frac{2}{3}}\label{tt 17}
    \ls \oot\xnm{\re}+\xnm{\re}^2+\oot,\\
    \pnm{(\ik-\bpk)[\re]}{6}\ls&\tnm{(\ik-\bpk)[\re]}^{\frac{1}{3}}\lnmm{(\ik-\bpk)[\re]}^{\frac{2}{3}}\label{tt 18}\\
    \ls& \Big(\oot\xnm{\re}^{\frac{1}{3}}+\xnm{\re}^{\frac{1}{3}}+\oot\Big)\left(\e^{\frac{1}{2}}\lnmm{\re}\right)^{\frac{2}{3}}\ls\oot\xnm{\re}+\xnm{\re}^2+\oot,\no\\
    \pnms{\m^{\frac{1}{4}}(1-\pp)[\re]}{4}{\gamma_+}\ls&\tnms{(1-\pp)[\re]}{\gamma_+}^{\frac{1}{2}}\lnmms{(1-\pp)[\re]}{\gamma_+}^{\frac{1}{2}}\label{tt 19}\\
    \ls& \Big(\oot\xnm{\re}^{\frac{1}{2}}+\xnm{\re}^{\frac{1}{2}}+\oot\Big)\Big(\e^{\frac{1}{2}}\lnmms{\re}{\gamma_+}\Big)^{\frac{1}{2}}
    \ls\oot\xnm{\re}+\xnm{\re}^2+\oot.\no
\end{align}
Then the desired result follows from \eqref{tt 17}\eqref{tt 18}\eqref{tt 19}.
\end{proof}

\begin{lemma}[Ukai's Trace Theorem, Lemma 2.1 of \cite{Esposito.Guo.Kim.Marra2013} ]\label{remainder lemma 1}
Define the near-grazing set of $\gamma_+$ or $\gamma_-$ as
\begin{align}
\gamma_{\pm}^{\d}=\left\{(\vx,\vv)\in\gamma_{\pm}:
\abs{\vn(\vx)\cdot\vv}\leq\d\ \text{or}\ \abs{\vv}\geq\frac{1}{\d}\ \text{or}\
\abs{\vv}\leq\d\right\}.
\end{align}
Then for $p\in[1,\infty)$
\begin{align}
\abs{f\id_{\gamma_{\pm}\backslash\gamma_{\pm}^{\d}}}_{L^p}\leq
C(\delta)\bigg(\nm{f}_{L^p}+\nm{\vv\cdot\nx f}_{L^p}\bigg).
\end{align}
\end{lemma}

\begin{proposition}\label{lem:final 12}
Under the assumption \eqref{assumption:boundary}, we have
\begin{align}\label{qq 32}
    \tnms{\pp[\re]}{\gamma}
    \ls&\tnm{\P}+\tnm{\bb}+\tnm{c}+\oot\xnm{\re}+\xnm{\re}+\oot.
\end{align}
\end{proposition}

\begin{proof}
From Proposition \ref{thm:energy} and Lemma \ref{remark 01}, we have
\begin{align}
\tnm{\re}\ls& \tnm{\P}+\tnm{\bb}+\tnm{c}+\nm{\xi}_{H^1}+\tnm{\be}+\tnm{(\ik-\bpk)[\re]}\\
\ls&\tnm{\P}+\tnm{\bb}+\tnm{c}+\oot\xnm{\re}+\xnm{\re}^2+\oot.\no
\end{align}
Multiplying $\re$ on both sides of \eqref{remainder.}, we know
\begin{align}
\frac{1}{2}v\cdot\nx\big(\re\big)^2
&+\re\bigg(\mhh\left(\ab\cdot\frac{\nx \tq}{4\tq^2}\right)\re-\e^{-1}\xi(\nx\cdot\ab)+\e^{-1}\be\cdot\ab+\e^{-1}\lc\Big[(\ik-\bpk)[\re]\Big]-\ss\bigg)=0.
\end{align}
Taking absolute value and integrating over $\Omega\times\r^3$, with the help of Cauchy's inequality and Lemma \ref{remark 01}, we deduce
\begin{align}
\pnm{v\cdot\nx\big(\re\big)^2}{1}
\ls&\abs{\br{\re, \mhh\left(\ab\cdot\frac{\nx \tq}{4\tq^2}\right)\re-\e^{-1}\xi(\nx\cdot\ab)+\e^{-1}\be\cdot\ab+\e^{-1}\lc\Big[(\ik-\bpk)[\re]\Big]-\ss}}\\
\ls&\tnm{\P}^2+\tnm{\bb}^2+\tnm{c}^2+\oot\xnm{\re}^2+\xnm{\re}^2+\oot+\abs{\bbr{\re,\ss}}\no\\
\ls&\tnm{\P}^2+\tnm{\bb}^2+\tnm{c}^2+\oot\xnm{\re}^2+\xnm{\re}^2+\oot,\no
\end{align}
where we utilize the source term estimates in Section \ref{sec:bs}
\begin{align}
    \bbr{\re,\ss}
    \ls& \oot\xnm{\re}^2+\xnm{\re}^4+\oot\e.
\end{align}
Then by Lemma \ref{remainder lemma 1}, we have
\begin{align}
\tnms{\id_{\gamma\backslash\gamma^{\d}}\re}{\gamma}^2
\leq&
C(\d)\left(\tnm{\re}^2+\pnm{\vv\cdot\nx\big(\re\big)^2}{1}\right)
\ls\tnm{\P}^2+\tnm{\bb}^2+\tnm{c}^2+\oot\xnm{\re}^2+\xnm{\re}^2+\oot.
\end{align}
Notice that
\begin{align}
    \pp\Big[\id_{\gamma\backslash\gamma^{\d}}\re\Big]=&\pp\Big[\id_{\gamma\backslash\gamma^{\d}}\Big(\pp[\re]+\id_{\gamma_+}(1-\pp)[\re]+\id_{\gamma_-}h\Big)\Big]
\end{align}
Based on the definition, we can rewrite $\pp[\re]=z_{\gamma}(\vx)\mh$ for a suitable function
$z_{\gamma}(\vx)$ and
for $\d$ small, we deduce
\begin{align}
\tnms{\pp\Big[\id_{\gamma\backslash\gamma^{\d}}\pp[\re]\Big]}{\gamma}^2
=&\int_{\p\Omega}\abs{z_{\gamma}(\vx)}^2
\left(\int_{\abs{\vv\cdot\vn(\vx)}\geq\d,\d\leq\abs{\vv}\leq\frac{1}{\d}}\m(\vv)\abs{\vv\cdot\vn(\vx)}\ud{\vv}\right)\ud{\vx}\\
\geq&\half\left(\int_{\p\Omega}\abs{z_{\gamma}(\vx)}^2\ud{\vx}\right)\left(\int_{\r^2}\m(\vv)\abs{\vv\cdot\vn(\vx)}\ud{\vv}\right)=\half\tnms{\pp[\re]}{\gamma}^2,\no
\end{align}
where we utilized the fact that
\begin{align}
\int_{\abs{\vv\cdot\vn(\vx)}\leq\d}\m(\vv)\abs{\vv\cdot\vn(\vx)}\ud{\vv}&\leq C\d,\quad
\int_{\abs{\vv}\leq\d\ or\
\abs{\vv}\geq\frac{1}{\d}}\m(\vv)\abs{\vv\cdot\vn(\vx)}\ud{\vv}\leq C\d.
\end{align}
Also, we have
\begin{align}
    \tnms{\pp\Big[\id_{\gamma\backslash\gamma^{\d}}\Big(\id_{\gamma_+}(1-\pp)[\re]+\id_{\gamma_-}h\Big)\Big]}{\gamma}\ls\tnms{(1-\pp)[\re]}{\gamma_+}+\tnms{h}{\gamma_-}.
\end{align}
Therefore, from $\ds\tnms{\pp\Big[\id_{\gamma\backslash\gamma^{\d}}\re\Big]}{\gamma}\leq C
\tnms{\id_{\gamma\backslash\gamma^{\d}}\re}{\gamma}$,
we conclude
\begin{align}
\tnms{\pp[\re]}{\gamma}
\leq&\tnms{\pp\Big[\id_{\gamma\backslash\gamma^{\d}}\pp[\re]\Big]}{\gamma}
\leq\tnms{\pp\Big[\id_{\gamma\backslash\gamma^{\d}}\re\Big]}{\gamma}+\tnms{\pp\Big[\id_{\gamma\backslash\gamma^{\d}}\Big(\id_{\gamma_+}(1-\pp)[\re]+\id_{\gamma_-}h\Big)\Big]}{\gamma}\\
\leq& C\tnms{\id_{\gamma\backslash\gamma^{\d}}\re}{\gamma}+C\tnms{(1-\pp)[\re]}{\gamma_+}+C\tnms{h}{\gamma_-}.\no
\end{align}
Hence, using Lemma \ref{h-estimate} and Proposition \ref{thm:energy}, we conclude
\begin{align}\label{qq 32,}
\tnms{\pp[\re]}{\gamma}
\ls&\tnm{\P}+\tnm{\bb}+\tnm{c}+\oot\xnm{\re}+\xnm{\re}+\oot.
\end{align}
\end{proof}

%%%%%%%%%%%%%%%%%%%%%%%%%%%%%%%%%%%%%%%%%%%%%%%%%%%%%%%%%%%%%%%%%%%%%%%%
\section{\texorpdfstring{$\P$}{} and \texorpdfstring{$\z$}{} Estimates}
%%%%%%%%%%%%%%%%%%%%%%%%%%%%%%%%%%%%%%%%%%%%%%%%%%%%%%%%%%%%%%%%%%%%%%%%

%%%%%%%%%%%%%%%%%%%%%%%%%%%%%%%%%%%%%%%%%%%%%%%%%%%%%%%%%%%%%%%%%%%%%%%%
\subsection{\texorpdfstring{$L^{\N}$}{} Estimates of \texorpdfstring{$\P$}{}}\label{sec:p-est}
%%%%%%%%%%%%%%%%%%%%%%%%%%%%%%%%%%%%%%%%%%%%%%%%%%%%%%%%%%%%%%%%%%%%%%%%

\begin{proof}[Proof of Proposition \ref{prop:p-bound}]
To establish \eqref{oo 28=}, thanks to energy estimates, it suffices to establish for $2\leq\N\leq 6$
\begin{align}\label{oo 28}
    \pnm{\P}{\N}\ls&\pnm{(\ik-\bpk)[\re]}{\N}+\oot\e^{\frac{2}{\N}}.
\end{align}
Denote a test function $\psi=\mh(\vv)\Big(\vv\cdot\phi(x)\Big)$,
where $\phi$ is defined via solving the Bogovskii problem
\begin{align}
\left\{
\begin{array}{l}
-\nx\cdot\phi=\ds \P\abs{\P}^{\N-2}-\frac{1}{\abs{\Omega}}\int_{\Omega}\P(x)\abs{\P(x)}^{\N-2}\ud x\ \ \text{in}\ \ \Omega,\\\rule{0ex}{1.0em}
\phi=0\ \ \text{on}\ \ \p\Omega.
\end{array}
\right.
\end{align}
Based on \cite{Bogovskii1979} and \cite[Theorem 2.4]{Acosta.Duran2017}, there exists a solution $\phi$ satisfying
\begin{align}
\nm{\psi}_{W^{1,\frac{\N}{\N-1}}L^{\infty}_{\varrho,\vartheta}}\ls\nm{\phi}_{W^{1,\frac{\N}{\N-1}}}\ls\pnm{\P}{\N}^{\N-1}.
\end{align}
Taking $\test=\psi $ in \eqref{weak formulation}, considering $\psi\big|_{\gamma}=0$, we obtain
\begin{align}
    -\bbr{\re,\vv\cdot\nx\psi }+\br{\mhh\ab\cdot\frac{\nx\tq}{4\tq^2}\re,\psi }+\e^{-1}\bbr{\lc[\re],\psi}
    &=\bbr{\ss,\psi }.
\end{align}
Direct computation reveals that
\begin{align}
-\bbr{\re,\vv\cdot\nx\psi }+\br{\mhh\ab\cdot\frac{\nx\tq}{4\tq^2}\re,\psi }=-\br{\mh\re,\vv\cdot\nx\left(\mhh\psi \right)}.
\end{align}
We will split $\re$ according to \eqref{pp 01} and analyze each term.
Due to oddness, we have
\begin{align}
-\br{(\bb\cdot\vv)\m+\mh(\bd\cdot\a),\vv\cdot\nx\left(\mhh\psi \right)} =0.
\end{align}
On the other hand, using H\"{o}lder's inequality, we have
\begin{align}
\abs{\br{\mh(\ik-\bpk)[\re],\vv\cdot\nx\left(\mhh\psi \right)}}
&\ls\pnm{(\ik-\bpk)[\re]}{\N}\nm{\phi }_{W^{1,\frac{\N}{\N-1}}}\ls\pnm{(\ik-\bpk)[\re]}{\N}\pnm{\P}{\N}^{\N-1}.
\end{align}
For $c\left(\abs{\vv}^2-5\tq\right)\mh$ part, due to orthogonality, we know
\begin{align}\label{kernel 8}
&-\br{\m c\left(\abs{\vv}^2-5\tq\right),\vv\cdot\nx\left(\mhh\psi \right)}
=-\br{\mh c,\ab\cdot\nx\left(\mhh\psi \right)}=0.
\end{align}
For $\P\mh$ part, considering \eqref{final 31} and $\ds\int_{\r^3}\abs{\vv}^2\m=3P$, we have
\begin{align}
-\br{\P\m ,\vv\cdot\nx\left(\mhh\psi \right)}
=&-\br{\P\m,\vv\cdot\nx\Big(\vv\cdot\phi \Big)}
=-\frac{1}{3}\int_{\Omega}p\big(\nx\cdot\phi \big)\int_{\r^3}\abs{\vv}^2\m
=P\pnm{\P}{\N}^{\N}.
\end{align}
Also, due to orthogonality, we have
\begin{align}\label{kernel 12}
\bbr{\lc[\re],\psi}&=\br{\lc\Big[\bd\cdot\a+(\ik-\bpk)[\re]\Big],\psi}
=0 .
\end{align}
Collecting all above, we have shown that
\begin{align}
   \pnm{\P}{\N}^{\N}\ls \pnm{(\ik-\bpk)[\re]}{\N}\pnm{\P}{\N}^{\N-1}+\abs{\bbr{\ss,\psi}}.
\end{align}
Due to orthogonality and oddness, we have
\begin{align}
    \bbr{\ss,\psi}=\bbr{\ss_3,\psi}+\bbr{\sp,\psi}.
\end{align}
Notice that $\ss_3$ contains derivatives of $\fb_1$ in the rescaled variable $\eta=\e^{-1}\mn$. We define a natural extension of $\p_{\mn}\psi$ with zero value for $\mn>1$ (for the convenience of Hardy's inequality). Recall that $\psi=0$ on $\p\Omega$. Using Hardy's inequality, Lemma \ref{ss3-estimate} for $\sy,\sz$ and Remark \ref{ss3-remark} for $\sx$, we have
\begin{align}\label{oo 30}
    &\abs{\bbr{\ss_3,\psi}}=\abs{\br{\ss_3,\int_0^{\mn}\p_{\mn}\psi}}=\e\abs{\br{\eta\ss_3,\frac{1}{\mn}\int_0^{\mn}\p_{\mn}\psi}}\\
    \ls&\e\Big(\jnm{\eta\left(\sy+\sz\right)}+\jnm{\eta\fb_1}\Big)\knm{\frac{1}{\mn}\int_0^{\mn}\p_{\mn}\psi}\ls\e\Big(\jnm{\eta\left(\sy+\sz\right)}+\jnm{\eta\fb_1}\Big)\knm{\p_{\mn}\psi}\no\\
    \ls&\e\Big(\jnm{\eta\left(\sy+\sz\right)}+\jnm{\eta\fb_1}\Big)\pnm{\P}{\N}^{\N-1}\ls\oot\e^{\frac{2}{\N}}\pnm{\P}{\N}^{\N-1}.\no
\end{align}
Also, using Lemma \ref{ssp-estimate}, we know
\begin{align}
    \abs{\bbr{\sp,\psi}}\ls \pnm{\sp}{\N}\pnm{\psi}{\frac{\N}{\N-1}}\ls\oot\e\pnm{\P}{\N}^{\N-1}.
\end{align}
In summary, we have shown
\begin{align}
    \pnm{\P}{\N}^{\N}\ls&\pnm{(\ik-\bpk)[\re]}{\N}\pnm{\P}{\N}^{\N-1}+\oot\e^{\frac{2}{\N}}\pnm{\P}{\N}^{\N-1}+\oot\e\pnm{\P}{\N}^{\N-1}.
\end{align}
Hence, we have
\begin{align}\label{oo 28,}
    \pnm{\P}{\N}\ls&\pnm{(\ik-\bpk)[\re]}{\N}+\oot\e^{\frac{2}{\N}}.
\end{align}
Then the desired estimate \eqref{oo 28} follows from Proposition \ref{thm:energy} and Lemma \ref{final 11}.
\end{proof}

%%%%%%%%%%%%%%%%%%%%%%%%%%%%%%%%%%%%%%%%%%%%%%%%%%%%%%%%%%%%%%%%%%%%%%%%
\subsection{$L^{\N}$ Estimates of $\z$}
%%%%%%%%%%%%%%%%%%%%%%%%%%%%%%%%%%%%%%%%%%%%%%%%%%%%%%%%%%%%%%%%%%%%%%%%

%%%%%%%%%%%%%%%%%%%%%%%%%%%%%%%%%%%%%%%%%%%%%%%%%%%%%%%%%%%%%%%%%%%%%%%%
\subsubsection{Estimates of $\z$}\label{sec:z-est}
%%%%%%%%%%%%%%%%%%%%%%%%%%%%%%%%%%%%%%%%%%%%%%%%%%%%%%%%%%%%%%%%%%%%%%%%

\begin{proof}[Proof of Proposition \ref{prop:z-bound}]
To establish \eqref{oo 33=}, thanks to energy estimates, it suffices to establish for $2\leq\N\leq6$
\begin{align}\label{oo 33}
    \pnm{\z}{\N}\ls&\pnm{(\ik-\bpk)[\re]}{\N}+\jnms{\m^{\frac{1}{4}}(1-\pp)[\re]}{\gamma_+}+\oot\xnm{\re}+\xnm{\re}^2+\oot.
\end{align}
In \eqref{pp 03.}, consider $\phi(x)$ satisfying
\begin{align}
\left\{
\begin{array}{l}
-\nx\cdot\big(\k\nx\phi \big)=\z\abs{\z}^{\N-2}\ \ \text{in}\ \
\Omega,\\\rule{0ex}{1.2em}
\phi =0\ \ \text{on}\ \ \p\Omega.
\end{array}
\right.
\end{align}
Then we have
\begin{align}
    &\jnm{\z}^{\N}+\e^{-1}\bbrx{\nx\phi,\k\be}+\brx{\nx\phi,\frac{\nx T}{2T^2}\big(\kappa\P+\sigma c\big)}\\
    =&\bbrb{\nx\phi\cdot\a,h}{\gamma_-}-\bbrb{\nx\phi\cdot\a,(1-\pp)[\re]}{\gamma_+}+\br{v\cdot\nx\Big(\nx\phi\cdot\a\Big),(\ik-\bpk)[\re]}\no\\
    &-\br{\nx\phi\cdot\a,\left(\mhh\ab\cdot\frac{\nx T}{2T^2}\right)(\ik-\bpk)[\re]}+\bbr{\nx\phi\cdot\a,\ss}.\no
\end{align}
Using H\"{o}lder's inequality, we have
\begin{align}
    &\abs{\br{v\cdot\nx\Big(\nx\phi\cdot\a\Big),(\ik-\bpk)[\re]}}+\abs{\br{\nx\phi\cdot\a,\left(\mhh\ab\cdot\frac{\nx T}{2T^2}\right)(\ik-\bpk)[\re]}}\\
    \ls&\pnm{(\ik-\bpk)[\re]}{\N}\nm{\phi }_{W^{2,\frac{\N}{\N-1}}}\ls\pnm{(\ik-\bpk)[\re]}{\N}\pnm{\z}{\N}^{\N-1},\no
\end{align}
and
\begin{align}
    \abs{\brx{\nx\phi,\frac{\nx T}{2T^2}\big(\kappa\P+\sigma c\big)}}\ls&\oot \pnm{\k\P+\sigma c}{\N}\pnm{\phi }{\frac{\N}{\N-1}}\ls\oot\Big(\pnm{\P}{\N}+\pnm{c}{\N}\Big)\pnm{\z}{\N}^{\N-1}.
\end{align}
Due to \eqref{final 63}, we have
\begin{align}
    \e^{-1}\bbrx{\nx\phi,\k\be}=0.
\end{align}
By Sobolev embedding and trace theorem in 3D, we have
\begin{align}
\abs{\bbrb{h,\a\cdot\nx\phi}{\gamma_-}}&\ls\jnms{h}{\gamma_-}\knms{\nx\phi }{\gamma_-}\ls\jnms{h}{\gamma_-}\pnm{\z}{\N}^{\N-1},\\
\abs{\bbrb{(1-\pp)[\re],\a\cdot\nx\phi}{\gamma_+}}&\ls\jnms{\m^{\frac{1}{4}}(1-\pp)[\re]}{\gamma_+}\knms{\nx\phi }{\gamma_+}\ls\jnms{\m^{\frac{1}{4}}(1-\pp)[\re]}{\gamma_+}\pnm{\z}{\N}^{\N-1}.
\end{align}
Finally, using Section \ref{sec:bs}, we have the source term estimates
\begin{align}
    \abs{\bbr{\nx\phi\cdot\a,\ss-\ss_2}}\ls&\tnm{\nx\phi}\Big(\oot\um{\re}+\oot\Big)\ls \pnm{\z}{\N}^{\N-1}\Big(\oot\um{\re}+\oot\Big),\\
    \abs{\bbr{\nx\phi\cdot\a,\ss_2}}\ls&\tnm{\nx\phi}\xnm{\re}^2\ls \pnm{\z}{\N}^{\N-1}\xnm{\re}^2.
\end{align}
In summary, we have shown 
\begin{align}
\pnm{\z}{\N}^{\N}
\ls& \oot\Big(\pnm{\P}{\N}+\pnm{c}{\N}\Big)\pnm{\z}{\N}^{\N-1}+\pnm{(\ik-\bpk)[\re]}{\N}\pnm{\z}{\N}^{\N-1}\\
&+\jnms{h}{\gamma_-}\pnm{\z}{\N}^{\N-1}+\jnms{\m^{\frac{1}{4}}(1-\pp)[\re]}{\gamma_+}\pnm{\z}{\N}^{\N-1}+\pnm{\z}{\N}^{\N-1}\Big(\oot\xnm{\re}+\xnm{\re}^2+\oot\Big).\no
\end{align}
Then the desired estimate \eqref{oo 33} follows from Proposition \ref{thm:energy}.
\end{proof}

%%%%%%%%%%%%%%%%%%%%%%%%%%%%%%%%%%%%%%%%%%%%%%%%%%%%%%%%%%%%%%%%%%%%%%%%
\subsubsection{Decomposition of $\z$}\label{sec:z-split}
%%%%%%%%%%%%%%%%%%%%%%%%%%%%%%%%%%%%%%%%%%%%%%%%%%%%%%%%%%%%%%%%%%%%%%%%

We will introduce a decomposition $\z:=\z^R+\z^S$ such that that $\nm{\z^R}_{H^1_0}\ls\xnm{\re}$ and $\tnm{\z^S}\ls\e^{\frac{1}{2}}\xnm{\re}$. In other words, $\z^R$ is smoother and thus we can apply techniques like integration by parts or rewriting it with the fundamental theorem of calculus, while $\z^S$ is rougher but relatively small.

\begin{proof}[Proof of Proposition \ref{prop:z-splitting}]
The decomposition is rooted from the weak formulation. Equation \eqref{pp 03.} and \eqref{final 63} yield that for any test function $\phi(x)\in H^2_0(\Omega)$
\begin{align}\label{final 52}
\bbrx{\nx\cdot\big(\kappa\nx\phi\big),\z}
=&\brx{\nx\phi,\frac{\nx T}{2T^2}\big(\kappa\P+\sigma c\big)}-\bbrb{\nx\phi\cdot\a,h}{\gamma_-}+\bbrb{\nx\phi\cdot\a,(1-\pp)[\re]}{\gamma_+}\\
&-\br{v\cdot\nx\Big(\nx\phi\cdot\a\Big),(\ik-\bpk)[\re]}+\br{\nx\phi\cdot\a,\left(\mhh\ab\cdot\frac{\nx T}{2T^2}\right)(\ik-\bpk)[\re]}-\bbr{\nx\phi\cdot\a,\ss}.\no
\end{align}
In the following, we will split $\z$ and \eqref{final 52}.
\\
\paragraph{\underline{Construction of $\z^R$}}
By Lax-Milgram theorem, there exists a unique $\z^R\in H^1_0(\Omega)$ satisfying for any test function $\psi(x)\in H_0^1(\Omega)$,
\begin{align}\label{pp 125}
    \bbr{\k\nx\psi,\nx \z^R}
    =&-\br{\nx\psi,\frac{\nx T}{2\tq^2}(\k\P+\sigma c)}-\br{\a\cdot\nx\psi,\left(\mhh\ab\cdot\frac{\nx T}{2T^2}\right)(\ik-\bpk)[\re]}+\bbr{\a\cdot\nx\psi,\ss}.
\end{align}
Taking $\psi=\z^R$, using Poincar\'e's inequality, we obtain
\begin{align}
\nm{\z^R}_{H^1_0}\ls& \tnm{\frac{\nx T}{2\tq^2}(\k\P+\sigma c)}+\tnm{\br{\a,\mhh\ab\cdot\frac{\nx T}{4T^2}(\ik-\bpk)[\re]}}+\tnm{\brv{\a,\ss}}\\
\ls&\oot\tnm{\P}+\oot\tnm{c}+\oot\tnm{(\ik-\bpk)[\re]}+\oot\xnm{\re}+\xnm{\re}^2+\oot.\no
\end{align}
Then the desired estimate \eqref{qq 06} follows from Proposition \ref{thm:energy}.
\\
\paragraph{\underline{Construction of $\z^S$}} Denote $\z^S=\z-\z^R$. Let $\psi(x)=\phi(x)\in H^2_0(\Omega)$ in \eqref{pp 125}. Integration by parts yields 
\begin{align}\label{qq 110.}
    \bbrx{\nx\cdot\big(\kappa\nx\phi\big),\z^R}
    =&\br{\nx\phi,\frac{\nx T}{2\tq^2}(\k\P+\sigma c)}+\br{\a\cdot\nx\phi,\left(\mhh\ab\cdot\frac{\nx T}{2T^2}\right)(\ik-\bpk)[\re]}-\bbr{\a\cdot\nx\phi,\ss},
\end{align}
and thus $\z^S\in L^2(\Omega)$ satisfies
\begin{align}\label{qq 110..}
    \bbrx{\nx\cdot\big(\kappa\nx\phi\big),\z^S}
    =&-\bbrb{\nx\phi\cdot\a,h}{\gamma_-}+\bbrb{\nx\phi\cdot\a,(1-\pp)[\re]}{\gamma_+}-\br{v\cdot\nx\Big(\nx\phi\cdot\a\Big),(\ik-\bpk)[\re]}.
\end{align}
We choose the test function $\phi(x)$ satisfying
\begin{align}
\nx\cdot(\k\nx\phi) =\z^S,\quad \phi\Big|_{\p\Omega}=0.
\end{align}
The standard elliptic estimate implies
\begin{align}
    \nm{\phi}_{H^2}\ls\tnm{\z^S}.
\end{align}
It is obvious that
\begin{align}
    \bbrx{\nx\cdot\big(\kappa\nx\phi\big),\z^S}=\tnm{\z^S}^2.
\end{align}
On the other hand, we know
\begin{align}
    \abs{\br{v\cdot\nx\Big(\nx\phi\cdot\a\Big),(\ik-\bpk)[\re]}}
    \ls&\tnm{(\ik-\bpk)[\re]}\nm{\phi}_{H^2}\ls \tnm{(\ik-\bpk)[\re]}^2+\oo\tnm{\z^S}^2.
\end{align}
The remaining boundary terms can be bounded as
\begin{align}
    \abs{\bbrb{\nx\phi\cdot\a,h}{\gamma_-}}
    \ls&\tnms{\nx\phi}{\p\Omega}\tnms{\bbrv{(v\cdot n)\a, h\id_{\gamma_-}}}{\p\Omega}
    \ls\oo\tnm{\z^S}^2+\tnms{h}{\gamma_-}^2,
\end{align}
and
\begin{align}
    \abs{\bbrb{\nx\phi\cdot\a,(1-\pp)[\re]}{\gamma_+}}
    \ls&\tnms{\nx\phi}{\p\Omega}\tnms{\bbrv{(v\cdot n)\a,(1-\pp)[\re]\id_{\gamma_+}}}{\p\Omega}
    \ls\oo\tnm{\z^S}^2+\tnms{(1-\pp)[\re]}{\gamma_+}^2.
\end{align}
In total, using Lemma \ref{h-estimate}, we have
\begin{align}
    \tnm{\z^S}\ls \um{(\ik-\bpk)[\re]}+\tnms{(1-\pp)[\re]}{\gamma_+}+\oot\e.
\end{align}
Then the desired estimate \eqref{qq 06} follows from Proposition \ref{thm:energy}.
\end{proof}

%%%%%%%%%%%%%%%%%%%%%%%%%%%%%%%%%%%%%%%%%%%%%%%%%%%%%%%%%%%%%%%%%%%%%%%%
\section{Dual Stokes Problem and \texorpdfstring{$\bb$}{} Estimates}\label{sec:b-est}
%%%%%%%%%%%%%%%%%%%%%%%%%%%%%%%%%%%%%%%%%%%%%%%%%%%%%%%%%%%%%%%%%%%%%%%%

In this section, we plan to prove Proposition \ref{prop:b-bound}. To establish \eqref{oo 34=}, thanks to energy estimates and \eqref{ss 00}, it suffices to establish for $2\leq\N\leq6$
    \begin{align}\label{oo 34}
    \jnm{\bb}^{\N}\ls& \xnm{\re}\jnm{\bb}^{\N}+ \Big(\xnm{\re}+1\Big)\Big(\jnm{\be}^{\N}+\jnm{(\ik-\bpk)[\re]}^{\N}+\jnms{\m^{\frac{1}{4}}(1-\pp)[\re]}{\gamma_+}^{\N}\Big)\\
    &+\oot\e^{\frac{\N}{2}}\xnm{\re}^{\N}+\oot\e^{\frac{\N}{2}}\xnm{\re}^{2\N}+\oot\e.\no
    \end{align}

%%%%%%%%%%%%%%%%%%%%%%%%%%%%%%%%%%%%%%%%%%%%%%%%%%%%%%%%%%%%%%%%%%%%%%%%%%%%
\subsection{Combined Weak Formulation}
%%%%%%%%%%%%%%%%%%%%%%%%%%%%%%%%%%%%%%%%%%%%%%%%%%%%%%%%%%%%%%%%%%%%%%%%%%%%

Assume $(\psi,q)\in\r^3\times\r$ (where $q$ has zero average) is the unique strong solution to the perturbed Stokes problem 
\begin{align}\label{pp 17}
\left\{
    \begin{array}{r}
    -\lambda\Delta_x\psi-\bbb\cdot\big(\nx\psi:\nx\b\big)-\big(\nx\cdot\bbb\big)\big(\nx\psi:\b\big)+\bbrv{\nx\psi:\b,v\left(\ab\cdot\dfrac{\nx T}{2T^2}\right)}+\nx q=\bb\abs{\bb}^{\N-2}\ \ \text{in}\ \ \Omega,\\\rule{0ex}{1.2em}
    \nx\cdot\psi=0\ \ \text{in}\ \ \Omega,\qquad\psi=0\ \ \text{on}\ \ \p\Omega.
    \end{array}
\right.
\end{align}
Considering that the perturbation terms $\big(\nx\cdot\bbb\big)\big(\nx\psi:\b\big)$ and $\bbrv{\nx\psi:\b,v\left(\ab\cdot\frac{\nx T}{2T^2}\right)}$ are small due to the smallness of $\nx T$, we have the standard estimate \cite{Cattabriga1961}, Sobolev embedding for $2\leq\N\leq 6$ and trace estimates
\begin{align}\label{qq 15}
    \nm{\psi}_{W^{2,\frac{\N}{\N-1}}}+\nm{q}_{W^{1,\frac{\N}{\N-1}}}+\nm{\psi}_{H^1}+\tnm{q}+\abs{\psi}_{W^{1,\frac{2\N}{2\N-3}}_{\p\Omega}}+\knms{q}{\p\Omega}\ls\knm{\bb\abs{\bb}^{\N-2}}\ls \jnm{\bb}^{\N-1}.
\end{align}
Multiplying $\bb$ on both sides of \eqref{pp 17} and integrating by parts for $\brx{\nx q,\bb}$, we have
\begin{align} 
    -\bbrx{\lambda\Delta_x\psi,\bb}-\bbr{\bbb\cdot\big(\nx\psi:\nx\b\big)+\big(\nx\cdot\bbb\big)\big(\nx\psi:\b\big),\bb}+\br{\nx\psi:\b,(\bb\cdot v)\left(\ab\cdot\frac{\nx T}{2T^2}\right)}\\
    -\bbrx{ q,\nx\cdot\bb}+\int_{\p\Omega}q(\bb\cdot n)&=\jnm{\bb}^{\N},\no
\end{align}
which, by combining \eqref{conservation law 1} and Remark \ref{remark 02}, implies
\begin{align}\label{pp 18}
    -\bbrx{\lambda\Delta_x\psi,\bb}-\bbr{\bbb\cdot\big(\nx\psi:\nx\b\big)+(\nx\cdot\bbb)(\nx\psi:\b),\bb}+\br{\nx\psi:\b,(\bb\cdot v)\left(\ab\cdot\frac{\nx T}{2T^2}\right)}\\
    -P^{-1}\br{q,\mh\ss_3+\mh\ss_4}&=\jnm{\bb}^{\N}.\no
\end{align}
Inserting \eqref{pp 18} into \eqref{pp 16} to replace $-\bbrx{\lambda\Delta_x\psi,\bb}-\bbr{\bbb\cdot\big(\nx\psi:\nx\b\big)+\big(\nx\cdot\bbb\big)\big(\nx\psi:\b\big),\bb}+\bbr{\nx\psi:\b,(\bb\cdot v)\left(\ab\cdot\frac{\nx T}{2T^2}\right)}$, we obtain
\begin{align}\label{pp 19}
&\jnm{\bb}^{\N}+P^{-1}\br{q,\mh\ss_3+\mh\ss_4}+\bbr{\nx\psi:\b,\llc[\re]}\\
=&\bbrb{\nx\psi:\b,h}{\gamma_-}-\bbrb{\nx\psi:\b,(1-\pp)[\re]}{\gamma_+}\no\\
&+\br{v\cdot\nx\Big(\nx\psi:\b\Big),\bd\cdot\a+(\ik-\bpk)[\re]}-\br{\nx\psi:\b,\left(\mhh\ab\cdot\frac{\nx T}{4T^2}\right)\Big(\bd\cdot\a+(\ik-\bpk)[\re]\Big)}\no\\
&+\e^{-1}\br{\psi\cdot v\mh,\ss_3+\sp}+\bbr{\nx\psi:\b,\bar\ss}.\no
\end{align} 

\begin{lemma}
Under the assumption \eqref{assumption:boundary}, we have
\begin{align}\label{final 12}
  \jnm{\bb}^{\N}\ls& \jnm{\be}^{\N}+\jnm{(\ik-\bpk)[\re]}^{\N}+\jnms{\m^{\frac{1}{4}}(1-\pp)[\re]}{\gamma_+}^{\N}+\oot\e\\
  &-P^{-1}\br{q,\mh\ss_3+\mh\ss_4}-\bbr{\nx\psi:\b,\llc[\re]}+\e^{-1}\br{\psi\cdot v\mh,\ss_3+\sp}+\bbr{\nx\psi:\b,\bar\ss}.\no
\end{align}
\end{lemma}

\begin{proof}
We need to estimate each term in \eqref{pp 19} except the source terms. Using Lemma \ref{h-estimate}, we have
\begin{align}
&\abs{\bbrb{\nx\psi:\b,h}{\gamma_-}}+\abs{\bbrb{\nx\psi:\b,(1-\pp)[\re]}{\gamma_+}}\\
\ls& \oo\knms{\nx\psi}{\p\Omega}^{\frac{\N}{\N-1}}+\jnms{\int_{v\cdot n>0}(1-\pp)[\re]\b(v\cdot n)}{\p\Omega}^{\N}+\jnms{\int_{v\cdot n<0}h\b(v\cdot n)}{\p\Omega}^{\N}\no\\
\ls& \oo\jnm{\bb}^{\N}+\jnms{\m^{\frac{1}{4}}(1-\pp)[\re]}{\gamma_+}^{\N}+\jnms{h}{\gamma_-}^{\N}\ls \oo\jnm{\bb}^{\N}+\jnms{\m^{\frac{1}{4}}(1-\pp)[\re]}{\gamma_+}^{\N}+\oot\e^{3}.\no
\end{align}
In addition, \eqref{qq 15}, H\"older's inequality and Young's inequality yield
\begin{align}
\abs{\bbr{v\cdot\nx(\nx\psi:\b),\bd\cdot\a+(\ik-\bpk)[\re]}}
\ls&\oo\knm{\nx^2\psi}^{\frac{\N}{\N-1}}+\jnm{\bd}^{\N}+\jnm{(\ik-\bpk)[\re]}^{\N}\\
\ls&\oo\jnm{\bb}^{\N}+\jnm{\bd}^{\N}+\jnm{(\ik-\bpk)[\re]}^{\N},\no
\end{align}
and
\begin{align}
    &\abs{\br{\nx\psi:\b,\left(\mhh\ab\cdot\frac{\nx T}{4T^2}\right)\Big(\bd\cdot\a+(\ik-\bpk)[\re]\Big)}}\\
    \ls&\oot\knm{\nx\psi}^{\frac{\N}{\N-1}}+\oot\jnm{\bd}^{\N}+\oot\jnm{(\ik-\bpk)[\re]}^{\N}
    \ls\oot\jnm{\bb}^{\N}+\oot\jnm{\bd}^{\N}+\oot\jnm{(\ik-\bpk)[\re]}^{\N}.\no
\end{align}
Collecting all above and using Lemma \ref{remark 01}, we have the desired result.
\end{proof}

%%%%%%%%%%%%%%%%%%%%%%%%%%%%%%%%%%%%%%%%%%%%%%%%%%%%%%%%%%%%%%%%%%%%%%%%%%%%
\subsection{Estimates of Source Terms}
%%%%%%%%%%%%%%%%%%%%%%%%%%%%%%%%%%%%%%%%%%%%%%%%%%%%%%%%%%%%%%%%%%%%%%%%%%%%

In this section, we will estimate the source term in \eqref{final 12}:
\begin{align}\label{final 58}
  -P^{-1}\br{q,\mh\ss_3+\mh\ss_4}-\bbr{\nx\psi:\b,\llc[\re]}+\e^{-1}\br{\psi\cdot v\mh,\ss_3+\sp}+\bbr{\nx\psi:\b,\bar\ss}.
\end{align}

\begin{lemma}
Under the assumption \eqref{assumption:boundary}, we have
\begin{align}\label{final 11}
    &\abs{P^{-1}\br{q,\mh\ss_3+\mh\ss_4}}+\abs{\bbr{\nx\psi:\b,\llc[\re]}}+\abs{\e^{-1}\br{\psi\cdot v\mh,\ss_3+\sp}}+\abs{\bbr{\nx\psi:\b,\bar\ss}}\\
    \ls&\oot\jnm{\bb}^{\N}+\xnm{\re}\jnm{\bb}^{\N}+ \Big(\xnm{\re}+1\Big)\Big(\jnm{\be}^{\N}+\jnm{(\ik-\bpk)[\re]}^{\N}+\jnms{\m^{\frac{1}{4}}(1-\pp)[\re]}{\gamma_+}^{\N}\Big)\no\\
    &+\oot\e^{\frac{\N}{2}}\xnm{\re}^{\N}+\e^{\frac{\N}{2}}\xnm{\re}^{2\N}+\oot\e.\no
\end{align}
\end{lemma}

\begin{proof}\ 
\paragraph{\underline{$\llc[\re]$ Contribution}}
We need more detailed analysis than the estimates in Lemma \ref{ssl-estimate}. We split
\begin{align}
    \bbr{\nx\psi:\b,\llc[\re]}=\br{\nx\psi:\b,2\mhh Q^{\ast}\left[\mh f_1,\re\right]}=K_1+K_2,
\end{align}
where
\begin{align}
    K_1:=&\br{\nx\psi:\b,\mhh Q^{\ast}\left[\mh f_1,\P\m+(\bb\cdot v)\m+\bd\cdot\a\mh+\mh(\ik-\bpk)[\re]\right]},\\
    K_2:=&\br{\nx\psi:\b,\mhh Q^{\ast}\left[\mh f_1,c\left(\abs{v}^2-5T\right)\m\right]}.
\end{align}
We can check that
\begin{align}
    \mhh Q^{\ast}\left[\mh f_1,\P\m\right]=\P\lc[f_1]=-\P\left(\ab\cdot\frac{\nx\tq}{2\tq^2}\right),
\end{align}
and thus by oddness of $\ab$ and $\b$, we have
\begin{align}
    \br{\nx\psi:\b,\mhh Q^{\ast}\left[\mh f_1,\P\m\right]}=\brx{\bbrv{\nx\psi:\b,\ab},\P\frac{\nx\tq}{2\tq^2}}=0.
\end{align}
Also, considering that $\nm{f_1}_{\infty}\ls\oot$, H\"older's inequality and Young's inequality imply
\begin{align}
    &\abs{\br{\nx\psi:\b,\mhh Q^{\ast}\left[\mh f_1,(\bb\cdot v)\m+(\bd\cdot\a)\mh+\mh(\ik-\bpk)[\re]\right]}}\\
    \ls& \oot\nm{\psi}_{W^{2,\frac{\N}{\N-1}}}\Big(\jnm{\bb}+\jnm{\bd}+\jnm{(\ik-\bpk)[\re]}\Big)
    \ls\oot\jnm{\bb}^{\N}+\oot\jnm{\bd}^{\N}+\oot\um{(\ik-\bpk)[\re]}^{\N}.\no
\end{align}
Hence, we have shown that
\begin{align}
    K_1\ls \oot\jnm{\bb}^{\N}+\oot\jnm{\bd}^{\N}+\oot\jnm{(\ik-\bpk)[\re]}^{\N}.
\end{align}
In addition, due to oddness, we have
\begin{align}\label{pp 81}
    K_2=&\br{\nx\psi:\b,\mhh Q^{\ast}\bigg[-\frac{\a\cdot\nx\tq}{2\tq^2}\mh +\m\bigg(\frac{\rq_1}{\rq}+\frac{\uq_1\cdot v}{\tq}+\frac{\tq_1(\abs{v}^2-3\tq)}{2\tq^2}\bigg),c\left(\abs{v}^2-5T\right)\m\bigg]}\\
    =&\br{\nx\psi:\b,\mhh Q^{\ast}\bigg[\m\bigg(\frac{\rq_1}{\rq}+\frac{\tq_1(\abs{v}^2-3\tq)}{2\tq^2}\bigg),c\left(\abs{v}^2-5T\right)\m\bigg]}\no\\
    =&\br{\p_{x_1}\psi_1\b_{11}+\p_{x_2}\psi_2\b_{22}+\p_{x_3}\psi_3\b_{33},\mhh Q^{\ast}\bigg[\m\bigg(\frac{\rq_1}{\rq}+\frac{\tq_1(\abs{v}^2-3\tq)}{2\tq^2}\bigg),c\left(\abs{v}^2-5T\right)\m\bigg]}.\no
\end{align}
Since $\mhh Q^{\ast}\left[\m\left(\frac{\rq_1}{\rq}+\frac{\tq_1(\abs{v}^2-3\tq)}{2\tq^2}\right),c\big(\abs{v}^2-5T\big)\m\right]$ only depends on $\abs{v}^2$, for any $i=1,2,3$, using the definition of $\b_{ii}$, we have
\begin{align}\label{qq 16}
    &\brv{\b_{ii},\mhh Q^{\ast}\bigg[\m\bigg(\frac{\rq_1}{\rq}+\frac{\tq_1(\abs{v}^2-3\tq)}{2\tq^2}\bigg),c\big(\abs{v}^2-5T\big)\m\bigg]}\\
    =&\brv{\li\left[\left(v_i^2-\frac{1}{3}\abs{v}^2\right)\mh\right],\mhh Q^{\ast}\bigg[\m\bigg(\frac{\rq_1}{\rq}+\frac{\tq_1(\abs{v}^2-3\tq)}{2\tq^2}\bigg),c\big(\abs{v}^2-5T\big)\m\bigg]}\no\\
    =&\frac{1}{3}\brv{\li\left[\sum_{i=1,2,3}\left(v_i^2-\frac{1}{3}\abs{v}^2\right)\mh\right],\mhh Q^{\ast}\bigg[\m\bigg(\frac{\rq_1}{\rq}+\frac{\tq_1(\abs{v}^2-3\tq)}{2\tq^2}\bigg),c\big(\abs{v}^2-5T\big)\m\bigg]}=0,\no
\end{align}
since $\sum_{i=1,2,3}\left(v_i^2-\frac{1}{3}\abs{v}^2\right)=0$. 
Hence, we have $K_2=0$.

In total, using Lemma \ref{remark 01}, we have the $\llc$ contribution in \eqref{final 58}
\begin{align}
    \abs{\bbr{\nx\psi:\b,\llc[\re]}}\ls& \oot\jnm{\bb}^{\N}+\oot\jnm{\be}^{\N}+\oot\jnm{(\ik-\bpk)[\re]}^{\N}+\oot\e.
\end{align}

\paragraph{\underline{$\ss_3$ Contribution}}
We will apply a similar argument as \eqref{oo 30}.
Using $\psi\big|_{\p\Omega}=0$ and Hardy's inequality, combined with \eqref{ss3-estimate5} and \eqref{ss3-estimate6}, we have
\begin{align}\label{pp 124}
    &\abs{\br{\e^{-1}\psi\cdot v\mh,\ss_3}}\\
    \leq&\e^{-1}\abs{\br{\left(\int_0^{\mn}\p_{\mn}\psi(y)\ud y\right)\cdot v\mh,\sy+\sz}}+\e^{-1}\abs{\br{\left(\int_0^{\mn}\p_{\mn}\psi(y)\ud y\right)\cdot \nabla_v(v\mh),\fb_1}}\no\\
    =&\e^{-1}\abs{\br{\left(\frac{1}{\mn}\int_0^{\mn}\p_{\mn}\psi(y)\ud y\right)\cdot v\mh,\mn\big(\sy+\sz\big)}}+\e^{-1}\abs{\br{\left(\frac{1}{\mn}\int_0^{\mn}\p_{\mn}\psi(y)\ud y\right)\cdot \nabla_v(v\mh),\mn\fb_1}}\no\\
    \leq&\knm{\frac{1}{\mn}\int_0^{\mn}\p_{\mn}\psi(y)\ud y}\nm{\eta\big(\sy+\sz\big)}_{L^r_xL^1_v}+\knm{\frac{1}{\mn}\int_0^{\mn}\p_{\mn}\psi(y)\ud y}\nm{\eta\fb_1}_{L^r_xL^1_v}\no\\
    \leq&\knm{\p_{\mn}\psi}\nm{\eta\big(\sy+\sz\big)}_{L^r_xL^1_v}+\knm{\p_{\mn}\psi}\nm{\eta\fb_1}_{L^r_xL^1_v}
    \ls \oo \jnm{\bb}^{\N}+\nm{\eta\big(\sy+\sz\big)}_{L^r_xL^1_v}^{\N}+\nm{\eta\fb_1}_{L^r_xL^1_v}^{\N}.\no
\end{align}
Similarly, we have
\begin{align}
    \abs{\bbr{\nx\psi:\b,\ss_3}}\ls&\knm{\nx\psi}\jnm{\sy+\sz}+\knm{\nx\psi}\jnm{\fb_1}\\
    \ls& \oo\jnm{\bb}^{\N}+\nm{\sy+\sz}_{L^r_xL^1_v}^{\N}+\nm{\fb_1}_{L^r_xL^1_v}^{\N},\no
\end{align}
\begin{align}
    P^{-1}\abs{\bbr{q,\mh\ss_3}}\ls \oo\jnm{\bb}^{\N}+\nm{\sy+\sz}_{L^r_xL^1_v}^{\N}+\nm{\fb_1}_{L^r_xL^1_v}^{\N}.
\end{align}
In total, using Lemma \ref{ss3-estimate}, the $\ss_3$ contribution in \eqref{final 58} is bounded by
\begin{align}
    \oo\jnm{\bb}^{\N}+\oot\e.
\end{align}

\paragraph{\underline{$\ss_1$ Contribution}}
Notice that
\begin{align}\label{pp 123}
    \abs{\bbr{\nx\psi:\b,\ss_1}}\ls&\abs{\br{\nx\psi:\b,\Gamma\left[\fb_1,\P\mh+\bb\cdot\vv\mh+\bd\cdot\a+(\ik-\bpk)[\re]\right]}}\\
    &+\abs{\br{\nx\psi:\b,\Gamma\left[\fb_1,c\left(\abs{v}^2-5T\right)\mh\right]}}.\no
\end{align}
Using \eqref{qq 15}, we have
\begin{align}
    \text{First Term in \eqref{pp 123}}\ls& \tnm{\nx\psi}\lnmm{\fb_1}\Big(\tnm{p}+\tnm{\bb}+\tnm{\bd}+\um{(\ik-\bpk)[\re]}\Big)\\
    \ls&\oot\jnm{\bb}^{\N}+\oot\tnm{p}^{\N}+\oot\tnm{\bd}^{\N}+\oot\um{(\ik-\bpk)[\re]}^{\N}.\no
\end{align}
On the other hand, we know
\begin{align}
    \text{Second Term in \eqref{pp 123}}\ls&\abs{\br{\nx\psi:\b,\Gamma\left[\fb_1,\big(\z^R-\e^{-1}\xi\big)\left(\abs{v}^2-5T\right)\mh\right]}}\\
    &+\abs{\br{\nx\psi:\b,\Gamma\left[\fb_1,\z^S\left(\abs{v}^2-5T\right)\mh\right]}}.\no
\end{align}
Similar to \eqref{pp 124}, using \eqref{qq 06}, we have
\begin{align}
    &\abs{\br{\nx\psi:\b,\Gamma\left[\fb_1,\big(\z^R-\e^{-1}\xi\big)\left(\abs{v}^2-5T\right)\mh\right]}}\\
    =&\abs{\br{\nx\psi:\b,\Gamma\left[\fb_1,\int_0^{\mn}\p_{\mn}\big(\z^R-\e^{-1}\xi\big)\left(\abs{v}^2-5T\right)\mh\right]}}\no\\
    \ls&\e\tnm{\nx\psi}\lnmm{\eta\fb_1}\tnm{\nx\z^R-\e^{-1}\nx\xi}\ls \oot\jnm{\bb}^{\N}+\oot\Big(\e^{\N}\nm{\z^R}_{H^1}^{\N}+\nm{\xi}_{H^1}^{\N}\Big)\no\\
    \ls&\oot\jnm{\bb}^{\N}+\oot\e^{\N}\xnm{\re}^{\N}+\oot\e^{\N}\xnm{\re}^{2\N}+\oot\e^{\frac{\N}{2}}.\no
\end{align}
Also, using \eqref{qq 07}, we have
\begin{align}
    \abs{\br{\nx\psi:\b,\Gamma\left[\fb_1,\z^S\left(\abs{v}^2-5T\right)\mh\right]}}
    \ls&\tnm{\nx\psi}\lnmm{\fb_1}\tnm{\z^S}\\
    \ls&\oot\jnm{\bb}^{\N}+\oot\e^{\frac{\N}{2}}\xnm{\re}^{\N}+\oot\e^{\frac{\N}{2}}\xnm{\re}^{2\N}+\oot\e^{\frac{\N}{2}}.\no
\end{align}
In total, for $2\leq\N\leq6$, the $\ss_1$ contribution in \eqref{final 58}
\begin{align}
    \abs{\bbr{\nx\psi:\b,\ss_1}}\ls&\oot\e^{\frac{\N}{2}}\xnm{\re}^{\N}+\oot\e^{\frac{\N}{2}}\xnm{\re}^{2\N}+\oot\e^{\frac{\N}{2}}.
\end{align}

\paragraph{\underline{$\ss_2$ Contribution}}
We will need more detailed analysis than the estimates in Lemma \ref{ss2-estimate}. Notice that
\begin{align}\label{pp 103''}
    \bbr{\nx\psi:\b,\ss_2}=&\br{\nx\psi:\b,\Gamma\Big[\pk[\re]+\bd\cdot\a, \pk[\re]+\bd\cdot\a\Big]}\\
    &+\br{\nx\psi:\b,\Gamma\Big[\pk[\re]+\bd\cdot\a, (\ik-\bpk)[\re]\Big]}+\br{\nx\psi:\b,\Gamma\Big[(\ik-\bpk)[\re], (\ik-\bpk)[\re]\Big]}.\no
\end{align}
For the first term in \eqref{pp 103''}, by oddness, we have
\begin{align}\label{pp 104''}
    &\br{\nx\psi:\b,\Gamma\Big[\pk[\re]+\bd\cdot\a,\pk[\re]+\bd\cdot\a\Big]}\\
    =&\br{\nx\psi:\b,\Gamma\Big[\P\mh+c\left(\abs{v}^2-5T\right)\mh,\P\mh+c\left(\abs{v}^2-5T\right)\mh\Big]}\no\\
    &+\br{\nx\psi:\b,\Gamma\Big[\bb\cdot v\mh,\bb\cdot v\mh\Big]}+\br{\nx\psi:\b,\Gamma\Big[\bd\cdot \a,\bb\cdot v\mh\Big]}+\br{\nx\psi:\b,\Gamma\Big[\bd\cdot \a,\bd\cdot \a\Big]}.\no
\end{align}
Using oddness, similar to \eqref{qq 16}, we know 
\begin{align}
    \br{\nx\psi:\b,\Gamma\Big[\P\mh+c\left(\abs{v}^2-5T\right)\mh,\P\mh+c\left(\abs{v}^2-5T\right)\mh\Big]}=0.
\end{align}
Also, due to $\frac{2\N-3}{3\N}+\frac{1}{3}+\frac{1}{\N}=1$, considering the Sobolev embedding $\pnm{\nx\psi}{\frac{3\N}{2\N-3}}\ls\nm{\psi}_{W^{2,\frac{\N}{\N-1}}}$, and interpolation $\pnm{\bb}{3}\leq\tnm{\bb}+\pnm{\bb}{6}\ls\xnm{\re}$, we know
\begin{align}
    \abs{\br{\nx\psi:\b,\Gamma\Big[\bb\cdot v\mh,\bb\cdot v\mh\Big]}}\ls& \nm{\nx\psi}_{L^{\frac{3\N}{2\N-3}}}\pnm{\bb}{3}\jnm{\bb}\\
    \ls&\nm{\psi}_{W^{2,\frac{\N}{\N-1}}}\xnm{\re} \jnm{\bb}\ls \xnm{\re}\jnm{\bb}^{\N}.\no
\end{align}
Similarly, using $\pnm{\bd}{3}\ls\xnm{\re}$ due to interpolation, using \eqref{qq 15}, we have
\begin{align}
    \abs{\br{\nx\psi:\b,\Gamma\Big[\bd\cdot \a,\bb\cdot v\mh\Big]}}\ls \pnm{\bd}{3}\jnm{\bb}^{\N}\ls \xnm{\re}\jnm{\bb}^{\N},
\end{align}
and
\begin{align}
    \abs{\br{\nx\psi:\b,\Gamma\Big[\bd\cdot \a,\bd\cdot \a\Big]}}\ls&\jnm{\bb}^{\N-1}\pnm{\bd}{3}\pnm{\bd}{\N}\ls \jnm{\bb}^{\N-1}\xnm{\re}\pnm{\bd}{\N}
    \ls  \xnm{\re}\Big(\jnm{\bb}^{\N}+ \jnm{\bd}^{\N}\Big).
\end{align}
Hence, we know
\begin{align}
    \Babs{\text{First Term in \eqref{pp 103''}}}\ls& \xnm{\re}\Big(\jnm{\bb}^{\N}+ \jnm{\bd}^{\N}\Big).
\end{align}
On the other hand, noticing that $\lnmm{\pk[\re]+\bd\cdot\a}\ls\lnmm{\re}\ls \e^{-\frac{1}{2}}\xnm{\re}$, we obtain
\begin{align}
    \Babs{\text{Second Term in \eqref{pp 103''}}}\ls&\tnm{\nx\psi}\lnmm{\re}\um{(\ik-\bpk)[\re]}\\
    \ls&\e^{-\frac{1}{2}}\tnm{\nx\psi}\xnm{\re}\um{(\ik-\bpk)[\re]}
    \ls\e^{-\frac{1}{2}}\pnm{\bb}{\N}^{\N-1}\xnm{\re}\um{(\ik-\bpk)[\re]}\no\\
    \ls&\xnm{\re}\jnm{\bb}^{\N}+\e^{-\frac{\N}{2}}\xnm{\re}\um{(\ik-\bpk)[\re]}^{\N}.\no
\end{align}
Similarly,  noticing that $\lnmm{(\ik-\bpk)[\re]}\ls\lnmm{\re}\ls \e^{-\frac{1}{2}}\xnm{\re}$, we have
\begin{align}
    \Babs{\text{Third Term in \eqref{pp 103''}}}\ls \xnm{\re}\jnm{\bb}^{\N}+\e^{-\frac{\N}{2}}\xnm{\re}\um{(\ik-\bpk)[\re]}^{\N}.
\end{align}
In total, using Lemma \ref{remark 01}, the $\ss_2$ contribution in \eqref{final 58} 
\begin{align}
    \abs{\bbr{\nx\psi:\b,\ss_2}}&\ls  \xnm{\re}\Big(\jnm{\bb}^{\N}+ \jnm{\bd}^{\N}+\e^{-\frac{\N}{2}}\um{(\ik-\bpk)[\re]}^{\N}\Big)\\
    &\ls \xnm{\re}\Big(\jnm{\bb}^{\N}+ \jnm{\be}^{\N}+\e^{-\frac{\N}{2}}\um{(\ik-\bpk)[\re]}^{\N}+\oot\e\Big).\no
\end{align}

\paragraph{\underline{Other Source Contribution}}
Using \eqref{qq 15}, Lemma \ref{ss4-estimate} and Lemma \ref{ssp-estimate}, we have
\begin{align}
    \abs{\e^{-1}\br{\psi\cdot v\mh,\sp}}+\abs{\bbr{\nx\psi:\b,\sp}}\ls&\e^{-1}\nm{\psi}_{W^{2,\frac{\N}{\N-1}}}\tnm{\sp}
    \ls \oot\jnm{\bb}^{\N}+\oot\e^{\N},
\end{align}
and
\begin{align}
    P^{-1}\abs{\bbr{q,\mh\ss_4}}+\abs{\bbr{\nx\psi:\b,\ss_4}}&\ls \Big(\nm{q}_{L^{\frac{\N}{\N-1}}}+\nm{\psi}_{W^{2,\frac{\N}{\N-1}}}\Big)\tnm{\ss_4}
    \ls  \oot\jnm{\bb}^{\N}+\oot\e^{\N}.
\end{align}
Also, based on Lemma \ref{ss0-estimate} and Lemma \ref{ss5-estimate}, we know that
\begin{align}
    \abs{\bbr{\nx\psi:\b,\ss_0+\ss_5}}\ls&\nm{\psi}_{W^{2,\frac{\N}{\N-1}}}\Big(\tnm{\ss_0}+\tnm{\ss_5}\Big)
    \ls\oot\jnm{\bb}^{\N}+\oot\e^{\N}\xnm{\re}^{\N}+\oot\e^{\frac{\N}{2}}.
\end{align}
In total, the $\ss_0$, $\ss_4$, $\ss_5$ and $\sp$ contributions in \eqref{final 58} are bounded by
\begin{align}
    \oot\jnm{\bb}^{\N}+\oot\e^{\N}\xnm{\re}^{\N}+\oot\e^{\frac{\N}{2}}.
\end{align}
Collecting above estimates, we deduce \eqref{final 11}.
\end{proof}

\begin{proof}[Proof of Proposition \ref{prop:b-bound}]
Inserting \eqref{final 11} into \eqref{final 12},
we obtain the desired result \eqref{oo 34}.
\end{proof}

%%%%%%%%%%%%%%%%%%%%%%%%%%%%%%%%%%%%%%%%%%%%%%%%%%%%%%%%%%%%%%%%%%%%%%%%
\section{Dual Stokes-Poisson Problem, \texorpdfstring{$\e$}{}-Cutoff Boundary Layer, and \texorpdfstring{$c$}{} Estimates}\label{sec:c-est}
%%%%%%%%%%%%%%%%%%%%%%%%%%%%%%%%%%%%%%%%%%%%%%%%%%%%%%%%%%%%%%%%%%%%%%%%

In this section, we plan to prove Proposition \ref{prop:c-bound}. To establish \eqref{oo 35=}, thanks to energy estimates, it suffices to establish for $2\leq\N\leq6$
    \begin{align}\label{oo 35}
    \jnm{c}\ls&\jnm{(\ik-\bpk)[\re]}+\jnms{\m^{\frac{1}{4}}(1-\pp)[\re]}{\gamma_+}+\oot\xnm{\re}+\xnm{\re}^{2}+\xnm{\re}^{3}+\oot.
    \end{align}

%%%%%%%%%%%%%%%%%%%%%%%%%%%%%%%%%%%%%%%%%%%%%%%%%%%%%%%%%%%%%%%%%%%%%%%%%%%%
\subsection{Combined Weak Formulation}
%%%%%%%%%%%%%%%%%%%%%%%%%%%%%%%%%%%%%%%%%%%%%%%%%%%%%%%%%%%%%%%%%%%%%%%%%%%%

Adding $\e^{-1}\times$\eqref{pp 16} and \eqref{pp 10} yields
\begin{align}\label{pp 112}
    &5\e^{-1}P\bbrx{\phi,\nx T\cdot\bb}-\e^{-1}\bbrx{\lambda\Delta_x\psi,\bb}-\e^{-1}\bbr{\bbb\cdot\big(\nx\psi:\nx\b\big)+\big(\nx\cdot\bbb\big)\big(\nx\psi:\b\big),\bb}\\
    &+\e^{-1}\br{\nx\psi:\b,(\bb\cdot v)\left(\ab\cdot\frac{\nx T}{2T^2}\right)}-\bbrx{\nx\cdot(\kappa\nx\phi),c}+\brx{\nx\phi,\frac{\nx T}{2T^2}\big(\kappa\P+\sigma c\big)}\no\\
    &+\bbr{\e^{-1}\nx\psi:\b+\nx\phi\cdot\a,\llc[\re]}\no\\
    =&\bbrb{\nx\phi\cdot\a,h}{\gamma_-}-\bbrb{\nx\phi\cdot\a,(1-\pp)[\re]}{\gamma_+}-\e^{-1}\bbrb{\nx\psi:\b,h}{\gamma_-}-\e^{-1}\bbrb{\nx\psi:\b,(1-\pp)[\re]}{\gamma_+}\no\\
    &+\br{v\cdot\nx\Big(\nx\phi\cdot\a\Big),(\ik-\bpk)[\re]}+\e^{-1}\br{v\cdot\nx\Big(\nx\psi:\b\Big),\bd\cdot\a+(\ik-\bpk)[\re]}\no\\
    &-\br{\nx\phi\cdot\a,\left(\mhh\ab\cdot\frac{\nx T}{2T^2}\right)(\ik-\bpk)[\re]}-\e^{-1}\br{\nx\psi:\b,\left(\mhh\ab\cdot\frac{\nx T}{4T^2}\right)\Big(\bd\cdot\a+(\ik-\bpk)[\re]\Big)}\no\\
    &+\br{\e^{-2}\psi\cdot v\mh,\ss_3+\sp}+\br{\e^{-1}\phi\left(\abs{v}^2-5T\right)\mh,\ss_3+\ss_4+\sp}+\bbr{\nx\phi\cdot\a+\e^{-1}\nx\psi:\b,\bar\ss}.\no
\end{align}
Using oddness and $\nx\cdot\psi=0$, similarly to \eqref{qq 16}, we can further simplify the $\pk[\re]$ part in $\e^{-1}\bbr{\nx\psi:\b,\llc[\re]}$:
\begin{align}\label{pp 113}
    &\e^{-1}\br{\nx\psi:\b,\llc\Big[\pk[\re]\Big]}\\
    =&-\e^{-1}\br{\nx\psi:\b,2\mhh Q^{\ast}\bigg[-\mh\a\cdot\frac{\nx\tq}{2\tq^2}+\m\bigg(\frac{\rq_1}{\rq}+\frac{\uq_1\cdot v}{\tq}+\frac{\tq_1(\abs{v}^2-3\tq)}{2\tq^2}\bigg),\P\m+(\bb\cdot v)\m+c\left(\abs{v}^2-5T\right)\m\bigg]}\no\\
    =&-\e^{-1}\br{\nx\psi:\b,2\mhh Q^{\ast}\left[-\mh\a\cdot\frac{\nx\tq}{2\tq^2}+\m\frac{\uq\cdot v}{\tq},(\bb\cdot v)\m\right]}.\no
\end{align}
We extract all $\e^{-1}$-order terms involving $\bb$ in \eqref{pp 112}:
\begin{align}
    5P\bbrx{\phi,\nx T\cdot\bb}-\bbrx{\lambda\Delta_x\psi,\bb}-\bbr{\bbb\cdot\big(\nx\psi:\nx\b\big)+\big(\nx\cdot\bbb\big)\big(\nx\psi:\b\big),\bb}+\bbr{\nx\psi\cdot \mathfrak{M},\bb}
\end{align}
for 
\begin{align}\label{pp 114}
    \mathfrak{M}:=v\left(\ab\cdot\dfrac{\nx T}{2T^2}\right)-2\mhh Q^{\ast}\left[-\mh\a\cdot\frac{\nx\tq}{2\tq^2}+\m\frac{\uq_1\cdot v}{\tq},v\m\right].
\end{align}
which identifies the coefficients in front
of $\bb$ by combining the first term on the second line of \eqref{pp 112} and the third line of \eqref{pp 113}
\begin{align}
    \br{\nx\psi:\b,(\bb\cdot v)\left(\ab\cdot\frac{\nx T}{2T^2}\right)}-\br{\nx\psi:\b,2\mhh Q^{\ast}\left[-\mh\a\cdot\frac{\nx\tq}{2\tq^2}+\m\frac{\uq\cdot v}{\tq},(\bb\cdot v)\m\right]}.
\end{align}
Assume $(\psi,q,\phi)\in\r^3\times\r\times\r$ (where $q$ has zero average) is the unique strong solution to the perturbed Stokes-Poisson problem
\begin{align}\label{pp 20}
\left\{
    \begin{array}{r}
    -\lambda\Delta_x\psi-\bbb\cdot\big(\nx\psi:\nx\b\big)-\big(\nx\cdot\bbb\big)\big(\nx\psi:\b\big)+\bbrv{\nx\psi:\b,\mathfrak{M}}+\nx q=-5P\phi\nx T\ \ \text{in}\ \ \Omega,\\\rule{0ex}{1.0em}
    \nx\cdot\psi=0\ \ \text{in}\ \ \Omega,\qquad
    \psi=0\ \ \text{on}\ \ \p\Omega,
    \end{array}
    \right.
\end{align}
\begin{align}\label{pp 23}
    -\nx\cdot\big(\kappa\nx\phi\big)=c\abs{c}^{\N-2}\ \ \text{in}\ \ \Omega,\qquad
    \phi=0\ \ \text{on}\ \ \p\Omega.
\end{align}
Since the perturbation terms $\bbb\cdot\big(\nx\psi:\nx\b\big)+\big(\nx\cdot\bbb\big)\big(\nx\psi:\b\big)$ and $\bbrv{\nx\psi:\b,\mathfrak{M}}$ are small considering the smallness of $\nx T$ and $u_1$, we have the standard estimate \cite{Cattabriga1961}, Sobolev embedding for $2\leq\N\leq 6$ and trace estimate
\begin{align}\label{qq 17}
    &\nm{\psi}_{W^{4,\frac{\N}{\N-1}}}+\nm{q}_{W^{3,\frac{\N}{\N-1}}}+\nm{\psi}_{H^3}+\nm{q}_{H^2}+\abs{\psi}_{W^{3,\frac{2\N}{2\N-3}}_{\p\Omega}}+\abs{q}_{W^{2,\frac{2\N}{2\N-3}}_{\p\Omega}}+\abs{\psi}_{H^1_{\p\Omega}}\ls \oot \nm{\phi}_{W^{2,\frac{\N}{\N-1}}},\\\label{qq 18}
    &\nm{\phi}_{W^{2,\frac{\N}{\N-1}}}+\nm{\phi}_{H^1}+\knms{\phi}{\p\Omega}\ls\knm{c\abs{c}^{\N-2}}\ls\jnm{c}^{\N-1}.
\end{align}
Multiplying $\bb$ on both sides of \eqref{pp 20} and integrating by parts in $\br{\nx q,\bb}$, we have
\begin{align}
    -\bbrx{\lambda\Delta_x\psi,\bb}-\bbr{\bbb\cdot\big(\nx\psi:\nx\b\big)+\big(\nx\cdot\bbb\big)\big(\nx\psi:\b\big),\bb}+\bbr{\big(\nx\psi:\b\big)\mathfrak{M},\bb}\\
    -\bbrx{q,\nx\cdot\bb}+\int_{\p\Omega}q(\bb\cdot n)&=-5P\bbrx{\phi,\nx T\cdot\bb},\no
\end{align}
which, by combining \eqref{conservation law 1} and \eqref{final 53}, implies
\begin{align}\label{pp 21}
    -\bbrx{\lambda\Delta_x\psi,\bb}-\bbr{\bbb\cdot\big(\nx\psi:\nx\b\big)+\big(\nx\cdot\bbb\big)\big(\nx\psi:\b\big),\bb}+\bbr{\big(\nx\psi:\b\big) \mathfrak{M},\bb}\\
    -P^{-1}\br{q,\mh\ss_3+\mh\ss_4}&=-5P\bbrx{\phi,\nx T\cdot\bb}.\no
\end{align}
Multiplying $c$ on both sides of \eqref{pp 23}, we have
\begin{align}\label{pp 24}
    -\bbrx{\nx\cdot\big(\kappa\nx\phi\big),c}=\jnm{c}^{\N}.
\end{align}
Inserting \eqref{pp 21}\eqref{pp 24} into \eqref{pp 112} and using \eqref{pp 114} to replace $5P\bbrx{\phi,\nx T\cdot\bb}-\bbrx{\lambda\Delta_x\psi,\bb}-\bbr{\bbb\cdot\big(\nx\psi:\nx\b\big)+\big(\nx\cdot\bbb\big)\big(\nx\psi:\b\big),\bb}+\bbr{\big(\nx\psi:\b\big)\mathfrak{M},\bb}$ and  $-\bbrx{\nx\cdot(\kappa\nx\phi),c}$, we obtain
\begin{align}\label{pp 25}
    &\e^{-1}P^{-1}\br{q,\mh\ss_3+\mh\ss_4}\\
    &+\jnm{c}^{\N}+\brx{\nx\phi,\frac{\nx T}{2T^2}\big(\kappa\P+\sigma c\big)}+\e^{-1}\br{\nx\psi:\b,\llc\Big[(\ik-\pk)[\re]\Big]}+\bbr{\nx\phi\cdot\a,\llc[\re]}\no\\
    =&\bbrb{\nx\phi\cdot\a,h}{\gamma_-}-\bbrb{\nx\phi\cdot\a,(1-\pp)[\re]}{\gamma_+}-\e^{-1}\bbrb{\nx\psi:\b,h}{\gamma_-}-\e^{-1}\bbrb{\nx\psi:\b,(1-\pp)[\re]}{\gamma_+}\no\\
    &+\br{v\cdot\nx\Big(\nx\phi\cdot\a\Big),(\ik-\bpk)[\re]}+\e^{-1}\br{v\cdot\nx\Big(\nx\psi:\b\Big),\bd\cdot\a+(\ik-\bpk)[\re]}\no\\
    &-\br{\nx\phi\cdot\a,\left(\mhh\ab\cdot\frac{\nx T}{2T^2}\right)(\ik-\bpk)[\re]}-\e^{-1}\br{\nx\psi:\b,\left(\mhh\ab\cdot\frac{\nx T}{4T^2}\right)\Big(\bd\cdot\a+(\ik-\bpk)[\re]\Big)}\no\\
    &+\br{\e^{-2}\psi\cdot v\mh,\ss_3+\sp}+\br{\e^{-1}\phi\left(\abs{v}^2-5T\right)\mh,\ss_3+\ss_4+\sp}+\bbr{\nx\phi\cdot\a+\e^{-1}\nx\psi:\b,\bar\ss}.\no
\end{align}
Note that $\e^{-1}\br{\nx\psi:\b,\llc\Big[\pk[\re]\Big]}$ has been included in the definition of $\mathfrak{M}$ in \eqref{pp 114}, so we are left with $\e^{-1}\br{\nx\psi:\b,\llc\Big[(\ik-\pk)[\re]\Big]}$ in \eqref{pp 25}.

\begin{lemma}\label{lem:final 4}
Under the assumption \eqref{assumption:boundary}, we have
\begin{align}\label{pp 28'}
\jnm{c}^{\N}\leq&-\e^{-1}\bbrb{\nx\psi:\b,(1-\pp)[\re]}{\gamma_+}\\
&+\jnm{(\ik-\bpk)[\re]}^{\N}+\jnms{\m^{\frac{1}{4}}(1-\pp)[\re]}{\gamma_+}^{\N}+\oot\xnm{\re}^{\N}+\xnm{\re}^{2\N}+\oot\no\\
&-\e^{-1}P^{-1}\br{q,\mh\ss_3+\mh\ss_4}+\bbr{\nx\phi\cdot\a,\llc[\re]}-\e^{-1}\br{\nx\psi:\b,\llc\Big[(\ik-\pk)[\re]\Big]}\no\\
&+\br{\e^{-2}\psi\cdot v\mh,\ss_3+\sp}+\br{\e^{-1}\phi\left(\abs{v}^2-5T\right)\mh,\ss_3+\ss_4+\sp}+\bbr{\nx\phi\cdot\a+\e^{-1}\nx\psi:\b,\bar\ss}.\no
\end{align}
\end{lemma}

\begin{proof}
We need to estimate each term in \eqref{pp 25} except the source terms. 
Using \eqref{qq 18}, we have
\begin{align}
    \abs{\brx{\nx\phi,\frac{\nx T}{2T^2}\big(\kappa\P+\sigma c\big)}}\ls \oot\tnm{\nx\phi}\Big(\tnm{\P}+\tnm{c}\Big)\ls \oot\jnm{c}^{\N}+\oot\tnm{\P}^{\N}\ls\oot\xnm{\re}^{\N}.
\end{align}
Also, using \eqref{qq 18} and Lemma \ref{h-estimate}, we have
\begin{align}
&\abs{\bbrb{\nx\phi\cdot\a,h}{\gamma_-}}+\abs{\bbrb{\nx\phi\cdot\a,(1-\pp)[\re]}{\gamma_-}}\\
\ls& \oo\knms{\nx\phi}{\p\Omega}^{\frac{\N}{\N-1}}+\jnms{\int_{v\cdot n>0}(1-\pp)[\re]\a(v\cdot n)}{\p\Omega}^{\N}+\jnms{\int_{v\cdot n<0}h\a(v\cdot n)}{\p\Omega}^{\N}\no\\
\ls& \oo\jnm{c}^{\N}+\jnms{\m^{\frac{1}{4}}(1-\pp)[\re]}{\gamma_+}^{\N}+\jnms{h}{\gamma_-}^{\N}\ls \oo\jnm{c}^{\N}+\jnms{\m^{\frac{1}{4}}(1-\pp)[\re]}{\gamma_+}^{\N}+\oot\e^{\N}.\no
\end{align}
In addition, \eqref{qq 18}, H\"older's inequality and Young's inequality yield
\begin{align}
\abs{\br{v\cdot\nx\Big(\nx\phi\cdot\a\Big),(\ik-\bpk)[\re]}}
\ls&\oo\knm{\nx^2\phi}^{\frac{\N}{\N-1}}+\jnm{(\ik-\bpk)[\re]}^{\N}
\ls\oo\jnm{c}^{\N}+\jnm{(\ik-\bpk)[\re]}^{\N},
\end{align}
and
\begin{align}
    \abs{\br{\big(\nx\phi\cdot\a\big)\left(\mhh\ab\cdot\frac{\nx T}{2T^2}\right),(\ik-\bpk)[\re]}}
    \ls&\oot\knm{\nx\phi}^{\frac{\N}{\N-1}}+\oot\jnm{(\ik-\bpk)[\re]}^{\N}\\
    \ls&\oot\jnm{c}^{\N}+\oot\jnm{(\ik-\bpk)[\re]}^{\N}.\no
\end{align}
Further, \eqref{qq 17}\eqref{qq 18}, H\"older's inequality and Young's inequality (and further integration by parts for $\nx\xi$ in $\bd$ using $\xi\big|_{\p\Omega}=0$) yield
\begin{align}
&\abs{\e^{-1}\br{v\cdot\nx\Big(\nx\psi:\b\Big),\bd\cdot\a+(\ik-\bpk)[\re]}}\\
\leq&\abs{\e^{-1}\br{v\cdot\nx\Big(\nx\psi:\b\Big),\nx\xi\cdot\a}}+\abs{\e^{-1}\br{v\cdot\nx\Big(\nx\psi:\b\Big),\be\cdot\a+(\ik-\bpk)[\re]}}\no\\
\ls&\e^{-1}\nm{\psi}_{H^3}\Big(\tnm{\xi}+\tnm{\be}+\tnm{(\ik-\bpk)[\re]}\Big)
\ls\oot\jnm{c}^{\N}+\oot\e^{-\N}\tnm{\xi}^{\N}+\oot\e^{-\N}\tnm{\be}^{\N}+\oot\e^{-\N}\tnm{(\ik-\bpk)[\re]}^{\N},\no
\end{align}
and
\begin{align}
    &\abs{\e^{-1}\br{\nx\psi:\b,\left(\mhh\ab\cdot\frac{\nx T}{4T^2}\right)\Big(\bd\cdot\a+(\ik-\bpk)[\re]\Big)}}\\
    \leq&\abs{\e^{-1}\br{\nx\psi:\b,\left(\mhh\ab\cdot\frac{\nx T}{4T^2}\right)\Big(\nx\xi\cdot\a\Big)}}+\abs{\e^{-1}\br{\nx\psi:\b,\left(\mhh\ab\cdot\frac{\nx T}{4T^2}\right)\Big(\be\cdot\a+(\ik-\bpk)[\re]\Big)}}\no\\
    \ls&\oot\e^{-1}\nm{\psi}_{H^2}\Big(\tnm{\xi}+\tnm{\be}+\tnm{(\ik-\bpk)[\re]}\Big)
    \ls\oot\jnm{c}^{\N}+\oot\e^{-\N}\tnm{\xi}^{\N}+\oot\e^{-\N}\tnm{\be}^{\N}+\oot\e^{-\N}\tnm{(\ik-\bpk)[\re]}^{\N}.\no
\end{align}
Finally, from \eqref{qq 17} and Lemma \ref{h-estimate}, we know
\begin{align}
\abs{\e^{-1}\bbrb{\nx\psi:\b,h}{\gamma_-}}
\ls& \tnms{\nx\psi}{\p\Omega}^{\frac{\N}{\N-1}}+\e^{-\N}\tnms{h}{\gamma_-}^{\N}
\ls \oot\jnm{c}^{\N}+\oot.
\end{align}
In total, from \eqref{pp 25} and Proposition \ref{thm:energy} and Lemma \ref{lem:final 11}, we obtain the desired result.
\end{proof}

%%%%%%%%%%%%%%%%%%%%%%%%%%%%%%%%%%%%%%%%%%%%%%%%%%%%%%%%%%%%%%%%%%%%%%%%%%%%
\subsection{Estimates of Source Terms}
%%%%%%%%%%%%%%%%%%%%%%%%%%%%%%%%%%%%%%%%%%%%%%%%%%%%%%%%%%%%%%%%%%%%%%%%%%%%

In this subsection, we consider the source term contribution in \eqref{pp 28'}:
\begin{align}\label{pp 49}
    &-\e^{-1}P^{-1}\br{q,\mh\ss_3+\mh\ss_4}+\bbr{\nx\phi\cdot\a,\llc[\re]}+\e^{-1}\br{\nx\psi:\b,\llc\Big[(\ik-\pk)[\re]\Big]}\\
    &+\br{\e^{-2}\psi\cdot v\mh,\ss_3+\sp}+\br{\e^{-1}\phi\left(\abs{v}^2-5T\right)\mh,\ss_3+\ss_4+\sp}+\bbr{\nx\phi\cdot\a+\e^{-1}\nx\psi:\b,\sb}.\no
\end{align}

\begin{lemma}\label{lem:final 3}
Under the assumption \eqref{assumption:boundary}, we have
\begin{align}\label{pp 99''}
    &\abs{\e^{-1}P^{-1}\br{q,\mh\ss_3+\mh\ss_4}}+\abs{\bbr{\nx\phi\cdot\a,\llc[\re]}}+\abs{\e^{-1}\br{\nx\psi:\b,\llc\Big[(\ik-\pk)[\re]\Big]}}\\
    &+\abs{\br{\e^{-2}\psi\cdot v\mh,\ss_3+\sp}}+\abs{\br{\e^{-1}\phi\left(\abs{v}^2-5T\right)\mh,\ss_3+\ss_4+\sp}}+\abs{\bbr{\nx\phi\cdot\a+\e^{-1}\nx\psi:\b,\sb}}\no\\
    \ls
    &\oot\jnm{c}^{\N}+\oot\xnm{\re}^{\N}+\xnm{\re}^{2\N}+\oot.\no
\end{align}
\end{lemma}

\begin{proof}
\ 
\paragraph{\underline{$\llc[\re]$ Contribution}}
Lemma \ref{ssl-estimate} and \eqref{qq 17} imply
\begin{align}\label{pp 53}
    \abs{\bbr{\nx\phi\cdot\a,\llc[\re]}}\ls& \oot \nm{\phi}_{H^{1}}\um{\re}
    \ls \oot\jnm{c}^{\N}+\oot\xnm{\re}^{\N}.
\end{align}
Similarly, with the help of integration by parts for $\nx\xi$ with $\xi\big|_{\p\Omega}=0$, we have
\begin{align}
    &\abs{\e^{-1}\br{\nx\psi:\b,\llc\Big[(\ik-\pk)[\re]\Big]}}\\
    \ls&\abs{\e^{-1}\br{\nx\psi:\b,\llc\Big[\nx\xi\cdot\a\Big]}}+\abs{\e^{-1}\br{\nx\psi:\b,\llc\Big[\be\cdot\a+(\ik-\bpk)[\re]\Big]}}\no\\
    =&\abs{\e^{-1}\br{\nx^2\psi:\b,\llc\Big[\xi\cdot\a\Big]}}+\abs{\e^{-1}\br{\nx\psi:\b,\llc\Big[\be\cdot\a+(\ik-\bpk)[\re]\Big]}}\no\\
    \ls& \e^{-1}\nm{\psi}_{H^{2}}\Big(\tnm{\xi}+\tnm{\be}+\um{(\ik-\bpk)[\re]}\Big)
    \ls\oot\jnm{c}^{\N}+\oot\e^{-\N}\tnm{\xi}^{\N}+\oot\e^{-\N}\tnm{\be}^{\N}+\oot\e^{-\N}\um{(\ik-\bpk)[\re]}^{\N}.\no
\end{align}
In summary, the $\llc$ terms in \eqref{pp 49} can be bounded by
\begin{align}
    \oot\jnm{c}^{\N}+\oot\e^{-\N}\tnm{\be}^{\N}+\oot\e^{-\N}\um{(\ik-\bpk)[\re]}^{\N}+\oot\xnm{\re}^{\N}.
\end{align}

\paragraph{\underline{$\ss_3$ Contribution}}
We will apply a similar argument as \eqref{oo 30}.
Using $\psi\big|_{\p\Omega}=0$ and \eqref{qq 17}, with the help of the Taylor expansion with integral remainder and Hardy's inequality, we obtain
\begin{align}\label{pp 101}
    &\abs{\br{\e^{-2}\psi\cdot v\mh,\ss_3}}\\
    =&\e^{-2}\abs{\br{\left(\psi(0)+\mn\p_{\mn}\psi(0)+\int_0^{\mn}(\mn-y)\p_{\mn\mn}\psi(y)\ud y\right)\cdot v\mh,\ss_3}}\no\\
    =&\e^{-2}\abs{\br{\left(\mn\p_{\mn}\psi(0)+\int_0^{\mn}(\mn-y)\p_{\mn\mn}\psi(y)\ud y\right)\cdot v\mh,\ss_3}}\no\\
    =&\e^{-2}\abs{\br{\p_{\mn}\psi(0)\cdot v\mh,\mn\big(\sy+\sz+\fb_1\big)}}+\e^{-2}\abs{\br{\left(\frac{1}{\mn}\int_0^{\mn}\frac{\mn-y}{\mn}\p_{\mn\mn}\psi(y)\ud y\right)\cdot \nabla_v(v\mh),\mn^2\big(\sy+\sz+\fb_1\big)}}\no\\
    =&\e^{-1}\abs{\br{\p_{\mn}\psi(0)\cdot v\mh,\eta\big(\sy+\sz+\fb_1\big)}}+\abs{\br{\left(\frac{1}{\mn}\int_0^{\mn}\frac{\mn-y}{\mn}\p_{\mn\mn}\psi(y)\ud y\right)\cdot \nabla_v(v\mh),\eta^2\big(\sy+\sz+\fb_1\big)}}\no\\
    \ls&\e^{-1}\nm{\p_{\mn}\psi(0)\cdot v\mh}_{L^2_xL^{\infty}_v}\nm{\eta\big(\sy+\sz+\fb_1\big)}_{L^2_{\iota_i}L^1_{\mn}L^1_v}+\tnm{\frac{1}{\mn}\int_0^{\mn}\p_{\mn\mn}\psi(y)\ud y}\nm{\eta^2\big(\sy+\sz+\fb_1\big)}_{L^2_xL^1_v}\no\\
    \ls&\e^{-1}\nm{\p_{\mn}\psi}_{W^{3,\frac{\N}{\N-1}}}\nm{\eta\big(\sy+\sz+\fb_1\big)}_{L^2_{\iota_i}L^1_{\mn}L^1_v}+\nm{\p_{\mn\mn}\psi}_{W^{2,\frac{\N}{\N-1}}}\nm{\eta^2\big(\sy+\sz+\fb_1\big)}_{L^2_xL^1_v}\no\\
    \ls& \oot \jnm{c}^{\N}+\e^{-\N}\nm{\eta\big(\sy+\sz\big)}_{L^2_{\iota_i}L^1_{\mn}L^1_v}^{\N}+\e^{-\N}\nm{\eta\fb_1}_{L^2_{\iota_i}L^1_{\mn}L^1_v}^{\N}+\nm{\eta^2\big(\sy+\sz\big)}_{L^2_xL^1_v}^{\N}+\nm{\eta^2\fb_1}_{L^2_xL^1_v}^{\N}.\no
\end{align}
Following a similar argument, the other $\ss_3$ contribution can be estimated
\begin{align}
    &\abs{\br{\e^{-1}\phi\left(\abs{v}^2-5T\right)\mh+\nx\phi\cdot\a,\ss_3}}
    =\e^{-1}\abs{\br{\left(\int_0^{\mn}\p_{\mn}\phi\right)\left(\abs{v}^2-5T\right)\mh,\ss_3}}+\abs{\bbr{\nx\phi\cdot\a,\ss_3}}\\
    \ls&\tnm{\nx\phi}\nm{\eta\left(\sy+\sz+\fb_1\right)}_{L^2_xL^1_v}\ls \oot \jnm{c}^{\N}+\nm{\eta\left(\sy+\sz\right)}_{L^2_xL^1_v}^{\N}+\nm{\eta\fb_1}_{L^2_xL^1_v}^{\N},\no
\end{align}
and
\begin{align}
    \abs{\bbr{\e^{-1}\nx\psi:\b,\ss_3}}=&\e^{-1}\abs{\bbr{\left(\nx\psi(0)+\int_0^{\mn}\p_{\mn}\nx\psi\right):\b,\ss_3}}\label{final 54}\\
    \ls&\e^{-1}\nm{\p_{\mn}\psi(0):\b}_{L^2_xL^{\infty}_v}\nm{\ss_3}_{L^2_{\iota_i}L^1_{\mn}L^1_v}+\tnm{\p_{\mn}\nx\psi}\nm{\eta\left(\sy+\sz+\fb_1\right)}_{L^2_xL^1_v}\no\\
    \ls&\e^{-1}\nm{\p_{\mn}\psi}_{W^{3,\frac{\N}{\N-1}}}\nm{\ss_3}_{L^2_{\iota_i}L^1_{\mn}L^1_v}+\nm{\p_{\mn}\nx\psi}_{W^{2,\frac{\N}{\N-1}}}\nm{\eta\left(\sy+\sz+\fb_1\right)}_{L^2_xL^1_v}\no\\
    \ls& \oot \jnm{c}^{\N}+\e^{-\N}\nm{\ss_3}_{L^2_{\iota_i}L^1_{\mn}L^1_v}^{\N}+\nm{\eta\left(\sy+\sz\right)}_{L^2_xL^1_v}^{\N}+\nm{\eta\fb_1}_{L^2_xL^1_v}^{\N}.\no
\end{align}
Considering the Sobolev embedding $\lnm{q}\ls\nm{q}_{W^{3,\frac{\N}{\N-1}}}\ls\oot\jnm{c}^{\N-1}$, H\"older's inequality, Young's inequality and \eqref{qq 17} and \eqref{qq 18} yield
\begin{align}
    \abs{\e^{-1}P^{-1}\br{q,\mh\ss_3}}\ls&\e^{-1}\lnm{q}\nm{\mh\ss_3}_{L^1}
    \ls \lnm{q}^{\frac{\N}{\N-1}}+\e^{-\N}\nm{\ss_3}_{L^1}^{\N}
    \ls \oot\jnm{c}^{\N}+\e^{-\N}\nm{\ss_3}_{L^1}^{\N}.
\end{align}
In summary, using Lemma \ref{ss3-estimate}, the $\ss_3$ terms in \eqref{pp 49} can be bounded by
\begin{align}
    \oot \jnm{c}^{\N}+\oot.
\end{align}

\paragraph{\underline{$\ss_1$ Contribution}}
Direct estimate yields
\begin{align}\label{pp 120.}
    &\abs{\bbr{\nx\phi\cdot\a,\ss_1}}
    \ls \tnm{\nx\phi}^{\frac{\N}{\N-1}}\lnmm{\fb_1}\um{\re}^{\N}
    \ls \oot\jnm{c}^{\N}+\oot\xnm{\re}^{\N}.
\end{align}
Also, similar to \eqref{pp 123}, we have
\begin{align}\label{pp 123=}
    \abs{\bbr{\e^{-1}\nx\psi:\b,\ss_1}}\ls&\abs{\br{\e^{-1}\nx\psi:\b,\Gamma\left[\fb_1,\P\mh+\bb\cdot\vv\mh+\bd\cdot\a+(\ik-\bpk)[\re]\right]}}\\
    &+\abs{\br{\e^{-1}\nx\psi:\b,\Gamma\left[\fb_1,c\left(\abs{v}^2-5T\right)\mh\right]}}.\no
\end{align}
Considering the Sobolev embedding $\nm{\psi}_{W^{1,\infty}}\ls\nm{\psi}_{W^{4,\frac{\N}{\N-1}}}\ls\jnm{c}^{\N-1}$, we have
\begin{align}
    \text{First Term in \eqref{pp 123=}}\ls& \e^{-1}\lnm{\nx\psi}\tnm{\fb_1}\Big(\tnm{p}+\tnm{\bb}+\tnm{\bd}+\um{(\ik-\bpk)[\re]}\Big)\\
    \ls&\oot\jnm{c}^{\N}+\oot\e^{-\frac{\N}{2}}\tnm{p}^{\N}+\oot\e^{-\frac{\N}{2}}\tnm{\bb}^{\N}+\oot\e^{-\frac{\N}{2}}\tnm{\bd}^{\N}+\oot\e^{-\frac{\N}{2}}\um{(\ik-\bpk)[\re]}^{\N}.\no
\end{align}
On the other hand, we know
\begin{align}
    \text{Second Term in \eqref{pp 123=}}\ls&\e^{-1}\abs{\br{\nx\psi:\b,\Gamma\left[\fb_1,\big(\z^R-\e^{-1}\xi\big)\left(\abs{v}^2-5T\right)\mh\right]}}\\
    &+\e^{-1}\abs{\br{\nx\psi:\b,\Gamma\left[\fb_1,\z^S\left(\abs{v}^2-5T\right)\mh\right]}}.\no
\end{align}
Similar to \eqref{pp 124}, using \eqref{qq 06}, we have
\begin{align}
    &\e^{-1}\abs{\br{\nx\psi:\b,\Gamma\left[\fb_1,\big(\z^R-\e^{-1}\xi\big)\left(\abs{v}^2-5T\right)\mh\right]}}\\
    \ls&\lnm{\nx\psi}\tnm{\eta\fb_1}\tnm{\nx\z^R-\e^{-1}\nx\xi}\ls \oot\jnm{c}^{\N}+\oot\Big(\e^{\frac{\N}{2}}\nm{\z^R}_{H^1}^{\N}+\e^{-\frac{\N}{2}}\nm{\xi}_{H^1}^{\N}\Big)\no\\
    \ls&\oot\jnm{c}^{\N}+\oot\e^{\frac{\N}{2}}\tnm{\P}^{\N}+\e^{-\frac{\N}{2}}\tnm{\bb}^{\N}+\oot\e^{\frac{\N}{2}}\tnm{c}^{\N}+\oot\e^{\frac{\N}{2}}\tnm{(\ik-\bpk)[\re]}^{\N}+\oot\e^{\frac{\N}{2}}\xnm{\re}^{\N}+\oot\e^{\frac{\N}{2}}\xnm{\re}^{2\N}+\oot.\no
\end{align}
Also, using \eqref{qq 07}, we have
\begin{align}
    &\e^{-1}\abs{\br{\nx\psi:\b,\Gamma\left[\fb_1,\z^S\left(\abs{v}^2-5T\right)\mh\right]}}\\
    \ls&\e^{-1}\lnm{\nx\psi}\tnm{\fb_1}\tnm{\z^S}\ls \oot\jnm{c}^{\N}+\oot\e^{-\frac{\N}{2}}\tnm{\z^S}^{\N}\no\\
    \ls&\oot\jnm{c}^{\N}+\oot\e^{-\frac{\N}{2}}\um{(\ik-\bpk)[\re]}^{\N}+\oot\e^{-\frac{\N}{2}}\tnms{(1-\pp)[\re]}{\gamma_+}^{\N}+\oot\e^{\frac{\N}{2}}.\no
\end{align}
In total, we have the $\ss_1$ terms in \eqref{pp 49} is bounded by
\begin{align}
    &\oot\jnm{c}^{\N}+\oot\xnm{\re}^{\N}+\oot\e^{\frac{\N}{2}}\xnm{\re}^{2\N}+\oot.
\end{align}

\paragraph{\underline{$\ss_2$ Contribution}}
Notice that 
\begin{align}\label{pp 103'}
    \e^{-1}\bbr{\nx\psi:\b,\ss_2}=&\e^{-1}\br{\nx\psi:\b,\Gamma\Big[\pk[\re]+\bd\cdot\a,\pk[\re]+\bd\cdot\a\Big]}\\
    &+\e^{-1}\br{\nx\psi:\b,\Gamma\Big[\pk[\re]+\bd\cdot\a,(\ik-\bpk)[\re]\Big]}+\e^{-1}\br{\nx\psi:\b,\Gamma\Big[(\ik-\bpk)[\re],(\ik-\bpk)[\re]\Big]}.\no
\end{align}
For the first term in \eqref{pp 103'}, by oddness, we have
\begin{align}\label{pp 104'}
    &\e^{-1}\br{\nx\psi:\b,\Gamma\Big[\pk[\re]+\bd\cdot\a,\pk[\re]+\bd\cdot\a\Big]}\\
    =&\e^{-1}\br{\nx\psi:\b,\Gamma\Big[\P\mh+c\left(\abs{v}^2-5T\right)\mh,\P\mh+c\left(\abs{v}^2-5T\right)\mh\Big]}\no\\
    &+\e^{-1}\br{\nx\psi:\b,\Gamma\Big[\bb\cdot v\mh,\bb\cdot v\mh\Big]}+\e^{-1}\br{\nx\psi:\b,\Gamma\Big[\bd\cdot \a,\bb\cdot v\mh\Big]}+\e^{-1}\br{\nx\psi:\b,\Gamma\Big[\bd\cdot \a,\bd\cdot \a\Big]}.\no
\end{align}
For using oddness, similar to \eqref{qq 16}, we know
\begin{align}\label{final 05}
    \e^{-1}\br{\nx\psi:\b,\Gamma\Big[\P\mh+c\left(\abs{v}^2-5T\right)\mh,\P\mh+c\left(\abs{v}^2-5T\right)\mh\Big]}=0.
\end{align}
Also, using $\tnm{\bb}\ls\e^{\frac{1}{2}}\xnm{\re}$, we know
\begin{align}
    \abs{\e^{-1}\br{\nx\psi:\b,\Gamma\Big[\bb\cdot v\mh,\bb\cdot v\mh\Big]}}\ls& \oot\e^{-1}\lnm{\nx\psi}\tnm{\bb}^2\\
    \ls&\oot\jnm{c}^{{\N-1}}\xnm{\re}^2
    \ls\oot\jnm{c}^{\N}+ \xnm{\re}^{2\N}.\no
\end{align}
Similarly, using $\tnm{\bd}\ls\e^{\frac{1}{2}}\xnm{\re}$, we have
\begin{align}
    \abs{\e^{-1}\br{\nx\psi:\b,\Gamma\Big[\bd\cdot \a,\bb\cdot v\mh\Big]}}\ls \oot\jnm{c}^{\N}+\xnm{\re}^{2\N},
\end{align}
and
\begin{align}
    \abs{\e^{-1}\br{\nx\psi:\b,\Gamma\Big[\bd\cdot \a,\bd\cdot \a\Big]}}\ls \oot\jnm{c}^{\N}+ \xnm{\re}^{2\N}.
\end{align}
Hence, in total, we know
\begin{align}
    \Babs{\text{First Term in \eqref{pp 103'}}}\ls \oot\jnm{c}^{\N}+\xnm{\re}^{2\N}.
\end{align}
On the other hand, using $\um{(\ik-\bpk)[\re]}\ls\e\xnm{\re}$ we have
\begin{align}
    \Babs{\text{Second Term in \eqref{pp 103'}}}\ls&\e^{-1}\lnm{\nx\psi}\um{\pk[\re]+\bd\cdot\a}\um{(\ik-\bpk)[\re]}
    \ls \oot\jnm{c}^{\N}+\xnm{\re}^{2\N}.
\end{align}
In a similar fashion, we obtain
\begin{align}
    \Babs{\text{Third Term in \eqref{pp 103'}}}\ls \oot\jnm{c}^{\N}+\xnm{\re}^{2\N}.
\end{align}
In total, \eqref{pp 103'} can be controlled
\begin{align}
\abs{\e^{-1}\bbr{\nx\psi:\b,\ss_2}}\ls \oot\jnm{c}^{\N}+ \xnm{\re}^{2\N}.
\end{align}
Similarly, we know
\begin{align}\label{pp 103=}
    \bbr{\nx\phi\cdot\a,\ss_2}=&\br{\nx\phi\cdot\a,\Gamma\Big[\pk[\re]+\bd\cdot\a,\pk[\re]+\bd\cdot\a\Big]}\\
    &+\br{\nx\phi\cdot\a,\Gamma\Big[\pk[\re]+\bd\cdot\a,(\ik-\bpk)[\re]\Big]}+\br{\nx\phi\cdot\a,\Gamma\Big[(\ik-\bpk)[\re],(\ik-\bpk)[\re]\Big]}.\no
\end{align}
For the first term in \eqref{pp 103=}, by oddness and orthogonality of $\a$, we have
\begin{align}\label{pp 104=}
    &\br{\nx\phi\cdot\a,\Gamma\Big[\pk[\re]+\bd\cdot\a,\pk[\re]+\bd\cdot\a\Big]}\\
    =&\br{\nx\phi\cdot\a,\Gamma\Big[c\left(\abs{v}^2-5T\right)\mh,\bb\cdot v\mh\Big]}+\br{\nx\phi\cdot\a,\Gamma\Big[c\left(\abs{v}^2-5T\right)\mh,\bd\cdot \a\Big]}.\no
\end{align}
Using $\lnm{c}\ls\lnmm{\re}\ls\e^{-\frac{1}{2}}\xnm{\re}$ and $\tnm{\bb}\ls\e^{\frac{1}{2}}\xnm{\re}$, we know
\begin{align}
    \abs{\br{\nx\phi\cdot\a,\Gamma\Big[c\left(\abs{v}^2-5T\right)\mh,\bb\cdot v\mh\Big]}}\ls& \tnm{\nx\phi}\lnm{c}\tnm{\bb}
    \ls \jnm{c}^{{\N-1}}\xnm{\re}^2\ls\oot\jnm{c}^{\N}+\xnm{\re}^{2\N}.
\end{align}
Similarly, we have
\begin{align}
    \abs{\br{\nx\phi\cdot\a,\Gamma\Big[c\left(\abs{v}^2-5T\right)\mh,\bd\cdot \a\Big]}}\ls \oot\jnm{c}^{\N}+\xnm{\re}^{2\N}.
\end{align}
Hence, in total, we know
\begin{align}
    \Babs{\text{First Term in \eqref{pp 103=}}}\ls \oot\jnm{c}^{\N}+\xnm{\re}^{2\N}.
\end{align}
Similarly, we have
\begin{align}
    \Babs{\text{Second Term in \eqref{pp 103=}}}\ls&\tnm{\nx\phi}\lnmm{\pk[\re]+\bd\cdot\a}\um{(\ik-\bpk)[\re]}
    \ls\oot \jnm{c}^{\N}+\xnm{\re}^{2\N},\\
    \Babs{\text{Third Term in \eqref{pp 103=}}}\ls& \oo \jnm{c}^{\N}+\xnm{\re}^{2\N}.
\end{align}
In total, we know that \eqref{pp 103=} can be controlled 
\begin{align}
\abs{\bbr{\nx\phi\cdot\a,\ss_2}}\ls\oot \jnm{c}^{\N}+\xnm{\re}^{2\N}.
\end{align}
In summary, the $\ss_2$ contribution in \eqref{pp 49} can be bounded by
\begin{align}
\oot \jnm{c}^{\N}+\xnm{\re}^{2\N}.
\end{align}

\paragraph{\underline{$\ss_4$ and $\sp$ Contribution}}
Based on Lemma \ref{ss4-estimate}, we have
\begin{align}
    \abs{\br{\e^{-1}\phi\left(\abs{v}^2-5T\right)\mh,\ss_4}}\ls& \e^{-1}\nm{\phi}_{W^{1,\frac{\N}{\N-1}}}\jnm{\ss_4}\ls \oot\jnm{c}^{\N}+\oot.
\end{align}
Similarly, we have
\begin{align}
    \abs{\bbr{\nx\phi\cdot\a+\e^{-1}\nx\psi:\b,\ss_4}}\ls& \oot\jnm{c}^{\N}+\oot,\\
    \abs{\e^{-1}P^{-1}\br{q,\mh\ss_4}}\ls& \oot\jnm{c}^{\N}+\oot.
\end{align}
Based on Lemma \ref{ssp-estimate}, we have
\begin{align}
    \abs{\br{\e^{-2}\psi\cdot v\mh,\sp}}\ls \e^{-2}\tnm{\psi}\tnm{\sp}\ls \oot\jnm{c}^{\N}+\oot.
\end{align}
Similarly, we have
\begin{align}
    \abs{\br{\e^{-1}\phi\big(\abs{v}^2-5T\big)\mh,\sp}+\bbr{\nx\phi\cdot\a+\e^{-1}\nx\psi:\b,\sp}}\ls \oot\jnm{c}^{\N}+\oot.
\end{align}
In summary, the $\ss_4$ and $\sp$ terms \eqref{pp 49} can be bounded by
\begin{align}
\oot\jnm{c}^{\N}+\oot.
\end{align}

\paragraph{\underline{$\ss_0$ and $\ss_5$ Contribution}}
Based on Lemma \ref{ss0-estimate}, we have
\begin{align}\label{pp 120}
    \abs{\bbr{\nx\phi\cdot\a+\e^{-1}\nx\psi:\b,\ss_0}}
    \ls& \oot\tnm{\nx\phi}^{\frac{\N}{\N-1}}+\oot\tnm{\nx\psi}^{\frac{\N}{\N-1}}+\oot\xnm{\re}^{\N}
    \ls \oot\jnm{c}^{\N}+\oot\xnm{\re}^{\N}.
\end{align}
Also, 
\begin{align}
    \abs{\bbr{\nx\phi\cdot\a+\e^{-1}\nx\psi:\b,\ss_5}}\ls&\abs{\bbr{\nx\phi\cdot\a+\e^{-1}\nx\psi:\b,\e\Gamma[\f_i,\f_j]+\Gamma[\fb_1,\f_i]}}.
\end{align}
Clearly, we have
\begin{align}
    \abs{\bbr{\nx\phi\cdot\a+\e^{-1}\nx\psi:\b, \e\Gamma[\f_i,\f_j]}}
    \ls& \oot\tnm{\nx\phi}^{\frac{\N}{\N-1}}+\oot\tnm{\nx\psi}^{\frac{\N}{\N-1}}+\oot
    \ls\oot\jnm{c}^{\N}+\oot.
\end{align}
Then following similar techniques as \eqref{pp 101} and \eqref{final 54}, we obtain
\begin{align}
    &\abs{\bbr{\nx\phi\cdot\a+\e^{-1}\nx\psi:\b, \Gamma[\fb_1,\f_i]}}\\
    \ls&\oot\tnm{\nx\phi}^{\frac{\N}{\N-1}}+\oot\tnm{\nx\psi}^{\frac{\N}{\N-1}}+\e^{-\N}\nm{\fb_1}_{L^2_{\iota_i}L^1_{\mn}L^2_v}^{\N}+\tnm{\eta\fb_1}^{\N}\ls\oot\jnm{c}^{\N}+\oot.\no
\end{align}
In summary, the $\ss_0$ and $\ss_5$ terms \eqref{pp 49} can be bounded by
\begin{align}
\oot\jnm{c}^{\N}+\oot.
\end{align}
We therefore deduce Lemma \ref{lem:final 3} by collecting the above estimates.
\end{proof}

Combining Lemma \ref{lem:final 3} and Lemma \ref{lem:final 4}, and applying Lemma \ref{h-estimate}, we conclude  
\begin{lemma}
Under the assumption \eqref{assumption:boundary}, we have
\begin{align}\label{pp 28}
\jnm{c}^{\N}\ls&-\e^{-1}\bbrb{\nx\psi:\b,(1-\pp)[\re]}{\gamma_+}\\
&+\jnm{(\ik-\bpk)[\re]}^{\N}+\jnms{\m^{\frac{1}{4}}(1-\pp)[\re]}{\gamma_+}^{\N}+\oot\xnm{\re}^{\N}+\xnm{\re}^{2\N}+\oot.\no
\end{align}
\end{lemma}

Now the remaining difficulty is to control the boundary term $\ds\e^{-1}\bbrb{\nx\psi:\b,(1-\pp)[\re]}{\gamma_+}$.

%%%%%%%%%%%%%%%%%%%%%%%%%%%%%%%%%%%%%%%%%%%%%%%%%%%%%%%%%%%%%%%%%%%%%%%%%%%%
\subsection{Auxiliary Function $\gb$}
%%%%%%%%%%%%%%%%%%%%%%%%%%%%%%%%%%%%%%%%%%%%%%%%%%%%%%%%%%%%%%%%%%%%%%%%%%%%

In order to handle the troublesome boundary term $\ds\e^{-1}\bbrb{\nx\psi:\b,(1-\pp)[\re]}{\gamma_+}$, we plan to design an auxiliary function $\gb(x,v)$ satisfying
\begin{align}\label{pp 46}
    \gb\big|_{\gamma_+}=-\Big(\nx\psi:\b\Big)\Big|_{\gamma_+}+\gb_h\Big|_{\gamma_+},
\end{align}
and the almost zero mass-flux condition
\begin{align}\label{pp 116}
    \int_{\p\Omega\times\r^3}\gb\mh(v\cdot n)=\gb_m.
\end{align}
for some small $\gb_h$ and $\gb_m$ which will be specified later (see \eqref{final 61'} and \eqref{final 62'}). Taking the test function $\test=\gb$ in \eqref{weak formulation}, and considering that $\lc$ is self-adjoint, we obtain
\begin{align}\label{pp 26}
    &\int_{\gamma}(\gb R)(v\cdot n)+\br{-v\cdot\nx\gb+\left(\mhh\ab\cdot\frac{\nx T}{4T^2}\right)\gb+\e^{-1}\lc[\gb],\re}=\bbr{\gb,\ss}.
\end{align}
Noting that \eqref{pp 116} yields
\begin{align}
    \int_{\p\Omega\times\r^3}\Big(\gb\pp[\re]\Big)(v\cdot n)=\gb_m\int_{\p\Omega}\mhh\pp[\re],
\end{align}
with the help of \eqref{qq 109}, we know that \eqref{pp 26} becomes
\begin{align}\label{pp 27}
    -\bbrb{\nx\psi:\b,(1-\pp)[\re]}{\gamma_+}-\bbrb{\gb, h}{\gamma_-}+\bbrb{\gb_h,(1-\pp)[\re]}{\gamma_+}+\gb_m\int_{\p\Omega}\mhh\pp[\re]\\
    +\br{-v\cdot\nx\gb
    +\left(\mhh\ab\cdot\frac{\nx T}{4T^2}\right)\gb+\e^{-1}\lc[\gb],\re}&=\bbr{\gb,\ss}.\no
\end{align}
Adding $\e^{-1}\times$\eqref{pp 27} and \eqref{pp 28} to eliminate $\ds-\e^{-1}\bbrb{\nx\psi:\b,(1-\pp)[\re]}{\gamma_+}$, we obtain
\begin{align}\label{pp 29}
    \jnm{c}^{\N}\ls&-\e^{-1}\br{-v\cdot\nx\gb+\left(\mhh\ab\cdot\frac{\nx T}{4T^2}\right)\gb+\e^{-1}\lc[\gb],\re}\\
    &+\e^{-1}\bbr{\gb,\ss}+\e^{-1}\bbrb{\gb,h}{\gamma_-}-\e^{-1}\bbrb{\gb_h,(1-\pp)[\re]}{\gamma_+}-\e^{-1}\gb_m\int_{\p\Omega}\mhh\pp[\re]\no\\
    &+\jnm{(\ik-\bpk)[\re]}^{\N}+\jnms{\m^{\frac{1}{4}}(1-\pp)[\re]}{\gamma_+}^{\N}+\oot\xnm{\re}^{\N}+\xnm{\re}^{2\N}+\oot.\no
\end{align}
Now the key is to estimate the auxiliary terms in \eqref{pp 29}
\begin{align}\label{pp 30}
    \e^{-1}\br{-v\cdot\nx\gb+\left(\mhh\ab\cdot\frac{\nx T}{4T^2}\right)\gb+\e^{-1}\lc[\gb],\re}
\end{align}
and the interaction terms
\begin{align}\label{pp 51'}
    &\e^{-1}\bbr{\gb,\ss}+\e^{-1}\bbrb{\gb,h}{\gamma_-}-\e^{-1}\bbrb{\gb_h,(1-\pp)[\re]}{\gamma_+}-\e^{-1}\gb_m\int_{\p\Omega}\mhh\pp[\re].
\end{align}

%%%%%%%%%%%%%%%%%%%%%%%%%%%%%%%%%%%%%%%%%%%%%%%%%%%%%%%%%%%%%%%%%%%%%%%%%%%%
\subsection{Construction of $\gb$}
%%%%%%%%%%%%%%%%%%%%%%%%%%%%%%%%%%%%%%%%%%%%%%%%%%%%%%%%%%%%%%%%%%%%%%%%%%%%

We define $\gb:=\gbg+\gg$ where $\gbg$ is an auxiliary $\e$-cutoff boundary layer and $\gg$ is an auxiliary interior solution. 

\paragraph{\underline{$\gbg$ Construction}}
Using the substitution in Section \ref{sec:boundary-form}, define $\gbg$ as an $\e$-cutoff boundary layer similar to $\fb_1$:
\begin{align}\label{final 03}
    \gbg(\eta,\iota_1,\iota_2,\vvv):=\ch\left(\e^{-1}\va\right)\chi(\e\eta)\Big(\blg(\eta,\iota_1,\iota_2,\vvv)-\blg_{\infty}(\iota_1,\iota_2,\vvv)\Big):=\ch\left(\e^{-1}\va\right)\chi(\e\eta)\blgg(\eta,\iota_1,\iota_2,\vvv)
\end{align}
where
\begin{align}\label{pp 36'}
    \left\{
    \begin{array}{l}
    -\va\dfrac{\p\blg}{\p\eta}+\lc_w\big[\blg\big]=0\ \ \text{in}\ \ [0,\infty)\times\r^3,\\\rule{0ex}{1.0em}
    \blg(0,\iota_1,\iota_2,\vvv)=\fbf(\vvv):=-\Big(\nx\psi(0,\iota_1,\iota_2,):\b(0,\iota_1,\iota_2,\vvv)\Big)\ \ \text{for}\ \ \va<0,
    \end{array}
    \right.
\end{align}
satisfying (from the almost zero mass-flux condition \eqref{pp 116})
\begin{align}\label{pp 35}
    &\int_{\r^3}\blg(0,\vvv)\m_w^{\frac{1}{2}}(\vvv)\va=0,\quad
    \lim_{\eta\rt\infty}\blg(\eta,\iota_1,\iota_2,\vvv)=\blg_{\infty}(\iota_1,\iota_2,\vvv)\in\nk.
\end{align}

Let $\nm{\cdot}_{L^p_{\iota_1\iota_2}L^{\infty}_{\eta}L^{\infty}_{\varrho,\vartheta}}$ denote the $L^p$ norm in $\iota_1$ and $\iota_2$ and $L^{\infty}_{\varrho,\vartheta}$ norm in $\eta$ and $\vvv$. The similar notation also applies to $\nm{\cdot}_{L^p_{\iota_1\iota_2}L^{\infty}_{\gamma,\varrho,\vartheta}}$
\begin{lemma}\label{lem:final 5}
    Under the assumption \eqref{assumption:boundary}, there exists $\blg(\eta,\iota_1,\iota_2,\vvv)$ satisfying \eqref{pp 36'} and \eqref{pp 35} such that for some $K_0>0$
    \begin{align}\label{qq 24}
    \lnmm{\blg_{\infty}}+\lnmm{\ue^{K_0\eta}\blgg}\ls& \oot\jnm{c}^{\N-1},\\
    \label{qq 23'}
    \lnmm{\ue^{K_0\eta}\p_{\vb}\blgg}+\lnmm{\ue^{K_0\eta}\p_{\vc}\blgg}\ls& \oot\jnm{c}^{\N-1},
    \end{align}
    \begin{align}
    \label{qq 23}
    \lnmm{\p_{\iota_1}\blg_{\infty}}+\lnmm{\p_{\iota_2}\blg_{\infty}}+\nm{\ue^{K_0\eta}\p_{\iota_1}\blgg}_{L^2_{\iota_1\iota_2}L^{\infty}_{\eta}L^{\infty}_{\varrho,\vartheta}}+\nm{\ue^{K_0\eta}\p_{\iota_2}\blgg}_{L^2_{\iota_1\iota_2}L^{\infty}_{\eta}L^{\infty}_{\varrho,\vartheta}}
    \ls&\oot\jnm{c}^{\N-1},
    \end{align}
    \begin{align}\label{final 55}
        \bnm{\blgg}\ls\oot\jnm{c}^{\N-1}.
    \end{align}
    and for any $\eta$,
    \begin{align}\label{final 64}
    \int_{\r^3}\blg(\eta,\iota_1,\iota_2 ,\vvv)\m_w^{\frac{1}{2}}(\vvv)\va=0.
    \end{align}
    Then $\gbg$ can be defined as \eqref{final 03}.
\end{lemma}

\begin{proof}
    For fixed $(\iota_1,\iota_2)$, using Theorem \ref{boundary well-posedness}, we have
    \begin{align}\label{qq 24=}
    \lnmm{\blg_{\infty}}+\lnmm{\ue^{K_0\eta}\blgg}\ls& \lnmms{\fbf}{\gamma_+},\\
    \label{qq 23'=}
    \lnmm{\ue^{K_0\eta}\p_{\vb}\blgg}+\lnmm{\ue^{K_0\eta}\p_{\vc}\blgg}\ls& \lnmms{\fbf}{\gamma_+}+\lnmms{\p_{\vb}\fbf}{\gamma_+}+\lnmms{\p_{\vc}\fbf}{\gamma_+},
    \end{align}
    and
    \begin{align}
    \label{qq 23=}
    &\lnmm{\p_{\iota_1}\blg_{\infty}}+\lnmm{\p_{\iota_2}\blg_{\infty}}+\nm{\ue^{K_0\eta}\p_{\iota_1}\blgg}_{L^2_{\iota_1\iota_2}L^{\infty}_{\eta}L^{\infty}_{\varrho,\vartheta}}+\nm{\ue^{K_0\eta}\p_{\iota_2}\blgg}_{L^2_{\iota_1\iota_2}L^{\infty}_{\eta}L^{\infty}_{\varrho,\vartheta}}\\
    \ls& \nm{\fbf}_{L^2_{\iota_1\iota_2}L^{\infty}_{\gamma_+,\varrho,\vartheta}}+\nm{\p_{\iota_1}\fbf}_{L^2_{\iota_1\iota_2}L^{\infty}_{\gamma_+,\varrho,\vartheta}}+\nm{\p_{\iota_2}\fbf}_{L^2_{\iota_1\iota_2}L^{\infty}_{\gamma_+,\varrho,\vartheta}}.\no
    \end{align}
    Then for $2\leq \N\leq 6$, due to trace theorem and Sobolev embedding 
    \begin{align}
    \tnms{\nx^2\psi}{\p\Omega}\ls&\nm{\psi}_{H^3}\ls\nm{\psi}_{W^{4,\frac{\N}{\N-1}}}\ls\oot\jnm{c}^{\N-1},\\
    \lnms{\nx\psi}{\p\Omega}\ls&\nm{\psi}_{H^3}\ls\nm{\psi}_{W^{4,\frac{\N}{\N-1}}}\ls\oot\jnm{c}^{\N-1},
    \end{align}
    from the definition of $\fbf$ in \eqref{pp 36'}, we have the desired results. The estimate \eqref{final 55} follows from Theorem \ref{boundary regularity}. Also, \eqref{final 64} comes from direct integrating in \eqref{pp 36'}.
\end{proof}

\begin{lemma}\label{lem:final 6}
    Under the assumption \eqref{assumption:boundary}, we have that 
    $\blg$, $\blgg$ and $\blg_{\infty}$ defined in \eqref{pp 36'} are all odd in $\vb$ and $\vc$. In particular, we have
    \begin{align}\label{pp 39}
    \blg_{\infty}(\vvv)=\big(\vvv\cdot\bb_{\infty}\big)\m_w^{\frac{1}{2}}(\vvv)=\Big(\vb b_{\infty,\iota_1}+\vc b_{\infty,\iota_2}\Big)\m_w^{\frac{1}{2}}(\vvv),
    \end{align}
    for $\bb_{\infty}:=\Big(0,b_{\infty,\iota_1},b_{\infty,\iota_2}\Big)$ with some $b_{\infty,\iota_1}$ and $b_{\infty,\iota_2}$ only depending on $(\iota_1,\iota_2)$. In addition, $\gbg$ is odd in $\vb$ and $\vc$:
    \begin{align}\label{gb-oddness}
        \gbg(\vxx,\va,\vb,\vc)=-\gbg(\vxx,\va,-\vb,\vc),\quad \gbg(\vxx,\va,\vb,\vc)=-\gbg(\vxx,\va,\vb,-\vc).
    \end{align}
\end{lemma}
\begin{proof}
Since $\psi=(\psi_{\iota_1},\psi_{\iota_2},\psi_{\mn})=\od$ on $\p\Omega$, we know $\p_{\iota_1}\psi=\p_{\iota_2}\psi=0$. Also, due to $\nx\cdot\psi=0$, we know $\p_{\mn}\psi_n=0$. Hence,
\begin{align}
    \fbf(\vvv)=-\p_{\mn}\psi_{\iota_1}(0)\b_{\mn\iota_1}(0)-\p_{\mn}\psi_{\iota_2}(0)\b_{\mn\iota_2}(0).
\end{align}
Due to the oddness of $\fbf(\vvv)$ in $\vb$ and $\vc$ (since $\b_{\mn\iota_1}$ and $\b_{\mn\iota_2}$ are odd), using Theorem \ref{boundary well-posedness}, we know that the solution $\blg(\eta,\vvv)$ is also odd in $\vb$ and $\vc$. Hence, for any two-variable function $\mathcal{M}$, we know 
\begin{align}\label{final 01}
    \int_{\r^3}\mathcal{M}\big(\abs{v}^2,\va\big)\blg(\eta,\vvv)=0.
\end{align}
Denote
\begin{align}
    \pk[\blg](\eta,\vvv):=\bigg( \P_{\blg}(\eta )+\vvv\cdot \bb_{\blg}(\eta )+\left(\abs{\vvv}^2-5\right) c_{\blg}(\eta )\bigg)\m_w^{\frac{1}{2}}(\vvv).
\end{align}
The oddness dictates that $\P_{\blg}=c_{\blg}=0$. Also, for $\bb_{\blg}:=\Big(b_{\blg,\iota_1},b_{\blg,\iota_2},b_{\blg,\mn}\Big)$, \eqref{final 64} implies $b_{\blg,\mn}=0$. Hence, we have
\begin{align}
    \pk[\blg](\eta ,\vvv)=\Big(\vb b_{\blg,\iota_1}(\eta )+\vc b_{\blg,\iota_2}(\eta )\Big)\m_w^{\frac{1}{2}}(\vvv).
\end{align}
Similarly, considering
\begin{align}
    \blg_{\infty}(\vvv)=\pk[\blg_{\infty}](v):=\bigg( \P_{\infty}+\vvv\cdot \bb_{\infty}+\left(\abs{\vvv}^2-5\right) c_{\infty}\bigg)\m_w^{\frac{1}{2}}(\vvv),
\end{align}
we know $\bb_{\infty}=\Big(b_{\infty,\iota_1},b_{\infty,\iota_2},b_{\infty,\mn}\Big)$ satisfies $b_{\infty,\mn}=0$ from \eqref{pp 35}, and thus
\begin{align}
    \blg_{\infty}(\vvv)=\Big(\vb b_{\infty,\iota_1}+\vc b_{\infty,\iota_2}\Big)\m_w^{\frac{1}{2}}(\vvv).
\end{align}
Similar to \eqref{final 01}, using \eqref{final 03}, we have the oddness of $\gbg$, i.e. for any two-variable function $\mathcal{M}$, we know
\begin{align}\label{final 02}
    \int_{\r^3}\mathcal{M}\big(\abs{v}^2,\va\big)\gbg(\eta,\vvv)=0.
\end{align}
\end{proof}

\paragraph{\underline{$\gg$ Construction}}
Now we discuss the construction of $\gg$.
In order to fulfill \eqref{pp 46}, we need
\begin{align}
    \gg\Big|_{\p\Omega}=\blg_{\infty}=\big(\vvv\cdot \bb_{\infty}\big)\m_w^{\frac{1}{2}}.
\end{align}

\begin{lemma}\label{lem:final 7}
    Under the assumption \eqref{assumption:boundary}, there exists an extension
    $\bbq$ of $\bb_{\infty}$ satisfying
    \begin{align}\label{pp 47}
    \nx\cdot\bbq=0\ \text{in}\ \Omega,\quad
    \vv\cdot\bbq=\vvv\cdot\bb_{\infty}\ \text{on}\ \p\Omega,
    \end{align}
    such that 
    \begin{align}
        \nm{\bbq}_{H^1}\ls \oot\jnm{c}^{\N-1}.
    \end{align}
    Then $\gg$ can be defined as
    \begin{align}\label{final 10}
    \gg&=\mh(\vv)\big(\vv\cdot\bbq \big)
    \end{align}
\end{lemma}
\begin{remark}
    Since the change-of-variable $\vv\rt\vvv$ is orthogonal, we know $\mh$ is invariant. Then thanks to \eqref{aa 44}, the boundary condition in \eqref{pp 47} is valid by applying a linear orthogonal transformation. Hence, we deduce that $\bbq\cdot\vn=\bb_{\infty}\cdot\vn=0$. 
\end{remark}
\begin{proof}[Proof of Lemma \ref{lem:final 7}]
The construction of $\bbq$ consists of two steps:

\subparagraph{$\overline{\bb}_{\infty}$ Construction:}
Let $\overline{\bb}_{\infty}(x)$ for $x\in\Omega$ be a Sobolev extension of $\bb_{\infty}$ satisfying $\vv\cdot\overline{\bb}_{\infty}=\vvv\cdot\bb_{\infty}$ on $\p\Omega\times\r^3$ and
\begin{align}
    \nm{\overline{\bb}_{\infty}}_{H^1}\ls \nm{\bb_{\infty}}_{H^{\frac{1}{2}}_{\p\Omega}}.
\end{align}

\subparagraph{$\overline\bbq$ Construction:}
Since $\ds\int_{\Omega}\nx\cdot\overline{\bb}_{\infty}=\int_{\p\Omega}\overline{\bb}_{\infty}\cdot n=\int_{\p\Omega}\bb_{\infty,\mn}=0$, based on \cite[Theorem IV.5.2]{Boyer.Fabrie2013}, we know that there exists a unique solution $(\overline\bbq,\overline\ppq)\in H^1\times L^2$ (with $\overline\ppq$ zero average) to the Stokes problem
\begin{align}
    -\dx\overline\bbq+\nx\overline\ppq=0,\quad
    \nx\cdot\overline\bbq=-\nx\cdot\overline{\bb}_{\infty}\ \text{in}\ \Omega,\qquad
    \overline\bbq=0\ \text{on}\ \p\Omega,
\end{align}
satisfying
\begin{align}
    \nm{\overline\bbq}_{H^1}+\tnm{\overline\ppq}\ls \nm{\nx\cdot\overline\bb_{\infty}}_{L^2}\ls \nm{\overline{\bb}_{\infty}}_{H^1}\ls \nm{\bb_{\infty}}_{H^{\frac{1}{2}}_{\p\Omega}}.
\end{align}

\subparagraph{Summary of $\gg$ Construction:}
Hence, we know $\bbq=\overline\bbq+\overline{\bb}_{\infty}$ is the solution to \eqref{pp 47} and using \eqref{qq 23} it satisfies
\begin{align}
    \nm{\bbq}_{H^1}\ls \nm{ \bb_{\infty}}_{H^{\frac{1}{2}}_{\p\Omega}}\ls \oot\jnm{c}^{\N-1}.
\end{align}

\end{proof}

\paragraph{\underline{Summary of $\gb$ Construction}}
Therefore, we split \eqref{pp 30} based on $\gb=\gbg+\gg$.
Similar to the derivation of $\ss_3$ in \eqref{mm 00}, direct computation using the substitution in Section \ref{sec:boundary-form} reveals 
\begin{align}\label{pp 38}
    &\e^{-1}\br{-v\cdot\nx\gbg+\left(\mhh\ab\cdot\frac{\nx T}{4T^2}\right)\gbg+\e^{-1}\lc\left[\gbg\right],\re}\\
    =&-\e^{-1}\bbr{\dfrac{1}{R_1-\e\eta}\bigg(\vb^2\dfrac{\p }{\p\va}\bigg)\gbg+\dfrac{1}{R_2-\e\eta}\bigg(\vc^2\dfrac{\p }{\p\va}\bigg)\gbg,\re}\no\\
    &+\e^{-1}\bbr{\dfrac{1}{R_1-\e\eta}\bigg(\va\vb\dfrac{\p}{\p\vb}\bigg)\gbg+\dfrac{1}{R_2-\e\eta}\bigg(\va\vc\dfrac{\p }{\p\vc}\bigg)\gbg,\re}\no\\
    &-\e^{-1}\bbr{\dfrac{1}{\pl_1\pl_2}\left(-\dfrac{\p_{\iota_1\iota_1}\vr\cdot\p_{\iota_2}\vr}{\pl_1(\e\kk_1\eta-1)}\vb\vc
    +\dfrac{\p_{\iota_1\iota_2}\vr\cdot\p_{\iota_2}\vr}{\pl_2(\e\kk_2\eta-1)}\vc^2\right)\dfrac{\p\gbg }{\p\vb},\re}\no\\
    &-\e^{-1}\bbr{\dfrac{1}{\pl_1\pl_2}\left(\dfrac{\p_{\iota_2\iota_2}\vr\cdot\p_{\iota_1}\vr}{\pl_2(\e\kk_2\eta-1)}\vb\vc
    +\dfrac{\p_{\iota_1\iota_2}\vr\cdot\p_{\iota_1}\vr}{\pl_1(\e\kk_1\eta-1)}\vb^2\right)\dfrac{\p\gbg }{\p\vc},\re}\no\\
    &-\e^{-1}\bbr{\left(\dfrac{\vb}{\pl_1(\e\kk_1\eta-1)}\dfrac{\p }{\p\iota_1}+\dfrac{\vc}{\pl_2(\e\kk_2\eta-1)}\dfrac{\p }{\p\iota_2}\right)\gbg,\re}+\e^{-1}\br{\left(\mhh\ab\cdot\frac{\nx\tq}{4\tq^2}\right)\gbg,\re}-\e^{-2}\br{(v\cdot n)\frac{\p\chi(\e\eta)}{\p\eta}\blgg,\re}\no\\
    &-\e^{-2}\br{\chi(\e\eta)\bigg(\chi(\e^{-1}\va)K\Big[\blgg\Big]-K\Big[\chi(\e^{-1}\va)\blgg\Big]\bigg),\re}+\e^{-2}\br{\lc\left[\gbg\right]-\lc_w\left[\gbg\right],\re}.\no
\end{align}
Also, due to the direct computation
\begin{align}
    -\mh v\cdot\nx\Big(\mhh\gg\Big)=&-\mh v\cdot\left(\mhh\nx\gg-\frac{1}{2}\gg\m^{-\frac{3}{2}}\nx\m\right)\\
    =&-\mh v\cdot\left(\mhh\nx\gg-\gg\m^{-\frac{1}{2}}\left(\abs{v}^2-5\tq\right)\frac{\nx T}{4T^2}\right)
    =-v\cdot\nx\gg+\left(\mhh\ab\cdot\frac{\nx T}{4T^2}\right)\gg,\no
\end{align}
and $\gg\in\nk$ from \eqref{final 10}, we have
\begin{align}\label{pp 45}
    &\e^{-1}\br{-v\cdot\nx\gg+\left(\mhh\ab\cdot\frac{\nx T}{4T^2}\right)\gg+\e^{-1}\lc[\gg],\re}=-\e^{-1}\br{\mh v\cdot\nx\Big(\mhh\gg\Big),\re}.
\end{align}
Also, due to the cutoff $\ch\left(\e^{-1}\va\right)$ in $\gbg$ definition \eqref{pp 36'}, we have
\begin{align}\label{final 61'}
    \gb_h:=\chi\left(\e^{-1}\va\right)\fbf(\vvv),
\end{align}
with
\begin{align}\label{final 61}
    \tnms{\gb_h}{\gamma_+}\ls \oot\e\jnm{c}^{\N-1},
\end{align}
and
\begin{align}\label{final 62'}
    \gb_m:=&\int_{\r^3}\va\mbh\chi\left(\e^{-1}\va\right)\blgg(0,\vvv)\ud \vvv=\int_{\r^3}\va\mbh\chi\left(\e^{-1}\va\right)\fbf(0,\vvv)\ud \vvv,
\end{align}
with
\begin{align}\label{final 62}
    \abs{\gb_m}\ls\oot\e^2\jnm{c}^{\N-1}.
\end{align}

%%%%%%%%%%%%%%%%%%%%%%%%%%%%%%%%%%%%%%%%%%%%%%%%%%%%%%%%%%%%%%%%%%%%%%%%%%%%
\subsection{Estimates of Auxiliary $\e$-Cutoff Boundary Layer $\gbg$ Terms}\label{sec:gb estimate}
%%%%%%%%%%%%%%%%%%%%%%%%%%%%%%%%%%%%%%%%%%%%%%%%%%%%%%%%%%%%%%%%%%%%%%%%%%%%

\begin{lemma}\label{lemma:k1}
    Let $\hb(\eta,\iota_1,\iota_2,\vv)$ with $\eta=\e^{-1}\mn$ be a boundary-layer-type quantity satisfying 
    \begin{align}\label{final 65}
        \tnm{\hb}+\nm{\hb}_{L^2_{\iota_1\iota_2}L^2_{\eta}L^1_{v}}+\nm{\eta\hb}_{L^2_{\iota_1\iota_2}L^2_{\eta}L^1_{v}}\ls\oot\jnm{c}^{\N-1}.
    \end{align} 
    Then under the assumption \eqref{assumption:boundary}, we have
    \begin{align}
    \abs{\e^{-1}\bbr{\hb,\re}}\ls&\oot\jnm{c}^{\N}+\oot\xnm{\re}^{\N}+\xnm{\re}^{2\N}+\oot. 
    \end{align}
\end{lemma}
\begin{remark}
    When applying Lemma \ref{lemma:k1}, we will let $\hb=\gbg$, or its derivatives. 
\end{remark}
\begin{proof}[Proof of Lemma \ref{lemma:k1}]
We decompose
\begin{align}\label{qq 40}
    \e^{-1}\bbr{\hb,\re}
    =&I_1+I_2+I_3\\
    :=&\e^{-1}\bbr{\hb,(\ik-\bpk)[\re]}+\e^{-1}\br{\hb,\P\mh+\bb\cdot v\mh+\bd\cdot\a}+\e^{-1}\br{\hb,c\left(\abs{v}^2-5T\right)\mh}.\no
\end{align}
We directly bound
\begin{align}
    \babs{I_{1}}\ls&\e^{-1}\tnm{\hb}\tnm{(\ik-\bpk)[\re]}\ls\oot\jnm{c}^{\N}+\oot\e^{-\N}\tnm{(\ik-\bpk)[\re]}^{\N}\ls \oot\jnm{c}^{\N}+\oot\xnm{\re}^{\N}.
\end{align}
Using the scaling $\eta=\e^{-1}\mn$, we can bound
\begin{align}\label{qq 41}
    \babs{I_{2}}\ls& \e^{-1}\nm{\hb}_{L^2_{\iota_1\iota_2}L^2_{\mn}L^1_{v}}\Big(\tnm{\P}+\tnm{\bb}+\tnm{\bd}\Big)
    \ls\e^{-\frac{1}{2}} \nm{\hb}_{L^2_{\iota_1\iota_2}L^2_{\eta}L^1_{v}}\Big(\tnm{\P}+\tnm{\bb}+\tnm{\bd}\Big)\\
    \ls& \oot\jnm{c}^{\N}+\oot\Big(\e^{-\frac{\N}{2}}\tnm{\P}^{\N}+\e^{-\frac{\N}{2}}\tnm{\bb}^{\N}+\e^{-\frac{\N}{2}}\tnm{\bd}^{\N}\Big)\ls \oot\jnm{c}^{\N}+\oot\xnm{\re}^{\N}.\no
\end{align}
For $I_{3}$, notice that
\begin{align}\label{pp 41}
    c=Z-\e^{-1}\xi=\Big(\z^R-\e^{-1}\xi\Big)+\z^S:=& K_1+K_2.
\end{align}
We define a natural extension of $\z^R$ and $\xi$ with zero value for $\mn>1$ (for the convenience of Hardy's inequality). Then using \eqref{final 65}, Hardy's inequality, Lemma \ref{remark 01} and \eqref{qq 06}, we have
\begin{align}\label{pp 96=}
    \babs{I_{3, K_1}}\ls&\e^{-1}\br{\hb,\int_0^{\mn}\big(\p_{\mn}\z^R-\e^{-1}\p_{\mn}\xi\big)}
    \ls\e^{-1}\nm{\mn\hb}_{L^2_{\iota_1\iota_2}L^2_{\mn}L^1_{v}}\tnm{\frac{1}{\mn}\int_0^{\mn}\big(\p_{\mn}\z^R-\e^{-1}\p_{\mn}\xi\big)}\\
    \ls&\nm{\eta\hb }_{L^2_{\iota_1\iota_2}L^2_{\mn}L^1_{v}}\Big(\tnm{\p_{\mn}\z^R}+\e^{-1}\tnm{\p_{\mn}\xi}\Big)\ls\e^{\frac{1}{2}} \nm{\eta\hb}_{L^2_{\iota_1\iota_2}L^2_{\eta}L^1_{v}}\Big(\nm{\z^R}_{H_0^1}+\e^{-1}\nm{\xi}_{H_0^1}\Big)\no\\
    \ls& \oot\jnm{c}^{\N}+\oot\e^{\frac{\N}{2}}\xnm{\re}^{\N}+\e^{\frac{\N}{2}}\xnm{\re}^{2\N}+\oot\e^{\frac{\N}{2}}.\no
\end{align}
Similar to the estimate of $I_{1}$, using \eqref{qq 07}, we have
\begin{align}\label{pp 96='}
    \babs{I_{3, K_2}}\ls& \e^{-1}\nm{\hb}_{L^2_{\iota_1\iota_2}L^2_{\eta}L^1_{v}}\tnm{\z^S}\ls\e^{-\frac{1}{2}} \nm{\hb}_{L^2_{\iota_1\iota_2}L^2_{\eta}L^1_{v}}\tnm{\z^S}
    \ls \jnm{c}^{\N}+\oot\xnm{\re}^{\N}+\xnm{\re}^{2\N}+\oot.
\end{align}
We deduce Lemma \ref{lemma:k1} by collecting the above estimates.
\end{proof}

\begin{lemma}
    Under the assumption \eqref{assumption:boundary}, we have
    \begin{align}\label{pp 44'}
    \babs{\text{\eqref{pp 38}}}\ls&\oot\jnm{c}^{\N}+\oot\xnm{\re}^{\N}+\xnm{\re}^{2\N}+\oot.
\end{align}
\end{lemma}

\begin{proof}
Comparing the first term in \eqref{qq 40} with \eqref{mm 00}, we know
\begin{align}
    M:=\dfrac{1}{R_1-\e\eta}\bigg(\vb^2\dfrac{\p }{\p\va}\bigg)\gbg+\dfrac{1}{R_2-\e\eta}\bigg(\vc^2\dfrac{\p }{\p\va}\bigg)\gbg
\end{align}
is essentially $\sx$ with $\fb_1$ replaced by $\gbg$ and $\blff$ replaced by $\blgg$. Hence, using the proof of Lemma \ref{ss3-estimate}, we have
\begin{align}
    \pnm{M}{\N}&\ls\oot\e^{\frac{2}{\N}-1}\jnm{c}^{\N-1}.\label{gg4}
\end{align}
Similar to \eqref{qq 40}, we split
\begin{align}\label{final 56}
\\
    \text{First Term in \eqref{pp 38}}
    =&\e^{-1}\bbr{M,(\ik-\bpk)[\re]}+\e^{-1}\bbr{M,\P\mh+\bb\cdot v\mh+\bd\cdot\a}
    +
    \e^{-1}\bbr{M,c\left(\abs{v}^2-5T\right)\mh}.\no
\end{align}
For the first term in \eqref{final 56}, using \eqref{gg4}, we bound
\begin{align}
    \abs{\e^{-1}\bbr{M,(\ik-\bpk)[\re]}}\ls&\e^{-1}\tnm{M}\tnm{(\ik-\bpk)[\re]}\ls\oot\jnm{c}^{\N}+\oot\e^{-\N}\tnm{(\ik-\bpk)[\re]}^{\N}\\
    \ls&\oot\jnm{c}^{\N}+\oot\xnm{\re}^{\N}.\no
\end{align}
For the second term in \eqref{final 56}, we first integrate by parts with respect to $\va$, and then use Lemma \ref{lem:final 5} with a similar argument as \eqref{qq 41} to obtain
\begin{align}
    &\abs{\e^{-1}\bbr{M,\P\mh+\bb\cdot v\mh+\bd\cdot\a}}\ls\e^{-1}\abs{\bbr{\gbg,\P\nabla_v\mh+\bb\cdot \nabla_v(v\mh)+\bd\cdot\nabla_v\a}} \\
    \ls&\e^{-\frac{1}{2}}\nm{\gbg}_{L^2_{\iota_1\iota_2}L^2_{\eta}L^1_{v}}\Big(\tnm{\P}+\tnm{\bb}+\tnm{\bd}\Big)
    \ls\oot\jnm{c}^{\N}+\oot\xnm{\re}^{\N}.\no
\end{align}
For the third term in \eqref{final 56}, we first integrate by parts with respect to $\va$ and then apply the splitting in \eqref{pp 41}. A similar argument as \eqref{pp 96=} and \eqref{pp 96='} with Lemma \ref{lem:final 5} imply
\begin{align}
    \abs{\e^{-1}\bbr{M,c\left(\abs{v}^2-5T\right)\mh}}\ls&\oot\jnm{c}^{\N}+\oot\xnm{\re}^{\N}+\xnm{\re}^{2\N}+\oot.
\end{align}
Hence, we have shown that
\begin{align}
    &\Babs{\text{First Term in \eqref{pp 38}}}
   \ls\oot\jnm{c}^{\N}+\oot\e^{\N}\xnm{\re}^{\N}+\xnm{\re}^{2\N}+\oot.
\end{align}
Note that in the definition of $\gbg$ \eqref{final 03}, the cutoff is only in the normal direction, so $\gbg$ should satisfy the same tangential estimates as $\blgg$ in \eqref{qq 23} and \eqref{qq 23'}. From Lemma \ref{lem:final 5}, we have for $\hb=\gbg,\p_{\vb}\gbg,\p_{\vc}\gbg,\p_{\iota_1}\gbg,\p_{\iota_2}\gbg, \blgg$
\begin{align}
    \tnm{\hb}+\nm{\hb}_{L^2_{\iota_1\iota_2}L^2_{\eta}L^1_{v}}+\nm{\eta\hb}_{L^2_{\iota_1\iota_2}L^2_{\eta}L^1_{v}}\ls &\jnm{c}^{\N-1}.
\end{align}
Hence, applying Lemma \ref{lemma:k1} with $\hb=\gbg,\p_{\vb}\gbg,\p_{\vc}\gbg,\p_{\iota_1}\gbg,\p_{\iota_2}\gbg, \blgg$ and Lemma \ref{lem:final 5}, we have
\begin{align}
    &\Babs{\text{Second Term 
    to Sixth Term in \eqref{pp 38}}}\ls\oot\jnm{c}^{\N}+\oot\xnm{\re}^{\N}+\xnm{\re}^{2\N}+\oot.
\end{align}
In addition, noticing that $\lnm{\frac{\p\chi(\e\eta)}{\p\eta}}\ls \e$, applying Lemma \ref{lemma:k1} with $\hb=\e^{-1}\frac{\p\chi(\e\eta)}{\p\eta}\blgg$ and Lemma \ref{lem:final 5}, we have
\begin{align}
    &\Babs{\text{Seventh Term in \eqref{pp 38}}}\ls\oot\jnm{c}^{\N}+\oot\xnm{\re}^{\N}+\xnm{\re}^{2\N}+\oot.
\end{align}
Using the same idea as in the proof of Lemma \ref{ss3-estimate} and Lemma \ref{lem:final 5}, we obtain
\begin{align}
    \tnm{\chi(\e^{-1}\va)K\Big[\blgg\Big]}+\tnm{K\Big[\chi(\e^{-1}\va)\blgg\Big]}\ls&\e\jnm{c}^{\N-1},\\
    \nm{\chi(\e^{-1}\va)K\Big[\blgg\Big]}_{L^2_{\mn}L^1_v}+\nm{K\Big[\chi(\e^{-1}\va)\blgg\Big]}_{L^2_{\mn}L^1_v}\ls&\e^{\frac{3}{2}}\jnm{c}^{\N-1}.
\end{align}
Hence, using Lemma \ref{lemma:k1} with $\hb=\chi\left(\e^{-1}\va\right)K\left[\blgg\right],K\left[\chi\left(\e^{-1}\va\right)\blgg\right]$, we obtain that
\begin{align}
    &\Babs{\text{Eighth Term in \eqref{pp 38}}}\ls\oot\jnm{c}^{\N}+\oot\xnm{\re}^{\N}+\xnm{\re}^{2\N}+\oot.
\end{align}
Using Corollary \ref{m-estimate.}, Corollary \ref{m-estimate..} and Lemma \ref{lem:final 5}, we know 
\begin{align}
    &\Babs{\text{Ninth Term in \eqref{pp 38}}}\ls\oot\jnm{c}^{\N}+\oot\xnm{\re}^{\N}+\xnm{\re}^{2\N}+\oot.
\end{align}
\end{proof}

%%%%%%%%%%%%%%%%%%%%%%%%%%%%%%%%%%%%%%%%%%%%%%%%%%%%%%%%%%%%%%%%%%%%%%%%%%%%
\subsection{Estimates of Auxiliary Interior Solution $\gg$ Terms}
%%%%%%%%%%%%%%%%%%%%%%%%%%%%%%%%%%%%%%%%%%%%%%%%%%%%%%%%%%%%%%%%%%%%%%%%%%%%

\begin{lemma}
    Under the assumption \eqref{assumption:boundary}, we have
    \begin{align}\label{pp 44''}
    \babs{\text{\eqref{pp 45}}}\ls&\oot\jnm{c}^{\N}+\oot\xnm{\re}^{\N}.
    \end{align}
\end{lemma}

\begin{proof}
Oddness guarantees that in \eqref{pp 45}
\begin{align}
    &\e^{-1}\br{\mh v\cdot\nx\Big(\mhh\gg\Big),\re}=\e^{-1}\br{\mh v\cdot\nx\big(\bbq\cdot v\big),\re}\\
    =&\e^{-1}\br{\mh v\cdot\nx\Big(\bbq\cdot v\Big),\P\mh+\bb\cdot v\mh+c\left(\abs{v}^2-5T\right) \mh+\bd\cdot\a+(\ik-\bpk)[\re]}\no\\
    =&\e^{-1}\br{\mh v\cdot\nx\Big(\bbq\cdot v\Big),\P\mh+c\left(\abs{v}^2-5T\right) \mh+(\ik-\bpk)[\re]}.\no
\end{align}
Here, thanks to the oddness/parity of \eqref{pp 39} and the construction \eqref{pp 47} which yields the divergence-free $\nx\cdot\bbq=0$ for $\bbq:=(\bbq_{1},\bbq_{2},\bbq_{3})$, we observe the crucial cancellation of the $\P\mh+c\left(\abs{v}^2-5T\right) \mh$ contribution
\begin{align}\label{final 04}
    &\e^{-1}\br{\mh v\cdot\nx\Big(\bbq\cdot v\Big),\P\mh+c\left(\abs{v}^2-5T\right) \mh}\\
    =&\e^{-1}\br{\mh\Big(v_1^2\p_1\bbq_{1}+v_2^2\p_2\bbq_{2}+v_3^2\p_3\bbq_{3}\Big),\P\mh+c\left(\abs{v}^2-5T\right) \mh}
    =\frac{\e^{-1}}{3}\br{\mh \abs{v}^2\big(\nx\cdot\bbq\big),\P\mh+c\left(\abs{v}^2-5T\right) \mh}=0.\no
\end{align}
Also, using Lemma \ref{lem:final 7}, the $(\ik-\bpk)[\re]$ contribution can be controlled
\begin{align}
    \abs{\e^{-1}\br{\mh v\cdot\nx\big(\bbq\cdot v\big),(\ik-\bpk)[\re]}}
    \ls&\e^{-1}\nm{\bbq}_{H^1}\tnm{(\ik-\bpk)[\re]}
    \ls\oot\jnm{c}^{\N}+\oot\e^{-\N}\tnm{(\ik-\bpk)[\re]}^{\N}.
\end{align}
\end{proof}

%%%%%%%%%%%%%%%%%%%%%%%%%%%%%%%%%%%%%%%%%%%%%%%%%%%%%%%%%%%%%%%%%%%%%%%%%%%%
\subsection{Estimates of Interaction Terms}
%%%%%%%%%%%%%%%%%%%%%%%%%%%%%%%%%%%%%%%%%%%%%%%%%%%%%%%%%%%%%%%%%%%%%%%%%%%%

In this subsection, we consider the interaction terms in \eqref{pp 51'}.

\begin{lemma}
    Under the assumption \eqref{assumption:boundary} and $\xnm{\re}\ll 1$, we have
    \begin{align}\label{pp 99'''}
    \babs{\text{\eqref{pp 51'}}}
    \ls&\oot\jnm{c}^{\N}+\oot\xnm{\re}^{\N}+\xnm{\re}^{2\N}+\xnm{\re}^{3\N}+\oot.
    \end{align}
\end{lemma}

\begin{proof}
\ 
\paragraph{\underline{Boundary Terms Contribution}}
Naturally, using Lemma \ref{h-estimate}, the boundary term in \eqref{pp 51'} is controlled as
\begin{align}
    \abs{\e^{-1}\bbrb{\gb ,h}{\gamma_-}}\leq&\abs{\e^{-1}\bbrb{\gbg ,h}{\gamma_-}}+\abs{\e^{-1}\bbrb{\gg ,h}{\gamma_-}}
    \ls\e^{-1}\tnms{\fbf}{\gamma_-}\tnms{h}{\gamma_-}
    \ls \oot\jnm{c}^{\N}+\e^{-\N}\tnms{h}{\gamma_-}^{\N}\\
    \ls&\oot\jnm{c}^{\N}+\oot.\no
\end{align}
Also, using \eqref{final 61}, since $\gb_h$ is restricted to $\abs{\va}\ls\e$, we have
\begin{align}
    \abs{\e^{-1}\bbrb{\gb_h,(1-\pp)[\re]}{\gamma_+}}\ls& \e^{-1}\tnms{\gb_h}{\gamma_+}\tnms{(1-\pp)[\re]}{\gamma_+}\ls \oot\jnm{c}^{\N-1}\tnms{(1-\pp)[\re]}{\gamma_+}\\
    \ls&\oot\jnm{c}^{\N}+\oot\e^{\frac{\N}{2}}\xnm{\re}^{\N}+\e^{\frac{\N}{2}}\xnm{\re}^{2\N}+\oot\e^{\frac{\N}{2}}.\no
\end{align}
Using \eqref{final 62}, we know
\begin{align}
    \abs{\e^{-1}\gb_m\int_{\p\Omega}\mhh\pp[\re]}\ls \e^{-1}\abs{\gb_m}\tnms{\pp[\re]}{\gamma}\ls\oot\e\jnm{c}^{\N-1}\tnms{\pp[\re]}{\gamma}\ls \oot\jnm{c}^{\N}+\oot\e^{\N}\xnm{\re}^{\N}.
\end{align}

\paragraph{\underline{$\e^{-1}\bbr{\gb,\llc[\re]}$ Contribution}}
We can decompose
\begin{align}\label{pp 97}
    \e^{-1}\bbr{\gb,\llc[\re]}=&\e^{-1}\br{\gbg,\mhh Q^{\ast}\left[\mh f_1,\mh\re\right]}+\e^{-1}\br{\gg,\mhh Q^{\ast}\left[\mh f_1,\mh\re\right]}.
\end{align}
For the first term in \eqref{pp 97}, using $\lnm{\f_1}\ls\oot$, we can follow a similar argument as $\gbg$ estimate in Section \ref{sec:gb estimate} to split $\re$ and use Lemma \ref{lem:final 5} to bound
\begin{align}
    \abs{\e^{-1}\br{\gbg,\mhh Q^{\ast}\left[\mh f_1,\mh\re\right]}}
    \ls&\oot\jnm{c}^{\N}+\oot\xnm{\re}^{\N}+\xnm{\re}^{2\N}+\oot.
\end{align}
Since $\gg\in\nk$, for the second term in \eqref{pp 97}, from Lemma \ref{lem:final 7}, we deduce from the orthogonality of $Q$
\begin{align}
    \e^{-1}\br{\gg,\mhh Q^{\ast}\left[\mh f_1,\mh\re\right]}=0.
\end{align}
In total, we know 
\begin{align}
    \abs{\e^{-1}\bbr{\gb,\llc[\re]}}
    \ls&\oot\jnm{c}^{\N}+\oot\xnm{\re}^{\N}+\xnm{\re}^{2\N}+\oot.
\end{align}

\paragraph{\underline{$\e^{-1}\bbr{\gbg,\ss_0+\ss_5}$ Contribution}}
Using Lemma \ref{lem:final 5}, Lemma \ref{ss0-estimate} and Lemma \ref{ss5-estimate}, we have
\begin{align}
    \abs{\e^{-1}\bbr{\gbg,\ss_0+\ss_5}}\ls& \e^{-1}\tnm{\gbg}\Big(\e\um{\re}+\oot\e^{\frac{1}{2}}\Big)\ls\oot\e^{-\frac{1}{2}}\jnm{c}^{\N-1}\Big(\e\um{\re}+\oot\e^{\frac{1}{2}}\Big)\\
    \ls& \oot\jnm{c}^{\N}+\oot\e^{\frac{\N}{2}}\xnm{\re}^{\N}+\oot.\no
\end{align}

\paragraph{\underline{$\e^{-1}\bbr{\gbg,\ss_4+\sp}$ Contribution}}
Using Lemma \ref{lem:final 5}, Lemma \ref{ss4-estimate} and Lemma \ref{ssp-estimate}, we have
\begin{align}
    \abs{\e^{-1}\bbr{\gbg,\ss_4+\sp}}\ls&\e^{-1}\tnm{\gbg}\Big(\tnm{\ss_4}+\tnm{\sp}\Big)\ls\oot\e^{-\frac{1}{2}}\jnm{c}^{\N-1}\Big(\tnm{\ss_4}+\tnm{\sp}\Big)\\
    \ls& \oot\jnm{c}^{\N}+\oot\e^{\frac{\N}{2}}.\no
\end{align}

\paragraph{\underline{$\e^{-1}\bbr{\gbg,\ss_1}$ Contribution}}
Using $\lnm{\fb_1}\ls\oot$, we can follow a similar argument as $\gbg$ estimate in Section \ref{sec:gb estimate} to split $\re$ and use  Lemma \ref{lem:final 5} to bound
\begin{align}
    \abs{\e^{-1}\bbr{\gbg,\ss_1}}
    \ls&\oot\jnm{c}^{\N}+\oot\xnm{\re}^{\N}+\xnm{\re}^{2\N}+\oot.
\end{align}

\paragraph{\underline{$\e^{-1}\bbr{\gbg,\ss_3}$ Contribution}}
Using Lemma \ref{lem:final 5}, Lemma \ref{ss3-estimate}, we have
\begin{align}\label{final 36}
    \abs{\e^{-1}\br{\gbg,\ss_3}}\ls&\e^{-1}\nm{\gbg}_{L^{\infty}}\nm{\ss_3}_{L^1}
    \ls\oot\jnm{c}^{\N}+\oot\e^{-\N}\nm{\ss_3}_{L^1}^{\N}\ls\oot\jnm{c}^{\N}+\oot.
\end{align}

\paragraph{\underline{$\e^{-1}\bbr{\gbg,\ss_2}$ Contribution}}
We finally come to the most difficult term. Since $\gbg$ is odd in $\vb,\vc$, the estimate of $\abs{\e^{-1}\bbr{\gbg,\ss_2}}$ follows a similar argument as estimating \eqref{pp 103'}. We split
\begin{align}
    \e^{-1}\bbr{\gbg,\ss_2}=&\e^{-1}\br{\gbg,\Gamma\Big[\pk[\re]+\bd\cdot\a,\pk[\re]+\bd\cdot\a\Big]}\\
    &+\e^{-1}\br{\gbg,\Gamma\Big[\pk[\re]+\bd\cdot\a,(\ik-\bpk)[\re]\Big]}
    +\e^{-1}\br{\gbg,\Gamma\Big[(\ik-\bpk)[\re],(\ik-\bpk)[\re]\Big]},\no
\end{align}
and
\begin{align}
    &\e^{-1}\br{\gbg,\Gamma\Big[\pk[\re]+\bd\cdot\a,\pk[\re]+\bd\cdot\a\Big]}\\
    =&\e^{-1}\br{\gbg,\Gamma\Big[\P\mh+c\left(\abs{v}^2-5T\right)\mh,\P\mh+c\left(\abs{v}^2-5T\right)\mh\Big]}+\e^{-1}\br{\gbg,\Gamma\Big[\bb\cdot v\mh+\bd\cdot\a, \P\mh+c\left(\abs{v}^2-5T\right)\mh\Big]}\no\\
    &+\e^{-1}\br{\gbg,\Gamma\Big[\bb\cdot v\mh,\bb\cdot v\mh\Big]}+\e^{-1}\br{\gbg,\Gamma\Big[\bd\cdot \a,\bb\cdot v\mh\Big]}+\e^{-1}\br{\gbg,\Gamma\Big[\bd\cdot \a,\bd\cdot \a\Big]}.\no
\end{align}
Based on the oddness of $\gbg$ in \eqref{gb-oddness}, we have the crucial cancellation 
\begin{align}\label{final 06}
    \e^{-1}\br{\gbg,\Gamma\Big[\P\mh+c\left(\abs{v}^2-5T\right)\mh,\P\mh+c\left(\abs{v}^2-5T\right)\mh\Big]}=0.
\end{align}
Then using a similar argument as estimating \eqref{pp 103'} and Lemma \ref{lem:final 5}, we know 
\begin{align}
    &\abs{\e^{-1}\br{\gbg,\Gamma\Big[\bb\cdot v\mh+\bd\cdot\a, \P\mh\Big]}}+\abs{\e^{-1}\br{\gbg,\Gamma\Big[\bb\cdot v\mh,\bb\cdot v\mh\Big]}}\\
    &+\abs{\e^{-1}\br{\gbg,\Gamma\Big[\bd\cdot \a,\bb\cdot v\mh\Big]}}+\abs{\e^{-1}\br{\gbg,\Gamma\Big[\bd\cdot \a,\bd\cdot \a\Big]}}\no\\
    &+\abs{\e^{-1}\br{\gbg,\Gamma\Big[\pk[\re]+\bd\cdot\a,(\ik-\bpk)[\re]\Big]}}
    +\abs{\e^{-1}\br{\gbg,\Gamma\Big[(\ik-\bpk)[\re],(\ik-\bpk)[\re]\Big]}}\no\\
    \ls&\lnmm{\gbg}\Big(\e^{-\frac{1}{2}}\tnm{\P}+\e^{-\frac{1}{2}}\tnm{\bb}+\e^{-\frac{1}{2}}\tnm{\bd}+\e^{-\frac{1}{2}}\tnm{(\ik-\bpk)[\re]}\Big)^2\no\\
    \ls& \oot \jnm{c}^{\N}+\oot\xnm{\re}^{\N}+\xnm{\re}^{2\N}+\oot.\no
\end{align}
Then all terms in $\abs{\e^{-1}\bbr{\gbg,\ss_2}}$ have been bounded except
\begin{align}\e^{-12}\br{\gbg,\Gamma\Big[\bb\cdot v\mh,c\left(\abs{v}^2-5T\right)\mh\Big]}+\e^{-1}\br{\gbg,\Gamma\Big[\bd\cdot \a,c\left(\abs{v}^2-5T\right)\mh\Big]}.
\end{align}
Note that
\begin{align}\label{qq 01}
    &\e^{-1}\br{\gbg,\Gamma\Big[\bb\cdot v\mh,c\left(\abs{v}^2-5T\right)\mh\Big]}=\e^{-1}\brx{\brv{\gbg,\Gamma\Big[v\mh,\left(\abs{v}^2-5T\right)\mh\Big]},\bb c}\\
    =&-\e^{-1}\brx{\brv{\gbg,\Gamma\Big[v\mh,\left(\abs{v}^2-5T\right)\mh\Big]},\bb(\z^R-\e^{-1}\xi)}
    +\e^{-1}\brx{\brv{\gbg,\Gamma\Big[v\mh,\left(\abs{v}^2-5T\right)\mh\Big]},\bb\z^S}.\no
\end{align}
For the first term in \eqref{qq 01}, noting \eqref{qq 24} and $\tnm{\bb}\ls\e^{\frac{1}{2}}\xnm{\re}$ and using Hardy's inequality, \eqref{qq 06} and  Lemma \ref{lem:final 5}, we have
\begin{align}\label{qq 03}
    &\abs{\e^{-1}\brx{\brv{\gbg,\Gamma\Big[v\mh,\left(\abs{v}^2-5T\right)\mh\Big]},\bb(\z^R-\e^{-1}\xi)}}\\
    =&\e^{-1}\abs{\brx{\brv{\gbg,\Gamma\Big[v\mh,\left(\abs{v}^2-5T\right)\mh\Big]},\bb\int_0^{\mn}\p_{\mn}(\z^R-\e^{-1}\xi)}}\no\\
    =&\e^{-1}\abs{\brx{\brv{\mn\gbg,\Gamma\Big[v\mh,\left(\abs{v}^2-5T\right)\mh\Big]},\bb\frac{1}{\mn}\int_0^{\mn}\p_{\mn}(\z^R-\e^{-1}\xi)}}\no\\
    =&\abs{\brx{\brv{\eta\gbg,\Gamma\Big[v\mh,\left(\abs{v}^2-5T\right)\mh\Big]},\bb\frac{1}{\mn}\int_0^{\mn}\p_{\mn}(\z^R-\e^{-1}\xi)}}\no\\
    \ls&\lnm{\eta\gbg}\tnm{\bb}\tnm{\frac{1}{\mn}\int_0^{\mn}\p_{\mn}(\z^R-\e^{-1}\xi)}\ls\e^{\frac{1}{2}}\lnms{\fbf}{\gamma_-}\xnm{\re}\Big(\nm{\z^R}_{H^1}+\e^{-1}\nm{\xi}\Big)\no\\
    \ls&\oot\jnm{c}^{\N}+\xnm{\re}^{\N}\Big(\oot\e^{\frac{\N}{2}}\xnm{\re}^{\N}+\e^{\frac{\N}{2}}\xnm{\re}^{2\N}+\oot\e^{\frac{\N}{2}}\Big).\no
\end{align}
For the second term in \eqref{qq 01}, using  Lemma \ref{lem:final 5} and \eqref{qq 07}, we have
\begin{align}\label{qq 09}
    &\abs{\e^{-1}\brx{\brv{\gbg,\Gamma\Big[v\mh,\left(\abs{v}^2-5T\right)\mh\Big]},\bb \z^S}}\\
    \ls&\e^{-1}\lnm{\gbg}\tnm{\bb}\tnm{\z^S}\ls\e^{-\frac{1}{2}}\lnms{\fbf}{\gamma_-}\xnm{\re}\tnm{\z^S}
    \ls\oot\jnm{c}^{\N}+\xnm{\re}^{\N}\Big(\oot\xnm{\re}^{\N}+\xnm{\re}^{2\N}+\oot\Big).\no
\end{align}
Inserting \eqref{qq 03} and \eqref{qq 09} into \eqref{qq 01}, we arrive at
\begin{align}
    \abs{\e^{-1}\br{\gbg,\Gamma\Big[\bb\cdot v\mh,c\left(\abs{v}^2-5T\right)\mh\Big]}}
    \ls&\oot\jnm{c}^{\N}+\xnm{\re}^{\N}\Big(\oot\xnm{\re}^{\N}+\xnm{\re}^{2\N}+\oot\Big).
\end{align}
A similar argument justifies that
\begin{align}
    &\abs{\e^{-1}\br{\gbg,\Gamma\Big[\bd\cdot \a,c\left(\abs{v}^2-5T\right)\mh\Big]}}\ls
    \oot\jnm{c}^{\N}+\xnm{\re}^{\N}\Big(\oot\xnm{\re}^{\N}+\xnm{\re}^{2\N}+\oot\Big).
\end{align}
In summary, using Lemma \ref{remark 01}, we know
\begin{align}
    &\abs{\e^{-1}\bbr{\gbg,\ss_2}}\ls\oot \jnm{c}^{\N}+\oot\xnm{\re}^{\N}+\xnm{\re}^{2\N}+\xnm{\re}^{3\N}+\oot.
\end{align}

\paragraph{\underline{$\e^{-1}\bbr{\gg,\bar\ss}$ Contribution}}
Based on Lemma \ref{lem:final 7}, due to orthogonality, we have
\begin{align}
    \e^{-1}\bbr{\gg,\bar\ss}=\e^{-1}\bbr{\gg,\ss_3}+\e^{-1}\bbr{\gg,\ss_4}+\e^{-1}\bbr{\gg,\sp}.
\end{align}
Using Lemma \ref{lem:final 7}, Lemma \ref{ss4-estimate} and Lemma \ref{ssp-estimate}, we have
\begin{align}
    \abs{\e^{-1}\bbr{\gg,\ss_4+\sp}}\ls&\e^{-1}\tnm{\gg}\Big(\tnm{\ss_4}+\tnm{\sp}\Big)\ls\oot\e^{-1}\jnm{c}^{\N-1}\Big(\tnm{\ss_4}+\tnm{\sp}\Big)\\
    \ls& \oot\jnm{c}^{\N}+\e^{-\N}\tnm{\ss_4}^{\N}+\e^{-\N}\tnm{\sp}^{\N}\ls \oot \jnm{c}^{\N}+\oot.\no
\end{align}
Considering the rescaling $\eta=\e^{-1}\mn$, we have
\begin{align}
    \abs{\e^{-1}\bbr{\gg,\ss_3}}\ls&\e^{-1}\abs{\bbr{\gg(0)+\int_0^{\mn}\p_{\mn}\gg,\ss_3}}
    \ls\e^{-1}\abs{\bbr{\gg(0),\ss_3}}+\e^{-1}\abs{\bbr{\frac{1}{\mn}\int_0^{\mn}\p_{\mn}\gg(y)\ud y,\mn\ss_3}}\\
    \ls&\e^{-1}\tnm{\gg(0)}\nm{\ss_3}_{L^2_{\iota_i}L^1_{\mn}L^1_v}+\tnm{\frac{1}{\mn}\int_0^{\mn}\p_{\mn}\gg(y)\ud y}\nm{\eta\big(\sy+\sz+\fb_1\big)}_{L^2_xL^1_v}\no\\
    \ls&\e^{-1}\tnm{\gg(0)}\nm{\ss_3}_{L^2_{\iota_i}L^1_{\mn}L^1_v}+\nm{\gg}_{H^1}\nm{\eta\big(\sy+\sz+\fb_1\big)}_{L^2_xL^1_v}\no\\
    \ls& \oot \jnm{c}^{\N}+\e^{-\N}\nm{\ss_3}_{L^2_{\iota_i}L^1_{\mn}L^1_v}^{\N}+\nm{\eta\big(\sy+\sz\big)}_{L^2_xL^1_v}^{\N}+\nm{\eta\fb_1}_{L^2_xL^1_v}^{\N}\ls \oot \jnm{c}^{\N}+\oot.\no
\end{align}
We conclude that
\begin{align}
    \abs{\e^{-1}\bbr{\gg,\bar\ss}}
    \ls&\oot\jnm{c}^{\N}+\oot.
\end{align}
Collecting the above estimates, we arrive at \eqref{pp 99'''}.
\end{proof}

\begin{proof}[Proof of Proposition \ref{prop:c-bound}]
Inserting \eqref{pp 44'}\eqref{pp 44''}\eqref{pp 99'''} into \eqref{pp 29}, we have the desired result.
\end{proof}

%%%%%%%%%%%%%%%%%%%%%%%%%%%%%%%%%%%%%%%%%%%%%%%%%%%%%%%%%%%%%%%%%%%%%%%%
\section{\texorpdfstring{$L^{\infty}$}{} Estimates}\label{sec:linfty estimate}
%%%%%%%%%%%%%%%%%%%%%%%%%%%%%%%%%%%%%%%%%%%%%%%%%%%%%%%%%%%%%%%%%%%%%%%%

The goal of this section is to bootstrap the $L^2-L^6$
bounds into a $L^\infty$ bound. Following \cite{Guo.Jang.Jiang2010}, to circumvent the
well-known analytical difficulty of a cubic velocity growth from
non-homogeneous $\mu$, we re-formulate the basic linear problem \eqref{ll 01}
via a new weight from the global Maxwellian $\rem$ as follows.

Based on \cite{Guo.Jang.Jiang2010},  
we can decompose $K_M=K_M^m+K_M^c$, where
\begin{align}
    K_M^m[\rem]:=&\iint_{\r^3\times\s^2}q(\omega,\abs{\vuu-\vv})\chi_m(\abs{\vuu-\vv})\mu_M^{\frac{1}{2}}(\vuu)\m(\vv)\mu_M^{-\frac{1}{2}}(\vv)\rem(\vuu)\ud\omega\ud\vuu\\
    &-\iint_{\r^3\times\s^2}q(\omega,\abs{\vuu-\vv})\chi_m(\abs{\vuu-\vv})\mu_M^{\frac{1}{2}}(\vuu^{\ast})\m(\vv^{\ast})\mu_M^{-\frac{1}{2}}(\vv)\rem(\vv^{\ast})\ud\omega\ud\vuu\no\\
    &-\iint_{\r^3\times\s^2}q(\omega,\abs{\vuu-\vv})\chi_m(\abs{\vuu-\vv})\mu_M^{\frac{1}{2}}(\vv^{\ast})\m(\vuu^{\ast})\mu_M^{-\frac{1}{2}}(\vv)\rem(\vuu^{\ast})\ud\omega\ud\vuu,\no
\end{align}
and $K_M^c:=K_M-K_M^m$. Here $\chi_m=\chi(my)$ is a smooth cutoff function (see Section \ref{sec:boundary-form}) for any $m>0$. Based on \cite[Lemma 2.3]{Guo.Jang.Jiang2010}, when $m$ is sufficiently small, $K_M^m$ satisfies
\begin{align}\label{ll 04}
    \abs{K_M^m[\rem]}\ls \oo\nu_M\lnm{\rem},
\end{align}
and $K_M^c$ has kernel $l_M(\vuu,\vv)$
\begin{align}
    K_M^c[\rem]=\int_{\r^3}l_M(\vuu,\vv)\rem(\vuu)\ud\vuu.
\end{align}

Define a weight function scaled with parameters $0\leq\vrh<\frac{1}{2}$ and $\vth\geq0$,
\begin{align}\label{ktt 56}
\vh(\vv):=\bv.
\end{align}
Denote the weighted solution
\begin{align}
\remm(\vx,\vv):=&\vh(\vv) \rem(\vx,\vv),
\end{align}
and the weighted non-local operator
\begin{align}\label{final 66}
K^c_{M,\vh(\vv)}\left[\remm\right](\vv):=&\vh(\vv)K^c_M\left[\frac{\remm}{\vh}\right](\vv)=\int_{\r^3}l_{M,\vh(\vv)}(\vuu,\vv)\remm(\vuu)\ud{\vuu},
\end{align}
where 
\begin{align}
l_{M,\vh(\vv)}(\vuu,\vv):=l_M(\vuu,\vv)\frac{\vh(\vv)}{\vh(\vuu)}.
\end{align}
Multiplying $\e\vh(\vv)$ on both sides of \eqref{ll 01}, we have
\begin{align}\label{ktt 64}
\left\{
\begin{array}{l}
\e\vv\cdot\nx \remm+\nu_M \remm=K^c_{M,\vh}\left[\remm\right](\vx,\vv)+\vh(\vv)\widetilde S_M(\vx,\vv)\ \ \text{in}\ \ \Omega\times\r^3,\\\rule{0ex}{1.5em}
\remm(\vx_0,\vv)=\ds\vh(\vv)\mmss(\vx_0,\vv)\int_{\vuu\cdot\vn>0}\tvh(\vuu)\remm(\vx_0,\vuu)\ud\vuu+\vh(\vv) h_M(\vx_0,\vv)\ \ \text{for}\ \ \vx_0\in\p\Omega\ \
\text{and}\ \ \vv\cdot\vn<0,
\end{array}
\right.
\end{align}
where
\begin{align}
    \widetilde S_M(x,v)=\e S_M(x,v)+K_M^m[\rem].
\end{align}

Our main $L^\infty$ estimate is as follows:

\begin{proposition}\label{thm:infinity}
    Let $\re$ be a solution to \eqref{remainder}. Under the assumption \eqref{assumption:boundary}, we have
    \begin{align}\label{ss 06}
    \e^{\frac{1}{2}}\lnmm{\re}+\e^{\frac{1}{2}}\lnmms{\re}{\gamma}\ls&\oot\xnm{\re}+\xnm{\re}^2+\oot.
    \end{align}
\end{proposition}

Due to our choice of $\mm$, we have
for any $a<1$
\begin{align}\label{ll 02}
    \mm\ls\m\ls\mm^{a}.
\end{align}
From the first inequality in \eqref{ll 02}, we further know
\begin{align}\label{ll 03}
    \mmhh\mh\gs1.
\end{align}

The proof of Proposition \ref{thm:infinity} follows the
combination of two methods in \cite{Guo.Jang.Jiang2009} and \cite{Guo.Jang.Jiang2010}, as well as in \cite{Esposito.Guo.Kim.Marra2015} and
\cite{Guo2010}, provided the following key modification (see \cite{Guo2010}).

\begin{definition}[Exit Time and Position]\label{exit}
For any $(\vx,\vv)\in\overline\Omega\times\r^3$ with $\vv\neq\od$ and $(\vx,\vv)\notin\gamma_0$, define the backward exit time $t_b(\vx,\vv):=\inf\{t>0:\vx-\e t\vv\notin\Omega\}$. Also, define the backward exit position $\vx_b(\vx,\vv):=\vx-\e t_b(\vx,\vv)\vv\in\p\Omega$.
\end{definition}

\begin{definition}[Stochastic Cycle]
For any $(\vx,\vv)\in\overline\Omega\times\r^3$ with $\vv\neq\od$ and $(\vx,\vv)\notin\gamma_0$, let $(t_0,\vx_0,\vv_0)=(0,\vx,\vv)$. Define the first stochastic triple
\begin{align}
(t_1,\vx_1,\vv_1):=\Big(t_b(\vx_0,\vv_0),\vx_b(\vx_0,\vv_0),\vv_1\Big),
\end{align}
for some $\vv_1$ satisfying $\vv_1\cdot\vn(\vx_1)>0$.

Inductively, assume we know the $k^{th}$ stochastic triple $(t_k,\vx_k,\vv_k)$. Define the $(k+1)^{th}$ stochastic triple
\begin{align}
(t_{k+1},\vx_{k+1},\vv_{k+1}):=\Big(t_k+t_b(\vx_k,\vv_k),\vx_b(\vx_k,\vv_k),\vv_{k+1}\Big),
\end{align}
for some $\vv_{k+1}$ satisfying $\vv_{k+1}\cdot\vn(\vx_{k+1})>0$.
\end{definition}

\begin{definition}[Diffuse Reflection Integral]
Define $\nn_{k}=\{\vv\in \r^3:\vv\cdot\vn(\vx_{k})>0\}$, so the stochastic cycle must satisfy $\vv_k\in\nn_k$. Let the iterated integral for $k\geq2$ be defined as
\begin{align}
\int_{\prod_{j=1}^{k-1}\nn_j}\prod_{j=1}^{k-1}\ud{\sigma_j}:=\int_{\nn_1}\ldots\bigg(\int_{\nn_{k-1}}\ud{\sigma_{k-1}}\bigg)\ldots\ud{\sigma_1}
\end{align}
where $\ud{\sigma_j}:=\mm(\vv_j)\abs{\vv_j\cdot\vn(\vx_j)}\ud{\vv_j}$.
\end{definition}

The following is a key novel result our proof relies on (compared with that in \cite{Guo.Jang.Jiang2010}):
\begin{lemma}\label{ktt lemma 2}
For $\to>0$ sufficiently large, there exist constants $C_1,C_2>0$ independent of $\to$, such that for $k=C_1\to^{\frac{5}{4}}$, and
$(\vx,\vv)\in\times\overline\Omega\times\r^3$,
\begin{align}\label{ktt 59}
\int_{\Pi_{j=1}^{k-1}\nn_j}\id_{\{t_k(\vx,\vv,\vv_1,\ldots,\vv_{k-1})<\frac{\to}{\e}\}}\prod_{j=1}^{k-1}\ud{\sigma_j}\leq
\left(\frac{1}{2}\right)^{C_2\to^{\frac{5}{4}}}.
\end{align}
\end{lemma}
\begin{proof}
This is a rescaled version of \cite[Lemma 4.1]{Esposito.Guo.Kim.Marra2013}. Choosing $\zf>0$ sufficiently small, we further define the non-grazing set for $1\leq j\leq k-1$ as 
\begin{align}
    \nn_j^{\zf}=\Big\{v\in\nn_j: v\cdot n(x_j)>\zf\Big\}\bigcup \Big\{v\in\nn_j: \abs{v}<\frac{1}{\zf}\Big\}.
\end{align}
Then clearly, we have for some general $C>0$ independent of $j$, $\ds\int_{\nn_j\backslash\nn_j^{\zf}}\leq C\zf$.
From the stochastic cycle, we know for $v\in\nn_j^{\zf}$, if $t_k(\vx,\vv,\vv_1,\ldots,\vv_{k-1})<\frac{\to}{\e}$, then there can be at most $\left[\dfrac{C_{\Omega}\to}{\zf^3}\right]+1$ number of $v_j\in\nn_j^{\zf}$. Hence, we have
\begin{align}
    \int_{\Pi_{j=1}^{k-1}\nn_j}\id_{\{t_k(\vx,\vv,\vv_1,\ldots,\vv_{k-1})<\frac{\to}{\e}\}}\prod_{j=1}^{k-1}\ud{\sigma_j}
    \leq&\sum_{m=1}^{\left[\frac{C_{\Omega}\to}{\zf^3}\right]+1}\begin{pmatrix}k-1\\m\end{pmatrix}\abs{\sup_j\int_{\nn_j^{\zf}}\ud\sigma_j}^m\abs{\sup_j\int_{\nn_j\backslash\nn_j^{\zf}}\ud\sigma_j}^{k-m-1}.
\end{align}
Note that $\ud\sigma_j$ is not exactly a probability measure, but the $(2\pi T_M)^{-\frac{1}{2}}$ multiple of a probability measure. Hence, we have
\begin{align}
    \abs{\sup_j\int_{\nn_j\backslash\nn_j^{\zf}}\ud\sigma_j}^{k-m-1}
    \leq& \abs{\sup_j\int_{\nn_j\backslash\nn_j^{\zf}}\ud\sigma_j}^{k-\left[\frac{C_{\Omega}\to}{\zf^3}\right]-2}(2\pi T_M)^{-\frac{k}{2}-\frac{1}{2}\left[\frac{C_{\Omega}\to}{\zf^3}\right]-1}\\
    \leq&\big(C\zf\big)^{k-\left[\frac{C_{\Omega}\to}{\zf^3}\right]-2}(2\pi T_M)^{-\frac{k}{2}-\frac{1}{2}\left[\frac{C_{\Omega}\to}{\zf^3}\right]-1}.\no
\end{align}
Hence, considering that $\begin{pmatrix}k-1\\m\end{pmatrix}\leq\big(k-1\big)^m\leq\big(k-1\big)^{\left[\frac{C_{\Omega}\to}{\zf^3}\right]+1}$,
we have
\begin{align}
    \int_{\Pi_{j=1}^{k-1}\nn_j}\id_{\{t_k(\vx,\vv,\vv_1,\ldots,\vv_{k-1})<\frac{\to}{\e}\}}\prod_{j=1}^{k-1}\ud{\sigma_j}
    \leq&\left(\left[\frac{C_{\Omega}\to}{\zf^3}\right]+1\right)\big(k-1\big)^{\left[\frac{C_{\Omega}\to}{\zf^3}\right]+1}\big(C\zf\big)^{k-\left[\frac{C_{\Omega}\to}{\zf^3}\right]-2}(2\pi T_M)^{-\frac{k}{2}}.
\end{align}
Let $k-2=N\left(\left[\frac{C_{\Omega}\to}{\zf^3}\right]+1\right)$ for some $N$ to be determined later. Then if $\left[\frac{C_{\Omega}\to}{\zf^3}\right]\geq1$, we have for given $T_M$
\begin{align}
    \int_{\Pi_{j=1}^{k-1}\nn_j}\id_{\{t_k(\vx,\vv,\vv_1,\ldots,\vv_{k-1})<\frac{\to}{\e}\}}\prod_{j=1}^{k-1}\ud{\sigma_j}
    \leq&\left\{N\left(\left[\frac{C_{\Omega}\to}{\zf^3}\right]+1\right)\big(C\zf\big)^N(2\pi T_M)^{-\frac{N}{2}}\right\}^{\left[\frac{C_{\Omega}\to}{\zf^3}\right]+1}
    \leq \left\{C_{N,\Omega}\to\zf^{N-3}\right\}^{\left[\frac{C_{\Omega}\to}{\zf^3}\right]+1}.
\end{align}
We choose $\zf=\left(\dfrac{1}{2C_{N,\Omega}\to}\right)^{\frac{1}{N-3}}$, which yields $C_{N,\Omega}\to\zf^{N-3}=\dfrac{1}{2}$. Then we know for $\to$ sufficiently large $\left[\dfrac{C_{\Omega}\to}{\zf^3}\right]+1\approx C_{N,\Omega}\to^{1+\frac{3}{N-3}}$
and $\dfrac{C_{\Omega}\to}{\zf^3}\geq2$.

Finally, we choose $N=15$. Then our result follows.
\end{proof}

%%%%%%%%%%%%%%%%%%%%%%%%%%%%%%%%%%%%%%%%%%%%%%%%%%%%%%%%%%%%%%%%%%%%%%%%
\section{A Priori Remainder Estimates}
%%%%%%%%%%%%%%%%%%%%%%%%%%%%%%%%%%%%%%%%%%%%%%%%%%%%%%%%%%%%%%%%%%%%%%%%

\begin{theorem}\label{thm:apriori}
    Let $\re$ be a solution to \eqref{remainder}. Under the assumption \eqref{assumption:boundary}, \eqref{final 31}, with $\xnm{\re}\ll1$, for $h$ and $S$ defined in \eqref{aa 32} and \eqref{aa 31}, we have
    \begin{align}
        \xnm{\re}\ls\oot.
    \end{align}
\end{theorem}

\begin{proof}
Based on Lemma \ref{remark 01}, Proposition \ref{thm:energy} and Lemma \ref{final 11}, we have
\begin{align}\label{ss 01'}
    \e^{-1}\tnm{\be}+\e^{-1}\um{(\ik-\bpk)[\re]}+\pnm{\be}{6}+\pnm{(\ik-\bpk)[\re]}{6}+\e^{-\frac{1}{2}}\nm{\xi}_{H^2}&\\
    +\e^{-\frac{1}{2}}\tnms{(1-\pp)[\re]}{\gamma_+}+\pnms{\m^{\frac{1}{4}}(1-\pp)[\re]}{4}{\gamma_+}+\e^{-\frac{1}{2}}\tnms{\nx\xi}{\p\Omega}
    \ls&\oot\xnm{\re}+\xnm{\re}^2+\oot.\no
\end{align}
Collecting estimates from Proposition \ref{prop:p-bound}, Proposition \ref{prop:z-bound}, Proposition \ref{prop:b-bound} and Proposition \ref{prop:c-bound}, we have
\begin{align}\label{ss 02'}
    \e^{-1}\tnm{\P}+\pnm{\P}{6}+\e^{-\frac{1}{2}}\tnm{\bb}+\pnm{\bb}{6}+\tnm{\z}+\tnm{c}+\pnm{c}{6}
    \ls\oot\xnm{\re}+\xnm{\re}^{2}+\xnm{\re}^{3}+\oot.
\end{align}
Based on Proposition \ref{thm:infinity}, we have
\begin{align}\label{ss 06'}
\e^{\frac{1}{2}}\lnmm{\re}+\e^{\frac{1}{2}}\lnmms{\re}{\gamma}\ls&\oot\xnm{\re}+\xnm{\re}^2+\oot.
\end{align}
Based on Proposition \ref{lem:final 12}, using \eqref{ss 02'},  we have
\begin{align}\label{qq 32'}
    \tnms{\pp[\re]}{\gamma}
    \ls&\oot\xnm{\re}+\xnm{\re}^{2}+\xnm{\re}^{3}+\oot.
\end{align}
Combining \eqref{ss 01'}, \eqref{ss 02'}, \eqref{ss 06'} and \eqref{qq 32'}, we obtain that
\begin{align}\label{final 32}
    \xnm{\re}\ls \oot\xnm{\re}+\xnm{\re}^2+\xnm{\re}^{3}+\oot.
\end{align}
Hence, we have
\begin{align}
    \xnm{\re}\ls \xnm{\re}^2+\xnm{\re}^{3}+\oot.
\end{align}
Then our desired result follows from an iteration/fixed-point argument, for $\oot$ and $\re$ small.
\end{proof}

%%%%%%%%%%%%%%%%%%%%%%%%%%%%%%%%%%%%%%%%%%%%%%%%%%%%%%%%%%%%%%%%%%%%%%%%
\section{Well-Posedness of Linear Remainder Equation}\label{sec:linear remainder}
%%%%%%%%%%%%%%%%%%%%%%%%%%%%%%%%%%%%%%%%%%%%%%%%%%%%%%%%%%%%%%%%%%%%%%%%

In this section, we consider the well-posedness of the linear remainder equation \eqref{remainder} for given $\ss$ and $h$.

\begin{proposition}\label{prop:linear}
Assume $S$ and $h$ are given a priori satisfying \eqref{compatibility} and independent of $\re$. Under the assumption \eqref{assumption:boundary}, for fixed $\e>0$, there exists a unique solution $\re$ to the linear remainder equation \eqref{remainder} satisfying
    \begin{align}
    \xnm{\re}\ls_{\e} \xnm{\ss}+\lnmms{h}{\gamma_-},
\end{align}
and
\begin{align}
    \iint_{\Omega\times\r^3}\abs{v}^2\mh\re=3P\int_{\Omega}\P(\vx)\ud\vx=0.
\end{align}
\end{proposition}

Note that for fixed $\e>0$, $X$ norm is equivalent to $\widetilde X$ norm
\begin{align}
    \xnm{\re}\approx\xxnm{\re}:=\e^{-1}\nm{\rem}_{L^2_{\nu_M}}+\e^{-\frac{1}{2}}\tnms{\rem}{\gamma}+\e^{\frac{1}{2}}\lnmm{\rem}+\e^{\frac{1}{2}}\lnmms{\rem}{\gamma}.
\end{align}
Our goal is to justify the well-posedness of \eqref{remainder} in $\widetilde X$. \\

\paragraph{\underline{Reformulation with Global Maxwellian}}
Using the global Maxwellian \eqref{final 13}, we can rewrite \eqref{remainder} as
\begin{align}\label{global remainder}
\left\{
\begin{array}{l}
\vv\cdot\nx\widetilde{\rem}+\e^{-1}\mathfrak{A}\left[\widetilde{\rem}\right]+\e^{-1}\widetilde{\lc_M}\left[\widetilde{\rem}\right]=\widetilde{\ss_M}\ \ \text{in}\ \
\Omega\times\r^3,\\\rule{0ex}{1.0em} \widetilde{\rem}(\vx_0,\vv)=\widetilde{\mathcal{P}_M}\left[\widetilde{\rem}\right](\vx_0,\vv)+\widetilde{h_M}(\vx_0,\vv) \ \ \text{for}\ \ \vx_0\in\p\Omega\ \ \text{and}\ \ \vv\cdot\vn(\vx_0)<0,
\end{array}
\right.
\end{align}
where $\widetilde{\rem}:=\mmhh\mh\re$, $\widetilde{\ss_M}:=\mmhh\mh\ss$ and $\widetilde{h_M}:=\mmhh\mh h$
\begin{align}
    \widetilde{\lc_M}\left[\widetilde{\rem}\right]:=&-2\mmhh Q^{\ast}\left[\mm,\mmh \widetilde{\rem}\right]:=\widetilde{\nu_M}\widetilde{\rem}-\widetilde{K_M}\left[\widetilde{\rem}\right],\\
\label{oo 31}
    \mathfrak{A}\left[\widetilde{\rem}\right]:=&2\mmhh Q^{\ast}\left[\mm-\m,\mmh \widetilde{\rem}\right],\\
    \widetilde{\mathcal{P}_M}\left[\widetilde{\rem}\right](\vx_0,\vv):=&\mmss(\vx_0,\vv)\displaystyle\int_{\vuu\cdot\vn(\vx_0)>0}
    \mmh\widetilde{\rem}(\vx_0,\vuu)\abs{\vuu\cdot\vn(\vx_0)}\ud{\vuu},
\end{align}
for $\mmss:=\mmhh\mh(\vx_0,\vv)\mss(\vx_0,\vv)$.
Note that 
\begin{align}
    \iint_{\Omega\times\r^3}\abs{v}^2\mmh\widetilde{\rem}=\iint_{\Omega\times\r^3}\abs{v}^2\mh\re.
\end{align}
Notice that the notation here is slightly different that in \eqref{ll 01}. We will focus on the penalized version for $\theta>0$:
\begin{align}\label{linear remainder==}
\left\{
\begin{array}{l}
\theta\abs{v}^2\widetilde{\rem}+\vv\cdot\nx\widetilde{\rem}+\e^{-1}\mathfrak{A}\left[\widetilde{\rem}\right]+\e^{-1}\widetilde{\lc_M}\left[\widetilde{\rem}\right]=\widetilde{\ss_M}\ \ \text{in}\ \
\Omega\times\r^3,\\\rule{0ex}{1.0em} \widetilde{\rem}(\vx_0,\vv)=\widetilde{\mathcal{P}_M}\left[\widetilde{\rem}\right](\vx_0,\vv)+\widetilde{h_M}(\vx_0,\vv) \ \ \text{for}\ \ \vx_0\in\p\Omega\ \ \text{and}\ \ \vv\cdot\vn(\vx_0)<0,
\end{array}
\right.
\end{align}
and then take limit $\theta\rt0$. \\

\paragraph{\underline{Non-Homogeneous Boundary Data}}
Due to the presence of $\widetilde{h_M}\neq0$, the solution set for $\widetilde{\rem}$ is not a Banach space. We will first try to eliminate this $\widetilde{h_M}$ term. We introduce an auxiliary function $\widetilde{h}$ satisfying
\begin{align}\label{oo 32}
\left\{
\begin{array}{l}
\theta\abs{v}^2\widetilde{h}+\vv\cdot\nx\widetilde{h}=0\ \ \text{in}\ \
\Omega\times\r^3,\\\rule{0ex}{1.0em} \widetilde{h}(\vx_0,\vv)=\widetilde{\mathcal{P}_M}\left[\widetilde{h}\right](\vx_0,\vv)+\widetilde{h_M}(\vx_0,\vv) \ \ \text{for}\ \ \vx_0\in\p\Omega\ \ \text{and}\ \ \vv\cdot\vn(\vx_0)<0,
\end{array}
\right.
\end{align}
There is no non-local term in \eqref{oo 32}, and thus we can show that there exists a unique $\widetilde{h}\in\widetilde{X}$ satisfying
\begin{align}
    \xxnm{\widetilde{h}}\ls\lnmms{\widetilde{h_M}}{\gamma_-}\ls_{\e}\oot.
\end{align}
and (using Lemma \ref{lemma:final 2})
\begin{align}
    \iint_{\Omega\times\r^3}\abs{v}^2\mmh\widetilde{h}=\int_{\gamma_-}\mmh\widetilde{h_M}\ud\gamma=\int_{\gamma_-}\mh h\ud\gamma=0.
\end{align}
Now under the substitution $\widetilde\rem\rt\rem:=\widetilde\rem-\widetilde h$, \eqref{linear remainder==} becomes 
\begin{align}\label{linear remainder.}
\left\{
\begin{array}{l}
\theta\abs{v}^2\rem+\vv\cdot\nx{\rem}+\e^{-1}\mathfrak{A}\left[{\rem}\right]+\e^{-1}\widetilde{\lc_M}\left[{\rem}\right]=\ss_M:=\widetilde{\ss_M}+\e^{-1}\mathfrak{A}\left[\widetilde h\right]+\e^{-1}\widetilde{\lc_M}\left[\widetilde h\right]\ \ \text{in}\ \
\Omega\times\r^3,\\\rule{0ex}{1.0em} \rem(\vx_0,\vv)=\widetilde{\mathcal{P}_M}[\rem](\vx_0,\vv) \ \ \text{for}\ \ \vx_0\in\p\Omega\ \ \text{and}\ \ \vv\cdot\vn(\vx_0)<0.
\end{array}
\right.
\end{align}
In particular, the boundary condition becomes homogeneous.
Hence, in order to study the well-posedness of \eqref{remainder} for $\re$ in $\widetilde X$, it suffices to consider \eqref{linear remainder.} for $\rem$ in $\widetilde X$.\\

\paragraph{\underline{Setup}}
For \eqref{linear remainder.}, using Lemma \ref{lemma:final 2},
we can directly verify that
\begin{align}
    \iint_{\Omega\times\r^3}\mmh\ss_M=\iint_{\Omega\times\r^3}\mmh\widetilde{\ss_M}=\iint_{\Omega\times\r^3}\mh\ss=0,
\end{align}
and thus direct integration in \eqref{linear remainder.} yields
\begin{align}
    \iint_{\Omega\times\r^3}\abs{v}^2\mmh\rem=\iint_{\Omega\times\r^3}\abs{v}^2\mh\re=3P\int_{\Omega}\P=0.
\end{align}
Note that $\widetilde{\mathcal{P}_M}$ is not the natural diffuse-reflection boundary condition associated with $\mm$, and thus we do not have the desired energy structure. Hence, we need to split
\begin{align}
    \widetilde{\mathcal{P}_M}[\rem]=\mathcal{P}_M[\rem]+\mathfrak{B}[\rem],
\end{align}
where
\begin{align}
    \mathcal{P}_M[\rem]:=&P^{-1}\big(2\pi\tm\big)^{\frac{1}{2}}\mmh\int_{\vuu\cdot\vn(\vx_0)>0}
    \mmh{\rem}(\vx_0,\vuu)\abs{\vuu\cdot\vn(\vx_0)}\ud{\vuu},\\
   \mathfrak{B}[\rem]:=& \left(\mmhh\mh(\vx_0,\vv)\mss(\vx_0,\vv)-P^{-1}\big(2\pi\tm\big)^{\frac{1}{2}}\mmh\right)\int_{\vuu\cdot\vn(\vx_0)>0}
    \mmh{\rem}(\vx_0,\vuu)\abs{\vuu\cdot\vn(\vx_0)}\ud{\vuu}.
\end{align}
For the convenience of energy estimates, we further rewrite \eqref{linear remainder.} as
\begin{align}\label{linear remainder'}
\left\{
\begin{array}{l}
\theta\abs{v}^2\rem+v\cdot\nx\rem+\e^{-1}\mathfrak{A}\left[\rem\right]+\e^{-1}\widetilde{\lc_M}[\rem]+\e^{-1}\pk_M[\rem]=\ss_M+\e^{-1}\pk_M[\rem]\ \ \text{in}\ \
\Omega\times\r^3,\\\rule{0ex}{1.0em} \rem(\vx_0,\vv)=\mathcal{P}_M[\rem](\vx_0,\vv)+\mathfrak{B}[\rem](\vx_0,\vv) \ \ \text{for}\ \ \vx_0\in\p\Omega\ \ \text{and}\ \ \vv\cdot\vn(\vx_0)<0.
\end{array}
\right.
\end{align}
Here, $\pk_M$ and $\ik-\pk_M$ denote the projections onto $\nk_M$ -- the nullspace  of $\lc_M$, and its orthogonal complement $\nnk_M$.

Denote the linear operator
\begin{align}
    \lc_{\e}[\rem]:=\theta\abs{v}^2\rem+v\cdot\nx\rem+\e^{-1}\mathfrak{A}\left[\rem\right]+\e^{-1}\widetilde{\lc_M}[\rem]+\e^{-1}\pk_M[\rem].
\end{align}

\begin{lemma}
    $\lc_{\e}^{-1}$ is a bounded operator $\widetilde X\rt\widetilde X$ with $\nm{\lc_{\e}^{-1}}_{\widetilde X\rt\widetilde X}\ls\e$.
\end{lemma}
\begin{proof}
Consider $\lc_{\e}^{-1}[f]=g$, which is equivalent to 
\begin{align}\label{oo 54}
\left\{
\begin{array}{l}
\theta\abs{v}^2 g+v\cdot\nx g+\e^{-1}\mathfrak{A}\left[g\right]+\e^{-1}\widetilde{\lc_M}[g]+\e^{-1}\pk_M[g]=f\ \ \text{in}\ \
\Omega\times\r^3,\\\rule{0ex}{1.5em} g(\vx_0,\vv)=\mathcal{P}_M[g](\vx_0,\vv)+\mathfrak{B}[g](\vx_0,\vv) \ \ \text{for}\ \ \vx_0\in\p\Omega\ \ \text{and}\ \ \vv\cdot\vn(\vx_0)<0.
\end{array}
\right.
\end{align}
We design an iteration with $g^0=0$ and
\begin{align}\label{oo 61}
\left\{
\begin{array}{l}
\theta\abs{v}^2 g^{n+1}+v\cdot\nx g^{n+1}+\e^{-1}\widetilde{\lc_M}[g^{n+1}]+\e^{-1}\pk_M[g^{n+1}]=f-\e^{-1}\mathfrak{A}\left[g^{n}\right]\ \ \text{in}\ \
\Omega\times\r^3,\\\rule{0ex}{1.5em} g^{n+1}(\vx_0,\vv)=\mathcal{P}_M[g^{n+1}](\vx_0,\vv)+\mathfrak{B}[g^{n}](\vx_0,\vv) \ \ \text{for}\ \ \vx_0\in\p\Omega\ \ \text{and}\ \ \vv\cdot\vn(\vx_0)<0.
\end{array}
\right.
\end{align}
We can show that $g^{n}$ is well-posed in $\widetilde X$. In the following, we will focus on showing that $g^n\rt g$ in $\widetilde X$ by contraction mapping theorem.\\

\paragraph{\underline{$L^2$-Estimate}}
Note that
\begin{align}
    \lnmms{\mmhh\mh(\vx_0,\vv)\mss(\vx_0,\vv)-P^{-1}\big(2\pi\tm\big)^{\frac{1}{2}}\mmh}{\gamma_-}\ls&\oot,\\
    \int_{\gamma_-}\mmh\Big(\mmhh\mh(\vx_0,\vv)\mss(\vx_0,\vv)-P^{-1}\big(2\pi\tm\big)^{\frac{1}{2}}\mmh\Big)\ud\gamma=&0.
\end{align}
Also, since $\abs{\tm-\tq}\ls\oot$, using Lemma \ref{m-estimate}, we have
\begin{align}
    \abs{\mm-\m}\ls \oot\left(\abs{\vv}^2\ue^{\oot\abs{\vv}^2}\right)\mm.
\end{align}
Standard energy estimate in \eqref{oo 61} implies
\begin{align}\label{final 67}
    \e^{-1}\tnms{(1-\mathcal{P}_M)[g^{n+1}]}{\gamma_+}^2+\e^{-2}\nm{g^{n+1}}_{L^2_{\nu_M}}^2\ls\oot\e^{-1}\tnms{\mathcal{P}_M[g^{n}]}{\gamma_+}^2+\oot\e^{-2}\um{g^{n}}^2+\tnm{f}^2.
\end{align}
Next, we will try to bound $\tnms{\mathcal{P}_M[g^{n}]}{\gamma_+}$. Consider the cutoff function $\widetilde{\chi}(x,v):=\chi_0\big(\d^{-1}(v\cdot n)\big)\chi(\d^{-1}\mn)$
for some $0<\d\ll1$, where $\chi$ are defined as in Section \ref{sec:boundary-form}, and $\chi_0\in C^{\infty}$ satisfying $\chi_0(y)=1$ if $y\geq1$ and $\chi_0(y)=0$ if $y\leq0$. 

Multiplying $\widetilde{\chi}$ on both sides of \eqref{oo 61} yields
\begin{align}\label{oo 55}
\left\{
\begin{array}{l}
\theta\abs{v}^2 \widetilde{\chi} g^{n+1}+v\cdot\nx \big(\widetilde{\chi} g^{n+1}\big)+\e^{-1}\widetilde{\chi}\widetilde{\lc_M}[g^{n+1}]+\e^{-1}\widetilde{\chi}\pk_M[g^{n+1}]=f-\e^{-1}\widetilde{\chi}\mathfrak{A}\left[g^{n}\right]+v\cdot\nx\widetilde{\chi} g^{n+1}\ \ \text{in}\ \
\Omega\times\r^3,\\\rule{0ex}{1.5em} \widetilde{\chi} g^{n+1}(\vx_0,\vv)=\widetilde{\chi}\mathcal{P}_M[g^{n+1}](\vx_0,\vv)+\widetilde{\chi}\mathfrak{B}[g^{n}](\vx_0,\vv) \ \ \text{for}\ \ \vx_0\in\p\Omega\ \ \text{and}\ \ \vv\cdot\vn(\vx_0)<0.
\end{array}
\right.
\end{align}
Notice that
\begin{align}
    \widetilde{\chi}\widetilde{\lc_M}[g^{n+1}]=\widetilde{\lc_M}[\widetilde{\chi} g^{n+1}]-(1-\widetilde{\chi})\widetilde{K_M}[g^{n+1}]+\widetilde{K_M}[(1-\widetilde{\chi})g^{n+1}]
\end{align}
and
\begin{align}
   \widetilde{\chi}\pk_M[g^{n+1}]= \pk_M[\widetilde{\chi} g^{n+1}]-(1-\widetilde{\chi})\pk_M[g^{n+1}]+\pk_M[(1-\widetilde{\chi})g^{n+1}].
\end{align}
Hence, \eqref{oo 55} is equivalent to
\begin{align}\label{oo 56}
\left\{
\begin{array}{l}
\theta\abs{v}^2 \widetilde{\chi} g^{n+1}+v\cdot\nx \big(\widetilde{\chi} g^{n+1}\big)+\e^{-1}\widetilde{\lc_M}[\widetilde{\chi} g^{n+1}]+\e^{-1}\pk_M[\widetilde{\chi} g^{n+1}]=f+\sss\ \ \text{in}\ \
\Omega\times\r^3,\\\rule{0ex}{1.5em} \widetilde{\chi} g^{n+1}(\vx_0,\vv)=\widetilde{\chi}\mathcal{P}_M[g^{n+1}](\vx_0,\vv)+\widetilde{\chi}\mathfrak{B}[g^{n}](\vx_0,\vv) \ \ \text{for}\ \ \vx_0\in\p\Omega\ \ \text{and}\ \ \vv\cdot\vn(\vx_0)<0.
\end{array}
\right.
\end{align}
where
\begin{align}
    \sss:=&-\e^{-1}\widetilde{\chi}\mathfrak{A}\left[g^{n}\right]+v\cdot\nx\widetilde{\chi} g^{n+1}+\e^{-1}(1-\widetilde{\chi})\widetilde{K_M}[g^{n+1}]-\e^{-1}\widetilde{K_M}[(1-\widetilde{\chi})g^{n+1}]\\
    &+\e^{-1}(1-\widetilde{\chi})\pk_M[g^{n+1}]-\e^{-1}\pk_M[(1-\widetilde{\chi})g^{n+1}].\no
\end{align}
Since the overlapping region of $\widetilde{\chi}$ and $1-\widetilde{\chi}$ is very small, and $\nx\widetilde{\chi}\ls\d^{-1}$ in a small region, we know when $\e\ll\d$
\begin{align}
    \br{\sss,\e^{-1}\widetilde{\chi} g^{n+1}}\ls \big(\e^{-2}\d+\e^{-1}\big)\tnm{g^{n+1}}\ls \e^{-2}\d\tnm{g^{n+1}}.
\end{align}
Then standard energy estimate in \eqref{oo 56} yields (noticing that there is no $\gamma_-$ contribution due to the cutoff)
\begin{align}
    \e^{-1}\tnms{\widetilde{\chi} g^{n+1}}{\gamma_+}^2+\e^{-2}\tnm{\widetilde{\chi} g^{n+1}}^2\ls \tnm{f}^2+\e^{-2}\d\tnm{g^{n+1}}^2.
\end{align}
Then noticing that 
\begin{align}
    \mathcal{P}_M[g^{n+1}]\simeq&\int_{v'\cdot n>0}g^{n+1}(v')\mmh(v')(v'\cdot n)\ud v'\\
    =&\int_{v'\cdot n>2\d}g^{n+1}(v')\mmh(v')(v'\cdot n)\ud v'+\int_{0<v'\cdot n<2\d}g^{n+1}(v')\mmh(v')(v'\cdot n)\ud v',\no
\end{align}
we have
\begin{align}\label{final 68}
    \tnms{\mathcal{P}_M[g^{n+1}]}{\gamma}^2\ls&\tnms{\int_{v'\cdot n>\d}g^{n+1}(v')\mmh(v')(v'\cdot n)\ud v'}{\gamma_+}^2+\tnms{\int_{0<v'\cdot n<\d}g^{n+1}(v')\mmh(v')(v'\cdot n)\ud v'}{\gamma_+}^2\\
    \ls&\tnms{\widetilde{\chi} g^{n+1}}{\gamma_+}^2+\d\tnms{g^{n+1}}{\gamma}^2\ls\d\Big(\e^{-1}\tnm{g^{n+1}}^2+\tnms{g^{n+1}}{\gamma}^2\Big)+\e\tnm{f}^2.\no
\end{align}
Hence, combining \eqref{final 67} and \eqref{final 68}, we know
\begin{align}
    \e^{-1}\tnms{\mathcal{P}_M[g^{n+1}]}{\gamma}^2+\e^{-1}\tnms{(1-\mathcal{P}_M)[g^{n+1}]}{\gamma_+}^2+\e^{-2}\nm{g^{n+1}}_{L^2_{\nu_M}}^2\ls\oot\e^{-1}\tnms{\mathcal{P}_M[g^{n}]}{\gamma}^2+\oot\e^{-2}\nm{g^{n}}_{L^2_{\nu_M}}^2+\tnm{f}^2.
\end{align}
Then for fixed $\e$, when $\d$ and $\oot$ are small, by contraction mapping theorem, we know the iteration \eqref{oo 61} is convergent $g^n\rt g$ under the norm
\begin{align}
    \e^{-1}\tnms{\mathcal{P}_M[\ \cdot\ ]}{\gamma}^2+\e^{-1}\tnms{(1-\mathcal{P}_M)[\ \cdot\ ]}{\gamma_+}^2+\e^{-2}\nm{\ \cdot\ }_{L^2_{\nu_M}}^2.
\end{align}
Also, we have the uniform bound (for all $g^n$ and $g$)
\begin{align}
    \e^{-1}\tnms{\mathcal{P}_M[g]}{\gamma}^2+\e^{-1}\tnms{(1-\mathcal{P}_M)[g]}{\gamma_+}^2+\e^{-2}\nm{g}_{L^2_{\nu_M}}^2\ls\tnm{f}^2,
\end{align}
which yields
\begin{align}\label{final 71}
    \e^{-\frac{1}{2}}\tnms{g}{\gamma}+\e^{-1}\nm{g}_{L^2_{\nu_M}}\ls\tnm{f}.
\end{align}

\paragraph{\underline{$L^{\infty}$-Estimate}}
Noticing that $\pk_M$ is also a non-local operator similar to $\widetilde{K_M}$, we can follow the standard $L^{\infty}$ estimate in the proof of Proposition \ref{thm:infinity} 
to bound
\begin{align}\label{final 72}
    \lnmm{g}+\lnmms{g}{\gamma}\ls\e^{-\frac{3}{2}}\tnm{g}+\e\lnmm{\nu_M^{-1}f}\ls\e^{-\frac{1}{2}}\tnm{f}+\e\lnmm{\nu_M^{-1}f}.
\end{align}

\paragraph{\underline{Synthesis}}
Summarizing \eqref{final 71} and \eqref{final 72}, we know
\begin{align}\label{final 73}
    \xxnm{g}\ls\tnm{f}+\e^{\frac{3}{2}}\lnmm{f}\ls\e\xxnm{f}.
\end{align}
Hence, $\lc_{\e}^{-1}$ is a bounded operator $\widetilde X\rt\widetilde X$.
\end{proof}

Now we consider the solvability of \eqref{linear remainder.} or equivalently \eqref{linear remainder'}. Note that $\e^{-1}\lc_{\e}^{-1}\pk_M[\rem]$ is $O(1)$ from \eqref{final 73}. Hence, to prove the existence, we need to apply the Fredholm alternative instead of the contraction mapping principle. 

\begin{lemma}
    There exists a unique solution $\rem\in\widetilde{X}$ to \eqref{linear remainder'} satisfying
    \begin{align}
        \xxnm{\rem}\ls\xxnm{\ss_M}.
    \end{align}
\end{lemma}
\begin{proof}
Apply $\lc_{\e}^{-1}$ on both sides of \eqref{linear remainder'}, we obtain
\begin{align}\label{temp 13}
    \rem=\lc_{\e}^{-1}[\ss_M]+\e^{-1}\lc_{\e}^{-1}\pk_M[\rem].
\end{align}
To solve \eqref{temp 13}, it suffices to solve
\begin{align}\label{temp 14}
    \pk_M[\rem]=\pk_M\lc_{\e}^{-1}[\ss_M]+\e^{-1}\pk_M\lc_{\e}^{-1}\pk_M[\rem],
\end{align}
and 
\begin{align}\label{temp 14'}
    (\ik-\pk_M)[\rem]=(\ik-\pk_M)\lc_{\e}^{-1}[\ss_M]+\e^{-1}(\ik-\pk_M)\lc_{\e}^{-1}\pk_M[\rem].
\end{align}

\paragraph{\underline{Solvability of $\pk_M[\rem]$ and $(\ik-\pk_M)[\rem]$}}
Rewrite \eqref{temp 14} as
\begin{align}\label{oo 09}
    \Big(I-\e^{-1}\pk_M\lc_{\e}^{-1}\Big)\pk_M[\rem]=\pk_M\lc_{\e}^{-1}[\ss_M].
\end{align}
Direct estimate using \eqref{final 73} reveals that
\begin{align}
    \xxnm{\e^{-1}\pk_M\lc_{\e}^{-1}[\ss_M]}&\ls\xxnm{\ss_M}.
\end{align}
Hence, it suffices to show the compactness of operator $\pk_M\lc_{\e}^{-1}: \widetilde X\cap\nk_M\rt\widetilde X\cap\nk_M$.

Our argument relies on the application of the Fredholm alternative. Considering that the remainder estimates in Theorem \ref{thm:apriori} implies the injection, as long as the compactness is guaranteed, we know the operator 
\begin{align}
    I-\e^{-1}\pk_M\lc_{\e}^{-1}: \widetilde X\cap\nk_M\rt \widetilde X\cap\nk_M
\end{align}
has a bounded inverse, and thus
there exists a unique solution $\pk_M[\re]\in\widetilde X\cap\nk_M$ to \eqref{oo 09} or equivalently \eqref{temp 14} satisfying 
\begin{align}
    \xxnm{\pk_M[\rem]}\ls\xxnm{\ss_M}.
\end{align}
Also, from invertibility, there exists a unique solution $(\ik-\pk_M)[\rem]\in\widetilde X\cap\nnk_M$ to \eqref{temp 14'} satisfying
\begin{align}
    \xxnm{(\ik-\pk_M)[\rem]}\ls\xxnm{\ss_M}.
\end{align}
In the following, we will focus on proving the compactness.\\

\paragraph{\underline{Mild Formulation}}
Assume $\lc_{\e}^{-1}[g]=f$ for $g\in\widetilde X\cap\nk_M$. Notice that $g(x,v)$ is a linear combination of $\mathcal{N}^{\ast}:=\Big\{1,\vv,\abs{\vv}^2-5T\Big\}$, and thus due to linearity of $\lc_{\e}^{-1}$, it suffices to consider the case $g(x,v)= r(x)q(v)$ for some $r(x)$ and $q(v)\in\mathcal{N}^{\ast}$, which means
\begin{align}\label{oo 14}
    \theta\abs{v}^2 f+v\cdot\nx f+\e^{-1}\mathfrak{A}[f]+\e^{-1}\widetilde{\lc_M}[f]+\e^{-1}\pk_M[f]=g=r(x)q(v).
\end{align}
Notice that both $\mathfrak{A}$ and $\pk_M$ are non-local operators. We abuse the notation and let
\begin{align}
    \lc_M:=\widetilde{\lc_M}+\mathfrak{A}+\e^{-1}\pk_M:=\nu_M-K_M.
\end{align}
For the non-local integral operator $K_M$, denote its kernel as $k_M$. Denote the weighted solution
\begin{align}
f_w(\vx,\vv):=&\vh(\vv) f(\vx,\vv),
\end{align}
for $\vh$ defined in \eqref{ktt 56} and the weighted non-local operator
\begin{align}
K_{M,w}[f_w](\vv):=&\vh(\vv)K_M\left[\frac{f_w}{\vh}\right](\vv)=\int_{\r^3}k_{M,w}(\vuu,\vv)f_w(\vuu)\ud{\vuu},
\end{align}
where $k_{M,w}(\vuu,\vv):=k_M(\vuu,\vv)\frac{\vh(\vv)}{\vh(\vuu)}$.
Multiplying $\e\vh$ on both sides of \eqref{oo 14}, we have
\begin{align}\label{oo 24}
\e\vv\cdot\nx f_w+\nu_M f_w=K_{M,w}[f_w]+\e\vh r q.
\end{align}
Similar to the derivation in Section \ref{sec:linfty estimate}, by Duhamel's principle, we can track along the characteristics in \eqref{oo 24} to rewrite $f_w$ in terms of $K_{M,w}[f_w]$ and $wr\tilde q$
\begin{align}\label{oo 17}
    f_w(\vx,\vv)=& \e\int_{0}^{t_{1}} w(v)r\Big(\vx-\e(t_1-s)\vv\Big)q\big(\vv\big)\ue^{-\int_s^{t_1}\nu_M(\vv)
    }\ud{s}\\
    &+\int_{0}^{t_{1}}\int_{\r^3}k_{M,w}(v,v')f_w\Big(\vx-\e(t_1-s)\vv,\vuu\Big)\ue^{-\int_s^{t_1}\nu_M(\vv)
    }\ud v'\ud{s}+\frac{\ue^{-\int_0^{t_1}\nu_M(\vv)
    }}{\tvhh(\vv)}\int_{\nn_1}f_w(\vx_1,\vv_1)\tvh(\vv_1)\ud{\sigma_1}.\no
\end{align}

\paragraph{\underline{Reflection Term Estimates}}
Similar to the proof of Proposition \ref{thm:infinity}, we can deduce that the boundary reflection term (i.e. the third term on the RHS of \eqref{oo 17})
\begin{align}\label{oo 25}
    \frac{\ue^{-\int_0^{t_1}\nu_M(\vv)
    }}{\tvhh(\vv)}\int_{\nn_1}f_w(\vx_1,\vv_1)\tvh(\vv_1)\ud{\sigma_1}
    =&\frac{\ue^{-\int_0^{t_1}\nu_M(\vv)}}{\tvhh(\vv)}\sum_{\ell=1}^{k-1}\int_{\prod_{j=1}^{\ell}\nn_j}\Big(G_{\ell}[\vx,\vv]+H_{\ell}[\vx,\vv]\Big)\tvh(\vv_\ell)
    \bigg(\prod_{j=1}^{\ell}\ue^{-\int_{t_j}^{t_{j+1}}\nu_M(\vv_j)}\ud{\sigma_j}\bigg)\\
    &+\frac{\ue^{-\int_0^{t_1}\nu_M(\vv)}}{\tvhh(\vv)}\int_{\prod_{j=1}^{k}\nn_j}f_w(\vx_k,\vv_k)\tvh(\vv_{k})
    \bigg(\prod_{j=1}^{k}\ue^{-\int_{t_j}^{t_{j+1}}\nu_M(\vv_{j})}\ud{\sigma_j}\bigg),\no
\end{align}
where
\begin{align}
G_{\ell}[\vx,\vv]:=&\e\int_{t_{\ell}}^{t_{\ell+1}}
\vh(\vv_\ell)r\Big(\vx_{\ell}-\e(t_{\ell+1}-s)\vv_{\ell}\Big)\tilde q\Big(\vx_{\ell}-\e(t_{\ell+1}-s)\vv_{\ell},\vv_{\ell}\Big)\ue^{\int_0^{s}\nu_M(\vv_\ell)}\ud{s}\\
H_{\ell}[\vx,\vv]:=&\int_{t_{\ell}}^{t_{\ell+1}}\int_{r^3}k_{M,w}(\vv_{\ell},\vv_{\ell}')f_w\Big(\vx_{\ell}-\e(t_{\ell+1}-s)\vv_{\ell},\vv_{\ell}'\Big)\ue^{\int_0^{s}\nu_M(\vv_{\ell})}\ud \vv_{\ell}'\ud{s}.
\end{align}
Based on the similar proof of Proposition \ref{thm:infinity}, we know for sufficiently large $\ell$
\begin{align}
    \lnm{\frac{\ue^{-\int_0^{t_1}\nu_M(\vv)}}{\tvhh(\vv)}\int_{\prod_{j=1}^{k}\nn_j}f_w(\vx_k,\vv_k)\tvh(\vv_{k})
    \bigg(\prod_{j=1}^{k}\ue^{-\int_{t_j}^{t_{j+1}}\nu_M(\vv_{j})}\ud{\sigma_j}\bigg)}\ls\d\lnm{f_w}.
\end{align}
Also, note that $G_{\ell}$ and $H_{\ell}$ have the similar structure as the first and second terms on the RHS of \eqref{oo 17}, and thus can be handled in a similar fashion. Hence, we just need to consider the bound of the first and second terms on the RHS of \eqref{oo 17}.\\

\paragraph{\underline{Non-Local Term Estimates}}
For $k_M$ term (i.e. the second term on the RHS of \eqref{oo 17})
\begin{align}\label{oo 26}
    \int_{0}^{t_{1}}\int_{\r^3}k_{M,w}(v,v')f_w\Big(\vx-\e(t_1-s)\vv,\vuu\Big)\ue^{-\int_s^{t_1}\nu_M(\vv)
    }\ud v'\ud{s}
\end{align}
we can
introduce the truncated kernel $k_{M,N}(\vv,\vuu)$ which is smooth and has compactly supported range such that
\begin{align}
\sup_{\abs{\vv}\leq N}\int_{\abs{\vuu}\leq
N}\abs{k_{M,N}(\vv,\vuu)-k_{M,w}(\vv,\vuu)}\ud{\vuu}\leq\frac{1}{N}.
\end{align}
Similar to the proof of Proposition \ref{thm:infinity}, we deduce that 
\begin{align}
    \lnm{\int_{0}^{t_{b}}\int_{\r^3}\id_{\{\abs{v}\geq N\}}k_{M,w}(v,v')f_w\Big(\vx-\e(t_1-s)\vv,\vuu\Big)\ue^{-\int_s^{t_1}\nu_M(\vv)
    }\ud v'\ud{s}}\ls\frac{1}{N}\lnm{f_w},
\end{align}
\begin{align}
    \lnm{\int_{0}^{t_{b}}\int_{\r^3}\id_{\{\abs{v}\leq N\}}\id_{\{\abs{v'}\geq N\}}k_{M,w}(v,v')f_w\Big(\vx-\e(t_1-s)\vv,\vuu\Big)\ue^{-\int_s^{t_1}\nu_M(\vv)
    }\ud v'\ud{s}}\ls\frac{1}{N}\lnm{f_w},
\end{align}
\begin{align}
    \lnm{\int_{0}^{t_{1}}\int_{\r^3}\id_{\{\abs{v}\leq N\}}\id_{\{\abs{v'}\leq N\}}\id_{\{t_1-s\leq\d\}}k_{M,w}(v,v')f_w\Big(\vx-\e(t_1-s)\vv,\vuu\Big)\ue^{-\int_s^{t_1}\nu_M(\vv)
    }\ud v'\ud{s}}\ls\d\lnm{f_w},
\end{align}
and
\begin{align}
    \lnm{\int_{0}^{t_{1}}\int_{\r^3}\id_{\{\abs{v}\leq N\}}\id_{\{\abs{v'}\leq N\}}\id_{\{t_1-s\geq\d\}}\abs{k_{M,N}(\vv,\vuu)-k_{M,w}(\vv,\vuu)}f_w\Big(\vx-\e(t_1-s)\vv,\vuu\Big)\ue^{-\int_s^{t_1}\nu_M(\vv)
    }\ud v'\ud{s}}\ls\frac{1}{N}\lnm{f_w}.
\end{align}
When $N$ is large, all of these terms will only contribute $o(1)\lnm{f_w}$, which will not essentially change the compactness (which can be shown via contradiction). Hence, it suffices to consider 
\begin{align}
    \lnm{\int_{0}^{t_{1}}\int_{\r^3}\id_{\{\abs{v}\leq N\}}\id_{\{\abs{v'}\leq N\}}\id_{\{t_1-s\geq\d\}}k_{M,N}(\vv,\vuu)f_w\Big(\vx-\e(t_1-s)\vv,\vuu\Big)\ue^{-\int_s^{t_1}\nu_M(\vv)
    }\ud v'\ud{s}}.
\end{align}

\paragraph{\underline{Source Term Estimates}}
Consider the source term (i.e. the first term on the RHS of \eqref{oo 17})
\begin{align}
    \e\int_{0}^{t_{1}} w(v)r\Big(\vx-\e(t_1-s)\vv\Big)\tilde q\big(\vv\big)\ue^{-\int_s^{t_1}\nu_M(\vv)
    }\ud{s}.
\end{align}
Similar to the proof of Proposition \ref{thm:infinity}, we deduce that 
\begin{align}
    \lnm{\int_{0}^{t_{1}}\id_{\{\abs{v}\geq N\}} w(v)r\Big(\vx-\e(t_1-s)\vv\Big) q\big(\vv\big)\ue^{-\int_s^{t_1}\nu_M(\vv)
    }\ud{s}}\ls \frac{1}{N}\lnm{r}
\end{align}
and
\begin{align}
    \lnm{\int_{0}^{t_{1}}\id_{\{\abs{v}\leq N\}}\id_{\{t_1-s\leq\d\}} w(v)r\Big(\vx-\e(t_1-s)\vv\Big) q\big(\vv\big)\ue^{-\int_s^{t_1}\nu_M(\vv)
    }\ud{s}}\ls \d\lnm{r}.
\end{align}
When $N$ is large, all of these terms will only contribute $o(1)\lnm{r}$, which will not essentially change the compactness (which can be shown via contradiction). Hence, it suffices to consider 
\begin{align}
    \lnm{\int_{0}^{t_{1}}\id_{\{\abs{v}\leq N\}}\id_{\{t_1-s\geq\d\}} w(v)r\Big(\vx-\e(t_1-s)\vv\Big) q\big(\vv\big)\ue^{-\int_s^{t_1}\nu_M(\vv)
    }\ud{s}}\ls \d\lnm{r}.
\end{align}

\paragraph{\underline{Summary of Simplification}}
Summarizing the analysis all above, after applying $\pk_M$ (i.e. multiplying $q(v)$ and integrating over $\r^3$) in \eqref{oo 17}, it suffices to prove the compactness of 
\begin{align}\label{oo 19}
    \int_{0}^{t_{1}}\int_{\r^3}\int_{\r^3}\id_{\{\abs{v}\leq N\}}\id_{\{\abs{v'}\leq N\}}\id_{\{t_1-s\geq\d\}}k_{M,N}(\vv,\vuu)q(v)f_w\Big(\vx-\e(t_1-s)\vv,\vuu\Big)\ue^{-\int_s^{t_1}\nu_M(\vv)
    }\ud v'\ud v\ud{s}
\end{align}
% and
\begin{align}\label{oo 20}
    \int_{0}^{t_{1}}\int_{\r^3}\id_{\{\abs{v}\leq N\}}\id_{\{t_1-s\geq\d\}}w(v)q(v)r\Big(\vx-\e(t_1-s)\vv\Big)\tilde q\Big(\vx-\e(t_1-s)\vv,\vv\Big)\ue^{-\int_s^{t_1}\nu_M(\vv)
    }\ud v\ud{s}.
\end{align}

\paragraph{\underline{Compactness of \eqref{oo 19}}}
For \eqref{oo 19}, by substitution $v\rt y=x-\e(t_1-s)v$ (or equivalently $v=\e^{-1}\frac{x-y}{t_1-s}$), the integral reduces to
\begin{align}\label{oo 21}
&\int_{\Omega}\int_{\abs{\vuu}\leq N}\int_{0}^{\max\{t_1-\d,0\}} q(\e^{-1}\tfrac{x-y}{t_1-s})k_{M,N}(\e^{-1}\tfrac{x-y}{t_1-s},v')w(v')f(y,v')\ue^{-\int_{s}^{t_1}\nu_M(\e^{-1}\tfrac{x-y}{t_1-s})
}\ud{s}\ud{\vuu}\ud{y}.
\end{align}

On the one hand, \eqref{oo 14} naturally implies the $L^2$ bound $\tnm{f}\ls\xxnm{g}$,
which yields weak compactness of $f$ in $L^2$. 

On the other hand, note that $q$, $k_{M,N}$, $w$ terms are all smooth and uniformly bounded with respect to $(t,x)$ in $\abs{v'}\leq N$ and $t_1-s\geq\d$. Hence, after the substitution, for fixed $(t,x)$
\begin{align}
    \int_{\abs{\vuu}\leq N}\int_{0}^{\max\{t_1-\d,0\}}q(\e^{-1}\tfrac{x-y}{t_1-s})k_{M,N}(\e^{-1}\tfrac{x-y}{t_1-s},v')w(v')\ue^{-\int_{s}^{t_1}\nu_M(\e^{-1}\tfrac{x-y}{t_1-s})}\ud{s}\ud{\vuu}
\end{align}
is a smooth $L^2$ function of $y$, and its $L^2$ norm is uniformly bounded with respect to $(t,x)$.

Therefore, for the weakly $L^2$ convergent sequence in terms of $f$, we can pass to the limit in the $y$ integral of \eqref{oo 21} to get strongly convergent sequence (uniformly in $(t,x)$). Hence, we have shown the strong compactness of \eqref{oo 19} in $L^{\infty}$.\\

\paragraph{\underline{Compactness of \eqref{oo 20}}}
Similarly, for \eqref{oo 20}, using substitution $v\rt y=x-\e(t_1-s)v$ (or equivalently $v=\e^{-1}\frac{x-y}{t_1-s}$), the integral reduces to 
\begin{align}\label{oo 22}
    \int_{\Omega}\int_0^{\max\{t_1-\d,0\}}w(\e^{-1}\tfrac{x-y}{t_1-s})q(\e^{-1}\tfrac{x-y}{t_1-s})\ue^{-\int_s^t\nu_M(\e^{-1}\frac{x-y}{t_1-s})}r(y) q(\e^{-1}\tfrac{x-y}{t_1-s})\ud s\ud y.
\end{align}

On the one hand, $g\in\widetilde X$ implies $\tnm{r}\ls\tnm{g}\ls\xxnm{g}$, 
which yields weak compactness of $r$ in $L^2$. 

On the other hand,note that $q$, $w$ terms are all smooth and uniformly bounded with respect to $(t,x)$ in $\abs{v}\leq N$ and $t_1-s\geq\d$. Hence, after the substitution, for fixed $(t,x)$
\begin{align}
    \int_0^{\max\{t_1-\d,0\}}w(\e^{-1}\tfrac{x-y}{t_1-s})q(\e^{-1}\tfrac{x-y}{t_1-s})\ue^{-\int_s^t\nu_M(\e^{-1}\frac{x-y}{t_1-s})} q(\e^{-1}\tfrac{x-y}{t_b-s})\ud s
\end{align}
is a smooth $L^2$ function of $y$, and its $L^2$ norm is uniformly bounded with respect to $(t,x)$. 

Therefore, for the weakly $L^2$ convergent sequence in terms of $r$, we can pass to the limit in the $y$ integral of \eqref{oo 22} to get strongly convergent sequence (uniformly in $(t,x)$). Hence, we have shown the strong compactness of \eqref{oo 20} in $L^{\infty}$.\\

\paragraph{\underline{Summary of Compactness}}
In summary, we have shown that compactness of $f_w$ in $L^{\infty}$. Considering the weight function $w$ and \eqref{oo 14}, this is equivalent to the compactness of
$\pk_M\lc_{\e}^{-1}: \widetilde X\cap\nk_M\rt \widetilde X\cap\nk_M$ in $\widetilde X$.
\end{proof}

\begin{proof}[Proof of Proposition \ref{prop:linear}]
For given $\theta$, we have shown that there exists a unique solution $\re_{M,\theta}\in\widetilde X$ (we add subscript $\theta$ to highlight the dependence on $\theta$) to \eqref{linear remainder.} or equivalently \eqref{linear remainder'} satisfying
\begin{align}\label{oo 08'}
    \xxnm{\re_{M,\theta}}\ls_{\e} \oot.
\end{align}
Note that this estimate is uniform in $\theta$.
In addition, direct integration on both sides of \eqref{linear remainder'} with multiplier $\mh$ yields that $\widetilde\re_{\theta}$ satisfies
\begin{align}\label{oo 07'}
    \iint_{\Omega\times\r^3}\abs{v}^2\mmh(v)\re_{M,\theta}(x,v)\ud x\ud v=0.
\end{align}
Further we can recover the solution $\widetilde\rem$ to \eqref{linear remainder==} by adding $\widetilde h$ such that
\begin{align}\label{oo 08}
    \xxnm{\widetilde{\re_{M,\theta}}}\ls_{\e}\oot,
\end{align}
\begin{align}\label{oo 07}
    \iint_{\Omega\times\r^3}\abs{v}^2\mmh(v)\widetilde{\re_{M,\theta}}(x,v)\ud x\ud v=0.
\end{align}
Since the estimates \eqref{oo 08} is uniform in $\theta>0$, we can find a subsequence that is weak-$\ast$ convergent in weighted $L^{\infty}$ and weak convergent in $L^2$ as $\theta_i\rt0$. By weak lower semi-continuity, we have $\widetilde{\re_{M,\theta_j}}\rt \widetilde{\rem}$ in $\widetilde X$,
and
\begin{align}
    \xxnm{\widetilde{\rem}}\leq\liminf_{j\rt\infty}\xxnm{\widetilde{\re_{M,\theta_j}}}\ls_{\e} \oot.
\end{align}
Also, by passing to the limit in the weak formulation, we know $\widetilde\rem$ or equivalently $\re$ is a solution to the remainder equation \eqref{remainder} satisfying
\begin{align}
    \xxnm{\re}\ls_{\e} \oot,
\end{align}
% and
\begin{align}
    \iint_{\Omega\times\r^3}\abs{v}^2\mh(v)\re(x,v)\ud x\ud v=\iint_{\Omega\times\r^3}\abs{v}^2\mmh(v)\widetilde\rem(x,v)\ud x\ud v=0.
\end{align}
By the a priori remainder estimates in Theorem \ref{thm:apriori}, we know that such $\re$ is unique.
\end{proof}

%%%%%%%%%%%%%%%%%%%%%%%%%%%%%%%%%%%%%%%%%%%%%%%%%%%%%%%%%%%%%%%%%%%%%%%%
\section{Proof of Main Theorem}
%%%%%%%%%%%%%%%%%%%%%%%%%%%%%%%%%%%%%%%%%%%%%%%%%%%%%%%%%%%%%%%%%%%%%%%%

%%%%%%%%%%%%%%%%%%%%%%%%%%%%%%%%%%%%%%%%%%%%%%%%%%%%%%%%%%%%%%%%%%%%%%%%
\subsection{Well-Posedness and Positivity}\label{sec:well-posedness}
%%%%%%%%%%%%%%%%%%%%%%%%%%%%%%%%%%%%%%%%%%%%%%%%%%%%%%%%%%%%%%%%%%%%%%%%

%%%%%%%%%%%%%%%%%%%%%%%%%%%%%%%%%%%%%%%%%%%%%%%%%%%%%%%%%%%%%%%%%%%%%%%%
\subsubsection{Well-Posedness of the Auxiliary Equation}
%%%%%%%%%%%%%%%%%%%%%%%%%%%%%%%%%%%%%%%%%%%%%%%%%%%%%%%%%%%%%%%%%%%%%%%%

\paragraph{\underline{Setup of Iteration}}
Based on \eqref{remainder} and \eqref{wt 02}, we design an iterative scheme with $\re^0=0$ and
\begin{align}\label{wt 01}
    &\vv\cdot\nx \left(\mh\re^{n+1}\right)-\e^{-1}\mh\lc[\re^{n+1}]\\
    =&\ \mathscr{\ss}+\e^{-1}\bigg(Q^{\ast}\left[\mh\re^n,\ffe\right]+Q^{\ast}\left[\ffe,\mh\re^n\right]+\e Q^{\ast}\left[\mh\re^n,\mh\re^n\right]\bigg)\no\\
    &+\e^{-2}\Big(Q_{\text{gain}}\left[\ff+\e\mh\re^n,\ff+\e\mh\re^n\right]-Q_{\text{gain}}\left[(\ff+\e\mh\re^n)_+,(\ff+\e\mh\re^n)_+\right]\Big)\no\\
    &+\e^{-2}\mathfrak{z}\ds\iint_{\Omega\times\r^3}\bigg(Q_{\text{gain}}\left[\ff+\e\mh\re^n,\ff+\e\mh\re^n\right]-Q_{\text{gain}}\left[(\ff+\e\mh\re^n)_+,(\ff+\e\mh\re^n)_+\right]\bigg).\no
\end{align}
This is equivalent to 
\begin{align}\label{remainder final}
\left\{
\begin{array}{l}
\vv\cdot\nx\left(\mh\re^{n+1}\right)+\e^{-1}\mh\lc\left[\re^{n+1}\right]=\mh\ss^n\ \ \text{in}\ \
\Omega\times\r^3,\\\rule{0ex}{1.0em} \re^{n+1}(\vx_0,\vv)=\pp\left[\re^{n+1}\right](\vx_0,\vv)+h(\vx_0,\vv) \ \ \text{for}\ \ \vx_0\in\p\Omega\ \ \text{and}\ \ \vv\cdot\vn(\vx_0)<0.
\end{array}
\right.
\end{align}
Here $\ss^n$ denotes $\ss$ with $\re$ substituted by $\re^n$. The sequence $\big\{\re^n\big\}_{n=1}^{\infty}\subset X$ is guaranteed by Proposition \ref{prop:linear}.\\

\paragraph{\underline{Convergence}}
Applying the a priori remainder estimates in Theorem \ref{thm:apriori}, there exists $C_0>0$ such that letting $\re=\re^n$ on the RHS of \eqref{final 32}, we have
\begin{align}
    \xnm{\re^{n+1}}\leq \frac{1}{2}\xnm{\re^n}+C_0\left(\xnm{\re^n}^2+\xnm{\re^n}^{3}+\oot\right).
\end{align} 
Under the assumptions $\xnm{\re^n}\leq 3C_0\oot$,
for $\oot$ sufficiently small, by induction we obtain
\begin{align}
    \xnm{\re^{n+1}}\leq 3C_0\oot.
\end{align}
Hence, we obtain that the sequence $\big\{\re^n\big\}_{n=1}^{\infty}$ is uniformly bounded. Also, based on Proposition \ref{prop:linear}, we have
\begin{align}
    \iint_{\Omega\times\r^3}\abs{v}^2\mh(v)\re^{n+1}(x,v)\ud x\ud v=0.
\end{align}
By a similar proof, we know that this sequence is a contraction $\xnm{\re^{n+1}-\re^n}\ls\oot\xnm{\re^n-\re^{n-1}}$.

\paragraph{\underline{Existence and Uniqueness}}
Using contraction mapping theorem, we know that $\re^n\rt\re$ in $X$, and $\re$ is a solution to \eqref{remainder} satisfying $\xnm{\re}\ls\oot$
and $\ds\iint_{\Omega\times\r^3}\abs{v}^2\mh(v)\re(x,v)\ud x\ud v=0$.

The uniqueness follows from the a priori remainder estimates in Theorem \ref{thm:apriori}. Assume $\re_1$ and $\re_2$ are two distinct solutions to \eqref{remainder}. Then $\overline{\re}=\re_1-\re_2$ satisfies 
\begin{align}
\left\{
\begin{array}{l}
\vv\cdot\nx\left(\mh\overline{\re}\right)+\e^{-1}\mh\lc\left[\overline{\re}\right]=\mh\big(\ss^1-\ss^2\big)\ \ \text{in}\ \
\Omega\times\r^3,\\\rule{0ex}{1.0em} \overline{\re}(\vx_0,\vv)=\pp\left[\overline{\re}\right](\vx_0,\vv) \ \ \text{for}\ \ \vx_0\in\p\Omega\ \ \text{and}\ \ \vv\cdot\vn(\vx_0)<0.
\end{array}
\right.
\end{align}
Here $\ss^k$ for $k=1,2$ denotes $\ss$ with $\re$ substituted by $\re^k$. Then using Theorem \ref{thm:apriori}, we obtain
\begin{align}
    \xnm{\overline{\re}}\ls \oot\xnm{\overline{\re}},
\end{align}
which naturally yields $\xnm{\overline{\re}}=0$ and thus $\overline{\re}=0$.

%%%%%%%%%%%%%%%%%%%%%%%%%%%%%%%%%%%%%%%%%%%%%%%%%%%%%%%%%%%%%%%%%%%%%%%%
\subsubsection{Positivity of the Auxiliary Equation}
%%%%%%%%%%%%%%%%%%%%%%%%%%%%%%%%%%%%%%%%%%%%%%%%%%%%%%%%%%%%%%%%%%%%%%%%

\paragraph{\underline{Mild Formulation of the Auxiliary Equation}} Now we focus on \eqref{auxiliary system},
which can be rewritten as
\begin{align}
    \vv\cdot\nx \fs+\e^{-1}\nu(\fs)\fs=\e^{-1}Q_{\text{gain}}[\fs_+,\fs_+]+\mathfrak{z}\ds\iint_{\Omega\times\r^3}\e^{-1}\Big(Q_{\text{gain}}[\fs,\fs]-Q_{\text{gain}}[\fs_+,\fs_+]\Big).
\end{align}
Hence, using mild formulation to track the characteristics backward, we have
\begin{align}\label{oo 15}
    \fs(x,v)=&\ue^{-\e^{-1}\nu t_b}\fs(x_b,v)+\int_0^{t_b}\ue^{-\e^{-1}\nu (t_b-s)}\left(\e^{-1}Q_{\text{gain}}[\fs_+,\fs_+]+\mathfrak{z}\ds\iint_{\Omega\times\r^3}\e^{-1}\Big(Q_{\text{gain}}[\fs,\fs]-Q_{\text{gain}}[\fs_+,\fs_+]\Big)\right)(s)\ud s,
\end{align}
for $t_b$ and $x_b$ defined in Definition \ref{exit}.\\

\paragraph{\underline{Boundary Term in the Mild Formulation}}
We know
\begin{align}
    \int_{\gamma_+}\fs(x_b,v)(v\cdot n)\ud v=\int_{\gamma_+}\Big(\ff(x_b,v)+\e\mh\re(x_b,v)\Big)(v\cdot n)\ud v.
\end{align}
By a similar argument as the $\sp$ estimates, we know that this integral is positive (since the negative part is only effective when $v$ is very large). Hence, $\pp[\fs]>0$ and thus $\fs\big|_{\gamma_-}\geq C\mu>0$. In \eqref{oo 15}, the characteristic will hit $\gamma_-$ when tracking backward, and thus we know $\ue^{-\e^{-1}\nu t_b}\fs(x_b,v)\gs0$.\\

\paragraph{\underline{Bulk Term in the Mild Formulation}}
Using the similar technique as the $\sp$ estimates, we have
\begin{align}
    \ds\mathfrak{z}\iint_{\Omega\times\r^3}\e^{-1}\Big(Q_{\text{gain}}[\fs,\fs]-Q_{\text{gain}}[\fs_+,\fs_+]\Big)\gs -\e^{m-1}\mathfrak{z}.
\end{align}
Also, we know that $\fs\gs \m$ for $\abs{v}\ls \abs{\ln(\e)}^{\frac{1}{2}}$ and thus
\begin{align}
    \e^{-1}Q_{\text{gain}}[\fs_+,\fs_+]=&\e^{-1}\int_{\r^3}\int_{\s^2}q(\vo,\abs{\vuu-\vv})\fs_+(\vuu_{\ast})\fs_+(\vv_{\ast})\ud{\vo}\ud{\vuu}\\
    \geq&\e^{-1}\int_{\r^3}\int_{\s^2}\id_{\{\abs{\vuu_{\ast}}\leq 1\}}\id_{\{\abs{\vv_{\ast}}\leq 1\}}q(\vo,\abs{\vuu-\vv})\fs_+(\vuu_{\ast})\fs_+(\vv_{\ast})\ud{\vo}\ud{\vuu}\gs \e^{-1}.\no
\end{align}
Hence, considering that $\mathfrak{z}$ is compactly supported in $\{\abs{v}\leq 1\}$, we have
\begin{align}
    \int_0^{t_b}\ue^{-\e^{-1}\nu (t_b-s)}\left(\e^{-1}Q_{\text{gain}}[\fs_+,\fs_+]+\mathfrak{z}\ds\iint_{\Omega\times\r^3}\e^{-1}\Big(Q_{\text{gain}}[\fs,\fs]-Q_{\text{gain}}[\fs_+,\fs_+]\Big)\right)(s)\ud s\gs \e^{-1}-\e^{m-1}\mathfrak{z}\geq0.
\end{align}

\paragraph{\underline{Summary of Positivity}}
Combining the above results, when $\e$ is sufficiently small, we know that $\fs(x,v)\geq0$.
In this case, \eqref{auxiliary system} reduces to \eqref{large system-}, and thus we know that the solution $\fs$ to \eqref{large system-} satisfies 
\begin{align}
    \xnm{\re}\ls\oot,
\end{align}
\begin{align}
    \iint_{\Omega\times\r^3}\abs{v}^2\mh(v)\re(x,v)\ud x\ud v=0.
\end{align}

%%%%%%%%%%%%%%%%%%%%%%%%%%%%%%%%%%%%%%%%%%%%%%%%%%%%%%%%%%%%%%%%%%%%%%%%
\subsection{Ghost Effects}
%%%%%%%%%%%%%%%%%%%%%%%%%%%%%%%%%%%%%%%%%%%%%%%%%%%%%%%%%%%%%%%%%%%%%%%%

Based on \eqref{aa 08}, we have $\mhh\fs-\mhh\f-\mhh\fb-\e\re=0$,
which yields
\begin{align}
    \tnm{\mhh\fs-\mh-\e\f_1-\e\mhh\mbh\fb_1-\e\re}=\tnm{\e^2\f_2}\ls\oot\e^2.
\end{align}
Due to the rescaling, we have $\tnm{\e\mhh\mbh\fb_1}\ls\oot\e^{\frac{3}{2}}$,
and thus we know
\begin{align}\label{aa 09}
    \tnm{\mhh\fs-\mh-\e\f_1-\e\re}\ls\oot\e^{\frac{3}{2}}.
\end{align}
Let $\pk_{\mh}$, $\pk_{v\mh}$, $\pk_{\abs{v}^2\mh}$, $\pk_{\a}$ denote the $L^2$ projection onto the subspaces spanned by $\mh$, $v\mh$, $\abs{v}^2\mh$ or $\a$ (via taking inner product with these functions).

Applying these operators (denoted by $\pk_{\cdot}$) as well as $\ik-\bpk$ to $\mhh\fs-\mh-\e\f_1-\e\re$ and using \eqref{aa 09}, we have
\begin{align}
    \tnm{\pk_{\cdot}\Big[\mhh\fs-\mh-\e\f_1-\e\re\Big]}\ls&\oot\e^{\frac{3}{2}},\quad
    \tnm{(\ik-\bpk)\Big[\mhh\fs-\mh-\e\f_1-\e\re\Big]}\ls\oot\e^{\frac{3}{2}}.
\end{align}
Then a detailed computation reveals that
\begin{align}
    \tnm{\pk_{\mh}\left[\mhh\fs\right]-\Big(\rq+\e\rq_1+\e(\P-2Tc)\Big)\mh}\ls&\oot\e^{\frac{3}{2}},\\
    \tnm{\pk_{v\mh}\left[\mhh\fs\right]-\Big(\e\uq_1+\e\bb\Big)\cdot v\mh}\ls&\oot\e^{\frac{3}{2}},\\
    \tnm{\pk_{\abs{v}^2\mh}\left[\mhh\fs\right]-3\rq\tq\Big(1+\e\rq_1+\e\tq\tq_1+\e p\Big)\abs{v}^2\mh}\ls&\oot\e^{\frac{3}{2}},\\
    \tnm{\pk_{\a}\left[\mhh\fs\right]-\left(\e\left(-\frac{\nx\tq}{2\tq^2}\right)-\e \bd\right)\cdot\a}\ls&\oot\e^{\frac{3}{2}},\\
    \tnm{(\ik-\bpk)\left[\mhh\fs\right]-(\ik-\bpk)[\e\re]}\ls&\oot\e^{\frac{3}{2}}.
\end{align}
Since $\xnm{\re}\ls\oot$, we have $\e^{-\frac{1}{2}}\tnm{\bb}+\e^{-\frac{1}{2}}\tnm{\bd}+\e^{-1}\um{(\ik-\bpk)[\re]}\ls\oot$ and hence
\begin{align}
    \tnm{\pk_{\mh}\left[\mhh\fs\right]-\rq\mh}\ls&\oot\e,\quad \tnm{\pk_{\abs{v}^2\mh}\left[\mhh\fs\right]-3\rq\tq\abs{v}^2\mh}\ls\oot\e,\\
    \tnm{\pk_{v\mh}\left[\mhh\fs\right]-\e\uq_1\cdot v\mh}\ls&\oot\e^{\frac{3}{2}},\\
    \tnm{\pk_{\a}\left[\mhh\fs\right]+\e\left(\frac{\nx\tq}{2\tq^2}\right)\cdot \a}\ls&\oot\e^{\frac{3}{2}},
    \tnm{(\ik-\bpk)\left[\mhh\fs\right]}\ls\oot\e^{\frac{3}{2}}.
\end{align}
Therefore, we know that $(\rq,\e\uq_1,\tq)$ indeed captures the leading-order terms of $\mhh\fs$ in $L^2$ as $\e\rt0$.

%%%%%%%%%%%%%%%%%%%%%%%%%%%%%%%%%%%%%%%%%%%%%%%%%%%%%%%%%%%%%%%%%%%%%%%%%%%%%%%%%%
\appendix
%%%%%%%%%%%%%%%%%%%%%%%%%%%%%%%%%%%%%%%%%%%%%%%%%%%%%%%%%%%%%%%%%%%%%%%%%%%%%%%%%%

\makeatletter
\renewcommand \theequation {%
A.%
% \ifnum \c@section>\z@ \@arabic\c@section.%
% \fi
\ifnum\c@subsection>\z@\@arabic\c@subsection.%
%\fi\ifnum \c@subsubsection>\z@\@arabic\c@subsubsection.
\fi
\@arabic\c@equation} \@addtoreset{equation}{section}
\@addtoreset{equation}{subsection} \makeatother

%%%%%%%%%%%%%%%%%%%%%%%%%%%%%%%%%%%%%%%%%%%%%%%%%%%%%%%%%%%%%%%%%%%%%%%%
\section{Notation}
%%%%%%%%%%%%%%%%%%%%%%%%%%%%%%%%%%%%%%%%%%%%%%%%%%%%%%%%%%%%%%%%%%%%%%%%

%%%%%%%%%%%%%%%%%%%%%%%%%%%%%%%%%%%%%%%%%%%%%%%%%%%%%%%%%%%%%%%%%%%%%%%%
\subsection{Asymptotic Analysis}\label{sec:asymptotic}
%%%%%%%%%%%%%%%%%%%%%%%%%%%%%%%%%%%%%%%%%%%%%%%%%%%%%%%%%%%%%%%%%%%%%%%%

The details about the construction of the asymptotic expansion can be found in the companion paper \cite{AA024}.

%%%%%%%%%%%%%%%%%%%%%%%%%%%%%%%%%%%%%%%%%%%%%%%%%%%%%%%%%%%%%%%%%%%%%%%%
\subsubsection{Interior Solution}
%%%%%%%%%%%%%%%%%%%%%%%%%%%%%%%%%%%%%%%%%%%%%%%%%%%%%%%%%%%%%%%%%%%%%%%%

Following the analysis in the companion paper \cite{AA024}, we have
\begin{align}
\m(\vx,\vv):=\frac{\rq(\vx)}{\big(2\pi\tq(\vx)\big)^{\frac{3}{2}}}
\exp\bigg(-\frac{\abs{\vv}^2}{2\tq(\vx)}\bigg),
\end{align}
and
\begin{align}
    \f_1=&-\a\cdot\frac{\nx\tq}{2\tq^2}+\mh\bigg(\frac{\rq_1}{\rq}+\frac{\uq_1\cdot\vv}{\tq}+\frac{\tq_1\big(\abs{\vv}^2-3\tq\big)}{2\tq^2}\bigg),\label{final 34}\\
    \f_2=&-\li\Big[\mhh\left(\vv\cdot\nx\left(\mh\f_1\right)\right)\Big]+\li\Big[\Gamma[\f_1,\f_1]\Big]
    +\mh\bigg(\frac{\rq_2}{\rq}+\frac{\uq_2\cdot\vv}{\tq}+\frac{\tq_2\big(\abs{\vv}^2-3\tq\big)}{2\tq^2}\bigg),\label{final 35}
\end{align}
where $(\rho,0,T)$, $(\rho_1,u_1,T_1)$ and $(\rho_2,u_2,T_2)$ satisfy the ghost-effect equations \eqref{fluid system-} and an additional requirement
\begin{align}\label{fluid system'}
    \nx P_1=\nx\big(T\rho_1+\rho T_1\big)=&0.
\end{align}
Here $u_k:=(u_{k,1},u_{k,2},u_{k,3})$, $\mathfrak{p}:=T\rho_2+\rho T_2$.

%%%%%%%%%%%%%%%%%%%%%%%%%%%%%%%%%%%%%%%%%%%%%%%%%%%%%%%%%%%%%%%%%%%%%%%%
\subsubsection{$\e$-Cutoff Boundary Layer}\label{sec:boundary-form}
%%%%%%%%%%%%%%%%%%%%%%%%%%%%%%%%%%%%%%%%%%%%%%%%%%%%%%%%%%%%%%%%%%%%%%%%

Under the substitution $(x,v)\rt (\eta,\iota_1,\iota_2,\vvv)$ defined in Section \ref{sec:geometric-setup}, the transport operator becomes
\begin{align}
\vv\cdot\nx=&\frac{1}{\e}\va\dfrac{\p }{\p\eta}-\dfrac{1}{R_1-\e\eta}\bigg(\vb^2\dfrac{\p }{\p\va}-\va\vb\dfrac{\p}{\p\vb}\bigg)
-\dfrac{1}{R_2-\e\eta}\bigg(\vc^2\dfrac{\p }{\p\va}-\va\vc\dfrac{\p }{\p\vc}\bigg)\\
&+\frac{1}{\pl_1\pl_2}\Bigg(\frac{R_1\p_{\iota_1\iota_1}\vr\cdot\p_{\iota_2}\vr}{\pl_1(R_1-\e\eta)}\vb\vc
+\frac{R_2\p_{\iota_1\iota_2}\vr\cdot\p_{\iota_2}\vr}{\pl_2(R_2-\e\eta)}\vc^2\Bigg)\frac{\p}{\p\vb}
+\frac{1}{\pl_1\pl_2}\Bigg(\frac{R_2\p_{\iota_2\iota_2}\vr\cdot\p_{\iota_1}\vr}{\pl_2(R_2-\e\eta)}\vb\vc
+\frac{R_1\p_{\iota_1\iota_2}\vr\cdot\p_{\iota_1}\vr}{\pl_1(R_1-\e\eta)}\vb^2\Bigg)\frac{\p}{\p\vc}\no\\
&+\bigg(\frac{R_1\vb}{\pl_1(R_1-\e\eta)}\frac{\p}{\p\iota_1}+\frac{R_2\vc}{\pl_2(R_2-\e\eta)}\frac{\p}{\p\iota_2}\bigg),\no
\end{align}
where $R_i=\kk_i^{-1}$ represents the radius of principal curvature. 

Set $\lc_w[f]:=-2\m_w^{-\frac{1}{2}} Q^{\ast}\left[\m_w,\m_w^{\frac{1}{2}} f\right]:=\nu_w f-K_w[f]$.
Define $\blf(\eta,\vvv)$ for $(\eta,\vvv)\in[0,\infty)\times\r^3$ satisfying the Milne problem
\begin{align}
    \va\dfrac{\p\blf}{\p\eta}+\nu_w \blf-K_w\left[\blf \right]=0,\quad \blf (0,\vvv)=-\a\cdot\frac{\nx T}{2T^2}\ \ \text{for}\ \ \va>0,
\end{align}
and the zero mass-flux condition
\begin{align}\label{aa 05}
    \int_{\r^3}\va\mbh(\vvv)\blf (0,\vvv)\ud \vvv=0.
\end{align}
Based on Theorem \ref{boundary well-posedness}, we know that there exists 
\begin{align}
    \blf_{\infty}(\vvv):=\mbh\bigg(\frac{\rq^B}{\rq}+\frac{\uq^B\cdot\vv}{\tq}+\frac{\tq^B(\abs{\vv}^2-3\tq)}{2\tq^2}\bigg)\in\nk,
\end{align}
such that $\ds\int_{\r^3}\va\mbh(\vvv)\blf_{\infty}(\vvv)\ud \vvv=0$,
and for some $K_0>0$, $\blff(\eta,\vvv):=\blf(\eta,\vvv)-\blf_{\infty}(\vvv)$
satisfies
\begin{align}
    \abs{\blf_{\infty}}{}+\lnmm{\ue^{K_0\eta}\blff}\ls& \oot.
\end{align}
In addition, using Theorem \ref{boundary well-posedness} and Theorem \ref{boundary regularity}, we know
\begin{align}
    \lnmm{\ue^{K_0\eta}\p_{\iota_1}\blff}+\lnmm{\ue^{K_0\eta}\p_{\iota_2}\blff}+\lnmm{\ue^{K_0\eta}\p_{\vb}\blff}+\lnmm{\ue^{K_0\eta}\p_{\vc}\blff}+\bnm{\ue^{K_0}\nu\blff}\ls& \oot.
\end{align}
Let $\chi(y)\in C^{\infty}(\r)$
and $\ch(y)=1-\chi(y)$ be smooth cut-off functions satisfying $\chi(y)=1$ if $\abs{y}\leq1$ and $\chi(y)=0$ if $\abs{y}\geq2$.
Define the $\e$-cutoff boundary layer 
\begin{align}\label{final 14}
    \fb_1(\eta,\vvv)=\ch\left(\e^{-1}\va\right)\chi(\e\eta)\blff (\eta,\vvv).
\end{align}

%%%%%%%%%%%%%%%%%%%%%%%%%%%%%%%%%%%%%%%%%%%%%%%%%%%%%%%%%%%%%%%%%%%%%%%%
\subsubsection{Matching Procedure}\label{sec:matching}
%%%%%%%%%%%%%%%%%%%%%%%%%%%%%%%%%%%%%%%%%%%%%%%%%%%%%%%%%%%%%%%%%%%%%%%%

Based on the analysis in our companion paper \cite{AA024}, we derive that $(\rq,\uq_1,\tq)$ must satisfy the boundary conditions \eqref{boundary condition}. The well-posedness of the ghost-effect equation is well-studied in \cite{AA024} (see Theorem \ref{thm:ghost}).

Define $T_1$ satisfying the boundary condition $T_1\Big|_{\p\Omega}=T^B$ and
construct a Sobolev extension such that
\begin{align}
    \nm{T_1}_{W^{4+\frac{1}{\NN},\NN}}\ls\oot.
\end{align}
We choose constant $P_1=0$,
which implies
\begin{align}\label{final 07}
    \iint_{\Omega\times\r^3}\abs{v}^2\Big(\m+\e\mh\f_1+\e^2\mh\f_2+\e\mbh\fb_1\Big)=\e\iint_{\Omega\times\r^3}\abs{v}^2\mbh\fb_1.
\end{align}
Then based on \eqref{fluid system'}, we have $\rho_1=-T^{-1}\big(\rho T_1\big)$,
and thus
\begin{align}
    \nm{\rq_1}_{W^{4+\frac{1}{\NN},\NN}}\ls\oot.
\end{align}
Note that $\rho_1$ is not necessarily equal to $\rho^B$ on $\p\Omega$.

We will choose $u_2$ such that on the boundary $\p\Omega$, 
\begin{align}\label{aa 38}
    u_2\cdot n=-\e^{-1}P^{-1}\int_{\r^3}\va\mbh(\vvv)\fb_1(0,\vvv)\ud \vvv,
\end{align}
which yields 
\begin{align}\label{final 08}
    \int_{\r^3}\Big(\mh+\e\f_1+\e^2\f_2+\e\fb_1\Big)\mbh(v\cdot n)=0.
\end{align}
Then we can construct a Sobolev extension such that
{\begin{align}
    \nm{u_2}_{W^{4,\NN}}\ls\oot.
\end{align}}
We are free to take $\rho_2=0$ in $\Omega$, and thus $T_2$ is determined and satisfies
\begin{align}
    \nm{T_2}_{W^{2,\NN}}\ls\oot.
\end{align}

\begin{remark}
    The construction of $(\rq_1,\tq_1)$ and $(\rq_2,\uq_2,\tq_2)$ are not uniquely determined in our construction. 
\end{remark}

%%%%%%%%%%%%%%%%%%%%%%%%%%%%%%%%%%%%%%%%%%%%%%%%%%%%%%%%%%%%%%%%%%%%%%%%
\subsection{Linearized Boltzmann Operator}
%%%%%%%%%%%%%%%%%%%%%%%%%%%%%%%%%%%%%%%%%%%%%%%%%%%%%%%%%%%%%%%%%%%%%%%%

Based on \cite[Chapter 7]{Cercignani.Illner.Pulvirenti1994} and \cite[Chapter 1\&3]{Glassey1996}, define the symmetrized version of $Q$
\begin{align}\label{qstar}
Q^{\ast}[F,G]:=&\frac{1}{2}\iint_{\r^3\times\s^2}q(\vo,\abs{\vuu-\vv})\Big(F(\vuu_{\ast})G(\vv_{\ast})+F(\vv_{\ast})G(\vuu_{\ast})-F(\vuu)G(\vv)-F(\vv)G(\vuu)\Big)\ud{\vo}\ud{\vuu}.
\end{align}
Clearly, $Q[F,F]=Q^{\ast}[F,F]$.
Denote the linearized Boltzmann operator $\lc$
\begin{align}\label{att 11}
\lc[f]:=&-2\m^{-\frac{1}{2}}Q^{\ast}\big[\m,\m^{\frac{1}{2}}f\big]:=\nu f-K[f],
\end{align}
where
for some kernels $k(\vuu,\vv)=k_2(\vuu,\vv)-k_1(\vuu,\vv)$,
\begin{align}
\nu(x,\vv)=&\int_{\r^3}\int_{\s^2}q(\vo,\abs{\vuu-\vv})\m(\vuu)\ud{\vo}\ud{\vuu},\quad
K[f](x,\vv)=K_2[f]-K_1[f]=\int_{\r^3}k(\vuu,\vv)f(\vuu)\ud{\vuu},\\
K_1[f](x,\vv)=&\m^{\frac{1}{2}}(\vv)\int_{\r^3}\int_{\s^2}q(\vo,\abs{\vuu-\vv})\m^{\frac{1}{2}}(\vuu)f(\vuu)\ud{\vo}\ud{\vuu}=\int_{\r^3}k_1(\vuu,\vv)f(\vuu)\ud{\vuu},\\
K_2[f](x,\vv)=&\int_{\r^3}\int_{\s^2}q(\vo,\abs{\vuu-\vv})\m^{\frac{1}{2}}(\vuu)\bigg(\m^{\frac{1}{2}}(\vv_{\ast})f(\vuu_{\ast})
+\m^{\frac{1}{2}}(\vuu_{\ast})f(\vv_{\ast})\bigg)\ud{\vo}\ud{\vuu}=\int_{\r^3}k_2(\vuu,\vv)f(\vuu)\ud{\vuu}.
\end{align}
Note that $\lc$ is self-adjoint in $L^2_{\nu}(\r^3)$ satisfying the coercivity property 
\begin{align}\label{coercivity}
    \int_{\r^3}f\lc[f]\gs \um{(\ik-\pk)[f]}^2.
\end{align} 
Denote $\li: \nnk\rt\nnk$ the quasi-inverse of $\lc$.
Also, denote the nonlinear Boltzmann operator $\Gamma$
\begin{align}\label{gamma}
\Gamma[f,g]:=\mhh Q^{\ast}\left[\mh f,\mh g\right]\in\nnk.
\end{align}

%%%%%%%%%%%%%%%%%%%%%%%%%%%%%%%%%%%%%%%%%%%%%%%%%%%%%%%%%%%%%%%%%%%%%%%%%%%%
\subsection{Remainder Estimates}
%%%%%%%%%%%%%%%%%%%%%%%%%%%%%%%%%%%%%%%%%%%%%%%%%%%%%%%%%%%%%%%%%%%%%%%%%%%%

The boundary term is given by
\begin{align}\label{aa 32}
    h:=&\e^{-1}\Big(\pp\left[\mhh\ff\right]-\mhh\ff\Big)\\
    =&\e\left(\mss\displaystyle\int_{\vuu\cdot\vn>0}
    \mbh(v')\f_2(v')\abs{\vuu\cdot\vn}\ud{\vuu}-\f_2\Big|_{v\cdot n<0}\right)-\left(\mss\displaystyle\int_{\vuu\cdot\vn>0}
    \mbh\chi\left(\e^{-1}\va\right)\blff\abs{\vuu\cdot\vn}\ud{\vuu}-\mbh\chi\left(\e^{-1}\va\right)\blff\Big|_{v\cdot n<0}\right).\no
\end{align}
Also, the source term is given by
\begin{align}
    \ss:=&\ \mhh\mathscr{\ss}+\e^{-1}\mhh\bigg(2Q^{\ast}\left[\ffe,\mh\re\right]+\e Q^{\ast}\left[\mh\re,\mh\re\right]\bigg)\\
    &-\e^{-1}\mhh\bigg(\e^{-1}Q_{\text{gain}}\left[\ff+\e\mh\re,\ff+\e\mh\re\right]-\e^{-1}Q_{\text{gain}}\left[\left(\ff+\e\mh\re\right)_+,\left(\ff+\e\mh\re\right)_+\right]\bigg)\no\\
    &+\e^{-1}\mathfrak{z}\mhh\iint_{\Omega\times\r^3}\bigg(\e^{-1}Q_{\text{gain}}\left[\ff+\e\mh\re,\ff+\e\mh\re\right]-\e^{-1}Q_{\text{gain}}\left[\left(\ff+\e\mh\re\right)_+,\left(\ff+\e\mh\re\right)_+\right]\bigg),\no
\end{align}
where $\mathscr{\ss}:=-\e^{-1}v\cdot\nx\ff+\e^{-2}Q^{\ast}\left[\ff,\ff\right]$.
We write the detailed expression
\begin{align}\label{aa 31}
    \ss=-\llc[\re]+\sb=-\llc[\re]+\ss_0+\ss_1+\ss_2+\ss_3+\ss_4+\ss_5+\sp,
\end{align}
where
\begin{align}\label{final 51}
    \llc[\re]:=&-2\e^{-1}\mhh Q^{\ast}\left[\mh\Big(\e\f_1\Big),\mh\re\right]=-2\Gamma[\f_1,\re],\\
    \ss_0:=&2\e^{-1}\mhh Q^{\ast}\left[\mh\Big(\e^2\f_2\Big),\mh\re\right]=2\e\Gamma[f_2,\re],\label{def-s0}\\
    \ss_1:=&2\e^{-1}\mhh Q^{\ast}\left[\mh_w\Big(\e\fb_1\Big),\mh\re\right]=2\Gamma\left[\mhh\mh_w\fb_1,\re\right],\label{def-s1}\qquad\qquad\qquad\qquad\qquad\qquad\qquad\qquad\\
    \ss_2:=&\mhh Q^{\ast}\left[\mh\re,\mh\re\right]=\Gamma[\re,\re],\label{def-s2}
\end{align}
\begin{align}\label{mm 00}
    \ss_3:=&\mhh\dfrac{1}{R_1-\e\eta}\bigg(\vb^2\dfrac{\p }{\p\va}-\va\vb\dfrac{\p}{\p\vb}\bigg)\left(\mbh\fb_1\right)+\mhh\dfrac{1}{R_2-\e\eta}\bigg(\vc^2\dfrac{\p }{\p\va}-\va\vc\dfrac{\p }{\p\vc}\bigg)\left(\mbh\fb_1\right)\\
    &-\mhh\dfrac{1}{\pl_1\pl_2}\left(\dfrac{R_1\p_{\iota_1\iota_1}\vr\cdot\p_{\iota_2}\vr}{\pl_1(R_1-\e\eta)}\vb\vc
    +\dfrac{R_2\p_{\iota_1\iota_2}\vr\cdot\p_{\iota_2}\vr}{\pl_2(R_2-\e\eta)}\vc^2\right)\dfrac{\p }{\p\vb}\left(\mbh\fb_1\right)\no\\
    &-\mhh\dfrac{1}{\pl_1\pl_2}\left(\dfrac{R_2\p_{\iota_2\iota_2}\vr\cdot\p_{\iota_1}\vr}{\pl_2(R_2-\e\eta)}\vb\vc
    +\dfrac{R_1\p_{\iota_1\iota_2}\vr\cdot\p_{\iota_1}\vr}{\pl_1(R_1-\e\eta)}\vb^2\right)\dfrac{\p }{\p\vc}\left(\mbh\fb_1\right)\no\\
    &-\mhh\left(\dfrac{R_1\vb}{\pl_1(R_1-\e\eta)}\dfrac{\p }{\p\iota_1}+\dfrac{R_2\vc}{\pl_2(R_2-\e\eta)}\dfrac{\p }{\p\iota_2}\right)\left(\mbh\fb_1\right)\no\\
    &+\e^{-1}\mhh\va\ch(\e^{-1}\va)\frac{\p\chi(\e\eta)}{\p\eta}\left(\mbh\blff\right)+\e^{-1}\mhh\mbh\chi(\e\eta)\bigg(\ch(\e^{-1}\va)K_w\Big[\blff\Big]-K_w\left[\ch(\e^{-1}\va)\blff\right]\bigg),\no
\end{align}
which can be further split into $\ss_3=\sx+\sy+\sz$
\begin{align}
    \sx:=&\mhh\dfrac{1}{R_1-\e\eta}\bigg(\vb^2\dfrac{\p }{\p\va}\bigg)\left(\mbh\fb_1\right)+\mhh\dfrac{1}{R_2-\e\eta}\bigg(\vc^2\dfrac{\p }{\p\va}\bigg)\left(\mbh\fb_1\right),\\
    \sy:=&-\mhh\dfrac{1}{R_1-\e\eta}\bigg(\va\vb\dfrac{\p}{\p\vb}\bigg)\left(\mbh\fb_1\right)-\mhh\dfrac{1}{R_2-\e\eta}\bigg(\va\vc\dfrac{\p }{\p\vc}\bigg)\left(\mbh\fb_1\right)\\
    &-\mhh\dfrac{1}{\pl_1\pl_2}\left(\dfrac{R_1\p_{\iota_1\iota_1}\vr\cdot\p_{\iota_2}\vr}{\pl_1(R_1-\e\eta)}\vb\vc
    +\dfrac{R_2\p_{\iota_1\iota_2}\vr\cdot\p_{\iota_2}\vr}{\pl_2(R_2-\e\eta)}\vc^2\right)\dfrac{\p }{\p\vb}\left(\mbh\fb_1\right)\no\\
    &-\mhh\dfrac{1}{\pl_1\pl_2}\left(\dfrac{R_2\p_{\iota_2\iota_2}\vr\cdot\p_{\iota_1}\vr}{\pl_2(R_2-\e\eta)}\vb\vc
    +\dfrac{R_1\p_{\iota_1\iota_2}\vr\cdot\p_{\iota_1}\vr}{\pl_1(R_1-\e\eta)}\vb^2\right)\dfrac{\p }{\p\vc}\left(\mbh\fb_1\right)\no\\
    &-\mhh\left(\dfrac{R_1\vb}{\pl_1(R_1-\e\eta)}\dfrac{\p }{\p\iota_1}+\dfrac{R_2\vc}{\pl_2(R_2-\e\eta)}\dfrac{\p }{\p\iota_2}\right)\left(\mbh\fb_1\right)+\e^{-1}\mhh\va\ch(\e^{-1}\va)\frac{\p\chi(\e\eta)}{\p\eta}\left(\mbh\blff\right),\no\\
    \sz:=&\e^{-1}\mhh\mbh\chi(\e\eta)\bigg(\ch(\e^{-1}\va)K_w\Big[\blff\Big]-K_w\Big[\ch(\e^{-1}\va)\blff\Big]\bigg),
\end{align}
and
\begin{align}
    \ss_4:=&-\e^{-1}\mhh\left(\vv\cdot\nx\left(\mh\left(\e^2\f_2\right)\right)\right)=-\e\mhh\left(\vv\cdot\nx\left(\mh\f_2\right)\right),\label{def-s4}\\
    \ss_5:=&\e^{2}\mhh Q^{\ast}\left[\mh\f_2,\mh\f_2\right]+2\e\mhh Q^{\ast}\left[\mh\f_2,\mh\f_1\right]+2\e\mhh Q^{\ast}\left[\mh\f_2,\mbh\fb_1\right]\label{def-s5}\\
    &+2\mhh Q^{\ast}\left[\mbh\fb_1,\mh\f_1\right]+\mhh Q^{\ast}\left[\mbh\fb_1,\mbh\fb_1\right]+\e^{-1}\mhh Q^{\ast}\left[\m-\mb,\mbh\fb_1\right]\no\\
    =&\e^{2}\Gamma\left[\f_2,\f_2\right]+2\e\Gamma\left[\f_2,\f_1\right]+2\e\Gamma\left[\f_2,\mhh\mbh\fb_1\right]\no\\
    &+2\Gamma\left[\mhh\mbh\fb_1,\f_1\right]+\Gamma\left[\mhh\mbh\fb_1,\mhh\mbh\fb_1\right]+\e^{-1}\Gamma\left[\mhh(\m-\mb),\mhh\mbh\fb_1\right]\no,
\end{align}
% and
\begin{align}
\label{def-sp}
    \sp:=&-\e^{-1}\mhh\bigg(\e^{-1} Q_{\text{gain}}\left[\ff+\e\mh\re,\ff+\e\mh\re\right]-\e^{-1} Q_{\text{gain}}\left[\left(\ff+\e\mh\re\right)_+,\left(\ff+\e\mh\re\right)_+\right]\bigg)\\
    &+\e^{-1}\mathfrak{z}\mhh\ds\iint_{\Omega\times\r^3}\bigg(\e^{-1} Q_{\text{gain}}\left[\ff+\e\mh\re,\ff+\e\mh\re\right]-\e^{-1} Q_{\text{gain}}\left[\left(\ff+\e\mh\re\right)_+,\left(\ff+\e\mh\re\right)_+\right]\bigg).\no
\end{align}

%%%%%%%%%%%%%%%%%%%%%%%%%%%%%%%%%%%%%%%%%%%%%%%%%%%%%%%%%%%%%%%%%%%%%%%%
\subsection{Inner Products and Norms}\label{sec:norms}
%%%%%%%%%%%%%%%%%%%%%%%%%%%%%%%%%%%%%%%%%%%%%%%%%%%%%%%%%%%%%%%%%%%%%%%%

Let $\brv{\ \cdot\ ,\ \cdot\ }$ denote the inner product for $v\in\r^3$, $\brx{\ \cdot\ ,\ \cdot\ }$ the inner product for $x\in\Omega$, and $\br{\ \cdot\ ,\ \cdot\ }$ the inner product for $(x,v)\in\Omega\times\r^3$. Also, let $\br{\ \cdot\ ,\ \cdot\ }_{\gamma_{\pm}}$ denote the inner product on $\gamma_{\pm}$ with measure $\ud\gamma:=\abs{v\cdot n}\ud v\ud S_x$. 

Denote the bulk and boundary norms
\begin{align}
    \pnm{f}{r}:=\left(\iint_{\Omega\times\r^3}\abs{f(x,v)}^r\ud v\ud x\right)^{\frac{1}{r}},\quad \pnms{f}{r}{\gamma_{\pm}}:=\left(\int_{\gamma_{\pm}}\abs{f(x,v)}^r\abs{v\cdot n}\ud v\ud x\right)^{\frac{1}{r}}.
\end{align}
Define the weighted $L^{\infty}$ norms for properly chosen $\tm>0$, $0\leq\varrho<\dfrac{1}{2}$ and $\vartheta\geq0$
\begin{align}
    \lnmm{f}:=\esssup_{(x,v)\in\Omega\times\r^3}\bigg(\br{v}^{\vartheta}\ue^{\varrho\frac{\abs{v}^2}{2\tm}}\abs{f(x,v)}\bigg),\quad
    \lnmms{f}{\gamma_{\pm}}:=\esssup_{(x,v)\in\gamma_{\pm}}\bigg(\br{v}^{\vartheta}\ue^{\varrho\frac{\abs{v}^2}{2\tm}}\abs{f(x,v)}\bigg).
\end{align}
Denote the $\nu$-norm: $\ds\um{f}:=\left(\iint_{\Omega\times\r^3}\nu(x,v)\abs{f(x,v)}^2\ud v\ud x\right)^{\frac{1}{2}}$.
Let $\nm{\cdot}_{W^{k,p}}$ denote the usual Sobolev norm for $x\in\Omega$ and  $\abs{\cdot}_{W^{k,p}}$ for $x\in\p\Omega$, and $\nm{\cdot}_{W^{k,p}L^q}$ denote $W^{k,p}$ norm for $x\in\Omega$ and $L^q$ norm for $v\in\r^3$. The similar notation also applies when we replace $L^q$ by $L^{\infty}_{\varrho,\vartheta}$ or $L^q_{\gamma}$.

% %%%%%%%%%%%%%%%%%%%%%%%%%%%%%%%%%%%%%%%%%%%%%%%%%%%%%%%%%%%%%%%%%%%%%%%%
% \subsection{Symbols and Constants}\label{sec:notation}
% %%%%%%%%%%%%%%%%%%%%%%%%%%%%%%%%%%%%%%%%%%%%%%%%%%%%%%%%%%%%%%%%%%%%%%%%

We also define
\begin{align}
\label{final 23}
    &\k\id:=\int_{\r^3}\left(\a\otimes\ab\right)\ud v,\quad \sigma\id:=\int_{\r^3}\left(\abs{v}^2-5T\right)\left(\a\otimes\ab\right)\ud v,\quad \lambda:=\frac{1}{\tq}\int_{\r^3}\b_{ij}\bbb_{ij}\ \ \text{for}\ \ i\neq j.
\end{align}

%%%%%%%%%%%%%%%%%%%%%%%%%%%%%%%%%%%%%%%%%%%%%%%%%%%%%%%%%%%%%%%%%%%%%%%%
\section*{Acknowledgement}
%%%%%%%%%%%%%%%%%%%%%%%%%%%%%%%%%%%%%%%%%%%%%%%%%%%%%%%%%%%%%%%%%%%%%%%%

The authors would like to thank Ian Tice for providing useful references.

\bibliographystyle{siam}
\bibliography{Reference}

\begin{thebibliography}{100}

\bibitem{Acosta.Duran2017}
{\sc G.~Acosta and R.~G. Duran}, {\em Divergence operator and related inequalities}, Springer, New York, 2017.

\bibitem{Akhlaghi.Roohi.Stefanov2022}
{\sc H.~Akhlaghi, E.~Roohi, and S.~Stefanov}, {\em A comprehensive review on micro- and nano-scale gas flow effects: Slip-jump phenomena, {Knudsen} paradox, thermally-driven flows, and {Knudsen} pumps}, Physics Reports, 997 (2022), pp.~1--60.

\bibitem{Aoki.Golse.Kosuge2015}
{\sc K.~Aoki, F.~Golse, and S.~Kosuge}, {\em The steady {Boltzmann} and {Navier-Stokes} equations}, Bull. Inst. Math. Acad. Sin. (N.S.), 10 (2015), pp.~205--257.

\bibitem{Arkeryd.Esposito.Marra.Nouri2010}
{\sc L.~Arkeryd, R.~Esposito, R.~Marra, and A.~Nouri}, {\em Stability for {Rayleigh}-{Benard} convective solutions of the {Boltzmann} equation}, Arch. Ration. Mech. Anal., 198 (2010), pp.~125--187.

\bibitem{Arkeryd.Esposito.Marra.Nouri2011}
\leavevmode\vrule height 2pt depth -1.6pt width 23pt, {\em Ghost effect by curvature in planar {Couette} flow}, Kinet. Relat. Models, 4 (2011), pp.~109--138.

\bibitem{Arkeryd.Nouri2000}
{\sc L.~Arkeryd and A.~Nouri}, {\em {$L^1$} solutions to the stationary {Boltzmann} equation in a slab}, Ann. Fac. Sci. Toulouse Math., 9 (2000), pp.~375--413.

\bibitem{Arkeryd.Nouri2007}
\leavevmode\vrule height 2pt depth -1.6pt width 23pt, {\em Asymptotic techniques for kinetic problems of {Boltzmann} type}, Riv. Mat. Univ. Parma, 6 (2007), pp.~1--74.

\bibitem{Arsenio2012}
{\sc D.~Ars\'enio}, {\em From {Boltzmann}'s equation to the incompressible {Navier-Stokes-Fourier} system with long-range interactions}, Arch. Ration. Mech. Anal., 206 (2012), pp.~367--488.

\bibitem{Asano.Ukai1983}
{\sc K.~Asano and S.~Ukai}, {\em The {Euler} limit and the initial layer of the nonlinear {Boltzmann} equation}, Hokkaido Math. J., 12 (1983), pp.~303--324.

\bibitem{Bardos.Caflisch.Nicolaenko1986}
{\sc C.~Bardos, R.~E. Caflisch, and B.~Nicolaenko}, {\em The {Milne} and {Kramers} problems for the {Boltzmann} equation of a hard sphere gas}, Comm. Pure Appl. Math., 39 (1986), pp.~323--352.

\bibitem{Bardos.Golse.Levermore1991}
{\sc C.~Bardos, F.~Golse, and D.~Levermore}, {\em Fluid dynamical limits of kinetic equations {I}: formal derivations}, J. Statist. Phys., 63 (1991), pp.~323--344.

\bibitem{Bardos.Golse.Levermore1993}
\leavevmode\vrule height 2pt depth -1.6pt width 23pt, {\em Fluid dynamical limits of kinetic equations {II}: convergence proofs for the {Boltzmann} equation}, Comm. Pure Appl. Math., 46 (1993), pp.~667--753.

\bibitem{Bardos.Golse.Levermore1998}
\leavevmode\vrule height 2pt depth -1.6pt width 23pt, {\em Acoustic and {Stokes} limits for the {Boltzmann} equation}, C. R Acad. Sci. Paris, Serie 1 Math, 327 (1998), pp.~323--328.

\bibitem{Bardos.Golse.Levermore2000}
\leavevmode\vrule height 2pt depth -1.6pt width 23pt, {\em The acoustic limit for the {Boltzmann} equation}, Arch. Rational Mech. Anal., 153 (2000), pp.~177--204.

\bibitem{Bardos.Golse.Paillard2012}
{\sc C.~Bardos, F.~Golse, and L.~Paillard}, {\em The incompressible {Euler} limit of the {Boltzmann} equation with accommodation boundary condition}, Commun. Math. Sci., 10 (2012), pp.~159--190.

\bibitem{Bardos.Levermore.Ukai.Yang2008}
{\sc C.~Bardos, C.~D. Levermore, S.~Ukai, and T.~Yang}, {\em Kinetic equations: fluid dynamical limits and viscous heating}, Bull. Inst. Math. Acad. Sin. (N.S.), 3 (2008), pp.~1--49.

\bibitem{Bardos.Ukai1991}
{\sc C.~Bardos and S.~Ukai}, {\em The classical incompressible {Navier-Stokes} limit of the {Boltzmann} equation}, Math. Models Methods Appl. Sci., 1 (1991), pp.~235--257.

\bibitem{Bobylev1995}
{\sc A.~V. Bobylev}, {\em Quasistationary hydrodynamics for the {Boltzmann} equation}, J. Statist. Phys., 80 (1995), pp.~1063--1083.

\bibitem{Bogovskii1979}
{\sc M.~Bogovskii}, {\em Solution of the first boundary value problem for an equation of continuity of an incompressible medium}, Dokl. Akad. Nauk SSSR, 248 (1979), pp.~1037--1040.

\bibitem{Boyer.Fabrie2013}
{\sc F.~Boyer and P.~Fabrie}, {\em Mathematical tools for the study of the incompressible {Navier-Stokes} equations and related models}, Springer, New York, 2013.

\bibitem{Briant.Guo2016}
{\sc M.~Briant and Y.~Guo}, {\em Asymptotic stability of the {Boltzmann} equation with {Maxwell} boundary conditions}, J. Differential Equations, 261 (2016), pp.~7000--7079.

\bibitem{Briant.Merino-Aceituno.Mouhot2019}
{\sc M.~Briant, S.~Merino-Aceituno, and C.~Mouhot}, {\em From {Boltzmann} to incompressible {Navier-Stokes} in {Sobolev} spaces with polynomial weight}, Anal. Appl., 17 (2019), pp.~85--116.

\bibitem{Brull2008(=)}
{\sc S.~Brull}, {\em Problem of evaporation-condensation for a two component gas in the slab}, Kinet. Relat. Models, 1 (2008), pp.~185--221.

\bibitem{Brull2008}
\leavevmode\vrule height 2pt depth -1.6pt width 23pt, {\em The stationary {Boltzmann} equation for a two-component gas in the slab}, Math. Methods Appl. Sci., 31 (2008), pp.~153--178.

\bibitem{Caflisch1980}
{\sc R.~E. Caflisch}, {\em The fluid dynamic limit of the nonlinear {Boltzmann} equation}, Comm. Pure Appl. Math., 33 (1980), pp.~651--666.

\bibitem{Cao.Jang.Kim2021}
{\sc Y.~Cao, J.~Jang, and C.~Kim}, {\em Passage from the {Boltzmann} equation with diffuse boundary to the incompressible {Euler} equation with heat convection}, J. Differential Equations, 366 (2023), pp.~565--644.

\bibitem{Carlen.Carvalho1994}
{\sc E.~A. Carlen and M.~C. Carvalho}, {\em Entropy production estimates for {Boltzmann} equations with physically realistic collision kernels}, J. Statist. Phys., 74 (1994), pp.~743--782.

\bibitem{Cattabriga1961}
{\sc L.~Cattabriga}, {\em Su un problema al contorno relativo al sistema di equazioni di {Stokes}}, Rend. Sem. Mat. Univ. Padova, 31 (1961), pp.~308--340.

\bibitem{Cercignani.Illner.Pulvirenti1994}
{\sc C.~Cercignani, R.~Illner, and M.~Pulvirenti}, {\em The mathematical theory of dilute gases}, Springer-Verlag, New York, 1994.

\bibitem{Chaturvedi.Luk.Nguyen2022}
{\sc S.~Chaturvedi, J.~Luk, and T.~T. Nguyen}, {\em The {Vlasov-Poisson-Landau} system in the weakly collisional regime}, To appear in J. Amer. Math. Soc.,  (2022).

\bibitem{Chen.Chen.Liu.Sone2007}
{\sc C.-C. Chen, I.-K. Chen, T.-P. Liu, and Y.~Sone}, {\em Thermal transpiration for the linearized {Boltzmann} equation}, Comm. Pure Appl. Math., 60 (2007), pp.~147--163.

\bibitem{Masi.Esposito.Lebowitz1989}
{\sc A.~De~Masi, R.~Esposito, and J.~L. Lebowitz}, {\em Incompressible {Navier-Stokes} and {Euler} limits of the {Boltzmann} equation}, Comm. Pure and Appl. Math., 42 (1989), pp.~1189--1214.

\bibitem{Desvillettes.Villani2001}
{\sc L.~Desvillettes and C.~Villani}, {\em On the trend to global equilibrium in spatially inhomogeneous entropy-dissipating systems: the linear {Fokker-Planck} equation}, Comm. Pure Appl. Math., 54 (2001), pp.~1--42.

\bibitem{Desvillettes.Villani2005}
\leavevmode\vrule height 2pt depth -1.6pt width 23pt, {\em On the trend to global equilibrium for spatially inhomogeneous kinetic systems: the {Boltzmann} equation}, Invent. Math., 159 (2005), pp.~245--316.

\bibitem{Duan.Liu.Yang.Zhang2022}
{\sc R.~Duan, S.~Liu, T.~Yang, and Z.~Zhang}, {\em Heat transfer problem for the {Boltzmann} equation in a channel with diffusive boundary condition}, Chinese Ann. Math. Ser. B, 43 (2022), pp.~1071--1100.

\bibitem{Esposito.Guo.Kim.Marra2013}
{\sc R.~Esposito, Y.~Guo, C.~Kim, and R.~Marra}, {\em Non-isothermal boundary in the {Boltzmann} theory and {Fourier} law}, Comm. Math. Phys., 323 (2013), pp.~177--239.

\bibitem{Esposito.Guo.Kim.Marra2015}
\leavevmode\vrule height 2pt depth -1.6pt width 23pt, {\em Stationary solutions to the {Boltzmann} equation in the hydrodynamic limit}, Ann. PDE, 4 (2018), pp.~1--119.

\bibitem{Esposito.Guo.Marra2018}
{\sc R.~Esposito, Y.~Guo, and R.~Marra}, {\em Hydrodynamic limit of a kinetic gas flow past an obstacle}, Comm. Math. Phys., 364 (2018), pp.~765--823.

\bibitem{AA029}
{\sc R.~Esposito, Y.~Guo, R.~Marra, and L.~Wu}, {\em Dynamical stability of ghost effect}, In Preparation,  (2023).

\bibitem{AA024}
\leavevmode\vrule height 2pt depth -1.6pt width 23pt, {\em Ghost effect from {Boltzmann} theory: Expansion with remainder}, arXiv:2301.09560,  (2023).

\bibitem{Esposito.Lebowitz.Marra1994}
{\sc R.~Esposito, J.~L. Lebowitz, and R.~Marra}, {\em Hydrodynamic limit of the stationary {Boltzmann} equation in a slab}, Comm. Math. Phys., 160 (1994), pp.~49--80.

\bibitem{Esposito.Lebowitz.Marra1995}
\leavevmode\vrule height 2pt depth -1.6pt width 23pt, {\em The {Navier-Stokes} limit of stationary solutions of the nonlinear {Boltzmann} equation}, J. Statist. Phys., 78 (1995), pp.~389--412.

\bibitem{Esposito.Marra2020}
{\sc R.~Esposito and R.~Marra}, {\em Stationary non equilibrium states in kinetic theory}, J. Statist. Phys., 180 (2020), pp.~773--809.

\bibitem{Esposito.Marra2023}
\leavevmode\vrule height 2pt depth -1.6pt width 23pt, {\em On the derivation of new non-classical hydrodynamic equations for {Hamiltonian} particle systems}, arXiv:2305.06304,  (2023).

\bibitem{Gallagher.Tristani2020}
{\sc I.~Gallagher and I.~Tristani}, {\em On the convergence of smooth solutions from {Boltzmann} to {Navier-Stokes}}, Ann. H. Lebesgue, 3 (2020), pp.~561--614.

\bibitem{Glassey1996}
{\sc R.~T. Glassey}, {\em The {Cauchy} problem in kinetic theory}, Society for Industrial and Applied Mathematics (SIAM), Philadelphia, PA, 1996.

\bibitem{Golse2013}
{\sc F.~Golse}, {\em From the {Boltzmann} equation to the {Euler} equations in the presence of boundaries}, Comput. Math. Appl., 65 (2013), pp.~815--830.

\bibitem{Golse.Saint-Raymond2004}
{\sc F.~Golse and L.~Saint-Raymond}, {\em The {Navier-Stokes} limit of the {Boltzmann} equation for bounded collision kernels}, Invent. Math, 155 (2004), pp.~81--161.

\bibitem{Golse.Saint-Raymond2009}
\leavevmode\vrule height 2pt depth -1.6pt width 23pt, {\em The incompressible {Navier-Stokes} limit of the {Boltzmann} equation for hard cutoff potentials}, J. Math. Pures Appl., 91 (2009), pp.~508--552.

\bibitem{Gressman.Strain2011}
{\sc P.~T. Gressman and R.~M. Strain}, {\em Global classical solutions of the {Boltzmann} equation without angular cut-off}, J. Amer. Math. Soc., 24 (2011), pp.~771--847.

\bibitem{Guo2003}
{\sc Y.~Guo}, {\em The {Vlasov-Maxwell-Boltzmann} system near {Maxwellians}}, Invent. Math., 153 (2003), pp.~593--630.

\bibitem{Guo2006}
\leavevmode\vrule height 2pt depth -1.6pt width 23pt, {\em {Boltzmann} diffusive limit beyond the {Navier-Stokes} approximation}, Comm. Pure Appl. Math., 59 (2006), pp.~626--68.

\bibitem{Guo2010}
\leavevmode\vrule height 2pt depth -1.6pt width 23pt, {\em Decay and continuity of the {Boltzmann} equation in bounded domains}, Arch. Ration. Mech. Anal., 197 (2010), pp.~713--809.

\bibitem{Guo2012}
\leavevmode\vrule height 2pt depth -1.6pt width 23pt, {\em The {Vlasov-Poisson-Landau} system in a periodic box}, J. Amer. Math. Soc., 25 (2012), pp.~759--812.

\bibitem{Guo.Huang.Wang2021}
{\sc Y.~Guo, F.~Huang, and Y.~Wang}, {\em {Hilbert} expansion of the {Boltzmann} equation with specular boundary condition in half-space}, Arch. Ration. Mech. Anal., 241 (2021), pp.~231--309.

\bibitem{Guo.Jang2010}
{\sc Y.~Guo and J.~Jang}, {\em Global {Hilbert} expansion for the {Vlasov-Poisson-Boltzmann} system}, Comm. Math. Phys., 299 (2010), pp.~469--501.

\bibitem{Guo.Jang.Jiang2009}
{\sc Y.~Guo, J.~Jang, and N.~Jiang}, {\em Local {Hilbert} expansion for the {Boltzmann} equation}, Kinet. Relat. Models, 2 (2009), pp.~205--214.

\bibitem{Guo.Jang.Jiang2010}
\leavevmode\vrule height 2pt depth -1.6pt width 23pt, {\em Acoustic limit for the {Boltzmann} equation in optimal scaling}, Comm. Pure Appl. Math., 63 (2010), pp.~337--361.

\bibitem{Guo.Kim.Tonon.Trescases2016}
{\sc Y.~Guo, C.~Kim, D.~Tonon, and A.~Trescases}, {\em {BV}-regularity of the boltzmann equation in non-convex domains}, Arch. Ration. Mech. Anal., 220 (2016), pp.~1045--1093.

\bibitem{Guo.Kim.Tonon.Trescases2013}
\leavevmode\vrule height 2pt depth -1.6pt width 23pt, {\em Regularity of the {Boltzmann} equation in convex domain}, Invent. Math., 207 (2016), pp.~115--290.

\bibitem{AA007}
{\sc Y.~Guo and L.~Wu}, {\em Geometric correction in diffusive limit of neutron transport equation in {2D} convex domains}, Arch. Rational Mech. Anal., 226 (2017), pp.~321--403.

\bibitem{AA009}
\leavevmode\vrule height 2pt depth -1.6pt width 23pt, {\em Regularity of {Milne} problem with geometric correction in {3D}}, Math. Models Methods Appl. Sci., 27 (2017), pp.~453--524.

\bibitem{AA020}
\leavevmode\vrule height 2pt depth -1.6pt width 23pt, {\em ${L}^2$ diffusive expansion for neutron transport equation}, arXiv:2301.11996,  (2023).

\bibitem{Hardy1920}
{\sc G.~H. Hardy}, {\em Note on a theorem of {Hilbert}}, Mathematische Zeitschrift, 6 (1920), pp.~314--317.

\bibitem{Hilbert1900}
{\sc D.~Hilbert}, {\em Mathematical problems}, Gottinges Nachrichten,  (1900), pp.~253--297.

\bibitem{Hilbert1916}
\leavevmode\vrule height 2pt depth -1.6pt width 23pt, {\em Begrundung der kinetischen gastheorie}, Math. Ann., 72 (1912), pp.~331--407.

\bibitem{Hilbert1953}
\leavevmode\vrule height 2pt depth -1.6pt width 23pt, {\em Grundzugeiner allgemeinen Theorie der linearen Integralgleichungen}, Chelsea, New York, 1953.

\bibitem{Huang.Wang.Wang.Yang2016}
{\sc F.~Huang, Y.~Wang, Y.~Wang, and T.~Yang}, {\em Justification of limit for the {Boltzmann} equation related to {Korteweg} theory}, Quart. Appl. Math., 74 (2016), pp.~719--764.

\bibitem{Huang.Wang.Yang2010(=)}
{\sc F.~Huang, Y.~Wang, and T.~Yang}, {\em Fluid dynamic limit to the {Riemann} solutions of {Euler} equations: {I.} superposition of rarefaction waves and contact discontinuity}, Kinet. Relat. Models, 3 (2010), pp.~685--728.

\bibitem{Huang.Wang.Yang2010}
\leavevmode\vrule height 2pt depth -1.6pt width 23pt, {\em Hydrodynamic limit of the {Boltzmann} equation with contact discontinuities}, Comm. Math. Phys., 295 (2010), pp.~293--326.

\bibitem{Huang.Wang.Yang2013}
\leavevmode\vrule height 2pt depth -1.6pt width 23pt, {\em The limit of the {Boltzmann} equation to the {Euler} equations for {Riemann} problems}, SIAM J. Math. Anal., 45 (2013), pp.~1741--1811.

\bibitem{Jang2009}
{\sc J.~Jang}, {\em {Vlasov-Maxwell-Boltzmann} diffusive limit}, Arch. Ration. Mech. Anal., 194 (2009), pp.~531--584.

\bibitem{Jang.Kim2021}
{\sc J.~Jang and C.~Kim}, {\em Incompressible {Euler} limit from {Boltzmann} equation with diffuse boundary condition for analytic data}, Ann. PDE, 7 (2021), p.~103.

\bibitem{Jiang.Levermore.Masmoudi2010}
{\sc N.~Jiang, C.~D. Levermore, and N.~Masmoudi}, {\em Remarks on the acoustic limit for the {Boltzmann} equation}, Comm. Partial Differential Equations, 35 (2010), pp.~1590--1609.

\bibitem{Jiang.Masmoudi2016}
{\sc N.~Jiang and N.~Masmoudi}, {\em Boundary layers and incompressible {Navier-Stokes-Fourier} limit of the {Boltzmann} equation in bounded domain {I}}, Comm. Pure Appl. Math., 70 (2016), pp.~90--171.

\bibitem{Jiang.Masmoudi2022}
\leavevmode\vrule height 2pt depth -1.6pt width 23pt, {\em Low {Mach} number limits and acoustic waves}, Springer, Cham, 2018.

\bibitem{Kim2011}
{\sc C.~Kim}, {\em Formation and propagation of discontinuity for {Boltzmann} equation in non-convex domains}, Comm. Math. Phys., 308 (2011), pp.~641--701.

\bibitem{Kim2014}
\leavevmode\vrule height 2pt depth -1.6pt width 23pt, {\em {Boltzmann} equation with a large potential in a periodic box}, Comm. Partial Differential Equations, 39 (2014), pp.~1393--1423.

\bibitem{Kim.La2022}
{\sc C.~Kim and J.~La}, {\em Vorticity convergence from {Boltzmann} to {2D} incompressible {Euler} equations below {Yudovich} class}, arXiv:2206.00543,  (2022).

\bibitem{Kogan1958}
{\sc M.~Kogan}, {\em On the equations of motion of a rarefied gas}, Appl. Math. Mech., 22 (1958), pp.~597--607.

\bibitem{Kogan.Galkin.Fridlender1976}
{\sc M.~Kogan, V.~Galkin, and O.~Fridlender}, {\em Stresses produced in gases by temperature and concentration inhomogeneities. new types of free convection}, Soy. Phys. Usp., 19 (1976), pp.~420--428.

\bibitem{Lachowicz1987}
{\sc M.~Lachowicz}, {\em On the initial layer and the existence theorem for the nonlinear {Boltzmann} equation}, Math. Methods Appl. Sci., 9 (1987), pp.~342--366.

\bibitem{Lions.Masmoudi2001}
{\sc P.-L. Lions and N.~Masmoudi}, {\em From the {Boltzmann} equations to the equations of incompressible fluid mechanics. {I, II.}}, Arch. Ration. Mech. Anal., 158 (2001), pp.~173--193,195--211.

\bibitem{Liu.Yu2013}
{\sc T.-P. Liu and S.-H. Yu}, {\em Invariant manifolds for steady {Boltzmann} flows and applications}, Arch. Ration. Mech. Anal., 209 (2013), pp.~869--997.

\bibitem{Masmoudi2011}
{\sc N.~Masmoudi}, {\em About the {Hardy} inequality}, in Schleicher, D., Lackmann, M. (eds) An Invitation to Mathematics, Berlin, Heidelberg, 2011, Springer.

\bibitem{Masmoudi.Saint-Raymond2003}
{\sc N.~Masmoudi and L.~Saint-Raymond}, {\em From the {Boltzmann} equation to the {Stokes-Fourier} system in a bounded domain}, Comm. Pure and Appl. Math., 56 (2003), pp.~1263--1293.

\bibitem{Maxwell1879}
{\sc J.~C. Maxwell}, {\em On stresses in rarified gases arising from inequalities of temperature}, Philosophical Transactions of the Royal Society of London, 170 (1879), pp.~231--256.

\bibitem{Mouhot.Villani2015}
{\sc C.~Mouhot and C.~Villani}, {\em Kinetic theory}, in Princeton Companion to Applied Mathematics, Princeton University Press, 2015.

\bibitem{Nishida1978}
{\sc T.~Nishida}, {\em Fluid dynamical limit of the nonlinear {Boltzmann} equation to the level of the compressible {Euler} equation}, Comm. Math. Phys., 61 (1978), pp.~119--148.

\bibitem{Saint-Raymond2003}
{\sc L.~Saint-Raymond}, {\em Convergence of solutions to the {Boltzmann} equation in the incompressible euler limit}, Arch. Ration. Mech. Anal., 166 (2003), pp.~47--80.

\bibitem{Saint-Raymond2008}
\leavevmode\vrule height 2pt depth -1.6pt width 23pt, {\em From {Boltzmann}'s kinetic theory to {Euler}'s equations}, Phys. D, 237 (2008), pp.~2028--2036.

\bibitem{Saint-Raymond2009}
\leavevmode\vrule height 2pt depth -1.6pt width 23pt, {\em Hydrodynamic limits of the {Boltzmann} equation}, Springer-Verlag, Berlin, 2009.

\bibitem{Sone1972}
{\sc Y.~Sone}, {\em Flow induced by thermal stress in rarefied gas}, Phys. Fluids, 15 (1972), pp.~1418--1423.

\bibitem{Sone2002}
\leavevmode\vrule height 2pt depth -1.6pt width 23pt, {\em Kinetic theory and fluid dynamics.}, Birkhauser Boston, Inc., Boston, MA, 2002.

\bibitem{Sone2007}
\leavevmode\vrule height 2pt depth -1.6pt width 23pt, {\em Molecular gas dynamics. Theory, techniques, and applications.}, Birkhauser Boston, Inc., Boston, MA, 2007.

\bibitem{Sone.Aoki.Doi1992}
{\sc Y.~Sone, K.~Aoki, and T.~Doi}, {\em Kinetic theory analysis of gas flows condensing on a plane condensed phase: Case of a mixture of a vapor and noncondensable gas.}, Transp. Theory Stat. Phys., 21 (1992), pp.~297--328.

\bibitem{Sone.Aoki.Takata.Sugimoto.Bobylev1996}
{\sc Y.~Sone, K.~Aoki, S.~Takata, H.~Sugimoto, and A.~V. Bobylev}, {\em Inappropriateness of the heat-conduction equation for description of a temperature field of a stationary gas in the continuum limit: examination by asymptotic analysis and numerical computation of the {Boltzmann} equation}, Phys. Fluids, 8 (1996), pp.~628--638.

\bibitem{Speck.Strain2011}
{\sc J.~Speck and R.~M. Strain}, {\em {Hilbert} expansion from the {Boltzmann} equation to relativistic fluids}, Comm. Math. Phys., 304 (2011), pp.~229--280.

\bibitem{Villani2002}
{\sc C.~Villani}, {\em A review of mathematical topics in collisional kinetic theory}, Handbook of Mathematical Fluid Dynamics, Vol.{I} (2002), pp.~71--305.

\bibitem{AA004}
{\sc L.~Wu}, {\em Hydrodynamic limit with geometric correction of stationary {Boltzmann} equation}, J. Differential Equations, 260 (2016), pp.~7152--7249.

\bibitem{AA013}
\leavevmode\vrule height 2pt depth -1.6pt width 23pt, {\em Boundary layer of {Boltzmann} equation in {2D} convex domains}, Analysis\&PDE, 14 (2021), pp.~1363--1428.

\bibitem{AA003}
{\sc L.~Wu and Y.~Guo}, {\em Geometric correction for diffusive expansion of steady neutron transport equation}, Comm. Math. Phys., 336 (2015), pp.~1473--1553.

\bibitem{BB002}
{\sc L.~Wu and Z.~Ouyang}, {\em Asymptotic analysis of {Boltzmann} equation in bounded domains}, arXiv:2008.10507,  (2020).

\bibitem{Yu2005}
{\sc S.-H. Yu}, {\em Hydrodynamic limits with shock waves of the {Boltzmann} equation}, Comm. Pure Appl. Math., 58 (2005), pp.~409--443.

\end{thebibliography}

\end{document}